\documentclass[a4paper,11pt, english, reqno]{smfbook}

\usepackage{mathrsfs}  
\usepackage[english]{babel}
\usepackage{amsmath}
\usepackage{amssymb}
\usepackage{amsfonts}
\usepackage{amsthm}
\usepackage[all]{xy}

\newcommand{\Q}{\mathbf{Q}}
\newcommand{\C}{\mathbf{C}}
\newcommand{\N}{\mathbf{N}}
\newcommand{\Z}{\mathbf{Z}}
\newcommand{\D}{\mathscr{D}}
\newcommand{\DB}{\mathscr{D}}
\renewcommand{\O}{\mathscr{O}}
\newcommand{\R}{\mathbf{R}}
\newcommand{\A}{\mathscr{A}}

\newcommand{\E}{\mathscr{E}}

\newcommand{\Rhyp}{\mathbb{R}}
\newcommand{\Hyp}{\mathbb{H}}

\renewcommand{\P}{\mathbb{P}}

\newcommand{\LieU}{\mathfrak{u}}
\newcommand{\gl}{\mathfrak{gl}}
\newcommand{\Sy}{\mathfrak{S}}

\renewcommand{\epsilon}{\varepsilon}

\newcommand{\Hom}{\mathrm{Hom}}
\newcommand{\GL}{\mathrm{GL}}

\newcommand{\Mat}{\mathrm{Mat}}

\renewcommand{\Im}{\mathrm{Im}}
\newcommand{\im}{\mathrm{im}}
\newcommand{\Tr}{\mathrm{Tr}}

\newcommand{\sgn}{\mathrm{sgn}}

\newcommand{\id}{\mathrm{id}}
\newcommand{\rel}{\mathrm{rel}}
\newcommand{\Spec}{\mathrm{Spec}}
\newcommand{\Sp}{\mathrm{Sp}}
\newcommand{\Spwf}{\mathrm{Spwf}}
\newcommand{\Tot}{\mathrm{Tot}}
\newcommand{\Fil}{\mathrm{Fil}}
\newcommand{\Ch}{\mathrm{Ch}}
\newcommand{\univ}{\mathrm{univ}}
\newcommand{\Cone}{\mathrm{Cone}}
\newcommand{\Sing}{\mathscr{C}}
\renewcommand{\top}{\mathrm{top}}
\newcommand{\alg}{\mathrm{alg}}

\newcommand{\rig}{\mathrm{rig}}
\newcommand{\syn}{\mathrm{syn}}
\newcommand{\dR}{\mathrm{dR}}
\newcommand{\spec}{\mathrm{sp}}
\newcommand{\colim}{\mathrm{colim}}
\newcommand{\coker}{\mathrm{coker}}

\renewcommand{\smash}{\wedge}
\renewcommand{\subset}{\subseteq}

\newcommand{\trivcof}{\overset{\sim}{\hookrightarrow}}

\newcommand{\isom}{\simeq}

\newcommand{\del}{\partial}
\newcommand{\ol}[1]{\overline{#1}}
\newcommand{\ul}[1]{\underline{#1}}

\newcommand{\<}{\langle}
\renewcommand{\>}{\rangle}
\renewcommand{\projlim}{\underleftarrow{\lim}}
\newcommand{\dirlim}{\underrightarrow{\lim}}
\renewcommand{\dot}{\bullet}

\renewcommand{\labelenumi}{(\roman{enumi})}

\newtheorem{thm}{Theorem}[chapter]
\newtheorem*{thm*}{Theorem}
\newtheorem{cor}[thm]{Corollary}
\newtheorem{lemma}[thm]{Lemma}
\newtheorem{propn}[thm]{Proposition}
\newtheorem{prop}[thm]{Proposition}

\theoremstyle{definition}
\newtheorem{dfn}[thm]{Definition}
\newtheorem*{dfn*}{Definition}

\theoremstyle{remark}
\newtheorem{rem}[thm]{Remark}
\newtheorem{rems}[thm]{Remarks}
\newtheorem*{ex*}{Example}
\newtheorem{ex}[thm]{Example}
\newtheorem*{conv}{Convention}

\numberwithin{equation}{chapter}

\date{\today}
\begin{document}
\thispagestyle{empty}
\setcounter{page}{-1}

\begin{center}

\vspace*{1cm}

{\Large The Relative Chern Character and Regulators}

\vspace{4cm}

\textsc{Dissertation zur Erlangung des Doktorgrades}

\textsc{der Naturwissenschaften (Dr. rer. nat.)}

\textsc{an der Naturwissenschaftlichen Fakult\"at I -- Mathematik der Universit\"at Regensburg}

\vspace{4cm}

vorgelegt von

Georg Tamme

aus Sinzing

2010

\end{center}

\newpage

\thispagestyle{empty}

\vspace*{1cm}
\vfill

Promotionsgesuch eingereicht am: 4. Februar 2010

\medskip

Die Arbeit wurde angeleitet von: Prof. Dr. Guido Kings

\medskip

Pr\"ufungsausschuss:

Prof. Dr. Helmut Abels (Vorsitzender)

Prof. Dr. Guido Kings (1. Gutachter)

Prof. Amnon Besser, Ben Gurion University (2. Gutachter)

Prof. Dr. Klaus K\"unnemann

Prof. Dr. Uwe Jannsen (Ersatzpr\"ufer)

\cleardoublepage

\cleardoublepage

\tableofcontents

\chapter*{Introduction}

The starting point for the study of regulators is Dirichlet's regulator for a number field $F$. If $r_1$ (resp. $2r_2$) is the number of real (resp. complex) embeddings of $F$, one has the regulator map $r: \O_F^{\times} \to H \subset \R^{r_1+r_2}$ from the group of units in the ring of integers $\O_F$ of $F$ to a hyperplane in $\R^{r_1+r_2}$. Its kernel is finite and its image is a lattice, whose covolume is Dirichlet's regulator $R_F$. In the late $19^{\text{th}}$ century, Dedekind related this regulator to the residue at $s=1$ of the zeta function $\zeta_F(s)$ of the number field. Using the meromorphic continuation and the functional equation of $\zeta_F$ proved by Hecke one can formulate this relation in the class number formula
\[
\lim_{s\to 0} \zeta_F(s)s^{-(r_1+r_2-1)} = -\frac{hR_F}{w},
\]
where $h$ is the class number of $F$, $w$ is the number of roots of unity and the left hand side is the leading coefficient of the Taylor expansion of $\zeta_F$ at $s=0$.

In the 1970's Quillen introduced higher algebraic $K$-groups $K_i(\O_F)$, $i \geq 0$, generalizing $K_1(\O_F) = \O_F^{\times}$ and showed, that they are finitely generated. Borel constructed higher regulators $r_n:K_{2n-1}(\O_F) \to \R^{r_2}$ (resp. $\R^{r_1+r_2}$), if $n\geq 2$ is even (resp. odd). He was able to prove, that the kernel of $r_n$ is finite and its image is a lattice, whose covolume is a rational multiple of the leading coefficient of the Taylor expansion of $\zeta_F$ at the point $1-n$.

In the following, the construction of regulators was extended to the case of $K_2$ of a curve by Bloch, and then to all smooth projective varieties over $\Q$ by Beilinson. In this context the regulator maps for the variety $X$,
\[
K_i(X) \to H^{2n-i}_{\D}(X_{\R}, \R(n)),
\]
have values in the Deligne-Beilinson cohomology of $X$ and are obtained by composing the natural map $K_i(X) \to K_i(X_{\C})$ with the Chern character map $\Ch^{\D}_{n,i}\colon K_i(X_{\C}) \to H^{2n-i}_{\D}(X_{\C}, \R(n))$.\footnote{There is a natural action of complex conjugation on $H^{2n-i}_{\D}(X_{\C}, \R(n))$ and $K_i(X)$ lands in the ivariant part of this action, which by definition is $H^{2n-i}_{\D}(X_{\R}, \R(n))$.}
Beilinson establishes a whole system of conjectures relating these regulators to the leading coefficients of the Taylor expansions of the $L$-functions of $X$ at the integers \cite{Beilinson}. 

He also sketches a proof of the fact, that in the case of a number field, his regulator maps coincide with Borel's regulator maps. Then Borel's theorem implies Beilinson's conjectures in this case. Many details of this proof were given by Rapoport in \cite{Rapoport}.
With a completely different strategy, based on the comparison of Cheeger-Simons Chern classes with Deligne-Beilinson Chern classes, Dupont, Hain and Zucker \cite{DHZ} tried to compare both regulators and gave good evidence for their conjecture, that Borel's regulator is in fact twice Beilinson's regulator. Later on Burgos \cite{Burgos} worked out Beilinson's original argument and proved, that the factor is indeed $2$.
\smallskip

Nowadays there exists also a $p$-adic analogue of the above conjectures. Thanks to Perrin-Riou \cite{PerrinRiou} one has a conjectural picture about the existence and properties of $p$-adic $L$-functions, so that one can formulate a $p$-adic Beilinson conjecture for smooth projective varieties over a $p$-adic field. There the Deligne-Beilinson cohomology is replaced by (rigid) syntomic cohomology and the regulator maps by the corresponding rigid syntomic Chern character. 

In \cite{HK} Huber and Kings show, that one can also construct a $p$-adic Borel regulator parallel to the classical Borel regulator, and relate it to the syntomic regulator by an analogue of Beilinson's comparison argument.

\smallskip

In a different direction, Karoubi \cite{Kar87} constructed Chern character maps (resp. relative Chern character maps) on the algebraic (resp. relative) $K$-theory of any real, complex or even ultrametric Banach algebra with values in continuous cyclic homology, where relative $K$-theory is the homotopy fibre of the map from algebraic to topological $K$-theory. In the case, that the Banach algebra is just $\C$, Hamida \cite{HamidaBorel} related Karoubi's relative Chern character to the Borel regulator for $\C$\footnote{After a suitable renormalization, the Borel regulator of any number field $F$ factors through $K_{2n-1}(F) \to \prod_{\sigma:F\hookrightarrow\C} K_{2n-1}(\C)$ followed by the Borel regulator for $\C$.}. In the $p$-adic case Karoubi also conjectured a relation with $p$-adic polylogarithms for $p$-adic fields.

\smallskip

This is the starting point of this thesis. As Karoubi pointed out, the $p$-adic Borel regulator should be directly connected with his relative Chern character in the case, where the ultrametric Banach algebra is just a finite extension of $\Q_p$. In the preprint \cite{Tamme}, I was able to make this relation precise. Later on I realized, that there should be a comparison result for a suitably generalized ``geometric'' version of Karoubi's relative Chern character for smooth quasiprojective varieties over the ring of integers in a finite extension of $\Q_p$ on the one hand and the rigid syntomic Chern character on the other hand, and that one should get the comparison result of Huber and Kings as a corollary of this. In fact, Besser formulated such a conjecture in 2003 \cite{BesserTalk}. In the following, I developed a strategy to prove this conjectural relation, but did not succeed due to technical problems with rigid syntomic cohomology.

Nevertheless, this strategy works in the analogue complex situation to give a proof of the following theorem:
\begin{thm*}
Let $X$ be a smooth variety of finite type over $\C$. For any $i >0$ the diagram
\[\xymatrix{
K_i^{\rel}(X) \ar[r]\ar[d]^{(-1)^{n-1}\Ch_{n,i}^{\rel}} & K_i(X) \ar[d]^{\Ch_{n,i}^{\DB}}\\
H^{2n-i-1}(X, \C)/\Fil^nH^{2n-i-1}(X, \C) \ar[r] & H^{2n-i}_{\DB}(X, \Q(n))
}
\]
commutes. 
\end{thm*}
The interest in this result relies on the fact, that the relative Chern character is quite explicit in nature, and, that for projective $X$ the map from relative to algebraic $K$-theory is rationally surjective. Combined with the comparison of the relative Chern character with Borel's regulator, this gives a new proof of Burgos' theorem, that Borel's regulator is twice Beilinson's regulator. 

These results are contained in part I of this thesis.
In part II we give a construction of the relative Chern character for smooth affine varieties over the ring of integers $R$ in a finite extension of $\Q_p$, and prove, that, when the variety is $\Spec(R)$ itself, this essentially gives the $p$-adic Borel regulator.

\smallskip

Let us now describe the contents of the different chapters in more detail.

\smallskip

Karoubi's construction of the relative Chern character for a Banach (or Fr\'echet) algebra $A$ relies on a Chern-Weil theory for $\GL(A)$-bundles on simplicial sets using de Rham--Sullivan differential forms. In the first chapter we adapt this formalism to the geometric case of simplicial complex manifolds (if $A$ is the algebra of functions on a manifold $X$, Karoubi's bundles on the simplicial set $S$ correspond in our geometric setting to bundles on the simplicial manifold $X\otimes S$).
This is similar to the simplicial Chern-Weil theory developped by Dupont (\cite{Dup76}, \cite{DupLNM}) except for the consequent use of what we call \emph{topological morphisms} of simplicial manifolds (compatible families of morphisms defined on $\Delta^p \times X_p$ for a simplicial manifold $X_{\dot}$) and \emph{topological bundles}. 
The use of topological morphisms and bundles is motivated by the fact, that the relative $K$-theory of an affine scheme may be described in terms of (algebraic, hence) holomorphic bundles on certain simplicial varieties together with a trivialization of the underlying \emph{topological} bundle. The relative Chern character will then be given by certain secondary characteristic classes for such bundles. 

When one now wants to compare regulators on $K$-theory, one has by construction of these regulator maps to compare characteristic classes of certain bundles on simplicial varieties (or manifolds). This is often easy, when these classes exist and are functorial for \emph{all} (algebraic) bundles, since then it suffices to consider the universal case $B_{\dot}\GL$ and there the comparison result in question follows from the simple structure of the cohomology of $B_{\dot}\GL$. In our case one immediately arrives at the problem, that, whereas the Deligne-Beilinson Chern character classes are defined for every algebraic bundle, the relative Chern character classes are \emph{not}. Note, that it is exactly this kind of problem, that also arises in \cite{DHZ}.

The solution to this problem in our case is contained in the second chapter. It also yields a refinement of the secondary classes constructed in chapter 1 for \emph{algebraic} bundles, which are topologically trivialized, taking the Hodge filtration into account. The basic idea is to construct another kind of characteristic classes, which exist for all (algebraic) bundles, and from which, in the case of a topologically trivialized bundle, one can get the secondary classes constructed before by some simple procedure. These are the so called \emph{refined} Chern character classes, which live in a cohomology group, that depends on the bundle. In some sense, they have a \emph{primary} component, which is the de Rham Chern character class, and a \emph{secondary} component, which comes from the canonical trivialization of the pullback of a principal bundle to itself. Since these classes are obtained from the universal case simply by functoriality, it is clear, that they are well behaved with respect to the Hodge filtration. These classes then give the secondary classes in the topologically trivialized case simply by pulling back with a topological section of the corresponding principal bundle, which corresponds to a topological trivialization of the bundle itself. With these refined classes the above strategy then gives the comparison of secondary and Deligne-Beilinson Chern character classes.

In chapter 3, after constructing a good simplicial model for the relative $K$-theory, we construct the relative Chern character and compare it with the Deligne-Beilinson Chern character, first in the smooth affine case, and then for all smooth varieties of finite type using Jouanolou's trick. Since our construction of the relative Chern character differs slightly from Karoubi's one, we reprove the relation between the relative Chern character for $\Spec(\C)$ and Borel's regulator, using the explicit description of van Est's isomorphism due to Dupont. This then gives the comparison of Beilinson's and Borel's regulator for $\Spec(\C)$ (and hence for number fields).

\smallskip

In part II we try to carry the constructions and results from the first part over to the $p$-adic setting. Since rigid analytic spaces are not well suited for de Rham cohomology (and hence for Chern-Weil theory) due to convergence problems caused by integration, we make systematic use of the theory of dagger spaces developped by Grosse-Kl{\"o}nne \cite{GK}. After recalling some basic facts and notations in chapter 4, we show in chapter 5, that the simplicial Chern-Weil theory in the style of Dupont also works for simplicial dagger spaces, replacing the standard simplex $\Delta^p$ by the dagger space $\Sp(K\<x_0, \dots, x_p\>^{\dag}/(\sum x_i-1))$. This also gives a notion of topological morphisms in the $p$-adic setting and we construct secondary classes for topologically trivialized bundles as in the complex case.

In chapter 6 we construct the refined and secondary classes for algebraic bundles. This is a little bit harder than in the complex case, since we dot not have nice functorial complexes computing the different cohomology groups at hand.

The last chapter contains the construction of the relative Chern character in the $p$-adic case. Karoubi and Villamayor \cite{KV} defined topological $K$-theory for ultrametric Banach algebras using rings of convergent power series. Since dagger algebras are not Banach algebras, we first of all show, that one can calculate the topological $K$-theory of the completion of a dagger algebra, which is a Banach algebra, also in terms of the dagger algebra and \emph{overconvergent} power series. Then we can construct relative $K$-theory and the relative Chern character as before. Finally, we compare the relative Chern character in the case of the ring of integers in a finite extension of $\Q_p$ with the $p$-adic Borel regulator using the explicit description of the Lazard isomorphism due to Huber and Kings.

We will give some remarks on the problems encountered when trying to compare the relative Chern character  with the syntomic Chern character in a seperate introduction to part II.

\smallskip

In the appendix, we collect some mostly well-known facts used in the main body of the text, for which we couldn't find a good reference.

\smallskip

I should point out, that this whole work owes much to the ideas of Dupont and Karoubi.

\section*{Acknowledgements}

I would like to thank my advisor Guido Kings. He introduced me to the world of regulators and brought my attention to Karoubi's relative Chern character. He patiently listened to all my problems and questions, and often turned my thoughts to new directions.

It's a pleasure to thank Amnon Besser for some inspiring discussions during the ``Minerva school on $p$-adic Methods in Arithmetic Algebraic Geometry'' in Jerusalem 2009 and his interest for my work.

I am grateful to Annette Huber for pointing out a stupid mistake during presenting her my results. 

Furthermore I would like to thank my colleagues for the nice working atmosphere and especially Volker Neumaier, with whom I could discuss many mathematical problems.

Finally, I want to express my gratitude to my parents and my family. My wife Verena supported and encouraged me constantly and I want to thank her and my children Clara and Henrike for a wonderful non-mathematical life.

\section*{Notations and Conventions}
\subsection*{Homological algebra}
If $A$ is a cochain complex and $k$ an integer, $A[k]$ denotes the complex $A$ shifted $k$ times to the left, i.e. $A[k]^n= A^{n+k}$ with differential $d_{A[k]} = (-1)^k d_A$.

Let $f: A \to B$ be a morphism of cochain complexes. We define the \emph{Cone} of $f$ to be the complex $\Cone(f)$ which in degree $n$ is $A^{n+1} \oplus B^n$ with differential $d(a,b) = (-da,db-f(a))$. There is a short exact sequence of complexes
\[
 0 \to B \to \Cone(f) \to A[1] \to 0,
\]
where the maps are given by $b \mapsto (0,b)$ resp. $(a,b) \mapsto a$.

\subsection*{Simplicial objects}
We denote by $\Delta$ the category of finite ordered sets $[p] = \{0, 1, \dots, p\}$ with morphisms the increasing maps $[p]\to [q]$. A \emph{simplicial} resp. \emph{cosimplicial object} in a category $\mathscr C$ is contra- resp. covariant functor $X: \Delta \to \mathscr C$. We usually denote $X([p])$ by $X_p$ resp. $X^p$. We denote by $\delta^i: [p-1] \to [p]$, $i=0, \dots, p$ the strictly increasing map with $i \not\in \im(\delta^i)$. The induced map $X_p \to X_{p-1}$ of a simplicial object is denoted by $\del_i$ and called the $i$-th face operator. Similarly, $\sigma^i : [p+1]\to [p]$, $i=0, \dots p$ is the increasing surjective map with $\sigma^i(i) = \sigma^i(i+1)$. The induced map $X_p \to X_{p+1}$ is denoted by $s_i$ and called the $i$-th degeneracy map. We denote the corresponding maps on a cosimplicial object by $\delta^i: X^{p-1} \to X^p$ resp $\sigma^i: X^{p+1} \to X^p$.

If $C^{\dot}$ is a cosimplicial object in an abelian category, the associated cochain complex is by definition the complex $\dots \to C^{p-1} \xrightarrow{d} C^{p} \to \dots$ with $d = \sum_{i=0}^p (-1)^i\delta^i$.

\part{The complex theory}

\chapter{Simplicial Chern-Weil theory}

\section{De Rham cohomology of simplicial complex manifolds}\label{sec:SimplDeRham}

This section mainly recalls Dupont's computation of the de Rham cohomology of simplicial manifolds and adapts it to the case of complex manifolds, thereby fixing notations. This is fundamental for the Chern-Weil theory on simplicial manifolds.

For an arbitrary complex manifold $Y$, we denote by $\O_Y$ the sheaf of holomorphic functions, by $\Omega^n_Y$ the sheaf of holomorphic $n$-forms on $Y$ and by $\Omega^n(Y)$ its global sections.

Let $X_{\dot}$ be a simplicial complex manifold.  The sheaves $\Omega^n_{X_p}$, $p\in\N$, together with the pullback maps $\phi_X^*: \Omega^n_{X_p} \to \Omega^n_{X_q}$ for every increasing map $[p] \to [q]$ yield a sheaf\footnote{Cf. \cite[(5.1.6)]{HodgeIII} for the notion of a sheaf on an simplicial topological space.} 
$\Omega^n_{X_{\dot}}$ on the simplicial manifold $X_{\dot}$. With the usual differential we get the complex $\Omega^*_{X_{\dot}}$ of sheaves on $X_{\dot}$.
The \emph{(holomorphic) de Rham cohomology} is defined as the hypercohomology
\[
 \Hyp^*(X_{\dot}, \Omega^*_{X_{\dot}}).
\]

For an arbitrary complex manifold $Y$, we denote by $\A^n_Y$ the sheaf of smooth complex valued $n$-forms on $Y$ and by $\A^n(Y)$ its global sections. 
More precisely, $\A^*_Y$ is the total complex associated with the double complex $(\A^{p,q}_Y, \del, \bar\del)$, where $\A^{p,q}_Y$ is the sheaf of $(p,q)$-forms on $Y$, and for each $p$
\[
 \Omega^p_Y \hookrightarrow \A^{p,0}_Y \xrightarrow{\bar\del} \A^{p,1}_Y \to \dots
\]
is a resolution of $\Omega^p_Y$ by fine sheaves. Thus, if we denote by $\Omega^{\geq r}_Y$ the naive filtration of the holomorphic de Rham complex, then $\Hyp^*(Y, \Omega^{\geq r}_Y)$ may be computed as the cohomology of the complex $\Fil^r\A^*(Y) := \bigoplus_{p+q=*, p \geq r} \A^{p,q}(Y)$.

\smallskip

Similarly in the simplicial case: If $X_{\dot}$ is a simplicial complex manifold, then 
\[
\Hyp^*(X_{\dot}, \Omega^{\geq r}_{X_{\dot}}) = H^*(\Tot\, \Fil^r\A^*(X_{\dot})),
\]
where $\Tot\,\Fil^r \A^*(X_{\dot})$ is the total complex associated with the cosimplicial complex $[p]\mapsto \Fil^r\A^*(X_p)$ (cf. \cite[(5.2.7)]{HodgeIII}).
For the purpose of simplicial Chern-Weil theory we need another version of the simplicial de Rham complex. Let 
\[
\Delta^p := \left\{ (x_0, \dots, x_p) \in \R^{p+1}\;\big|\; x_i \geq 0, \sum\nolimits_{i=0}^p x_i = 1\right\} \subset \R^{p+1}
\]
denote the affine standard simplex. Then $[p]\mapsto \Delta^p$ is a cosimplicial topological space with coface operators $\delta^i : \Delta^{p-1} \to \Delta^p, (x_0, \dots, x_p) \mapsto (x_0, \dots, x_{i-1}, 0, x_i, \dots, x_{p-1})$. A function (or form) on $\Delta^p$ is called smooth, if it extends to a smooth function (form) on a neighbourhood of $\Delta^p$ in $\{\sum x_i = 1\}\subset \R^{p+1}$. We recall from \cite{Dup76}:
\begin{dfn}
A smooth simplicial $n$-form on a simplicial complex manifold $X_{\dot}$ is a family $\omega = (\omega_p)_{p\geq 0}$, where $\omega_p$ is a smooth $n$-form on $\Delta^p\times X_p$, and the compatibility condition
\[
 (\delta^i\times 1)^*\omega_p = (1 \times \del_i)^*\omega_{p-1} \quad\text{ on }\quad \Delta^{p-1} \times X_p
\]
$i = 0, \dots, p$, $p \geq 0$, is satisfied.
The space of smooth simplicial $n$-forms on $X_{\dot}$ is denoted by $A^n(X_{\dot})$.
\end{dfn}
Dupont considers real valued forms, but this makes no significant difference. 

The exterior derivative $d$ and the usual wedge product applied component-wise make $A^*(X_{\dot})$ into a commutative differential graded $\C$-algebra.

Next, $A^*(X_{\dot})$ is naturally the total complex of the double complex $(A^{k,l}(X_{\dot}), d_{\Delta}, d_X)$, where $A^{k,l}(X_{\dot})$ consists of the forms $\omega$ of type $(k,l)$, that is, $\omega_p$ is locally of the form $\sum_{I,J} f_{I,J} dx_{i_1} \smash \dots \smash dx_{i_k} \smash dy_{j_1} \smash \dots \smash dy_{j_l}$, where $x_0, \dots, x_p$ are the barycentric coordinates on $\Delta^p$ and the $y_j$ are (smooth) local coordinates on $X_p$, $d_{\Delta}$ resp. $d_X$ denote the exterior derivative in $\Delta$- resp. $X$-direction. On the other hand we have the double complex $(\A^{k,l}(X_{\dot}), \delta, d_X)$, where $\A^{k,l}(X_{\dot}) = \A^l(X_k)$ and $\delta: \A^{k,l}(X_{\dot}) \to \A^{k+1,l}(X_{\dot})$ is given by $\sum_{i=0}^k (-1)^i \del_i^*$.
Dupont proves \cite[Theorem 2.3]{Dup76}:
\begin{thm}\label{thm:Dupont}
For each $l$ the two chain complexes $(A^{*, l}(X_{\dot}), d_{\Delta})$ and $(\A^{*,l}(X_{\dot}), \delta)$ are naturally chain homotopy equivalent.

In fact, there are natural maps $I : A^{k,l}(X_{\dot}) \rightleftarrows \A^{k,l}(X_{\dot}) : E$ and chain homotopies $s: A^{k,l}(X_{\dot}) \to A^{k-1,l}(X_{\dot})$, such that
\begin{eqnarray}
\label{eq:1} I \circ d_{\Delta} = \delta \circ I, &\quad& I \circ d_X = d_X \circ I,\\
\label{eq:2} d_{\Delta} \circ E = E \circ \delta, &\quad&  E\circ d_X = d_X \circ E,\\
\label{eq:3} I\circ E = \id,\\
\label{eq:4} E\circ I - \id = s\circ d_{\Delta} + d_{\Delta} \circ s, &\quad& s\circ d_X = d_X \circ s.
\end{eqnarray}
\end{thm}

We need a filtered version of this theorem. First of all, observe that we have a natural decomposition
$\A^n(\Delta^p \times X_p) = \bigoplus_{k+l+m=n} \A^{k,l,m}(\Delta^p \times X_p)$, where $\A^{k,l,m}(\Delta^q \times X_p)$ consists of the forms of type $(k,l,m)$, i.e., which are locally of the form
\[
 \sum_{|I|=k, |J|=l, |K|=m} f_{I,J,K} dx_{i_1} \wedge \dots \wedge dx_{i_k} \wedge d\zeta_{j_1} \wedge\dots\wedge d\zeta_{j_l} \wedge d\bar\zeta_{k_1} \wedge\dots\wedge d\bar\zeta_{k_m},
\]
where $x_0, \dots, x_p$ are as usual the barycentric coordinates on $\Delta^p$ and the $\zeta_j$ are \emph{holomorphic} coordinates on $X_p$. Since the simplicial structure maps of $X_{\dot}$ are holomorphic, this direct sum decomposition is respected by the pullback maps $(\delta^i \times \id)^*$ resp. $(\id \times \del_i)^*$, and thus leads to a direct sum decomposition $A^n(X_{\dot}) = \bigoplus_{k+l+m=n} A^{k,l,m}(X_{\dot})$. Then $A^*(X_{\dot})$ is the total complex associated with the triple complex $(A^{k,l,m}(X_{\dot}), d_{\Delta}, \del_X, \bar\del_X)$ and we write $\Fil^rA^*(X_{\dot}) = \bigoplus_{k+l+m=*, l\geq r} A^{k,l,m}(X_{\dot})$. Similarly to the above, we also have the triple complex $(\A^{k,l,m}(X_{\dot}), \delta, \del_X, \bar\del_X)$ with $\A^{k,l,m}(X_{\dot}) = \A^{l,m}(X_k)$.
\begin{thm}\label{thm:filteredDupont}
Let $X_{\dot}$ be a simplicial complex manifold.
For each $l,m\geq 0$ the two complexes $(A^{*,l,m}(X_{\dot}), d_{\Delta})$ and $(\A^{*,l,m}(X_{\dot}), \delta)$ are naturally chain homotopy equivalent.

In fact, the maps $I, E$ and $s$ in theorem \ref{thm:Dupont} induce maps $I : A^{k,l,m}(X_{\dot}) \rightleftarrows \A^{k,l,m}(X_{\dot}) : E$ and $s: A^{k,l,m}(X_{\dot}) \to A^{k-1,l,m}(X_{\dot})$, such that
\begin{align}
\label{eq:5} I \circ d_{\Delta}& = \delta \circ I,&  I \circ \del_X& = \del_X \circ I,& I \circ \bar\del_X& = \bar\del_X \circ I,\\
\label{eq:6} d_{\Delta} \circ E& = E \circ \delta,&  E\circ \del_X& = \del_X \circ E,& E\circ \bar\del_X& = \bar\del_X \circ E,\\
\label{eq:7} I\circ E &= \id, \\
\label{eq:8} E\circ I - \id &= s d_{\Delta} + d_{\Delta} s,&  s\circ \del_X &= \del_X \circ s,& s\circ \bar\del_X& = \bar\del_X \circ s.
\end{align}
In particular, we get natural isomorphisms 
\[
 \Hyp^*(X_{\dot}, \Omega^{\geq r}_{X_{\dot}}) \cong H^*(\Tot\; \Fil^r\A^*(X_{\dot})) \cong H^*(\Fil^r A^*(X_{\dot})).
\]
\end{thm}
\begin{proof}
We recall the constructions of the maps $I, E$ and $s$ of theorem \ref{thm:Dupont}.

Let again $Y$ be an arbitrary complex manifold. Let $e_0, \dots, e_p$ denote the standard basis of $\R^{p+1}$ and $x_0, \dots, x_p$ the barycentric coordinates on $\Delta^p$. For each $j = 0, \dots, p$ define the operator $h_{(j)} : \A^n(\Delta^p \times Y) \to \A^{n-1}(\Delta^p \times Y)$ as follows: Let $g: [0,1] \times \Delta^p \to \Delta^p$ be the homotopy $g(s,t) = s\cdot e_j + (1-s)\cdot t$. Then $h_{(j)}(\omega) := \int_0^1 i_{\del/\del s}((g\times \id_Y)^*\omega)ds$, where $i_{\del/\del s}$ denotes the interior multiplication w.r.t the vector field $\frac{\del}{\del s}$.
\begin{lemma}
$h_{(j)}$ maps $\A^{k,l,m}(\Delta^p \times Y)$ to $\A^{k-1,l,m}(\Delta^p \times Y)$.
\end{lemma}
\begin{proof} The question being local on $Y$, we may assume $Y$ to be an open subset of some affine space $\C^N$ with holomorphic coordinates $\zeta_1, \dots, \zeta_N$. It is enough to consider a form of type
\[
\omega = f dx_{i_1} \wedge \dots \wedge dx_{i_k} \wedge d\zeta_{j_1} \wedge\dots\wedge d\zeta_{j_l} \wedge d\bar\zeta_{k_1} \wedge\dots\wedge d\bar\zeta_{k_m}
\]
with a smooth function $f$. Then $(g\times \id_Y)^*\omega = f\circ (g\times \id_Y) g^*(dx_{i_1} \wedge \dots \wedge dx_{i_k}) \wedge d\zeta_{j_1} \wedge\dots\wedge d\zeta_{j_l} \wedge d\bar\zeta_{k_1} \wedge\dots\wedge d\bar\zeta_{k_m}$ and 
\begin{multline*}
h_{(j)}(\omega) = \left(\int_0^1 f\circ (g\times \id_Y) i_{\del/\del s}(g^*(dx_{i_1} \wedge \dots \wedge dx_{i_k}))ds\right) \\ \wedge d\zeta_{j_1} \wedge\dots\wedge d\zeta_{j_l} \wedge d\bar\zeta_{k_1} \wedge\dots\wedge d\bar\zeta_{k_m}
\end{multline*}
is of type $(k-1, l, m)$.
\end{proof}
The map $I: A^{k,l}(X_{\dot}) \to \A^l(X_k)$ of theorem \ref{thm:Dupont} is now defined by the formula
\begin{equation}\label{eq:Int}
I(\omega) = (-1)^k(e_k \times \id_{X_k})^*(h_{(k-1)} \circ \dots \circ h_{(0)})(\omega_k).
\end{equation}
It follows from the lemma, that $I$ maps $A^{k,l,m}(X_{\dot})$ to $\A^{k,l,m}(X_{\dot}) = \A^{l,m}(X_k)$. Comparing types, the equalities \eqref{eq:5} follow from \eqref{eq:1}. 

\smallskip

Next we come to the definition of $E$. For $\omega \in \A^l(X_k)$ the simplicial form $E(\omega)$ is given by a $(k,l)$-form on $\Delta^p \times X_p$ for all $p \geq 0$. For $p < k$ this form is 0 and for $p \geq k$ it is given by $E(\omega)_p =$
\[
  =k! \sum_{\phi:[k]\hookrightarrow[p]}\left(\sum_{j=0}^k (-1)^j x_{\phi(j)}dx_{\phi(0)} \smash \dots \smash
\widehat{(dx_{\phi(j)})}\smash \dots \smash dx_{\phi(k)} \right) \smash \phi_X^*\omega.
\]
Here the sum runs over all strictly increasing maps $\phi : [k] \to [p]$ and $\phi_X : X_p \to X_k$ denotes the corresponding structure map of the simplicial manifold. Since $\phi_X$ is holomorphic, we see, that $E$ indeed induces a map $\A^{k,l,m}(X_{\dot}) \to A^{k,l,m}(X_{\dot})$. Again, the equalities \eqref{eq:6} follow from \eqref{eq:2}.

\smallskip

Finally, if $\omega \in A^{k,l}(X_{\dot})$ then $s(\omega)$ is given by the family
\begin{multline*}
s(\omega)_p = \sum_{i=0}^{k-1} i!\sum_{\phi:[i]\hookrightarrow [p]} \left(\sum_{j=0}^i (-1)^j x_{\phi(j)}dx_{\phi(0)} \smash \dots 
\smash \widehat{(dx_{\phi(j)})}\smash \dots \smash dx_{\phi(i)} \right)\\
 \smash h_{(\phi(i))} \circ \dots \circ h_{(\phi(0))}(\omega_p),
\end{multline*}
$p \geq 0$, and it follows from the above lemma, that $s(\omega)\in A^{k-1,l,m}(X_{\dot})$ if $\omega \in A^{k,l,m}(X_{\dot})$. Again, the identities \eqref{eq:7} and \eqref{eq:8} follow from \eqref{eq:3} and \eqref{eq:4}.
\end{proof}

\begin{rem}
The map $I$ in \eqref{eq:Int} is just given by integrating forms on $\Delta^k \times X_k$ over $\Delta^k$, where the orientation of $\Delta^k$ is given by the $k$-form $dx_1\wedge \dots \wedge dx_k$ \cite[Ch. 1, Exercise 3]{DupLNM}:
\[
I(\omega) = \int_{\Delta^k} \omega_k \quad \text{if}\quad \omega \in A^{k,l}(X_{\dot}).
\]
\end{rem}

\section{Bundles on simplicial manifolds}
\label{sec:Bundles}

Similar to Karoubi \cite[Ch. V]{Kar87}, we define bundles via their transition functions. This viewpoint is very well-suited for computations, and we will associate Chern-Weil theoretic characteristic classes with these bundles in the next section. To compare this construction with other approaches however, we have to study the precise relation of the bundles defined below with vector bundles. This is done in section \ref{sec:RelationVB}. The construction of Chern characters on relative $K$-groups in chapter \ref{ch:RelKandReg} naturally leads to the definition of \emph{topological bundles} in section \ref{sec:TopMorphAndBundles} below.

\begin{dfn}\label{def:ClassSpace}
The \emph{classifying simplicial manifold} for $\GL_r(\C)$ is the simplicial complex manifold $B_{\dot}\GL_r(\C)$, where
\[
 B_p\GL_r(\C) = \GL_r(\C) \times \dots \times \GL_r(\C) \quad (p \text{ factors}),
\]
with faces and degeneracies
\[
\del_i(g_1, \dots, g_p) = \begin{cases}
(g_2, \dots, g_p), & \text{if } i=0,\\
(g_1, \dots, g_ig_{i+1}, \dots, g_p), &\text{if } 1 \leq i \leq p-1,\\
(g_1, \dots, g_{p-1}), & \text{if } i = p,
\end{cases}
\]
\[
s_i(g_1, \dots, g_p) = (g_1, \dots, g_i, 1, g_{i+1}, \dots, g_p), \qquad i = 0, \dots, p.
\]
The \emph{universal principal $\GL_r(\C)$-bundle} is the simplicial complex manifold $E_{\dot}\GL_r(\C)$, where
\[
 E_p\GL_r(\C) = \GL_r(\C) \times \dots \times \GL_r(\C) \quad (p+1 \text{ factors}),
\]
with faces and degeneracies
\begin{gather}
\del_i(g_0, \dots, g_p) = (g_0, \dots, g_{i-1}, g_{i+1}, \dots, g_p), \qquad i = 0, \dots, p,\label{eq:DelEG}\\
s_i(g_0, \dots, g_p) = (g_0, \dots, g_i, g_i, \dots, g_p), \qquad i = 0, \dots, p. \label{eq:SEG}
\end{gather}
The canonical projection $p : E_{\dot}\GL_r(\C) \to B_{\dot}\GL_r(\C)$ is given in degree $p$ by
\[
 (g_0, \dots, g_p) \mapsto (g_0g_1^{-1}, \dots, g_{p-1}g_p^{-1}).
\]
\end{dfn}
Thus $B_{\dot}\GL_r(\C)$ is the quotient of $E_{\dot}\GL_r(\C)$ by the diagonal \emph{right} action of $\GL_r(\C)$.
Obviously $E_{\dot}\GL_r(\C)$ is a simplicial group and it operates from the left on $B_{\dot}\GL_r(\C) \cong E_{\dot}\GL_r(\C)/\GL_r(\C)$. Explicitly, this action is given by
\[
 (g_0, \dots, g_p) \cdot (h_1, \dots, h_p) = (g_0h_1g_1^{-1}, \dots, g_{p-1}h_pg_p^{-1}).
\]

\medskip

We define $B_{\dot}G$ and $E_{\dot}G$ in the same way, if $G$ is a discrete group, a group scheme, etc. 
\begin{dfn}
Let $X_{\dot}$ be a simplicial complex manifold. A \emph{holomorphic} $\GL_r(\C)$\emph{-bundle} on $X_{\dot}$ is a holomorphic morphism of simplicial complex manifolds
\[
 g : X_{\dot} \to B_{\dot}\GL_r(\C).
\]
We also denote such a bundle by $E/X_{\dot}$ and call $g$ the \emph{classifying map of} $E$. The \emph{universal $\GL_r(\C)$-bundle} $E^{\univ}$ is the bundle given by $\id : B_{\dot}\GL_r(\C) \to B_{\dot}\GL_r(\C)$.

A \emph{morphism} $\alpha: g \to h$ of $\GL_r(\C)$-bundles on $X_{\dot}$ is a morphism of simplicial complex manifolds
$\alpha: X_{\dot} \to E_{\dot}\GL_r(\C)$,
such that $\alpha \cdot g = h$ w.r.t the abovementioned action. Every morphism is an isomorphism.
\end{dfn}

\begin{rem}
Note, that $B_{\dot}\GL_r(\C)$ may also be viewed as (the $\C$-valued points of) a simplicial $\C$-scheme. In analogy with the above definition, we define an \emph{algebraic $\GL_r(\C)$-bundle} on a simplicial $\C$-\emph{scheme} $X_{\dot}$ to be a morphism $g: X_{\dot} \to B_{\dot}\GL_r(\C)$ of simplicial $\C$-schemes. 
\end{rem}

\begin{rem}\label{rem:MorphismBGL}
To give a holomorphic morphism $g:X_{\dot} \to B_{\dot}\GL_r(\C)$ is equivalent to give a morphism $g_1:X_1 \to \GL_r(\C)$ satisfying the cocycle condition $(g_1\circ \del_2) \cdot (g_1\circ \del_0) = g_1\circ \del_1$.

In fact, if $g : X_{\dot} \to B_{\dot}\GL_r(\C)$ is a morphism, the cocycle condition follows from the identities $\del_0 \circ g_2 = \mathrm{pr}_2 \circ g_2 = g_1 \circ \del_0$, $\del_2 \circ g_2 = \mathrm{pr}_1 \circ g_2 = g_1 \circ \del_2$ and $g_1 \circ \del_1 = \del_1 \circ g_2 = (\mathrm{pr}_1 \circ g_2)\cdot (\mathrm{pr}_2 \circ g_2) = (g_1 \circ \del_2)\cdot (g_1 \circ \del_0)$, where $\mathrm{pr}_1, \mathrm{pr}_2 : B_2\GL_r(\C) = \GL_r(\C) \times \GL_r(\C) \to B_1\GL_r(\C) = \GL_r(\C)$ denote the projections.

On the other hand, the composition $\del_0^{i-1} \circ \del_{i+1} \circ \del_{i+2} \circ\dots\circ \del_p : B_p\GL_r(\C) \to B_1\GL_r(\C)$ is given by the projection $\mathrm{pr}_i$ to the $i$-th factor. Hence $\mathrm{pr}_i \circ g_p = g_1 \circ \del_0^{i-1} \circ \del_{i+1} \circ \del_{i+2} \circ\dots\circ \del_p$ and $g_p : X_p \to B_p\GL_r(\C)$ is completely determined by $g_1$. One can check, that, given $g_1$, if one defines $g_p$ by the preceding formula, this indeed gives a morphism of simplicial manifolds $X_{\dot} \to B_{\dot}\GL_r(\C)$.
\end{rem}

\begin{ex}[Cf. section \ref{sec:RelationVB}]
Let $Y$ be an arbitrary complex manifold and $E$ a holomorphic vector bundle of rank $r$. Choose an open covering $\mathscr U = \{U_{\alpha}, \alpha\in A\}$ of $Y$ such that $E\big|_{U_{\alpha}}$ is trivial for each $\alpha \in A$. A set of transition functions $g_{\alpha\beta}: U_{\alpha} \cap U_{\beta} \to \GL_r(\C)$ defininig $E$ yields a holomorphic map $N_1\mathscr U = \coprod_{\alpha,\beta\in A} U_{\alpha} \cap U_{\beta} \to B_1\GL_r(\C) = \GL_r(\C)$
and the cocycle condition ensures, that this map extends uniquely to a holomorphic map $g : N_{\dot}\mathscr U \to B_{\dot}\GL_r(\C)$, where $N_{\dot}\mathscr U$ denotes the \v{C}ech nerve of $\mathscr U$, i.e. the simplicial manifold which in degree $p$ is given by $N_p\mathscr U = \coprod_{\alpha_0, \dots, \alpha_p \in A} U_{\alpha_0}\cap\dots\cap U_{\alpha_p}$. Thus we get a $\GL_r(\C)$-bundle on $N_{\dot}\mathscr U$ in the above sense.
\end{ex}

\begin{ex}\label{ex:SimplicialBundles}
Again let $Y$ be a complex manifold and in addition let $S$ be a simplicial set. Let $\O(Y)$ denote the ring of holomorphic functions on $Y$ and $G$ the constant simplicial group $\GL_r(\O(Y))$. Then a $G$-fibre bundle (``$G$-fibr\'e rep\'er\'e'') on $S$ in the sense of Karoubi \cite[5.1]{Kar87} may be defined as a morphism of simplicial sets $S \to B_{\dot}G$ (cf. the proof of {\it loc. cit.} Th\'eor\`eme 5.4). But $G = \GL_r(\O(Y))$ may be identified with the group of holomorphic maps $Y \to \GL_r(\C)$ and thus a morphism of simplicial sets $S \to B_{\dot}G$ is equivalent to a morphism of simplicial complex manifolds
$Y\otimes S \to B_{\dot}\GL_r(\C)$, where $Y\otimes S$ is the simplicial manifold given in degree $p$ by $\coprod_{\sigma \in S_p} Y$ with structure maps induced from those of $S$.
\end{ex}

\subsection{Relation with vector bundles}\label{sec:RelationVB}

The notion of a $\GL_r(\C)$-bundle on a simplicial manifold has the advantage, that it is very well suited for computations, at the drawback of being sometimes too rigid. For example, if $E$ is a $\GL_r(\C)$-bundle on $X_{\dot}$, we may form the associated projective bundle as a simplicial manifold $\P(E) \to X_{\dot}$, but the associated tautological bundle is \emph{not} a $\mathbb G_m$-bundle in the above sense.
There is the more flexible notion of vector bundles on simplicial manifolds (or schemes \dots), which we now recall (cf. \cite[Ex. 1.1]{GilletUCC}).

\begin{dfn}\label{def:VectorBundle}
A \emph{(holomorphic) vector bundle} on a simplicial complex manifold $X_{\dot}$ is a sheaf $\E_{\dot}$ of $\O_{X_{\dot}}$-modules, such that each $\E_p$ is locally free and for every $\phi: [p] \to [q]$ the associated map $\phi_X^*\E_p \to \E_q$ is an isomorphism.

A vector bundle $\E_{\dot}$ is called degreewise trivial, if each $\E_p$ is trivial, i.e. isomorphic to a free $\O_{X_p}$-module.
\end{dfn}
The precise relation between vector bundles and $GL_r(\C)$-bundles is described in the following lemmata. All this is certainly well-known, but I could not find an accurate reference.
\begin{lemma}\label{lemma:degwise_trivial_vb}
Let $X_{\dot}$ be a simplicial complex manifold. There is a natural 1--1 correspondence
\[
\begin{Bmatrix} \text{isomorphism classes of}\\\text{degreewise trivial}\\\text{holomorphic rank $r$ vector}\\ \text{bundles on $X_{\dot}$}\end{Bmatrix} \leftrightarrow \begin{Bmatrix} \text{isomorphism classes of}\\\text{holomorphic}\\ \text{$\GL_r(\C)$-bundles on $X_{\dot}$}\end{Bmatrix}. 
\]
\end{lemma}
\begin{proof}
Let $\E_{\dot}$ be a degreewise trivial holomorphic vector bundle on $X_{\dot}$. Fix an isomorphism $\psi^{(0)}:\E_0 \xrightarrow{\cong} \O_{X_0}^r$. For any $p\geq 0$ and any $i \in [p]$ let $\tau_i : [0]\to[p]$ denote the unique map, that sends $0$ to $i$. Then $\psi^{(0)}$ induces trivializations $\psi^{(p)}_i$ of $\E_p$ defined as the composition
\[
\E_p \xleftarrow{\cong} \tau_i^*\E_0 \xrightarrow{\tau_i^*\psi^{(0)}} \O_{X_p}^r.
\]
Then $\psi^{(p)}_i \circ (\psi^{(p)}_j)^{-1} : \O_{X_p}^r \to \O_{X_p}^r$ is given by a holomorphic map $g^{(p)}_{ij}: X_p \to \GL_r(\C)$. The required morphism $g:X_{\dot} \to B_{\dot}\GL_r(\C)$ is then given in degree $p$ by $(g^{(p)}_{01}, \dots, g^{(p)}_{p-1,p})$.

Let $\varphi: \E_{\dot} \to \E'_{\dot}$ be an isomorphism of degreewise trivial vector bundles on $X_{\dot}$. Fix trivializiations $\psi^{(0)}$, $\psi'^{(0)}$ of $\E_0$, $\E_0'$ respectively. They induce trivializations $\psi^{(p)}_i$, $\psi'^{(p)}_i$ and corresponding morphisms $g, g':X_{\dot} \to B_{\dot}\GL_r(\C)$ as above.
Then $\psi'^{(p)}_i\circ\varphi_p\circ(\psi^{(p)}_i)^{-1} : \O_{X_p}^r \to \O_{X_p}^r$ is given by a holomorphic map $\alpha^{(p)}_i : X_p \to \GL_r(\C)$. It follows from the constructions, that
\[
 \psi'^{(p)}_i(\psi'^{(p)}_j)^{-1} = \alpha^{(p)}_i \psi^{(p)}_i(\psi^{(p)}_j)^{-1} (\alpha^{(p)}_j)^{-1}
\quad\text{and}\quad \alpha^{(p)}_i = \alpha^{(0)}_0 \circ (\tau_i)_X.
\]
These conditions imply, that $\alpha : X_{\dot} \to E_{\dot}\GL_r(\C)$, given in degree $p$ by $(\alpha^{(p)}_0, \dots, \alpha^{(p)}_p)$, is a morphism of simplicial manifolds, that satisfies $\alpha\cdot g = g'$. This shows, that any isomorphism class of degreewise trivial rank $r$ vector bundles corresponds to a well defined isomorphism class of $\GL_r(\C)$-bundles.

On the other hand, given $g : X_{\dot} \to B_{\dot}\GL_r(\C)$, we define the associated vector bundle $\E_{\dot}$ as follows: Set $\E_p = \O_{X_p}^r$ for every $p \geq 0$. The structure maps $\del_i^*\E_{p-1} = \O_{X_p}^r \to \E_p = \O_{X_p}^r$ are given by $\id_{\O_{X_p}^r}$, if $i < p$, and by $(g^{(p)}_p)^{-1}$, if $i=p$. Here $g^{(p)}_p :X_p \to \GL_r(\C)$ is the $p$-th component of the map $g$ in simplicial degree $p$. The maps $s_i^*\E_{p+1} \to \E_p$ are given by $\id_{\O_{X_p}^r}$. One checks, that $\E_{\dot}$ is a well defined vector bundle.

Now let $g,g' : X_{\dot} \to B_{\dot}\GL_r(\C)$ be two $\GL_r(\C)$-bundles and $\alpha : g \to h$ a morphism, i.e. $\alpha : X_{\dot} \to E_{\dot}\GL_r(\C)$ and $\alpha \cdot g = g'$. Denote the associated vector bundles by $\E, \E'$ respectively. Then $\alpha$ induces an isomorphism $\E \xrightarrow{\cong} \E'$ in degree $p$ by $\E_p = \O_{X_p}^r \xrightarrow{\alpha^{(p)}_p} \O_{X_p}^r = \E'_p$ where $\alpha^{(p)}_p : X_p \to \GL_r(\C)$ denotes the last component of the map given by $\alpha$ in degree $p$.
The diagrams
\[\xymatrix@C+1cm{
\del_i^*\E_{p-1} \ar[r]^{\alpha_{p-1}^{(p-1)} \circ \del_i} \ar[d] & \del_i^*\E'_{p-1} \ar[d]\\
\E_p \ar[r]^{\alpha^{(p)}} & \E'_p
}
\]
commute: for $i < p$ this is clear since $\del_i^*\E_{p-1} \to \E_p$ is the identity and for $i=p$ it follows from the relation $\alpha \cdot g = g'$.

Thus any isomorphism class of $\GL_r(\C)$-bundles gives a well defined isomorphism class of degreewise trivial vector bundles. We have to show that these constructions are inverse to each other.

Thus let $g: X_{\dot} \to B_{\dot}\GL_r(\C)$ be a $\GL_r(\C)$-bundle with associated vector bundle $\E_{\dot}$.
Let $\widetilde g : X_{\dot} \to B_{\dot}\GL_r(\C)$ be the morphism associated with $\E_{\dot}$ and the trivialization $\id : \E_0 = \O_{X_0}$ by the above construction.
We want to prove, that $g = \widetilde g$. Since any morphism $X_{\dot} \to B_{\dot}\GL_r(\C)$ is determined by its component in simplicial degree 1 (cf. remark \ref{rem:MorphismBGL}), it suffices to show, that $g_1 = \widetilde g_1$. By construction $\widetilde g_1$ is the matrix of the morphism
\[
\O_{X_1}^r = \tau_1^*\E_0 \xrightarrow{\cong} \E_1 \xleftarrow{\cong} \tau_0^*\E_0 = \O_{X_1}^r.
\]
But since $\tau_0 = \delta^1, \tau_1=\delta^0 : [0]\to [1]$ this is just $\del_0^*\E_0 \xrightarrow{\id} \E_1 \xleftarrow{g_1^{-1}} \del_1^*\E_0$, that is $g_1$.

It remains to show that, given a vector bundle $\E_{\dot}$ and a trivialization $\psi^{(0)} : \E_0 \to \O_{X_0}^r$ with associated $\GL_r(\C)$-bundle $g: X_{\dot} \to B_{\dot}\GL_r(\C)$, the bundle constructed above associated with $g$ is isomorphic to $\E$. But an isomorphism is given explicitly by the sequence of maps $\psi_p^{(p)}:\E_p \xrightarrow{\cong} \O_{X_p}^r$ constructed at the beginning. Again it follows immediately from the constructions, that this really defines a morphism of vector bundles.
\end{proof}

We fix some terminology.
Let $X_{\dot}$ be a simplicial complex manifold. A morphism $U_{\dot} \to X_{\dot}$ is an \emph{open covering} if each $U_p \to X_p$ is an open covering in the usual sense, i.e. $U_p = \coprod_{\alpha} U_{p,\alpha}$ where each $U_{p,\alpha}$ is an open subset of $X_p$ and $\bigcup_{\alpha} U_{p,\alpha} = X_p$. The \emph{\v{C}ech nerve} of $U_{\dot} \to X_{\dot}$ is the bisimplicial manifold $N(U_{\dot}) = N_{X_{\dot}}(U_{\dot})$ defined by
\[
N_{X_{\dot}}(U_{\dot})_{p,q} = (N_{X_p}(U_p))_q,
\]
where $(N_{X_p}(U_p))_{\dot}$ is the usual simplicial \v{C}ech nerve of $U_p \to X_p$, i.e. $(N_{X_p}(U_p))_q = U_p \times_{X_p} \dots \times_{X_p} U_p$ ($q+1$ factors) with structure maps as in \eqref{eq:DelEG}, \eqref{eq:SEG}.

The diagonal simplicial manifold of $N(U_{\dot})$ is denoted by $\Delta N(U_{\dot})$.

\begin{lemma}\label{lemma:covering_iso}
Let $U_{\dot} \to X_{\dot}$ be an open covering and $\mathscr F_{\dot}^*$ a complex of abelian sheaves on $X_{\dot}$. Then the natural maps
\[
\Hyp^*(X_{\dot}, \mathscr F_{\dot}^*) \xrightarrow{\cong} \Hyp^*(N(U_{\dot}), \mathscr F_{\dot}^*|_{N(U_{\dot})})
\xrightarrow{\cong} \Hyp^*(\Delta N(U_{\dot}), \mathscr F_{\dot}^*|_{\Delta N(U_{\dot})})
\]
are isomorphisms.
\end{lemma}
Here $\mathscr F_{\dot}^*|_{N(U_{\dot})}$ denotes the inverse image of the complex of sheaves $\mathscr F_{\dot}^*$ on $N(U_{\dot})$.
\begin{proof}
The second isomorphism follows from the theorem of Eilenberg-Zilber \cite[(6.4.2.2)]{HodgeIII}.

Each $N(U_{\dot})_{p,\dot} \to X_p$ is the nerve of an open covering and thus of cohomological descent \cite[(5.3.7)]{HodgeIII}. Hence $N(U_{\dot}) \to X_{\dot}$ is of cohomological descent ({\it loc. cit.} (6.4.3)), hence the result. 
\end{proof}

We need this in the situation where $\mathscr F_{\dot}^*$ is a complex of differential forms. Since we only consider open coverings, $\Omega^i_{X_\dot} |_{N(U_{\dot})} = \Omega^i_{N(U_{\dot})}$ and similarly for smooth forms.

\begin{lemma}\label{lemma:open_covering}
Let $\E_{\dot}$ be a vector bundle on $X_{\dot}$. Then there exists an open covering $f:U_{\dot} \to X_{\dot}$ such that $f^*\E_{\dot}$ is trivial in each degree.
\end{lemma}
\begin{proof}
Choose an open covering $f_0:U_0 \to X_0$ such that $f_0^*\E_0$ is trivial. Define 
\[
f_p\colon U_p := U_0^{[p]} \times_{X_0^{[p]}} X_p \xrightarrow{\mathrm{pr}_2} X_p,
\]
where the $i$-th component of the map $X_p \to X_0^{[p]}$ is the morphism $\tau_i : X_p \to X_0$ induced by $\tau_i : [0] \xrightarrow{0\mapsto i} [p]$. Explicitely, if $U_0 = \coprod_{\alpha \in A} V_{\alpha}$, then $U_p = \coprod_{\alpha_0, \dots, \alpha_p\in A} \tau_0^{-1}(V_{\alpha_0}) \cap \dots \cap \tau_p^{-1}(V_{\alpha_p})$. Since $\tau_i^*\E_0 \to \E_p$ is an isomorphism, $\E_p|_{\tau_i^{-1}(V_{\alpha_i})}$ is trivial, hence also $f_p^*\E_p$.
\end{proof}

\begin{rem}\label{rem:CohomInterpretation}
We have the usual isomorphism
\[
\{\text{isomorphism classes of holomorphic line bundles on } X_{\dot}\} \cong H^1(X_{\dot}, \O_{X_{\dot}}^*)
\]
(cf. \cite[example 1.1]{GilletUCC}). 
The cohomology class associated with a degreewise trivial line bundle $\mathscr L_{\dot}$ is easy to describe: $\mathscr L_{\dot}$ is classified by a map $g : X_{\dot} \to B_{\dot}\GL_1(\C)$. Its component in degree 1, $g_1 : X_1 \to \C^*$, viewed as an element of $\Gamma(X_1, \O^*_{X_1})$ is, by the cocycle condition of remark \ref{rem:MorphismBGL}, a cocycle of degree 1 in the complex $\Gamma^*(X_{\dot}, \O^*_{X_{\dot}})$ (the complex associated with the cosimplicial group $[p]\mapsto \Gamma(X_p, \O^*_{X_p})$). There is a natural map
$H^*(\Gamma^*(X_{\dot}, \O_{X_{\dot}}^*)) \to H^*(X_{\dot}, \O_{X_{\dot}}^*)$ (an edge morphism in the spectral sequence $E_1^{p,q} = H^q(X_p, \O_{X_p}^*) \Rightarrow H^{p+q}(X_{\dot}, \O_{X_{\dot}}^*)$) and the cohomology class associated with $\mathscr L_{\dot}$ is just the image of the class of $g_1$ under this map.
\end{rem}

\subsection{Topological morphisms and bundles}
\label{sec:TopMorphAndBundles}

The definition of a differential form on a simplicial complex manifold leads to the following notion of what we call \emph{topological morphisms}.
\begin{dfn}\label{def:TopMorph}
A \emph{topological morphism} of simplicial manifolds $f: Y_{\dot} \rightsquigarrow X_{\dot}$ is a family of smooth maps 
\[
f_p: \Delta^p \times Y_p \to X_p, \qquad p\geq 0,
\]
satisfying the following compatibility condition: For every increasing map $\phi: [p] \to [q]$ the diagram
\[\xymatrix@C+1cm@R-0.3cm{
& \Delta^q \times Y_q \ar[r]^-{f_q} & X_q \ar[dd]^{\phi_X}\\
\Delta^p \times Y_q \ar[ur]^{\phi_{\Delta} \times \id} \ar[dr]^{\id \times \phi_Y}\\
& \Delta^p \times Y_p \ar[r]^-{f_p} & X_p 
}
\]
commutes.
Here $\phi_{\Delta}, \phi_Y, \phi_X$ denote the (co)simplicial structure maps induced by $\phi$.
\end{dfn}

Every holomorphic or smooth morphism of simplicial (complex) manifolds $f : Y_{\dot} \to X_{\dot}$ induces a topological morphism $f: Y_{\dot} \rightsquigarrow X_{\dot}$ by composition with the natural projections $\Delta^p \times Y_p \to Y_p$. 

\begin{rem}\label{rem:PullbackByTopMorph}
Let $f : Y_{\dot} \rightsquigarrow X_{\dot}$ be a topological morphism. 
Then we have commutative diagrams
\[\xymatrix{
& \Delta^q \times Y_q \ar[rr]^{(\mathrm{pr}_{\Delta^q},f_q)} && \Delta^q \times X_q \\
\Delta^p \times Y_q \ar[ur]^{\phi_{\Delta} \times \id} \ar[dr]^{\id \times \phi_Y}
\ar[rr]^{(\mathrm{pr}_{\Delta^p},f_q\circ(\phi_{\Delta}\times\id))}
&& \Delta^p \times X_q \ar[ur]^{\phi_{\Delta} \times \id} \ar[dr]^{\id \times \phi_X}\\
& \Delta^p \times Y_p \ar[rr]^{(\mathrm{pr}_{\Delta^p},f_p)} && \Delta^p \times X_p 
}
\]
for every increasing $\phi: [p] \to [q]$. Now let $\omega = (\omega_p)_{p\geq 0}$ be a simplicial form on $X_{\dot}$. Define $f^*\omega := \left((\mathrm{pr}_{\Delta^p}, f_p)^*\omega_p\right)_{p\geq 0}$. From the above diagram (in the special case where $\phi=\delta^i : [p-1]\to [p]$) one sees, that $f^*\omega$ is a well defined simplicial form on $Y_{\dot}$, the \emph{pullback} of $\omega$ by $f$. Thus we have a well defined pull-back map $f^*: A^n(X_{\dot}) \to A^n(Y_{\dot})$.

In a similar way we define the composition of two topological morphisms.
\end{rem}

\begin{dfn}
Let $X_{\dot}$ be a simplicial manifold. A \emph{topological} $\GL_r(\C)$\emph{-bundle} on $X_{\dot}$ is a topological morphism of simplicial manifolds
\[
 g: X_{\dot} \rightsquigarrow B_{\dot}\GL_r(\C).
\]
A \emph{morphism} $\alpha: g \to h$ of topological $\GL_r(\C)$-bundles on $X_{\dot}$ is a topological morphism of simplicial manifolds
$
 \alpha: X_{\dot} \rightsquigarrow E_{\dot}\GL_r(\C),
$
such that $\alpha \cdot g = h$.
\end{dfn}

\begin{ex}\label{ex:Karoubi_top_bundles}
Let $S$ be a simplicial set, $A$ a complex Fr\'echet algebra and $A_{\dot}$ the simplicial algebra $\mathscr{C}^{\infty}(\Delta^{\dot}_{\R}) \widehat\otimes_{\pi} A$, where $\mathscr{C}^{\infty}$ denotes smooth complex valued functions and $\widehat\otimes_{\pi}$ the projectively completed tensor product over $\C$. The simplicial classifying set $B_{\dot}\GL_r(A_{\dot})$ for the simplicial group $\GL_r(A_{\dot})$ is by definition the diagonal of the bisimplicial set $([p],[q])\mapsto B_p\GL_r(A_q)$.
Karoubi defines a topological $\GL_r(A)$-bundle (= a ``$\GL_r(A_{\dot})$-fibr\'e rep\'er\'e'') on the simplicial set $S$ to be a morphism $S \to B_{\dot}\GL_r(A_{\dot})$ \cite[5.1, proof of 5.4 and 5.26]{Kar87}. 

In the special case, where $A$ is the ring of smooth complex valued functions $\mathscr C^{\infty}(Y)$ on a complex manifold $Y$, this gives a topological bundle on the simplicial manifold $Y\otimes S$ (cf. example \ref{ex:SimplicialBundles}) as follows:

First, there is a natural map $\mathscr C^{\infty}(\Delta^p) \widehat\otimes_{\pi} \mathscr C^{\infty}(Y) \to \mathscr C^{\infty}(\Delta^p \times Y)$. Next, $B_p\GL_r\left(\mathscr C^{\infty}(\Delta^p \times Y)\right)=\mathscr C^{\infty}(\Delta^p \times Y, B_p\GL_r(\C))$. Thus, a morphism of simplicial sets $f: S \to B_{\dot}\GL_r(A_{\dot})$ gives rise to a family of smooth morphisms
\[
 \Delta^p \times Y \xrightarrow{f(\sigma)} B_p\GL_r(\C), \qquad \sigma \in S_p, p \geq 0.
\]
The fact that $f$ is a morphism of simplicial sets is reflected in the fact, that for every increasing $\phi: [p] \to [q]$ and $\sigma \in S_q$ the diagram
\[\xymatrix@C+1cm{
\Delta^q \times Y \ar[r]^-{f(\sigma)} & B_q\GL_r(\C) \ar[d]^{\phi_{B_{\dot}G}}\\
\Delta^p \times Y \ar[u]^{\phi_{\Delta}\times \id}\ar[r]^-{f(\phi_S^*\sigma)} & B_p\GL_r(\C)
}
\]
commutes. Here $\phi_S^* : S_q \to S_p$ denotes the simplicial structure map induced by $\phi$.
Now the collection of maps $f(\sigma)$, $\sigma \in S_p$, defines a smooth morphism
\[
\widetilde f_p: \Delta^p \times (Y\otimes S)_p = \coprod_{\sigma \in S_p} \Delta^p \times Y \xrightarrow{\coprod f(\sigma)} B_p\GL_r(\C)
\]
and the commutativity of the above diagrams is equivalent to the fact, that the family of maps $\widetilde f_p$, $p\geq 0$, defines a topological morphism $Y\otimes S \rightsquigarrow B_{\dot}\GL_r(\C)$ in our sense.
\end{ex}

\section{Connections, curvature and characteristic classes}
\label{sec:connections}

In this section we define connections, the associated curvature and construct the Chern-Weil theoretic characteristic classes. This is done by carrying Ka\-rou\-bi's definitions and constructions from the case of simplicial sets \cite[Ch. 5]{Kar87} over to our geometric setting. The systematic use of this formalism was a fundamental idea of Karoubi. 

In order to define the notion of a connection, we have to introduce some more notation.
Any $p$-simplex $x$ in the classifying space $B_{\dot}\GL_r(\C)$ may be written as $x=(g_{01}, g_{12}, \dots, g_{p-1,p})$. Thus, if $(g_0, \dots, g_p) \in E_p\GL_r(\C)$ is a $p$-simplex lying over $x$, then $g_{01} = g_0g_1^{-1}$ etc. and we define $g_{ji} := g_jg_i^{-1}$ for any $0 \leq i,j \leq p$. 
If $g: X_{\dot} \rightsquigarrow B_{\dot}\GL_r(\C)$ is a topological $\GL_r(\C)$-bundle, we write $g_{ji}$ for the smooth maps $\Delta^p \times X_p \to \GL_r(\C)$ obtained in the above way. If $g$ is a holomorphic bundle then $g_{ji}$ factors through a holomorphic map $X_p \to \GL_r(\C)$ which, by abuse of notation, will also be denoted by $g_{ji}$.

\begin{dfn}\label{def:Connection}
A \emph{connection} in a topological $\GL_r(\C)$-bundle $g: X_{\dot} \rightsquigarrow B_{\dot}\GL_r(\C)$ is given by the following data:
For any $p \geq 0$ and any $i \in [p] = \{0, \dots, p\}$ a matrix valued $1$-form $\Gamma_i = \Gamma_i^{(p)} \in \A^1(\Delta^p \times X_p; \Mat_r(\C)) = \Mat_r(\A^1(\Delta^p \times X_p))$ subject to the  conditions
\begin{enumerate}
\item $(\phi_{\Delta}\times \id)^*\Gamma_{\phi(i)}^{(q)} = (\id \times \phi_X)^*\Gamma_i^{(p)}$ for any increasing map $\phi: [p] \to [q]$ and
\item $\Gamma_i = g_{ji}^{-1} dg_{ji} + g_{ji}^{-1}\Gamma_j g_{ji}$.
\end{enumerate}
Here $\Mat_r$ denotes $r\times r$-matrices. 
We view $g_{ji}$ as a matrix of smooth functions on $\Delta^p \times X_p$. Thus $dg_{ji}$ is a matrix valued 1-form on $\Delta^p \times X_p$.

If $g$ is a holomorphic bundle, we call the connection (partially) \emph{holomorphic}, if $\Gamma_i \in \A^{0,1,0}(\Delta^p \times X_p, \Mat_r(\C)) \subset \A^1(\Delta^p \times X_p; \Mat_r(\C))$ 
(cf. the discussion before theorem \ref{thm:filteredDupont}).
\end{dfn}

\begin{ex}\label{ex:standard_connection}
Every topological $\GL_r(\C)$-bundle $g: X_{\dot} \rightsquigarrow B_{\dot}\GL_r(\C)$ may be equipped with the \emph{standard connection}
given by
\[
 \Gamma_i = \sum_k x_k g_{ki}^{-1}dg_{ki},
\]
where $x_0, \dots, x_p$ denote the barycentric coordinates of $\Delta^p$. If $g$ is holomorphic, this connection is holomorphic. The conditions of the definition are easily verified by direct computation.
\end{ex}

\begin{ex}
This example shows, how the classical notion of a connection fits into our framework. It will not be needed later on.

Let $Y$ be an arbitrary complex analytic manifold, $E/Y$ a smooth complex vector bundle of rank $r$ and $\nabla$ an ordinary connection on $E$, i.e. a $\C$-linear map $\E \to \A^1_Y \otimes_{\mathscr C^{\infty}_Y} \E$ satisfying Leibnitz' rule, where $\E$ denotes the sheaf of smooth sections of $E$.
Choose an open covering $\mathscr U = \{U_{\alpha}\}_{\alpha \in A}$ such that $E|_{U_{\alpha}}$ is trivial for each $\alpha$. Denote the pullback $E|_{N_{\dot}\mathscr U}$ by $E'_{\dot}$, the corresponding simplicial sheaf by $\E_{\dot}'$. The pullback of $\nabla$ is a compatible family of connections on each $E'_p$. As in lemma \ref{lemma:degwise_trivial_vb}, choose a trivialization $\psi^{(0)} : \E'_0 \to \mathscr (\mathscr C^{\infty}_{N_0\mathscr U})^r$. This induces trivializations $\psi^{(p)}_i$, $i = 0, \dots, p$, of $\E'_p$, and $\psi^{(p)}_i\circ (\psi^{(p)}_j)^{-1}$ is given by the smooth transition function $g^{(p)}_{ij} : N_p\mathscr U \to \GL_r(\C)$. 
Then $E'_{\dot}$ is classified by the smooth morphism $g : N_{\dot}\mathscr U \to B_{\dot}\GL_r(\C)$, analogous to the holomorphic case. In particular $E'_{\dot}$ gives rise to a topological $\GL_r(\C)$-bundle on $N_{\dot}\mathscr U$.

With respect to the trivialization $\psi^{(p)}_i$ the connection is given by a matrix valued 1-form $\Gamma_i = \Gamma_i^{(p)} \in \A^1(N_p\mathscr U;\Mat_r(\C))$ (see e.g. \cite[1.8]{Kar87}). These forms satisfy the transformation rule $\Gamma_i = g_{ji}^{-1} dg_{ji} + g_{ji}^{-1}\Gamma_j g_{ji}$ ({\it loc. cit.}) and the compatibility condition $\Gamma_{\phi(i)}^{(q)} = \phi_{N_{\dot}\mathscr U}^*\Gamma_i^{(p)}$ for every increasing $\phi: [p] \to [q]$. Hence, the pullbacks of the $\Gamma_i$ to $\Delta^p \times N_p\mathscr U$ yield a connection in the above sense.
\end{ex}
This example also motivates the following definitions.

\begin{dfn}
The \emph{curvature} of the connection $\{\Gamma_i\}$ is defined as the family of matrix valued 2-forms
\[
R_i := R_i^{(p)} := d\Gamma_i^{(p)} + \left(\Gamma_i^{(p)}\right)^2 \in \A^2(\Delta^p \times X_p; \Mat_r(\C)),  
\]
$p \geq 0, i = 0, \dots p$.
\end{dfn}

\begin{rems}\label{rem:FunctorialityConnection}
(i) Let $g, h : X_{\dot} \rightsquigarrow Y_{\dot}$ be two bundles, $\alpha : g \to h$ a morphism of bundles and $\Gamma = \{\Gamma_i\}$ a connection on $h$ with curvature $\{R_i\}$. Then the \emph{pullback} $\alpha^*\Gamma$ \emph{of the connection} $\Gamma$ is defined by the family of forms
\[
 (\alpha^*\Gamma)_i = \alpha_i^{-1} d\alpha_i + \alpha_i^{-1} \Gamma_i \alpha_i,
\]
where $\alpha_i : \Delta^p \times X_p \to \GL_r(\C)$ is the $i$-th component of the morphism $\alpha$ in simplicial degree $p$.
The curvature of $\alpha^*\Gamma$ is given by the family of 2-forms $\alpha_i^{-1} R_i \alpha_i$.
\smallskip

(ii) If $E/X_{\dot}$ is a topological bundle on $X_{\dot}$ given by $g : X_{\dot} \rightsquigarrow B_{\dot}\GL_r(\C)$, and $f: Y_{\dot} \rightsquigarrow X_{\dot}$ is a topological morphism, the pullback $f^*E$ is given by $g\circ f$. If $\Gamma = \{\Gamma_i\}$ is a connection on $E$, the induced connection $f^*\Gamma$ on $f^*E$ is given by
\[
 (f^*\Gamma)_i^{(p)} = (\mathrm{pr}_{\Delta^p}, f_p)^*\Gamma_i^{(p)}.
\]
Consequently, its curvature is given by the family of forms $(\mathrm{pr}_{\Delta^p}, f_p)^*R_i^{(p)}$.

If $\Gamma$ is the standard connection on $E$, then $f^*\Gamma$ is the standard connection on $f^*E$, as follows directly from the definitions.
\end{rems}

\begin{lemma}\label{lem:CurvatureCompatible}
The forms $R_i$ satisfy $R_i = g_{ji}^{-1}R_jg_{ji}$.
\end{lemma}
\begin{proof}
Again, this follows directly from the definitions. We give the proof as a prototype for all the calculations of this type.

Using the formula $d(g^{-1}) = - g^{-1}(dg)g^{-1}$ and Leibnitz' rule we get
\begin{eqnarray*}
R_i &=& d\Gamma_i + \Gamma_i^2\\
&=& d\left(g_{ji}^{-1}dg_{ji} + g_{ji}^{-1}\Gamma_j g_{ji}\right) + \left(g_{ji}^{-1}dg_{ji} + g_{ji}^{-1}\Gamma_j g_{ji}\right)^2 \\
&=& -g_{ji}^{-1}(dg_{ji})g_{ji}^{-1} dg_{ji} + -g_{ji}^{-1}(dg_{ji})g_{ji}^{-1}\Gamma_j g_{ji} + g_{ji}^{-1}(d\Gamma_j)g_{ji} - g_{ji}^{-1}\Gamma_j dg_{ji}\\
&& + (g_{ji}^{-1}dg_{ji})^2 + g_{ji}^{-1}(dg_{ji})g_{ji}^{-1}\Gamma_j g_{ji} + g_{ji}^{-1}\Gamma_j g_{ji} g_{ji}^{-1} dg_{ji} + g_{ji}^{-1}\Gamma_j^2 g_{ji}\\
&=& g_{ji}^{-1} (d\Gamma_j) g_{ji} + g_{ji}^{-1}\Gamma_j^2 g_{ji}\\
&=& g_{ji}^{-1} R_j g_{ji}. 
\end{eqnarray*}
\end{proof}

\begin{dfn}
We define the $n$\emph{-th Chern character form} $\Ch_n(\Gamma)$ of the connection $\Gamma =\{\Gamma_i\}$ to be the family of forms $\frac{1}{n!} \Tr\left(\left(R_i^{(p)}\right)^n\right)$ on $\Delta^p\times X_p$, $p \geq 0$. According to lemma \ref{lem:CurvatureCompatible}, this form does not depend on $i$.
\end{dfn}
\begin{propn}\label{prop:CharClasses} Let $g: X_{\dot} \rightsquigarrow B_{\dot}\GL_r(\C)$ be  a topological bundle and $\Gamma$ a connection on $g$.
\begin{enumerate}
\item $\Ch_n(\Gamma)$ is a closed $2n$-form on $X_{\dot}$, i.e. belongs to $A^{2n}(X_{\dot})$ and $d\Ch_n(\Gamma) = 0$.
\item The cohomology class of $\Ch_n(\Gamma)$ does not depend on the connection chosen.
\item If the bundle $g$ and the connection are holomorphic, $\Ch_n(\Gamma) \in \Fil^nA^{2n}(X_{\dot})$. Moreover, the class of $\Ch_n(\Gamma)$ in $H^{2n}(\Fil^nA^*(X_{\dot})) = \Hyp^{2n}(X_{\dot}, \Omega^{\geq n}_{X_{\dot}})$ does not depend on the holomorphic connection chosen.
\item If $h : X_{\dot} \rightsquigarrow B_{\dot}\GL_r(\C)$ is a second bundle, and $\alpha : h \to g$ is a morphism, then $\Ch_n(\alpha^*\Gamma) = \Ch_n(\Gamma)$.
\item If $f: Y_{\dot} \rightsquigarrow X_{\dot}$ is a topological morphism, $\Ch_n(f^*\Gamma) = f^*\Ch_n(\Gamma)$.
\end{enumerate}
\end{propn}
\begin{proof}
(i) It follows from condition (i) in  definition \ref{def:Connection}, that $(\phi_{\Delta} \times \id_{X_q})^*\Tr((R^{(q)}_{\phi(i)})^n) = (\id_{\Delta^p} \times \phi_X)^*\Tr((R^{(p)}_i)^n)$, hence the forms $\frac{1}{n!} \Tr\left(\left(R_i^{(p)}\right)^n\right)$, $p \geq 0$, are indeed compatible and define $\Ch_n(\Gamma) \in A^{2n}(X_{\dot})$. 
For the closedness cf. the proof of \cite[th\'eor\`eme 1.19]{Kar87}.

(ii) This follows from a standard homotopy argument. See lemma \ref{lem:RelChernCharForm} with $\alpha = \id$ below.

(iii) With the notations of section \ref{sec:SimplDeRham} write 
\[
\Fil^i\A^*(\Delta^p \times X_p) = \bigoplus_{k+l+m=*, l\geq i} \A^{k,l,m}(\Delta^p \times X_p)
\]
and similarly for matrix valued forms.
These are subcomplexes and the product maps $\Fil^i \times \Fil^j$ to $\Fil^{i+j}$.
Now, if the connection is holomorphic, $\Gamma_i \in \Fil^1\A^1(\Delta^p \times X_p, \Mat_r(\C))$, hence $R_i = d\Gamma_i + \Gamma_i^2 \in \Fil^1\A^2(\Delta^p \times X_p; \Mat_r(\C))$ and then also $\Ch_n(\Gamma) \in \Fil^nA^{2n}(X_{\dot})$.

Again, the independence of the associated cohomology class of the holomorphic connection chosen follows from a (slightly more complicated) homotopy argument, see lemma \ref{lem:ChRelHolom} below.

(iv), (v) These follow directly from remarks \ref{rem:FunctorialityConnection} (i) and (ii) respectively.
\end{proof}

\begin{dfn}
If $E/X_{\dot}$ is a topological bundle, we write $\Ch_n(E)$ for the cohomology class of $\Ch_n(\Gamma)$ in $H^{2n}(A^*(X_{\dot})) = H^{2n}(X_{\dot}, \C)$, where $\Gamma$ is any connection on $E$. If $E$ is \emph{holomorphic}, we also denote by $\Ch_n(E)$ the class of $\Ch_n(\Gamma)$ in $\Hyp^{2n}(X_{\dot}, \Omega^{\geq n}_{X_{\dot}})$, where $\Gamma$ is any (partially) holomorphic connection.
\end{dfn}

\paragraph*{Characteristic classes of holomorphic vector bundles}
In order to compare our construction of characteristic classes with other approaches, we have to extend the definition of Chern character classes to vector bundles using the results of section \ref{sec:RelationVB}.

Let $\E_{\dot}$ be an arbitrary holomorphic vector bundle of rank $r$ on the simplicial manifold $X_{\dot}$. We construct its Chern character classes as follows: Choose an open covering $U_{\dot} \to X_{\dot}$ such that $\E_{\dot}|_{U_{\dot}}$ is degreewise trivial. Denote by $X'_{\dot}$ the diagonal simplicial manifold of the \v{C}ech nerve $N_{X_{\dot}}(U_{\dot})$. Then $\E|_{X'_{\dot}}$ is degreewise trivial, hence corresponds to a holomorphic $\GL_r(\C)$-bundle $E'/X'_{\dot}$. We define the $n$-th Chern character class $\Ch_n(\E_{\dot}) \in \Hyp^{2n}(X_{\dot},\Omega^{\geq n}_{X_{\dot}})$ of $\E_{\dot}$ to be the inverse image of $\Ch_n(E')$ under the isomorphism $\Hyp^{2n}(X_{\dot},\Omega^{\geq n}_{X_{\dot}}) \xrightarrow{\cong} \Hyp^{2n}(X'_{\dot},\Omega^{\geq n}_{X'_{\dot}})$ of lemma \ref{lemma:covering_iso}.
\begin{lemma}
$\Ch_n(\E_{\dot})  \in \Hyp^{2n}(X_{\dot},\Omega^{\geq n}_{X_{\dot}})$ is well defined.
\end{lemma}
\begin{proof}
Let $V_{\dot} \to X_{\dot}$ be a second open covering such that $\E_{\dot}|_{V_{\dot}}$ is degreewise trivial. Denote the diagonal of the associated \v{C}ech nerve by $X''_{\dot}$ and let $\E_{\dot}|_{X''_{\dot}}$ correspond to the holomorphic $\GL_r(\C)$-bundle $E''$.

Consider a common refinement $W_{\dot}$ of $U_{\dot}$ and $V_{\dot}$, e. g. $W_{\dot} = U_{\dot} \times_{X_{\dot}} V_{\dot}$, and denote the diagonal of the \v{C}ech nerve of $W_{\dot}$ by $X'''_{\dot}$. We have a commutative diagram
\[\xymatrix@R-1cm{
&X'_{\dot} \ar[dr] \\
X'''_{\dot} \ar[dr]\ar[ur] & & X_{\dot},\\
& X''_{\dot}\ar[ur]
}
\]
all maps inducing isomorphisms in cohomology. The pullbacks of $E'/X'_{\dot}$ and $E''/X''_{\dot}$ to $X'''_{\dot}$ both correspond to the vector bundle $\E_{\dot}|_{X'''_{\dot}}$, hence are isomorphic, hence have the same Chern character classes. The claim follows.
\end{proof}
In order to be able to apply the splitting principle later on, we need the
\begin{propn}[Whitney sum formula]
Let $0 \to \E'_{\dot} \to \E_{\dot} \to \E''_{\dot} \to 0$ be a short exact sequence of holomorphic vector bundles on $X_{\dot}$. Then $\Ch_n(\E_{\dot}) = \Ch_n(\E'_{\dot}) + \Ch_n(\E''_{\dot})$.
\end{propn}
\begin{proof}
Without loss of generality we may assume, that $0 \to \E'_0 \to \E_0 \to \E''_0 \to 0$ is a split short exact sequence of free $\O_{X_0}$-modules.
In fact, choose an open covering $U_0 \to X_0$, such that $0 \to \E'_0|_{U_0} \to \E_0|_{U_0} \to \E''_0|_{U_0} \to 0$ is a split short exact sequence of free $\O_{U_0}$-modules. As in the proof of lemma \ref{lemma:open_covering} this induces an open covering $U_{\dot} \to X_{\dot}$ and we denote the diagonal of the corresponding \v{C}ech nerve by $X'_{\dot}$. 
Then $\Ch_n(\E_{\dot})$ maps to $\Ch_n(\E_{\dot}|_{X'_{\dot}})$ under the isomorphism $\Hyp^{2n}(X_{\dot}, \Omega^{\geq n}_{X_{\dot}}) \xrightarrow{\cong} \Hyp^{2n}(X'_{\dot}, \Omega^{\geq n}_{X'_{\dot}})$ and similarly for $\E'_{\dot}, \E''_{\dot}$.

Fix isomorphisms $\psi'^{(0)} : \E'_0 \xrightarrow{\cong} \O_{X_0}^s$, $\psi''^{(0)} : \E''_0 \xrightarrow{\cong} \O_{X_0}^{r-s}$ and a section of the secquence $0 \to \E'_0 \to \E_0 \to \E''_0 \to 0$. This yields compatible isomorphisms:
\[\xymatrix{
0 \ar[r] & \E'_0 \ar[d]^{\psi'^{(0)}}_{\cong} \ar[r] & \E_0 \ar[d]^{\psi^{(0)}}_{\cong} \ar[r] & \E''_0 \ar[d]^{\psi''^{(0)}}_{\cong} \ar[r] & 0\\
0 \ar[r] &\O_{X_0}^s \ar[r] & \O_{X_0}^s \oplus \O_{X_0}^{r-s} \ar[r] & \O_{X_0}^{r-s} \ar[r] & 0.
}
\]
As in the proof of lemma \ref{lemma:degwise_trivial_vb} we get compatible trivializations $\psi'^{(p)}_i, \psi^{(p)}_i$ and $\psi''^{(p)}_i$, $i = 0, \dots, p$, of $\E'_p, \E_p$ and $\E''_p$ respectively. Denote the transition functions $\psi'^{(p)}_i \circ (\psi'^{(p)}_j)^{-1}, \dots$ by $g'_{ij}, \dots$. From the compatibility of the trivializations it follows, that $g_{ij}$ is of the form
\[
\begin{pmatrix}
g'_{ij} & *\\ 0 & g''_{ij}
\end{pmatrix}.
\]
Denote the $\GL_r(\C)$-bundles associated with $\E'_{\dot}, \E_{\dot}$ and $\E''_{\dot}$ by $E', E$ and $E''$ respectively.
The standard connection $\Gamma^E$ on $E$ (cf. example \ref{ex:standard_connection}) is given by the family of matrix valued $1$-forms $\Gamma_i^E = \sum_k x_k g_{ki}^{-1}dg_{ki} = \bigl(\begin{smallmatrix} \Gamma_i^{E'}& * \\ 0 &\Gamma_i^{E''} \end{smallmatrix}\bigr)$ with curvature $R_i^E = d\Gamma_i^E + (\Gamma_i^E)^2 = \bigl(\begin{smallmatrix} R_i^{E'} & *\\ 0 & R_i^{E''}\end{smallmatrix}\bigr).$ Thus $\Ch_n(\Gamma^E) = \frac{1}{n!}\Tr((R^E_i)^n) = \frac{1}{n!} \left(\Tr((R_i^{E'})^n) + \Tr((R_i^{E''})^n)\right) = \Ch_n(\Gamma^{E'}) + \Ch_n(\Gamma^{E''})$ and the claim follows.
\end{proof}

\section{Secondary classes} 
\label{sec:SecondaryClasses}

In this section we construct secondary characteristic classes for holomorphic $\GL_r(\C)$-bundles, which are trivialized as topological bundles. These classes will be of fundamental interest in the following, since they appear in the construction of the relative Chern character on $K$-theory. The idea to consider these secondary classes and their construction is due to Karoubi.
The main technical tool is the homotopy operator from de Rham cohomology:
\begin{lemma}
Let $Y$ be an arbitrary complex (or differentiable) manifold. Let $K : \A^k(Y\times \C) \to \A^{k-1}(Y)$ be the standard homotopy operator of de Rham cohomology given by $\omega \mapsto \int_0^1 (i_{\del/\del t}\omega)dt$, where $t$ is the coordinate on $\C$. Then 
\[
 dK + Kd = i_1^* - i_0^* : \A^k(Y\times \C) \to \A^k(Y),
\]
where $i_0$ resp. $i_1 : Y \hookrightarrow Y \times \C$ denote the embeddings $y \mapsto (y,0)$ resp. $(y,1)$. 

If $f : Z \to Y$ is smooth, then $K \circ (f \times \id_{\C})^* = f^* \circ K : \A^k(Y\times \C) \to \A^{k-1}(Z)$. In particular, $K$ induces a homotopy operator $K:A^k(X_{\dot} \times \C) \to A^{k-1}(X_{\dot})$ for any simplicial manifold $X_{\dot}$, verifying the same properties.
\end{lemma}
\begin{proof}
The first formula is standard and follows by straight forward computation. Also the naturality is an easy consequence of the defining formula.
\end{proof}

\paragraph*{Construction of secondary forms}

Let $E$ and $F$ be two topological bundles  on $X_{\dot}$, given by $g$ and $h : X_{\dot} \rightsquigarrow B_{\dot}\GL_r(\C)$ respectively, with connections $\Gamma^E$ resp. $\Gamma^F$. Assume, that $\alpha$ is a
morphism from $g$ to $h$.
There is a canonical $(2n-1)$-form $\Ch_n^{\rel}(\Gamma^E, \Gamma^F, \alpha)$, whose boundary is $\Ch_n(\Gamma^E) - \Ch_n(\Gamma^F)$. It is constructed as follows:

Let $\pi$ denote the projection $X_{\dot} \times \C \to X_{\dot}$. On $\pi^*E$ we have the connections $\pi^*\Gamma^E$ and $\pi^*\alpha^*\Gamma^F$. We can also consider the connection $\Gamma= t \pi^*\Gamma^E + (1-t)\pi^*\alpha^*\Gamma^F$ on $\pi^*E$, given by the family
\begin{equation}\label{eq:HomotopyConnection}
 \Gamma_i = t(\pi^*\Gamma^E)_i + (1-t)(\pi^*\alpha^*\Gamma^F)_i,
\end{equation}
where $t$ is the coordinate on $\C$. It is easy to see, that this family indeed defines a connection. Then $\Ch_n(\Gamma)$ is a closed $2n$-form on $X_{\dot} \times \C$ and we define 
\[ 
\Ch_n^{\rel}(\Gamma^E, \Gamma^F, \alpha) := K\left(\Ch_n(\Gamma)\right).
\]
\begin{lemma}\label{lem:RelChernCharForm}
 $d\Ch_n^{\rel}(\Gamma^E, \Gamma^F, \alpha) = \Ch_n(\Gamma^E) - \Ch_n(\Gamma^F).$
\end{lemma}
\begin{proof}
Since $\Ch_n(\Gamma)$ is closed we have
\begin{eqnarray*}
dK\Ch_n(\Gamma) &=& i_1^*\Ch_n(\Gamma) - i_0^*\Ch_n(\Gamma) \\
&=& \Ch_n(i_1^*\Gamma) - \Ch_n(i_0^*\Gamma)\\
&=& \Ch_n(\Gamma^E) - \Ch_n(\alpha^*\Gamma^F)\\
&=& \Ch_n(\Gamma^E) - \Ch_n(\Gamma^F). 
\end{eqnarray*}
\end{proof}

This lemma shows in particular, that the class of $\Ch_n(\Gamma^E)$ in $H^{2n}(A^*(X_{\dot})) = H^{2n}(X_{\dot}, \C)$ only depends on the isomorphism class of $E$ and not on the particular connection chosen, thus proving proposition \ref{prop:CharClasses} (ii).  To prove the independence of the holomorphic connection in part (iii) of this proposition we need the following
\begin{lemma}\label{lem:ChRelHolom}
Let $E$ and $F$ be two holomorphic bundles on $X_{\dot}$ with holomorphic connections $\Gamma^E$ resp. $\Gamma^F$, and let $\alpha : E \to F$ be a holomorphic morphism. Then $\Ch_n^{\rel}(\Gamma^E, \Gamma^F, \alpha) \in \Fil^nA^{2n-1}(X_{\dot})$.
\end{lemma}
\begin{proof}
Let $\Fil'^i\A^*(\Delta^p \times X_p \times \C)$ denote the subcomplex of forms, which are locally of the form
\[
\sum_{I, J, K, l, m, |J| \geq i} f_{I,J,K,l,m} dx_I \wedge d\zeta_J \wedge d\bar\zeta_K \wedge dt^l \wedge d\bar t^m,
\]
where $x_0, \dots, x_p$ are the barycentric coordinates on $\Delta^p$, the $\zeta_j$ are holomorphic local coordinates on $X_p$, $t$ is the coordinate on $\C$, $I, J, K$ are multiindices, $dx_I = dx_{i_1}\dots dx_{i_r}$, etc., and $l,m \in \{0,1\}$. The wedge product of forms maps $\Fil'^i \times \Fil'^j$ to $\Fil'^{i+j}$ and the homotopy operator $K$ maps $\Fil'^i\A^k(\Delta^p \times X_p \times \C)$ to $\Fil^i\A^k(\Delta^p \times X_p)$. 

Since $\alpha$ is holomorphic, so is the connection $\alpha^*\Gamma^F$ (cf. the formula in remark \ref{rem:FunctorialityConnection}). Obviously, the matrices
\[
 \Gamma_i = t(\pi^*\Gamma^E)_i + (1-t)(\pi^*\alpha^*\Gamma^F)_i
\]
belong to $\Fil'^1\A^1(\Delta^p \times X_p \times \C; \Mat_r(\C))$, hence the forms $\Tr(R_i^n)$ live in $\Fil'^n\A^{2n}(\Delta^p\times X_p \times \C)$ and $K(\Tr(R_i^n)) \in \Fil^n\A^{2n-1}(\Delta^p\times X_p)$ and the claim follows.
\end{proof}

We have the following naturality property for secondary forms:
Let again $E,F$ be topological bundles on $X_{\dot}$ with connections $\Gamma^E$ resp. $\Gamma^F$ and in addition let $f: Y_{\dot} \rightsquigarrow X_{\dot}$ be a topological morphism. Then we can consider the pullbacks $f^*E, f^*F$ with connections $f^*\Gamma^E, f^*\Gamma^F$ and the morphism $f^*\alpha: E \to F$ given by $\alpha \circ f : Y_{\dot} \to E_{\dot}\GL_r(\C)$.
\begin{lemma} \label{lem:PullbackChRel}
$\Ch_n^{\rel}(f^*\Gamma^{E}, f^*\Gamma^{F}, f^*\alpha) = f^*\Ch_n^{\rel}(\Gamma^E, \Gamma^F,\alpha)$.
\end{lemma}
\begin{proof}
Denote the projections $Y_{\dot} \times \C \to Y_{\dot}$ and $X_{\dot} \times \C \to X_{\dot}$ by $\pi_Y$ and $\pi_X$ respectively. Then
\begin{eqnarray*}
 &&f^*\Ch_n^{\rel}(\Gamma^E, \Gamma^F, \alpha) =\\
&=& f^*K\left(\Ch_n(t\pi_X^*\Gamma^E + (1-t)\pi_X^*(\alpha^*\Gamma^F))\right)\\
&=& K\left( (f\times\id_{\C})^*\Ch_n(t\pi_X^*\Gamma^E + (1-t)\pi_X^*(\alpha^*\Gamma^F))\right)\\
&=& K\left( \Ch_n(t(f\times\id_{\C})^*\pi_X^*\Gamma^E + (1-t)(f\times\id_{\C})^*\pi_X^*(\alpha^*\Gamma^F))  \right)\\
&=& K\left( \Ch_n(t\pi_Y^*f^*\Gamma^E + (1-t)\pi_Y^*f^*(\alpha^*\Gamma^F)) \right)\\
&=& K\left( \Ch_n(t\pi_Y^*f^*\Gamma^E + (1-t)\pi_Y^*((f^*\alpha)^*(f^*\Gamma^F))) \right)\\
&=& \Ch_n^{\rel}(f^*\Gamma^E, f^*\Gamma^F, f^*\alpha).
\end{eqnarray*}
Here we used the fact, that $f^*(\alpha^*\Gamma^F) = (f^*\alpha)^*(f^*\Gamma^F)$. In fact, the connection on the left hand side is given by 
$
 (\mathrm{pr}_{\Delta^p}, f_p)^*(\alpha_i^{-1}d\alpha_i + \alpha_i^{-1}\Gamma^F_i\alpha_i),
$
which is equal to $(f^*\alpha)_i^{-1} d(f^*\alpha)_i + (f^*\alpha)_i^{-1}(f^*\Gamma^F)_i (f^*\alpha)_i$, that is, to the family defining the connection on the right hand side.
\end{proof}

\paragraph*{Secondary classes}
Now let $E$ and $F$ be two \emph{holomorphic} bundles on $X_{\dot}$ and $\alpha : E \to F$ a morphism of the underlying \emph{topological} bundles. Choose holomorphic connections $\Gamma^E$ and $\Gamma^F$ on $E$ and $F$, respectively. Then $\Ch_n^{\rel}(\Gamma^E, \Gamma^F, \alpha)$ is closed modulo $\Fil^nA^{*}(X_{\dot})$ and we can consider its cohomology class in $H^{2n-1}(A^*(X_{\dot})/\Fil^nA^*(X_{\dot})) = \Hyp^{2n-1}(X_{\dot}, \Omega^{<n}_{X_{\dot}})$.
\begin{propn}\label{prop:RelativeClasses}
The class of $\Ch_n^{\rel}(\Gamma^E, \Gamma^F, \alpha)$ in $\Hyp^{2n-1}(X_{\dot}, \Omega^{<n}_{X_{\dot}})$ does not depend on the holomorphic connections chosen. We denote this class by $\Ch_n^{\rel}(E, F, \alpha)$.
\end{propn}
\begin{proof}
Let $\widetilde\Gamma^E$ and $\widetilde\Gamma^F$ be other holomorphic connections on $E$ resp. $F$.
Denote by $\pi$ the projection $X_{\dot} \times \C \times \C \to X_{\dot}$, by $s,t$ the variables on $\C\times \C$ and consider the connection 
\[
\Gamma_{s,t} = (1-s)((1-t)\pi^*\Gamma^E + t \pi^*\alpha^*\Gamma^F) + s ((1-t)\pi^*\widetilde\Gamma^E + t\pi^*\alpha^*\widetilde\Gamma^F)
\]
on $\pi^*E$. Denote by $K_s$ resp. $K_t$ the homotopy operators with respect to $s$ resp. $t$ and consider the form
$K_s\circ K_t(\Ch_n(\Gamma_{s,t}))$. Then
\begin{eqnarray*}
&&d(K_s\circ K_t(\Ch_n(\Gamma_{s,t})))=\\
&=&  K_t(\Ch_n(\Gamma_{s,t}))|_{s=1} - K_t(\Ch_n(\Gamma_{s,t}))|_{s=0} - K_s(dK_t(\Ch_n(\Gamma_{s,t})))\\
&=& \Ch_n^{\rel}(\widetilde\Gamma^E, \widetilde\Gamma^F, \alpha) - \Ch_n^{\rel}(\Gamma^E, \Gamma^F, \alpha) - K_s(\Ch_n(\Gamma_{s,1}) - \Ch_n(\Gamma_{s,0}))\\
&=& \Ch_n^{\rel}(\widetilde\Gamma^E, \widetilde\Gamma^F, \alpha) - \Ch_n^{\rel}(\Gamma^E, \Gamma^F, \alpha) \\
&&- \Ch_n^{\rel}(\widetilde\Gamma^E, \Gamma^E, \id) + \Ch_n^{\rel}(\alpha^*\widetilde\Gamma^F, \alpha^*\Gamma^F, \id).
\end{eqnarray*}
Using remark \ref{rem:FunctorialityConnection} (i), it is not hard to see, that $\Ch_n^{\rel}(\alpha^*\widetilde\Gamma^F, \alpha^*\Gamma^F, \id) =\Ch_n^{\rel}(\widetilde\Gamma^F, \Gamma^F, \id)$. Then it follows from lemma \ref{lem:ChRelHolom}, that the last two summands above lie in $\Fil^n A^{2n-1}(X_{\dot})$, thus proving the claim.
\end{proof}

\paragraph*{Topologically trivialized bundles}
Now we specialize to the case, of a $\GL_r(\C)$-bundle together with a trivialization of the underlying topological bundle. On any simplicial manifold $Y_{\dot}$, we denote by $T$ or $T_r$ the trivial $\GL_r(\C)$-bundle, given by the constant map $1:Y_{\dot} \to \{1\} \in B_{\dot}\GL_r(\C)$.

\smallskip

Let $X_{\dot}$ be a simplicial complex manifold and $E/X_{\dot}$ a topological $\GL_r(\C)$-bundle given by $g: X_{\dot} \rightsquigarrow B_{\dot}\GL_r(\C)$. Then the morphisms of $\GL_r(\C)$-bundles $\alpha : T \to E$ are exactly the topological morphims $\alpha : X_{\dot} \rightsquigarrow E_{\dot}\GL_r(\C)$, such that $\alpha\cdot 1 = g$, or, equivalently, $p\circ\alpha=g$, where $p: E_{\dot}\GL_r(\C) \to B_{\dot}\GL_r(\C)$ denotes the projection.

This may be reformulated as follows:
The (holomorphic) $\GL_r(\C)$-bundle $p^*E^{\univ}$ given by $p : E_{\dot}\GL_r(\C) \to B_{\dot}\GL_r(\C)$ together with the trivialization $\id: E_{\dot}\GL_r(\C) \to E_{\dot}\GL_r(\C)$ is universal for bundles $E$ together with a trivialization $\alpha: T \to E$. In fact, if $E$ is any topological bundle on $X_{\dot}$ and $\alpha : T \to E$ a trivialization, then the pair $(E, \alpha)$ is the pullback $(\alpha^*p^*E^{\univ}, \alpha^*(\id_{E_{\dot}\GL_r(\C)}))$.

In particular, if we equip all bundles with the standard connection and write $\Ch_n^{\rel,\univ} := \Ch_n^{\rel}(\Gamma^T, \Gamma^{p^*E^{\univ}}, \id_{E_{\dot}\GL_r(\C)}) \in A^{2n-1}(E_{\dot}\GL_r(\C))$, then we have the following consequence of lemma \ref{lem:PullbackChRel}:
\begin{propn}\label{prop:ChRelUniv}
If $E/X_{\dot}$ is a topological $\GL_r(\C)$-bundle together with a trivialization $\alpha : T \to E$, then
\[
 \Ch_n^{\rel}(\Gamma^T, \Gamma^E, \alpha) = \alpha^*\Ch_n^{\rel, \univ},
\]
$\Gamma^T, \Gamma^E$ denoting the standard connections on $T, E$ respectively.
\end{propn}
\begin{rem}
This description will be needed in proposition \ref{prop:ComparisonOfRelClasses} to compare the class $\Ch_n^{\rel}(T,E,\alpha)$ for an algebraic bundle $E$ with a topological trivialization $\alpha$ with the class $\widetilde\Ch_n^{\rel}(T,E,\alpha)$ constructed by a completely different strategy. It will be this latter class, that can be compared with the Deligne-Beilinson Chern character class $\Ch_n^{\D}(E)$. Note, that this kind of ``universal'' description of relative Chern character forms, would not be possible in Karoubi's setting of bundles on simplicial sets.
\end{rem}

\chapter{Characteristic classes of algebraic bundles}
\label{ch:CharClassesAlg}

The heart of this chapter is the comparison of relative and Deligne-Beilinson Chern character classes in the last section. To do this, we first construct a refinement of the secondary classes of section \ref{sec:SecondaryClasses} for an \emph{algebraic} bundle on a simplicial variety $X_{\dot}$ together with a topological trivialization. These classes live in $H^{2n}(X_{\dot}, \C)/\Fil^nH^{2n}(X_{\dot},\C)$. Using the so called \emph{refined Chern character classes} constructed in section \ref{sec:RelativeChernCharClasses}, the comparison will be reduced to the comparison of (primary) Chern character classes, which is done in section \ref{sec:ChernClAlgBdl}. 
The first section recalls the definition of the Hodge filtration on the cohomology of a simplicial variety.

\section{Preliminaries}

By a \emph{variety} we will mean a smooth separated scheme of finite type over $\Spec(\C)$. We will always equip a variety with the analytic topology, and $\O_X, \Omega^*_X$ will denote the sheaves of holomorphic functions and differential forms respectively. 

Recall the following definitions and properties (see e. g. \cite[\S 3]{HodgeII}):
Let $X$ be a variety. According to Nagata and Hironaka there exists an open immersion $j : X \hookrightarrow \ol X$ in a smooth proper variety $\ol X$, such that $D = \ol X - X$ is a divisor with normal crossings. Such compactifications are called \emph{good compactifications}. The \emph{logarithmic de Rham complex}, denoted $\Omega^*_{\ol X}(\log D)$, is the smallest subcomplex of $j_*\Omega^*_X$, which contains $\Omega^*_{\ol X}$, is stable under the exterior product, and such that $df/f$ is a local section of $\Omega^*_{\ol X}(\log D)$, whenever $f$ is a local section of $j_*\O^*_X$, meromorphic along $D$.

There are quasiisomorphisms
\[
\Rhyp j_*\C \to \Rhyp j_*\Omega^*_X = j_*\Omega^*_X \leftarrow \Omega^*_{\ol X}(\log D)
\]
and in particular $H^*(X, \C) = \Hyp^*(\ol X, \Omega^*_{\ol X}(\log D))$. The complex $\Omega^*_{\ol X}(\log D)$ is filtered by the subcomplexes $\Fil^n\Omega^*_{\ol X}(\log D) := \Omega^{\geq n}_{\ol X}(\log D)$ and the natural maps $\Hyp^*(\ol X, \Fil^n\Omega^*_{\ol X}(\log D)) \to \Hyp^*(\ol X, \Omega^*_{\ol X}(\log D))$ are injective. The image of this map is by definition $\Fil^nH^*(X, \C)$, the $n$-th step of the \emph{Hodge filtration}. This definition does not depend on the chosen compactification and is functorial in $X$ \cite[(3.2.11)]{HodgeII}.

\smallskip
The cohomology of the complexes $\Fil^n\Omega^*_{\ol X}(\log D)$ may also be computed using $\mathscr C^{\infty}$-forms \cite[(3.2.3)]{HodgeII}: Write $\A^{p,q}_{\ol X}(\log D)$ for the subsheaf $\Omega^p_{\ol X}(\log D) \otimes_{\O_{\ol X}} \A^{0,q}_{\ol X}$ of $j_*\A^{p,q}_X$. This is a sheaf of $\mathscr C^{\infty}_{\ol X}$-modules, hence fine, hence acyclic for the global sections functor.
Let $\Fil^n\A^*_{\ol X}(\log D)$ be the complex of sheaves $\bigoplus_{p+q=*, p\geq n} \A^{p,q}_{\ol X}(\log D)$ (subcomplex of the total complex of $j_*\A^{*,*}_X$) and denote its global sections by $\Fil^n\A^*(\ol X, \log D)$. Then the natural map $\Fil^n\Omega^*_{\ol X}(\log D) \hookrightarrow \Fil^n\A^*_{\ol X}(\log D)$ is a quasiisomorphism and we get the isomorphisms 
\[
\Hyp^*(\ol X, \Fil^n\Omega^*_{\ol X}(\log D)) \cong H^*(\Fil^n\A^*(\ol X, \log D)) \cong \Fil^nH^*(X, \C).
\]

Now let $X_{\dot}$ be a simplicial variety. The mixed Hodge structure on $H^*(X_{\dot}, \Z)$ (in particular the Hodge filtration on $H^*(X_{\dot}, \C)$) may be constructed as follows (cf. \cite[1.2]{Soule}): Denote by $\Delta^{\mathrm{str}}$ the subcategory of the simplicial category $\Delta$ with the same objects $[p] = \{0, 1, \dots, p\}$, $p \geq 0$, but where the morphisms $[p] \to [q]$ are the \emph{strictly} increasing maps. A strict (co)simplicial object in a category $\mathscr C$ is a (covariant) contravariant functor $\Delta^{\mathrm{str}} \to \mathscr C$. Thus, a strict simplicial object is a ``simplicial object without degeneracies''. In particular, every simplicial object gives a strict simplicial object simply by restricting to the subcategory $\Delta^{\mathrm{str}}$. See the appendix \ref{app:StrictSimplicial} for facts on the cohomology of strict simplicial spaces.

For any strict simplicial variety $X_{\dot}$ one can inductively construct an open immersion $j : X_{\dot} \hookrightarrow \ol X_{\dot}$ into a proper\footnote{i.e. each $\ol X_p$ is proper over $\Spec(\C)$} strict simplicial variety $\ol X_{\dot}$, such that the complement $D_p := \ol X_p - X_p$ is a divisor with normal crossings for each $p$ \cite[1.2]{Soule}. Again we call $j$ a \emph{good compactification}.

Exactly as in \cite[(8.1.19)]{HodgeIII} (note, that (8.1.12), (8.1.13) and Th\'eor\`eme (8.1.15) of {\it loc. cit.} work equally well in the strict simplicial context) one constructs a functorial mixed Hodge structure on $H^*(X_{\dot}, \Z)$, which is independent of the chosen good compactification (cf. {\it loc. cit.} (8.3.3), in fact the verifications are even easier in the strict case).\footnote{If $j: X_{\dot}\hookrightarrow \ol X_{\dot}$ is a morphism of \emph{simplicial} varieties, this mixed Hodge structure (MHS) obviously coincides with the one constructed in {\it loc. cit.}. Hence they also coincide for general simplicial varieties (which might not admit a simplicial good compactification), as one sees considering a proper hypercovering $p: Z_{\dot} \to X_{\dot}$ as in {\it loc. cit.} (8.3.2) and noting that $p^*$ is a bijective morphism of MHSs for both possible MHSs on $H^*(X_{\dot}, \Z)$.} The Hodge filtration is again given as the image of the (injective) map $\Hyp^*(\ol X_{\dot}, \Omega^{\geq n}_{\ol X_{\dot}}(\log D_{\dot})) \hookrightarrow \Hyp^*(\ol X_{\dot}, \Omega^*_{\ol X_{\dot}}(\log D_{\dot})) = H^*(X_{\dot}, \C)$ and may be computed as the cohomology of the complex $\Tot\,\Fil^n\A^*(\ol X_{\dot}, \log D_{\dot})$.

To simplify notation we will often simply write $\Fil^n\A^*(\ol X_{\dot}, \log D_{\dot})$ and it will be clear from the context if this denotes the cosimplicial complex or the associated total complex.

\medskip

Note, that we have natural maps $\Omega^{\geq n}_{\ol X_{\dot}}(\log D_{\dot}) \to j_*\Omega^{\geq n}_{X_{\dot}}$. These induce morphisms
\[
\Hyp^*(\ol X_{\dot}, \Omega^{\geq n}_{\ol X_{\dot}}(\log D_{\dot})) = \Fil^nH^*(X_{\dot}, \C) \to \Hyp^*(X_{\dot}, \Omega^{\geq n}_{X_{\dot}}),
\]
which are obviously injective.

\begin{rem}
In the study of Chern character maps on higher $K$-theory, there naturally occur simplicial schemes of the form $X_{\dot} = X \otimes S$, where $X$ is a variety and $S$ a simplicial set. These are in general not of finite type.
Nevertheless, they admit a good compactifiaction $\ol X_{\dot}$ defined as $\ol X \otimes S$, where $X \overset{}{\hookrightarrow}  \ol X$ is a good compactification, and we can still consider the map $\Hyp^*(\ol X_{\dot}, \Omega^{\geq n}_{\ol X_{\dot}}(\log D_{\dot})) \to \Hyp^*(X_{\dot}, \Omega^*_{X_{\dot}})=H^*(X_{\dot}, \C)$. It follows from (an analogue of) lemma \ref{lemma:cohom_X_tensor_S}, that $\Hyp^k(\ol X_{\dot}, \Omega^{\geq n}_{\ol X_{\dot}}(\log D_{\dot})) \cong \bigoplus_{p+q=k} \Hom(H_p(S), \Hyp^q(\ol X, \Omega^{\geq n}_{\ol X}(\log D)))$ and similar for $H^k(X_{\dot}, \C)$, and one sees, that the above map is still injective. We will denote its image by $\Fil^nH^*(X_{\dot}, \C)$ also in this case.
\end{rem}

\begin{conv}
By a good compactification we will always mean a good compactification in the strict simplicial sense. Also, if no confusion can arise, we will denote the complement of any good compactification by $D_{\dot}$.
\end{conv}

\section{Chern classes of algebraic bundles}
\label{sec:ChernClAlgBdl}

Let $E$ be an algebraic $\GL_r(\C)$-bundle on the simplicial variety $X_{\dot}$, i.e. a morphism of simplicial varieties $g : X_{\dot} \to B_{\dot}\GL_r(\C)$. 
Since $E$ may be viewed as a holomorphic bundle, we have the classes $\Ch_n(E) \in \Hyp^{2n}(X_{\dot}, \Omega^{\geq n}_{X_{\dot}})$ constructed using Chern-Weil theory. On the other hand, one may also construct Chern (character) classes in $\Fil^nH^{2n}(X_{\dot}, \C)$ in the style of Grothendieck and Hirzebruch. We recall the construction, and show that these are (up to a sign) mapped to our classes under the natural map $\Fil^nH^{2n}(X_{\dot}, \C) \to \Hyp^{2n}(X_{\dot}, \Omega^{\geq n}_{X_{\dot}})$.

\subsection{The first Chern class of a line bundle}

Let $X$ be a complex manifold, or more generally a simplicial complex manifold. The group of isomorphism classes of \emph{holomorphic} line bundles on $X$ is $H^1(X, \O_X^*)$. 

\begin{dfn}
The first Chern class $c_1 : H^1(X, \O_X^*) \to \Hyp^2(X, \Omega^{\geq 1}_X)$ is the map on cohomology induced by the morphism of complexes $d\log : \O_X^*[-1] \to \Omega^{\geq 1}_X$.
\end{dfn}
\begin{lemma}
If $\mathscr L$ is an \emph{algebraic} line bundle on the \emph{variety} $X$, then $c_1(\mathscr L) \in \Fil^1H^2(X, \C) \subset \Hyp^2(X, \Omega^{\geq 1}_X)$.
\end{lemma}
\begin{proof}
By lemmata \ref{lemma:covering_iso} and \ref{lemma:open_covering} (keeping in mind, that a bijective morphism of mixed Hodge structures is an isomorphism) we may assume, that $X=X_{\dot}$ is a simplicial variety and that $\mathscr L$ is classified by a morphism of simplicial varieties $g^{(\dot)}: X_{\dot} \to B_{\dot}\mathbb G_m(\C)$. Then $g^{(1)} \in \Gamma(X_1, \O_{X_1}^*)$ represents a class in $H^1(\Gamma^*(X_{\dot}, \O_{X_{\dot}}^*))$, whose image in $H^1(X_{\dot}, \O_{X_{\dot}}^*)$ is the class of $\mathscr L$ (cf. remark \ref{rem:CohomInterpretation}). Thus $c_1(\mathscr L) \in \Hyp^2(X_{\dot}, \Omega^{\geq 1}_{X_{\dot}}) = H^2(\Tot \Fil^1\A^*(X_{\dot}))$ is the class represented by $(d\log(g^{(1)}))\oplus 0 \in \Gamma(X_1, \Omega^1_{X_1}) \oplus \Gamma(X_0, \Omega^2_{X_0}) \subset \Fil^1\A^1(X_1) \oplus \Fil^1\A^2(X_0)$. Let $\ol X_1$ be any good compactification of $X_1$, then $g^{(1)}$, being algebraic, is meromorphic along $\ol X_1 - X_1$, hence $d\log(g^{(1)}) \in \Fil^1\A^1(X_1, \log(\ol X_1 - X_1))$.
\end{proof}

With the above normalization, the first Chern class is the negative of the first Chern character class constructed in section \ref{sec:connections}:
\begin{lemma}\label{lemma:ChLineBundle}
Let $X_{\dot}$ be a simplicial complex manifold and $\mathscr L$ a holomorphic line bundle on $X_{\dot}$. Then 
\[
\Ch_1(\mathscr L_{\dot}) = -c_1(\mathscr L_{\dot})
\]
in $\Hyp^2(X_{\dot}, \Omega^{\geq 1})$.
\end{lemma}
\begin{proof}
Again, we may assume that $\mathscr L_{\dot}$ is classified by a holomorphic morphism of simplicial manifolds $g^{(\dot)} : X_{\dot} \to B_{\dot}\mathbb G_m(\C)$. Then $\Ch_1(\mathscr L_{\dot})$ can be computed explicitely: We equip the $\mathbb G_m(\C)$-bundle $L$ classified by $g^{(\dot)}$ with the standard connection, given by the family of matrices $\Gamma_i^{(p)} = \sum_{k=0}^p x_k (g^{(p)}_{ki})^{-1}dg^{(p)}_{ki} = \sum_k x_k d\log(g^{(p)}_{ki})$, where the notations are as in section \ref{sec:connections}. The curvature is then given by $R_i^{(p)} = \sum_k dx_k d\log(g^{(p)}_{ki}) + \sum_{k,l} x_kx_l d\log(g^{(p)}_{ki}) d\log(g^{(p)}_{li})$. This form does not depend on $i$, and the first Chern character form $\Ch_1(L)$ of $L$ in $\Fil^1 A^2(X_{\dot})$ is given by the family $(\Ch_1(L)_p)_{p\geq 0} = (R_i^{(p)})_{p\geq 0}$.

The isomorphism $H^2(\Fil^1 A^*(X_{\dot})) \to H^2(\Fil^1 \A^*(X_{\dot})) = \Hyp^2(X_{\dot}, \Omega^{\geq 1}_{X_{\dot}})$ is given by $\omega = (\omega_p)_{p\geq 0} \mapsto (\int_{\Delta^1} \omega_1, \int_{\Delta^0} \omega_0) \in \Fil^1 \A^1(X_1) \oplus \Fil^1 \A^2(X_0)$.

Since $g^{(p)}_{ii}$ is the constant map $1$, $d\log(g^{(p)}_{ii})=0$ for all $p\geq 0, i=0, \dots p$ and in particular $\Ch_1(L)_0=0$. Next, $\Ch_1(L)_1= R_1^{(1)}=dx_0 d\log(g^{(1)}_{01})$, and hence
$\int_{\Delta^1} \Ch_1(L)_1 = -d\log(g^{(1)}_{01}) = -d\log(g^{(1)})$. Comparing with the computation in the proof of the last lemma, this concludes the proof.
\end{proof}

\subsection{Higher Chern classes}\label{sec:HigherChernClasses}
These are constructed in the style of Grothendieck using the splitting principle.

We begin with the case of holomorphic vector bundles on arbitrary complex manifolds. Thus let $\E$ be a holomorphic vector bundle of rank $r$ on a complex manifold $X$. Denote by $\pi : \P(\E) \to X$ the associated projective bundle   and by $\O(1)$ the tautological line bundle on $\P(\E)$. Write $\xi := c_1(\O(1)) \in \Hyp^2(\P(\E), \Omega^{\geq 1}_\P(\E))$.

\begin{lemma}\label{lem:CohomProjBundle1}
The map
\[
\bigoplus_{i=0}^{r-1} \pi^*(\underline{\;\;})\cup \xi^i : \bigoplus_{i=0}^{r-1} \Hyp^{m-2i}(X, \Omega^{\geq n-i}_X) \to \Hyp^m(\P(\E), \Omega^{\geq n}_{\P(\E)})
\]
is an isomorphism.
\end{lemma}
\begin{proof}
By abuse of notation, we still denote by $\xi$ the image of $c_1(\O(1))$ in $H^1(\P(\E), \Omega^1_{\P(\E)}) = \Hyp^1(X, \Rhyp\pi_*\Omega^1_{\P(\E)})$. In the derived category $D^+(X)$ of bounded below complexes of abelian sheaves on $X$ we have the isomorphism
\[
\bigoplus_{i=0}^{r-1} \Omega^{p-i}_X[-i] \xrightarrow{\bigoplus_{i=0}^{r-1} \pi^*(\underline{\;\;})\cup \xi^i} \Rhyp\pi_*\Omega^p_{\P(\E)}
\]
\cite[Th\'eor\`eme 2]{Verdier}. Thus we have isomorphisms $\bigoplus_{i=0}^{r-1} H^{q-i}(X, \Omega^{p-i}_X) \xrightarrow{\cong} H^q(\P(\E), \Omega^p_{\P(\E)})$. 
On the other hand, we have the hypercohomology spectral sequences
\begin{align*}
E_1^{p,q} = H^q(\P(\E), (\Omega^{\geq n}_{\P(\E)})^p) &\Longrightarrow \Hyp^{p+q}(\P(\E), \Omega^{\geq n}_{\P(\E)}) \quad\text{ and}\\
E_1^{p-i,q-i}= H^{q-i}(X, (\Omega^{\geq n-i}_X)^{p-i}) &\Longrightarrow \Hyp^{p+q-2i}(X, \Omega^{\geq n-i}).
\end{align*}
Similar as in the proof of \cite[Sublemma 2.5]{GilletRR}, one shows, that the map in the statement of the lemma is the abutment of a map of spectral sequences (after suitable reindexing), which on the $E_1$ terms is the isomorphism
\[
\bigoplus_{i=0}^{r-1} H^{q-i}(X, (\Omega^{\geq n-i}_X)^{p-i}) \xrightarrow{\cong} H^q(\P(\E), (\Omega^{\geq n}_{\P(\E)})^p),
\]
and thus is an isomorphism, too.
\end{proof}

By a spectral sequence argument as in \cite[Lemma 2.4]{GilletRR}  this extends to simplicial complex manifolds:
\begin{lemma}\label{lem:CohomProjBundle}
Let $X_{\dot}$ be a simplicial complex manifold and $\E_{\dot}$ a holomorphic vector bundle of rank $r$ on $X_{\dot}$. Then the map
\[
\sum_{i=0}^{r-1} \pi^*(\underline{\;\;})\cup \xi^i : \bigoplus_{i=0}^{r-1} \Hyp^{m-2i}(X_{\dot}, \Omega^{\geq n-i}_{X_{\dot}}) \to \Hyp^m(\P(\E_{\dot}), \Omega^{\geq n}_{\P(\E_{\dot})})
\]
is an isomorphism.
\end{lemma} 

Now assume that $X_{\dot}$ is a simplicial \emph{variety} and that $\E_{\dot}$ is an \emph{algebraic} vector bundle on $X_{\dot}$. Then $\xi = c_1(\O(1)) \in \Fil^1H^2(\P(\E_{\dot}), \C)$ and we have
\begin{lemma}\label{lem:HodgeFiltProjBundle}
The map
\[
\sum_{i=0}^{r-1} \pi^*(\underline{\;\;})\cup \xi^i : \bigoplus_{i=0}^{r-1} \Fil^{n-i}H^{m-2i}(X_{\dot}, \C) \to \Fil^nH^m(\P(\E_{\dot}), \C)
\]
is an isomorphism.
\end{lemma}
\begin{proof}
As before, the simplicial case follows from the classical case. Hence let $\E$ be an algebraic vector bundle on the variety $X$. Then $\xi = c_1(\O(1)) = c_1^{\top}(\O(1))$ in $H^2(\P(\E), \C)$ (see below), where $c_1^{\top}(\O(1)) \in H^2(\P(\E), \Z(1))$ is the first topological Chern class of $\O(1)$. Then $\xi^i$ may be seen as a morphism of mixed Hodge structures $\Q(-i) \to H^{2i}(\P(\E), \Q)$ \cite[Corollaire (9.1.3)]{HodgeIII}. Thus the classical Leray-Hirsch isomorphism
\[
\sum_{i=0}^{r-1} \pi^*(\underline{\;\;})\cup \xi^i\colon \bigoplus_{i=0}^{r-1} H^{m-2i}(X, \Q)\otimes\Q(-i) \xrightarrow{\cong} H^m(\P(\E), \Q)
\]
is a morphism of mixed Hodge structures and the result follows by looking at $\Fil^n$.
\end{proof}

The higher Chern classes $c_n(\E) \in \Fil^nH^{2n}(X_{\dot}, \C)$ are now defined by the equation
\[
\sum_{i=0}^r \pi^*(c_{r-i}(\E_{\dot})) \cup c_1(\O_{\P(\E_{\dot})}(1))^i = 0
\]
and the condition $c_n(\E_{\dot}) = 0$ if $n > r$, $c_0(\E_{\dot}) = 1$.

As usual, one defines Chern character classes: Let $N_n \in \Z[X_1, \dots, X_n]$ be the $n$-th Newton polynomial, defined by $N_n(\sigma_1, \dots, \sigma_n) = Y_1^n + \dots + Y_n^n$, where $\sigma_i$ denotes the $i$-th elementary symmetric function in the indeterminates $Y_1, \dots, Y_n$. If now $\E_{\dot}$ is an algebraic vector bundle on the simplicial variety $X_{\dot}$ as above, 
\[
\widetilde\Ch_n(\E_{\dot}) := \frac{1}{n!} N_n(c_1(\E_{\dot}), \dots, c_n(\E_{\dot})) \in \Fil^nH^{2n}(X_{\dot}, \C).
\]

The theory of Chern classes and Chern character classes obtained in this way has the usual properties. In particular they are functorial and the Whitney sum formula holds \cite{GrothendieckClassesDeChern}.

\begin{prop}\label{prop:AlgChernCharClasses}
Let $\E_{\dot}$ be an algebraic vector bundle on the simplicial variety $X_{\dot}$. The natural morphism $\Fil^nH^{2n}(X_{\dot}, \C) \to \Hyp^{2n}(X_{\dot}, \Omega^{\geq n}_{X_{\dot}})$
maps $\widetilde\Ch_n(\E_{\dot})$ to $(-1)^n\Ch_n(\E_{\dot})$.
\end{prop}
\begin{proof}
Repeated use of the projective bundle construction gives a morphism of simplicial varieties $\pi : Q_{\dot} \to X_{\dot}$, such that $\pi^*\E_{\dot}$ has a filtration, whose subquotients are line bundles, and such that both maps $\pi^*: \Fil^nH^{2n}(X_{\dot}, \C) \to \Fil^nH^{2n}(Q_{\dot}, \C)$ and $\pi^*: \Hyp^{2n}(X_{\dot}, \Omega^{\geq n}_{X_{\dot}}) \to \Hyp^{2n}(Q_{\dot}, \Omega^{\geq n}_{Q_{\dot}})$ are injective (lemmata \ref{lem:CohomProjBundle} and \ref{lem:HodgeFiltProjBundle}).

By the Whitney sum formula it is thus enough to show, that for a line bundle $\mathscr L_{\dot}$, $\widetilde\Ch_n(\mathscr L_{\dot})$ maps to $(-1)^n\Ch_n(\mathscr L_{\dot})$.
But $\widetilde\Ch_n(\mathscr L_{\dot})$ is just $\frac{1}{n!}c_1(\mathscr L_{\dot})^n$ and similarly  $\Ch_n(\mathscr L_{\dot}) = \frac{1}{n!}\left(\Ch_1(\mathscr L_{\dot})\right)^n$. Indeed, for the Chern character classes $\widetilde\Ch_n$ this follows from the explicit form of the Newton polynomials and the fact that $c_i(\mathscr L_{\dot}) = 0$ if $i > 1$, while for the classes $\Ch_n(\mathscr L_{\dot})$ it follows directly from the construction. Hence the claim follows from lemma \ref{lemma:ChLineBundle}.
\end{proof}

In particular we see, that the Chern character classes $\Ch_n(\E_{\dot})$ of an algebraic vector bundle indeed lie in $\Fil^nH^{2n}(X_{\dot}, \C) \subset \Hyp^{2n}(X_{\dot}, \Omega^{\geq n}_{X_{\dot}})$.

\section{Relative Chern character classes}
\label{sec:RelativeChernCharClasses}

In this section we construct refinements of the secondary classes constructed in Proposition \ref{prop:RelativeClasses} for algebraic bundles together with a trivialization of the associated topological bundle, which take the Hodge filtration into account.

Let $E$ be an algebraic $\GL_r(\C)$-bundle on the simplicial variety $X_{\dot}$ classified by $g : X_{\dot} \to B_{\dot}\GL_r(\C)$. Define the principal bundle $E_{\dot} \xrightarrow{p} X_{\dot}$ associated with $E$ by the pullback diagram
\[
\xymatrix@C+0.3cm{
E_{\dot} \ar[r]\ar[d]_p\ar@{}[dr]|{\lrcorner} & E_{\dot}\GL_{r}(\C) \ar[d]^p\\
X_{\dot} \ar[r]^-{g} & B_{\dot}\GL_{r}(\C).
}
\]
Choose a good compactification $j: X_{\dot} \hookrightarrow \ol X_{\dot}$ of strict simplicial varieties and write $D_p = \ol X_p - X_p$. We define the complex $\Fil^nA^*(\ol X_{\dot}, \log D_{\dot})$ as the quasi-pullback of the diagram
\[\xymatrix{
& A^*(X_{\dot})\ar[d]^I_{\text{qis}}\\
\Fil^n\A^*(\ol X_{\dot}, \log D_{\dot}) \ar[r]^-{\iota_{\A}} & \A^*(X_{\dot}),
}
\]
see Appendix \ref{app:HomolAlg}.
Then the natural projection $\Fil^nA^*(\ol X_{\dot}, \log D_{\dot}) \to \Fil^n\A^*(\ol X_{\dot}, \log D_{\dot})$ is a quasiisomorphism and the diagram
\[\xymatrix{
\Fil^nA^*(\ol X_{\dot}, \log D_{\dot}) \ar[d]_{\text{qis}} \ar[r]^-{\iota_A} & A^*(X_{\dot})\ar[d]^I_{\text{qis}}\\
\Fil^n\A^*(\ol X_{\dot}, \log D_{\dot}) \ar[r]^-{\iota_{\A}} & \A^*(X_{\dot}),
}
\]
is commutative up to canonical homotopy.

\begin{dfn}
Define \emph{relative cohomology groups}
\begin{align*}
&H^{E,*}_{\rel}(X_{\dot}, n) := H^*\left(\Cone(\Fil^nA^*(\ol X_{\dot}, \log D_{\dot}) \xrightarrow{p^*\circ\iota_A} A^*(E_{\dot}))\right) \quad\text{and}\\
&H^*_{\rel}(X_{\dot}, n) := H^*\left(\Cone(\Fil^nA^*(\ol X_{\dot}, \log D_{\dot}) \xrightarrow{\iota_A} A^*(X_{\dot}))\right).
\end{align*}
\end{dfn}
Note, that $H^*_{\rel}(X_{\dot}, n) \cong H^*(X_{\dot}, \C)/\Fil^nH^*(X_{\dot}, \C)$. As for the Hodge filtration one shows that this definition is (up to isomorphism) independent of the chosen good compactification.\footnote{To get canonically defined groups one should denote the groups in the definition by $H^*_{\rel}(X_{\dot},n)_{\ol X_{\dot}}$ and define $H^*_{\rel}(X_{\dot},n) := \dirlim_{\ol X_{\dot}}H^*_{\rel}(X_{\dot},n)_{\ol X_{\dot}}$, where the limit runs over the directed family of all good compactifications of $X_{\dot}$.} Obviously there is a morphism $p^* : H^*_{\rel}(X_{\dot}, n) \to H^{E,*}_{\rel}(X_{\dot}, n)$, which yields a morphism of long exact sequences
\begin{equation}\label{seq:RelCohomCSchemes}
\xymatrix@C-0.55cm{
\dots \ar[r] & H^{E,i-1}_{\rel}(X_{\dot}, n) \ar[r] & \Fil^nH^i(X_{\dot},\C) \ar[r] & H^i(E_{\dot},\C)  \ar[r] & H^{E,i}_{\rel}(X_{\dot}, n)\ar[r] & \dots\\
\dots \ar[r] & H^{i-1}_{\rel}(X_{\dot}, n) \ar[r]\ar[u]^{p^*} & \Fil^nH^i(X_{\dot},\C) \ar[r]\ar@{=}[u] & H^i(X_{\dot},\C) \ar[u]^{p^*} \ar[r] & H^i_{\rel}(X_{\dot}, n) \ar[u]^{p^*} \ar[r] & \dots.
}
\end{equation}
Let $f: Y_{\dot}\to X_{\dot}$ be a morphism of simplicial varieties and $E/X_{\dot}$ as before. Given good compactifications $X_{\dot}\hookrightarrow \ol X_{\dot}$ and $Y_{\dot} \hookrightarrow \ol Y_{\dot}$, we may construct inductively (similar as in \cite[1.2]{Soule}) a good compactification $\widetilde Y_{\dot}$ together with a morphism of compactifications $\widetilde Y_{\dot} \to \ol Y_{\dot}$, such that $f$ extends to a morphism $\widetilde Y_{\dot} \to X_{\dot}$. Hence we can define pullback maps $f^*: H^*_{\rel}(X_{\dot}, n) \to H^*_{\rel}(Y_{\dot}, n)$ and $f^* : H^{E,*}_{\rel}(X_{\dot}, n) \to H^{f^*E,*}_{\rel}(Y_{\dot}, n)$.
\begin{prop}\label{prop:CRefinedClasses}
There exists a class $\widetilde\Ch_n^{\rel}(E) \in H^{2n-1,E}_{\rel}(X_{\dot},n)$, which is mapped to the $n$-th Chern character class $\Ch_n(E)$ in $\Fil^nH^{2n}(X_{\dot}, \C)$, and which is functorial in $X$. Moreover, the assignment $E \mapsto \widetilde\Ch_n^{\rel}(E)$ is uniquely determined by these two properties.
\end{prop}
\begin{proof}
Consider the universal situation: Since $H^i(E_{\dot}\GL_r(\C),\C)=0$ for all $i > 0$ by the following lemma, the natural map $H^{E^{\univ},2n-1}_{\rel}(B_{\dot}\GL_r(\C), n) \to \Fil^nH^{2n}(B_{\dot}\GL_r(\C),\C)$ is an isomorphism by the exactness of the top line in \eqref{seq:RelCohomCSchemes}, and the proposition follows.
\end{proof}
\begin{lemma}\label{lem:CohomEG}
Let $Y$ be any complex manifold. Define the simplicial manifold $E_{\dot}Y $ by $E_pY=Y\times \dots \times Y$ ($p+1$ factors) with faces and degeneracies as in \eqref{eq:DelEG} and \eqref{eq:SEG}. Then 
\[
H^n(E_{\dot}Y,\C) = 0, \text{ if } n>0, \quad H^0(E_{\dot}Y,\C) = \C.
\]
\end{lemma}
\begin{proof}
We have a spectral sequence
\[
E_1^{pq}(E_{\dot}Y) = H^q(E_pY,\C) \Longrightarrow H^{p+q}(E_{\dot}Y,\C),
\]
where the differential $d_1 : H^q(E_pY,\C) \to H^q(E_{p+1}Y,\C)$ is given by the alternating sum $d_1=\sum_{i=0}^p (-1)^i\del_i^*$.

Choose a point $e\in Y$.
For $p\geq 0, i = 0, \dots, p$, define $h_i : E_pY \to E_{p+1}Y$, $(y_0, \dots, y_p) \mapsto (y_0, \dots, y_i, e, \dots, e)$. Then the $h_i$'s satisfy the formal properties defining a simplicial homotopy between the constant map $e$ (given by $(y_0, \dots, y_p) \mapsto (e,\dots, e)$) and the identity \cite[Definitions 5.1]{May}, in particular
\begin{align*}
&\del_0h_0=e, \del_{p+1}h_p=\id_{E_pY}, \\
& \del_ih_j = h_{j-1}\del_i, &&\text{if } i<j,\\
& \del_{j+1}h_{j+1} = \del_{j+1}h_j, \\
& \del_ih_j= h_j\del_{i-1}, &&\text{if } i > j+1.
\end{align*}
The $h_i$ induce maps on the $E_1$-term of the above spectral sequence satisfying dual properties. Hence we get a chain homotopy $s$ between the identity and $e^*$ on the complex $(E_1^{*,q}(E_{\dot}Y), d_1)$ for any $q\geq 0$ by setting $s(x)= \sum_{i=0}^{p-1} (-1)^i h_i^*(x), x \in E_1^{pq}(E_{\dot}Y) = H^q(E_pY,\C)$ (cf. \cite[Proposition 5.3]{May}).
It follows, that $e: * \to E_{\dot}Y$, where $*$ is the one point constant simplicial manifold, and $E_{\dot}Y \to *$ induce homotopy inverse chain homotopy equivalences $E_1^{*,q}(E_{\dot}Y) \leftrightarrows E_1^{*,q}(*)$. Hence $E_2^{p,q}(E_{\dot}Y) \cong E_2^{p,q}(*)$.

But obviously $E_1^{pq}(*) = 0$, if $q>0$, and $E_1^{*,0}(*)$ is the complex $\C \xleftarrow{0} \C \xleftarrow{\id} \C \xleftarrow{0} \dots$,
i.e. 
\[
E_2^{pq}(E_{\dot}Y) = E_2^{pq}(*) = \begin{cases} \C, &\text{if } p=q=0,\\ 0, &\text{else.}\end{cases}
\]
Now the claim follows.
\end{proof}
\begin{dfn}
If $X_{\dot}$ is a simplicial variety and $E/X_{\dot}$ an algebraic $\GL_r$-bundle, the class $\widetilde\Ch_n^{\rel}(E) \in H^{2n-1,E}_{\rel}(X_{\dot},n)$ is called the \emph{$n$-th refined Chern character class of $E$}.
\end{dfn}
Now assume, that the algebraic bundle $E/X_{\dot}$, classified by $g: X_{\dot}\to B_{\dot}\GL_r(\C)$, admits a topological trivialization $\alpha: T \to E$, i.e. a topological morphism $\alpha: X_{\dot} \rightsquigarrow E_{\dot}\GL_r(\C)$, such that $p \circ \alpha = g$. Since $E_{\dot}$ is the pullback of $E_{\dot}\GL_r(\C)$ along $g$, $\alpha$ induces a topological morphism $\alpha: X_{\dot} \rightsquigarrow E_{\dot}$, such that $p\circ\alpha = \id_{X_{\dot}}$. Hence we can also define a map $\alpha^*: H^{E,*}_{\rel}(X_{\dot}, n) \to H^*_{\rel}(X_{\dot}, n)$ left inverse to $p^*$. 
\begin{dfn}
Let $E$ be an algebraic bundle on the simplicial variety $X_{\dot}$ and $\alpha$ a trivialization of the underlying topological bundle. Then we define
\[
\widetilde\Ch_n^{\rel}(T,E,\alpha) = -\alpha^*\widetilde\Ch_n^{\rel}(E) \in H^{2n-1}_{\rel}(X_{\dot}, n) \cong H^{2n-1}(X_{\dot},\C)/\Fil^n.
\]
\end{dfn}

More generally, we also allow $X_{\dot}$ to be of the form $X\otimes S$ with a variety $X$ and a simplicial set $S$.

\begin{prop}\label{prop:ComparisonOfRelClasses}
The class $\widetilde\Ch_n^{\rel}(T, E, \alpha)$ is mapped to the class $\Ch_n^{\rel}(T, E, \alpha)$ by the natural map $H^{2n-1}(X_{\dot}, \C)/\Fil^nH^{2n-1}(X_{\dot}, \C) \to \Hyp^{2n-1}(X_{\dot}, \Omega^{<n}_{X_{\dot}})$.
\end{prop}
\begin{proof}
Abbreviate $\GL_r(\C)$ to $G$. Let $g: X_{\dot} \to B_{\dot}G$ be the classifying map of $E$ and choose compatible good compactifications $B_{\dot}G \hookrightarrow \ol{B_{\dot}G}$ and $X_{\dot} \hookrightarrow \ol X_{\dot}$. 

Choose any representative $c$ of $\Ch_n(E^{\univ})$ in $\Fil^n\A^{2n}(\ol{B_{\dot}G}, \log D_{\dot})$. Then $\iota_{\A}(c) \in \A^{2n}(B_{\dot}G)$ lies in  $\Fil^n\A^{2n}(B_{\dot}G)$ and represents $\Ch_n(E^{\univ})$ considered as a class in $\Hyp^{2n}(B_{\dot}G, \Omega^{\geq n}_{B_{\dot}G})$. But this class is also represented by the form $I(\Ch_n(\Gamma^{\univ}))$, where $\Gamma^{\univ}$ denotes the standard connection on the universal bundle. Hence there exists $\eta \in \Fil^n\A^{2n-1}(B_{\dot}G)$ such that $d\eta= \iota_{\A}(c) - I(\Ch_n(\Gamma^{\univ}))$ and $ch_n := (c, \Ch_n(\Gamma^{\univ}), \eta)$ is a representative for $\Ch_n(E^{\univ})$ in $\Fil^nA^{2n}(\ol{B_{\dot}G}, \log D_{\dot})$. With this choice we have $p^*(\iota_A(ch_n)) = p^*\Ch_n(\Gamma^{\univ}) = -d\Ch_n^{\rel, \univ}$, where the form $\Ch_n^{\rel,\univ}$ was defined before proposition \ref{prop:ChRelUniv}.
Hence the universal refined class is represented by the cycle $(ch_n, -\Ch_n^{\rel,\univ})$.

Let $g': E_{\dot} \to E_{\dot}G$ be the map induced by $g$ on the principal bundles. Then $\widetilde\Ch_n^{\rel}(E)$ is represented by $(g^*ch_n, -g'^*\Ch_n^{\rel,\univ})$ and $\widetilde\Ch_n^{\rel}(T,E,\alpha)$ is represented by $(-g^*ch_n, \alpha^*g'^*\Ch_n^{\rel,\univ}) = (-g^*ch_n, \alpha^*\Ch_n^{\rel,\univ})=(-g^*ch_n, \Ch_n^{\rel}(\Gamma^T,\Gamma^E,\alpha))$, where on the left we view $\alpha$ as a morphism $X_{\dot}\rightsquigarrow E_{\dot}$, in the middle as a morphism $X_{\dot} \rightsquigarrow E_{\dot}G$, $\Gamma^T$ and $\Gamma^E$ denote the standard connections, and we used proposition \ref{prop:ChRelUniv}.
Now the natural map 
\begin{multline*}
H^*_{\rel}(X_{\dot}, n) = H^*\left(\Cone(\Fil^nA^*(\ol X_{\dot}, \log D_{\dot}) \xrightarrow{\iota_A} A^*(X_{\dot}))\right)  \\ \to \Hyp^*(X_{\dot}, \Omega^{<n}_{X_{\dot}}) = H^*(A^*(X_{\dot})/\Fil^nA^*(X_{\dot}))
\end{multline*}
is induced by the morphism of complexes
$\Cone(\Fil^nA^*(\ol X_{\dot}, \log D_{\dot}) \xrightarrow{\iota_A} A^*(X_{\dot})) \to A^*(X_{\dot})/\Fil^nA^*(X_{\dot})$, $(\omega, \eta) \mapsto \eta$. In particular $\widetilde\Ch_n^{\rel}(T, E, \alpha)$ is mapped to the class represented by $\Ch_n^{\rel}(\Gamma^T,\Gamma^E,\alpha)$, that is to $\Ch_n^{\rel}(T,E, \alpha)$.
\end{proof}

\section{Chern classes in Deligne-Beilinson cohomology}

Here we recall the definition of Deligne-Beilinson cohomology and Chern classes in Deligne-Beilinson cohomology. For the comparison with the relative Chern character classes in the next section, it is essential to have complexes computing Deligne-Beilinson cohomology of a simplicial variety, which behave well with respect to topological morphisms (in the appropriate sense). These are constructed in the first subsection.

\subsection{Definition of Deligne-Beilinson cohomology}
\label{sec:DefDBCohom}

Let $A$ be a subring of $\R$ and write $A(n):= (2\pi i)^nA \subset \C$.
Let $X_{\dot}$ be a simplicial algebraic variety and choose a good compactification $j : X_{\dot} \hookrightarrow \ol X_{\dot}$.

The \emph{Deligne-Beilinson cohomology} $H^*_{\DB}(X_{\dot}, A(n))$ of $X_{\dot}$ is by definition 
\[
\Hyp^*\left(\ol X_{\dot}, \Cone\left(\Rhyp j_*A(n) \oplus \Fil^n\Omega^*_{\ol X_{\dot}}(\log D_{\dot}) \xrightarrow{\epsilon - \iota} \Rhyp j_*\Omega^*_{X_{\dot}}\right)[-1]\right).
\]
This definition is independent of choices (cf. the definition of the mixed Hodge structure on $H^*(X_{\dot}, \Z)$).\footnote{As for the Hodge filtration, one could also define Deligne-Beilinson cohomology using simplicial varieties and suitable proper hypercoverings. This is the definition in \cite[\S 5]{EV}.} Since Deligne-Beilinson cohomology is constructed from a cone, we have long exact sequences
\begin{multline*}
\dots \to H^k_{\DB}(X_{\dot}, A(n)) \to H^k(X_{\dot}, A(n)) \oplus \Fil^nH^k(X_{\dot}, \C) \xrightarrow{\epsilon - \iota} H^k(X_{\dot}, \C) \\ \to H^{k+1}_{\DB}(X_{\dot}, A(n)) \to \dots
\end{multline*}

We need a concrete complex computing Deligne-Beilinson cohomology.
First some notation. For an arbitrary manifold $Y$ and an abelian group $G$ we denote by $\Sing^*(Y, G)$ the complex of smooth singular cochains with coefficients in $G$. 
The Theorem of de Rham asserts, that the natural map $\A^*(Y) \xrightarrow{\mathscr I} \Sing^*(Y, \C)$, which sends a differential $n$-form $\omega$ to the singular cochain sending any smooth $c : \Delta^n \to Y$ to $\int_{\Delta^n} c^*\omega$, is a quasiisomorphism (see e.~g. \cite[theorem 1.15]{DupLNM}).

For $A\subset \R$ as above, we define the complex of \emph{modified differential forms} $\widetilde\A^*(Y,A(n))$ to be the quasi-pullback of the diagram
\[\xymatrix{
 & \A^*(Y) \ar[d]^{\mathscr I}_{\text{qis}}\\
\Sing^*(Y, A(n)) \ar@{^{(}->}[r]^{\mathrm{incl}} & \Sing^*(Y, \C).
}
\]

Now let $X_{\dot}$ be a simplicial manifold. 
Let $\Sing^*(X_{\dot}, G)$ be the total comlex associated with the cosimplicial complex $[p] \mapsto \Sing^*(X_p, G)$. Then we have a natural isomorphism $H^*(X_{\dot}, G) = H^*(\Sing^*(X_{\dot}, G))$.

As in the case of de Rham cohomology, $H^*(X_{\dot}, G)$ may also be computed using \emph{compatible singular cochains}:
We define the complex of compatible singular cochains $C^*(X_{\dot}, G)$ in analogy with that of simplicial differential forms: 
\begin{multline*}
 C^n(X_{\dot}, G) := \big\{ (\sigma_p)_{p\geq 0} \;|\; \sigma_p \in \Sing^n(\Delta^p \times X_p, G),\\
 (\delta^i \times \id)^*\sigma_p = (\id \times \del_i)^*\sigma_{p-1}, i=0, \dots, p, p \geq 1\big\}
\end{multline*}
There is a natural quasiisomorphism $\Phi: C^*(X_{\dot}, G) \to \Sing^*(X_{\dot}, G)$ given as follows (cf. \cite[2.1.3]{Soule}): For a compatible $n$-cochain $\sigma = (\sigma_p)_{p \geq 0}$, define $\Phi(\sigma)_{p, n-p} \in \Sing^{n-p}(X_p, G)$ to be the cochain that sends a singular $(n-p)$-simplex $f : \Delta^{n-p} \to X_p$ to $\sigma_p(\id_{\Delta^p} \times f) \in G$. Here $\times$ denotes the cross product\footnote{defined using the shuffle-map, see e.g. \cite[Kap. V, 5.8]{Lamotke}} of singular chains, and $\id_{\Delta^p} : \Delta^p \to \Delta^p$ is the canonical singular $p$-chain.
Using the above compatibility condition and standard properties of the cross product, it is easy to see, that $\Phi$ is a chain map. 

\begin{lemma}
Integration over simplices induces an integration map $\mathscr I: A^*(X_{\dot}) \to C^*(X_{\dot}, \C)$
fitting in a commutative diagram
\[\xymatrix{
A^*(X_{\dot}) \ar[d]^I \ar[r]^{\mathscr I} & C^*(X_{\dot}, \C) \ar[d]^{\Phi}\\
\A^*(X_{\dot}) \ar[r]^{\mathscr I} & \Sing^*(X_{\dot}, \C).
}
\]
\end{lemma}
\begin{proof}
The map $\mathscr I: A^*(X_{\dot}) \to C^*(X_{\dot}, \C)$ is just given by applying the de Rham integration map $\mathscr I$ component-wise.
It is clearly well defined, and we have only to check, that the diagram commutes. Thus let $\omega = (\omega_p)_{p\geq 0} \in A^n(X_{\dot})$ be a simplicial $n$-form and let $f: \Delta^{n-p} \to X_p$ be a singular $n-p$-simplex  of $X_p$. Then the singular chain $\id_{\Delta^p} \times f$ is given by $\sum_{\mu} \sgn(\mu) (\id_{\Delta^p} \times f) \circ \mu$, where $\mu$ runs over all $(p,n-p)$-shuffles and $\mu$ also denotes the  $n$-simplex $\mu : \Delta^n \to \Delta^p \times \Delta^{n-p}$ corresponding to the shuffle $\mu$. On the right hand side of the formula $\id_{\Delta^p} \times f$ means just the cartesian product of maps.

Hence $\Phi\circ \mathscr I(\omega)$ sends the singular simplex $f$ to $\sum_{\mu} \sgn(\mu) \int_{\Delta^n} \mu^*(\id_{\Delta^p}\times f)^*\omega_p$. But since the signed sum over all $(p,n-p)$-shuffles corresponds to a oriented decomposition of $\Delta^p \times \Delta^{n-p}$ in $n$-simplices (cf. \cite[Section 5]{EilenbergMacLane}), this last sum is equal to $\int_{\Delta^p \times \Delta^{n-p}} (\id_{\Delta^p} \times f)^*\omega_p = \int_{\Delta^{n-p}} f^*(\int_{\Delta^p} \omega_p)$, which is also the result of applying $\mathscr I\circ I(\omega)$ to $f$.
\end{proof}

As before we define modified complexes $\widetilde A^*(X_{\dot}, A(n))$ resp. $\widetilde\A^*(X_{\dot}, A(n))$ as the quasi-pullbacks of the diagrams $C^*(X_{\dot}, A(n)) \to C^*(X_{\dot}, \C) \xleftarrow{\mathscr I} A^*(X_{\dot})$ resp.
$\Sing^*(X_{\dot}, A(n)) \to \Sing^*(X_{\dot}, \C) \xleftarrow{\mathscr I} \A^*(X_{\dot})$.

\begin{lemma}
Let $X_{\dot}$ be a simplicial variety and $X_{\dot} \overset{j}{\hookrightarrow} \ol X_{\dot}$ a good compactification. The Deligne-Beilinson cohomology $H^*_{\DB}(X_{\dot}, A(n))$ is naturally isomorphic to the cohomology of the complexes
\begin{gather*}
\Cone\left(\widetilde A^*(X_{\dot}, A(n)) \oplus \Fil^nA^*(\ol X_{\dot}, \log D_{\dot}) \xrightarrow{\epsilon - \iota} A^*(X_{\dot}) \right)[-1] \quad  \text{or} \\ 
\Cone\left(\widetilde \A^*(X_{\dot}, A(n)) \oplus \Fil^n\A^*(\ol X_{\dot}, \log D_{\dot}) \xrightarrow{\epsilon - \iota} \A^*(X_{\dot}) \right)[-1].
\end{gather*}
\end{lemma}
\begin{proof}
Using the fact, which follows from the constructions, that the diagram 
\[
\xymatrix{
\widetilde A^*(X_{\dot}, A(n)) \oplus \Fil^nA^*(\ol X_{\dot}, \log D_{\dot}) \ar[r]^-{\epsilon-\iota} \ar[d] & A^*(X_{\dot})\ar[d]^I\\
\widetilde \A^*(X_{\dot}, A(n)) \oplus \Fil^n\A^*(\ol X_{\dot}, \log D_{\dot}) \ar[r]^-{\epsilon-\iota} & \A^*(X_{\dot})
}
\]
commutes up to canonical homotopy, one constructs a map from the first complex to the second, which is a quasiisomorphism, since it is a quasiisomorphism on both components of the cone.

Furthermore, in the derived category $D^+(Ab)$ there are natural isomorphisms $\widetilde \A^*(X_{\dot}, A(n)) \simeq \Rhyp\Gamma(X_{\dot}, A(n)) \simeq \Rhyp\Gamma(\ol X_{\dot}, \Rhyp j_*A(n))$, $\A^*(X_{\dot}) \simeq \Rhyp\Gamma(\ol X_{\dot}, \Rhyp j_*\Omega^*_{X_{\dot}})$ and 
$\Fil^n\A^*(\ol X_{\dot}, \log D_{\dot}) \simeq \Rhyp\Gamma(\ol X_{\dot}, \Omega^{\geq n}_{\ol X_{\dot}}(\log D_{\dot}))$. Comparing with the definition of Deligne-Beilinson cohomology and using the long exact sequence of the cohomology of a cone, the claim follows.
\end{proof}

\begin{rem}
These complexes are also defined for simplicial schemes of the form $X \otimes S$ with an algebraic variety $X$ and a simplicial set $S$, and we use them to define the Deligne-Beilinson cohomology in this situation.
\end{rem}

The advantage of this description of the Deligne-Beilinson cohomology of simplicial varieties is, that we may define a pullback map $\alpha^* : \widetilde A^*(X_{\dot}, A(n)) \to \widetilde A^*(Y_{\dot}, A(n))$, whenever $\alpha: Y_{\dot} \rightsquigarrow X_{\dot}$ is a topological morphism:
\begin{lemma}
Let $\alpha: Y_{\dot} \rightsquigarrow X_{\dot}$ be a topological morphism of simplicial manifolds. Then there is a well defined pullback map $\alpha^*: \widetilde A^*(X_{\dot}, A(n)) \to \widetilde A^*(Y_{\dot}, A(n))$. It is compatible with the natural maps $\widetilde A^* \to A^*$.
\end{lemma}
\begin{proof}
By definition $\widetilde A^*(X_{\dot}, A(n))$ is the quasi-pullback of the diagram
$C^*(X_{\dot}, A(n)) \to C^*(X_{\dot}, \C) \xleftarrow{\mathscr I} A^*(X_{\dot})$. Obviously, $\alpha^*$ is well defined on each of the three complexes (cf. remark \ref{rem:PullbackByTopMorph}) and we only have to check, that it is compatible with the maps between them. This is is clear for the left hand map. For $\mathscr I$ this follows from the commutativity of the diagram
\[\xymatrix@C+1cm{
 \A^n(\Delta^p \times X_p) \ar[r]^{(\id_{\Delta^p},\alpha_p)^*} \ar[d]^{\mathscr I} & \A^n(\Delta^p \times Y_p) \ar[d]^{\mathscr I}\\
\Sing^n(\Delta^p \times X_p, \C) \ar[r]^{(\id_{\Delta^p}, \alpha_p)^*} & \Sing^n(\Delta^p \times Y_p, \C)
}
\]
which is established as follows: Let $\omega \in \A^n(\Delta^p \times X_p)$ and $\tau: \Delta^n \to \Delta^p\times Y_p$ be a smooth simplex. Then $(\id_{\Delta^p}, \alpha_p)^*\mathscr I(\omega)(\tau)= \int_{\Delta^n} ((\id_{\Delta^p},\alpha_p)\circ\tau)^*\omega=
\int_{\Delta^n} \tau^* ((\id_{\Delta^p},\alpha_p)^*\omega) = \mathscr I((\id_{\Delta^p},\alpha_p)^*\omega)(\tau)$.
\end{proof}

\subsection{Chern classes in Deligne-Beilinson cohomology}

There exists a theory of Chern (character) classes in Deligne-Beilinson cohomology for algebraic vector bundles on simplicial varieties (see \cite[\S 8]{EV}). We recall the relevant facts. To fix the normalizations we first of all recall the defintion of Chern classes in singular cohomology.

\begin{dfn}
Let $X$ be a (simplicial) complex manifold. The first Chern class $c_1^{\top}$ in singular cohomology (for holomorphic line bundles) is the connecting homomorphism
\[
c_1^{\top}: H^1(X, \O_X^*) \to H^2(X, \Z(1))
\]
associated with the short exact sequence of sheaves on $X$
\[
0 \to \Z(1) \to \O_X \xrightarrow{\exp} \O_X^* \to 0.
\]
\end{dfn}
\begin{rem}\label{rem:NormalizationChernClasses}
One can also use the sequence $0 \to \Z \to \O_X \xrightarrow{\exp(2\pi i\underline{\;\;})} \O_X^* \to 0$ to get integer valued Chern classes. This normalization for the first Chern class is also often used by algebraic geometers (e.g. \cite[Ch. I \S 1]{GH}). It differs from ours by the factor $2\pi i$. On the other hand topologists sometimes use yet another normalization: If $c_1^{\text{Milnor-Stasheff}}$ denotes the classical integer valued first Chern class as constructed e.g. in \cite{MS}, then $c_1^{\top} = -2\pi i c_1^{\text{Milnor-Stasheff}}$. This follows e.g. from \cite[Appendix C, Theorem (p. 306)]{MS} together with \cite[Ch. I \S1, Proposition (p. 141)]{GH}.

For later reference we note, that Burgos \cite{Burgos} uses topologists' normalization for his integer valued Chern classes $b_i$ and defines the ``twisted Chern classes'' $c_i^{\text{Burgos}} := (2\pi i)^i b_i$. 
In fact, the construction in \cite[section 4.2]{Burgos} is exactly the same as that in \cite[\S 14]{MS} (alternatively, one may look at the Chern-Weil theoretic approach in \cite[Proposition 5.27]{Burgos}). In particular, $c_1^{\top} = -c_1^{\text{Burgos}}$ and we have corresponding signs for the higher Chern and Chern character classes.
\end{rem}

The splitting principle also holds for singular cohomology and higher Chern classes $c_n^{\top}(\E) \in H^{2n}(X_{\dot}, \Z(n))$ and Chern character classes $\Ch_n^{\top}(\E) \in H^{2n}(X_{\dot}, \Q(n))$ for holomorphic vector bundles $\E$ are constructed as in section \ref{sec:HigherChernClasses}.

\begin{rem}\label{rem:CompatibilityChernClasses}
It is easy to see, that the diagram 
\[\xymatrix@R-0.4cm{
& H^2(X, \Z(1)) \ar[dr] \\
H^1(X, \O_X^*) \ar[ur]^{c_1^{\top}} \ar[dr]_{c_1} & & H^2(X, \C)\\
& \Hyp^2(X, \Omega^{\geq 1}_X) \ar[ur]
}
\]
commutes. In particular, if $\E$ is an algebraic vector bundle, the higher Chern (character) classes $c_n^{\top}(\E)$ resp. $\Ch_n^{\top}(\E)$ are mapped to  $c_n(\E)$ resp. $\widetilde\Ch_n(\E) \in \Fil^nH^{2n}(X, \C)$ under the natural map $H^{2n}(X, \Z(n)) \to H^{2n}(X, \C)$.
\end{rem}

The only thing we need to know (which is in fact easy to see using the long exact sequence of Deligne-Beilinson cohomology of $B_{\dot}\GL_r(\C)$) is:
Chern character classes $\Ch_n^{\DB}(\E)$ for \emph{algebraic} vector bundles $\E$ on (simplicial) varieties $X$ in Deligne-Beilinson cohomology $H^{2n}_{\DB}(X, \Q(n))$ are uniquely determined by the conditions, that they are functorial and compatible with the Chern character classes in singular cohomology under the natural map $H^{2n}_{\DB}(X, \Q(n)) \to H^{2n}(X, \Q(n))$ \cite[Prop. 8.2]{EV}.

\section[Comparison of Chern character classes]{Comparison of relative and Deligne-Beilinson Chern character classes}

Let $X_{\dot}$ be a simplicial algebraic variety and $A$ a subring of $\R$. There are natural morphisms
\[
H^{*-1}_{\rel}(X_{\dot}, n) = H^{*-1}(X_{\dot}, \C)/\Fil^nH^{*-1}(X_{\dot}, \C) \to H^*_{\DB}(X_{\dot}, A(n))
\]
induced on the defining cones by the maps in the commutative diagram
\[
\xymatrix{
\Fil^nA^*(\ol X_{\dot}, \log D_{\dot}) \ar[r]^-{\iota} \ar[d]_{\mathrm{incl.}} & A^*(X_{\dot}) \ar[d]^{-\id}\\
\widetilde A^*(X_{\dot}, A(n)) \oplus \Fil^nA^*(\ol X_{\dot}, \log D_{\dot}) \ar[r]^-{\epsilon-\iota} & A^*(X_{\dot}).
}
\]

Now let $E/X_{\dot}$ be an \emph{algebraic}  $\GL_r(\C)$-bundle, which may be viewed as an algebraic vector bundle,
and $\alpha : T \to E$ a trivialization of the associated topological bundle. 
Then we have the characteristic classes $\widetilde\Ch_n^{\rel}(T, E, \alpha) \in H^{2n-1}(X_{\dot}, \C)/\Fil^nH^{2n-1}(X_{\dot}, \C)$ and $\Ch_n^{\DB}(E) \in H^{2n}_{\DB}(X_{\dot}, \Q(n))$ and we may compare them by the above homomorphism.
\begin{thm}\label{thm:comparison}
$\widetilde\Ch_n^{\rel}(T, E, \alpha)$ is mapped to $(-1)^{n-1}\Ch_n^{\DB}(E)$ under the natural map $H^{2n-1}(X_{\dot}, \C)/\Fil^nH^{2n-1}(X_{\dot}, \C) \to H^{2n}_{\D}(X_{\dot}, \Q(n))$.
\end{thm}
\begin{proof}
Let $X_{\dot} \overset{j}{\hookrightarrow} \ol X_{\dot}$ be a good compactification and denote by $E_{\dot} \xrightarrow{p} X_{\dot}$ the principal bundle associated with $E$. Define
\begin{multline*}
H^{E,*}_{\DB}(X_{\dot}, \Q(n)) :=\\
H^*\left(\Cone\left(\widetilde A^*(E_{\dot}, \Q(n)) \oplus \Fil^nA^*(\ol X_{\dot},\log D_{\dot}) \xrightarrow{\epsilon - p^*\circ\iota} A^*(E_{\dot})\right)[-1]\right).
\end{multline*}
Similar as in the case of relative cohomology groups, we have a natural map $p^*: H^*_{\DB}(X_{\dot}, \Q(n)) \to H^{E,*}_{\DB}(X_{\dot},\Q(n))$ and a left inverse $\alpha^*$ of $p^*$ for a topological trivialization $\alpha$ of $E$. Moreover, there is a natural map $H^{E,*-1}_{\rel}(X_{\dot}, n) \to H^{E,*}_{\DB}(X_{\dot}, \Q(n))$ fitting in a commutative diagram (in the obvious sense)
\[
\xymatrix{
H^{E,*-1}_{\rel}(X_{\dot}, n) \ar[r] \ar@/^1pc/[d]^{\alpha^*} & H^{E,*}_{\DB}(X_{\dot}, \Q(n)) \ar@/^1pc/[d]^{\alpha^*}\\
H^{*-1}_{\rel}(X_{\dot}, n) \ar[u]^{p^*} \ar[r] & H^*_{\DB}(X_{\dot}, \Q(n)). \ar[u]^{p^*}
}
\]
We claim, that the refined class $\widetilde\Ch_n^{\rel}(E)$ is mapped to $p^*\Ch_n^{\D}(E)$ by the upper horizontal map. Since both classes are functorial, it suffices to treat the case of the universal bundle $E^{\univ}/B_{\dot}\GL_r(\C)$. Write $G:= \GL_r(\C)$. Since the cohomology of $E_{\dot}G$ vanishes in positive degrees and the cohomology of $B_{\dot}G$ vanishes in odd degrees, we have the following commutative diagram with exact rows:
\[
\xymatrix@C-0.43cm{
0 \ar[r] & H^{E^{\univ},2n}_{\D}(B_{\dot}G, \Q(n)) \ar[r] & \Fil^nH^{2n}(B_{\dot}G, \C) \ar[r] & 0\\
0 \ar[r] & H^{2n}_{\D}(B_{\dot}G, \Q(n)) \ar[r]\ar[u]^{p^*} & {\scriptstyle H^{2n}(B_{\dot}G,\Q(n)) \oplus \Fil^nH^{2n}(B_{\dot}G, \C)} \ar[r]^-{\epsilon-\iota}\ar[u]^{\mathrm{pr}_2} & H^{2n}(B_{\dot}G, \C).
}
\]
By definition, $\Ch_n^{\DB}(E^{\univ})$ is mapped to $\Ch_n^{\top}(E^{\univ})$ in $H^{2n}(B_{\dot}G,\Q(n))$. Since $\epsilon(\Ch_n^{\top}(E^{\univ})) = (-1)^n\iota(\Ch_n(E^{\univ}))$ (cf. proposition \ref{prop:AlgChernCharClasses}), it follows from the above diagram, that $p^*\Ch_n^{\D}(E^{\univ})$ is mapped to $(-1)^n\Ch_n(E^{\univ})$ in $\Fil^nH^{2n}(B_{\dot}G,\C)$. From the defining property of the refined classes and the commutativity of the diagram
\[
\xymatrix{
H^{E^{\univ},2n-1}_{\rel}(B_{\dot}G,n) \ar[d]\ar[dr]\\
H^{E^{\univ},2n}_{\D}(B_{\dot}G, \Q(n)) \ar[r] & \Fil^nH^{2n}(B_{\dot}G, \C),
}
\]
it follows, that $\widetilde\Ch_n^{\rel}(E^{\univ})$ is mapped to $(-1)^np^*\Ch_n^{\DB}(E^{\univ})$, whence our claim.

But then $\widetilde\Ch_n^{\rel}(T,E,\alpha) =-\alpha^*\widetilde\Ch_n^{\rel}(E)$ is mapped to $(-1)^{n-1}\alpha^*p^*\Ch_n^{\DB}(E)=(-1)^{n-1}\Ch_n^{\DB}(E)$.
\end{proof}

\begin{rem}
Obviously, the theorem remains true in the case, where $X_{\dot}$ is of the form $X\otimes S$ and this is the case we will be interested in.
\end{rem}

\chapter{Relative $K$-theory and regulators}
\label{ch:RelKandReg}

Let $X= \Spec(A)$ be a smooth affine scheme of finite type over $\C$. Then the algebraic and topological $K$-theory of $X$ resp. its underlying complex manifold are given (for $i > 0$) by 
\[
K_i(X) = \pi_i(B\GL(A)^+) \quad\text{resp.}\quad K^{-i}_{\top}(X) = \pi_i(BU^X)
\]
and there is a natural morphism $B\GL(A)^+ \to BU^X$ in the homotopy category of spaces (section 6 of Gillet's article in \cite{GilletTopK}). We define the \emph{relative $K$-group} $K_i^{\rel}(X)$ as the $i$-th homotopy group of the homotopy fibre of this map. The goal of this chapter is to construct relative Chern character maps $\Ch_{n,i}^{\rel}: K_i^{\rel}(X) \to H^{2n-i-1}(X, \C)/\Fil^nH^{2n-i-1}(X, \C)$ and to compare these with the Chern character in Deligne-Beilinson cohomology. 
The construction is roughly as follows: We have the Hurewicz morphism from the relative $K$-groups to the homology of a certain space and we construct a simplicial set $\mathscr F$ (sections \ref{sec:TopK} and \ref{sec:RelK}), whose geometric realization admits a natural acyclic map to this space, hence has the same homology. By construction, there will be a canonical topologically trivialized bundle on the simplicial variety $X\otimes \mathscr F$, whose relative Chern character class induces the desired map on the homology of $\mathscr F$ (section \ref{sec:RelChernChar}). 
Using Jouanolou's trick we will extend the relative Chern character to non affine varieties in section \ref{sec:NonAffineVar}. The comparison with the Deligne-Beilinson Chern character is done in section \ref{sec:ComparisonChernCharacters}.
In the last section, we will apply this to the case $X=\Spec(\C)$ to get a new proof of Burgos' theorem \cite{Burgos}, that Borel's regulator is twice Beilinson's regulator.

Throughout, we work in the category of compactly generated Hausdorff spaces. All constructions (in particular (fibered) products) are carried out in this category.

\section{Topological $K$-theory} 
\label{sec:TopK}

Our first task is to give an adequate simplicial model for the topological $K$-groups of a manifold $X$ in terms of smooth maps $\Delta^p \times X \to \GL_r(\C)$ in order to be able to apply our theory of topological bundles.

Let $X$ be a finite dimensional CW complex. By the \emph{topological $K$-theory of $X$} we mean the representable complex $K$-theory of $X$, i.e. 
\[
K^{-i}_{\top}(X) = \pi_i(BU^X) = [X_+, \Omega^i BU] = [\Sigma^i (X_+), BU],
\]
where $BU$ is a classifying space for the infinite unitary group $U= \dirlim_r U(r)$, $BU^X$ is the space of continuous maps from $X$ to $BU$ with the compact-open topology, $X_+$ is the union of $X$ with a disjoint basepoint, $\Sigma^i$ is the $i$-fold reduced suspension, $\Omega$ the loop space and $[.,.]$ means based homotopy classes of based continuous maps.

\begin{lemma}\label{lem:HomotopyBU}
The natural map $\dirlim_r \pi_i(BU(r)^X) \to \pi_i(BU^X)$ is an isomorphism.
\end{lemma}
\begin{proof}
The classifying space $BU$ may be realized as the direct limit of Grassmannians $BU = \dirlim_r BU(r) = \dirlim_r \dirlim_n G_r(\C^n)$ where $G_r(\C^n)$ denotes the Grassmannian of complex $r$-planes in $\C^n$. It has the structure of a CW complex with only finitely many cells in each dimension (cf. \cite[Corollary 6.7 and Problem 6-C]{MS}). In particular, every skeleton lies already in some $BU(r)$.

Using the cellular approximation theorem, it follows, that any element of $\pi_i(BU^X) = [\Sigma^i(X_+), BU]$ may be represented by a map $\Sigma^i(X_+) \to BU(r)$ for a suitable $r$, thus showing the surjectivity.

Given two maps $f,g:\Sigma^i(X_+) \to BU(r)$ and a homotopy between the induced maps $f,g: \Sigma^i(X_+) \to BU$, the same argument shows, that there is an $s \geq r$ and a homotopy between the induced maps $f,g: \Sigma^i(X_+) \to BU(s)$, proving injectivity.
\end{proof}
This lemma reduces the description of topological $K$-theory to the description of the homotopy groups $\pi_i(BU(r)^X) = \pi_{i-1}(U(r)^X)$. Note, that the inclusion $U(r) \hookrightarrow \GL_r(\C)$ is a homotopy equivalence.
\smallskip

For any topological space $Y$, we denote by $S_{\dot}(Y)$ the simplicial set of singular simplices in $Y$. This is a Kan complex (a fibrant simplicial set). The functor $S_{\dot}$ is right adjoint to the geometric realization functor $|\;.\;|$ from simplicial sets to spaces. The natural map $|S_{\dot}(Y)| \to Y$ is a weak equivalence. 
If $S_{\dot}$ is any Kan complex, there is a natural isomorphism $\pi_i(S_{\dot}) \xrightarrow{\cong} \pi_i(|S_{\dot}|)$. In particular, we have canonical isomorphisms $\pi_i(S_{\dot}(Y)) \xrightarrow{\cong} \pi_i(|S_{\dot}(Y)|) \xrightarrow{\cong} \pi_i(Y)$. See \cite[\S 16]{May} for the proofs.

Now assume that $Y$ is a smooth manifold. Then it is well known, that there is also a homotopy equivalence $S^{\infty}_{\dot}(Y) \xrightarrow{\simeq} S_{\dot}(Y)$, where $S^{\infty}_{\dot}(Y)$ denotes the simplicial set of \emph{smooth} singular simplices. We want to extend this result to spaces of mappings between smooth manifolds.

Thus let $X$ and $Y$ be smooth manifolds. There is a natural homeomorphism $(Y^X)^{\Delta^p} \cong Y^{\Delta^p \times X}$, hence any singular $p$-simplex $\sigma$ of $Y^X$ may be viewed as a map $\Delta^p \times X \to Y$ and we call $\sigma$ \emph{smooth}, if this last map is smooth. Denote by $S^{\infty}_{\dot}(Y^X)$ the simplicial set of smooth singular simplices in $Y^X$.
\begin{prop}
The natural inclusion $i: S^{\infty}_{\dot}(Y^X) \hookrightarrow S_{\dot}(Y^X)$ is a homotopy equivalence. 
\end{prop}
We have to approximate every singular simplex by a smooth one, in a compatible way. This is done in the following lemma (cf. \cite[Lemma 16.7]{Lee}). We denote by $I$ the unit interval $[0,1]$.
\begin{lemma}
For each singular $p$-simplex $\sigma : \Delta^p \to Y^X$ there exists a continuous map $H_{\sigma} : I \times \Delta^p \to Y^X$ such that the following properties hold:
\begin{enumerate}
\item $H_{\sigma}$ is a homotopy from $\sigma = H_{\sigma}(0, \,.\,)$ to a smooth $p$-simplex $\widetilde\sigma = H_{\sigma}(1,\,.\,)$.
\item For any increasing $\phi: [q] \to [p]$ we have $H_{\phi^*\sigma} = H_{\sigma} \circ (\id_I \times \phi_{\Delta})$, where as usual $\phi_{\Delta} : \Delta^q \to \Delta^p$ is the map induced by $\phi$.
\item If $\sigma$ is smooth, then $H_{\sigma}$ is the constant homotopy.
\end{enumerate}
\end{lemma}
\begin{proof}
Note that it is enough to fulfill (ii) for the face and degeneracy operators. The $H_{\sigma}$ are constructed by induction on the dimension of $\sigma$. If $\sigma: X \to Y$ is a $0$-simplex, we choose any homotopy $H_{\sigma}$ to a smooth $\widetilde\sigma : X \to Y$, constant if $\sigma$ is already smooth. 

Now suppose that we have constructed $H_{\sigma'}$ for any $\sigma'$ of dimension $<p$, and let $\sigma : \Delta^p \to Y^X$ be a $p$-simplex. It may uniquely be written as $\sigma = \phi^*\tau = \tau \circ \phi_{\Delta}$, where $\phi$ is surjective and $\tau$ is non-degenerate \cite[Satz 3.9]{Lamotke}.
If $\sigma$ is smooth, we let $H_{\sigma}$ be the constant homotopy. If $\sigma$ is degenerated, $\phi\not= \id$ and $\tau$ is of dimension strictly less than $p$.
Clearly $\phi_{\Delta}$ is smooth and we define $H_{\sigma} := H_{\tau} \circ (\id_I \times \phi_{\Delta})$. Note, that, since $\phi_{\Delta}$ has a smooth section, if $\sigma$ is smooth, so is $\tau$, so that $H_{\sigma}$ is well defined.

If $\sigma: \Delta^p \to Y^X$ is non-degenerate (i.e. $\phi=\id, \tau=\sigma$) and not smooth, we construct $H_{\sigma}$ as in \cite[Lemma 16.7]{Lee} viewing $\sigma$ as a map $\Delta^p \times X \to Y$. Everything goes through word by word.

Condition (ii) is checked in {\it loc. cit.} for the face maps. By our construction, it is also satisfied for the degeneracies (use the unicity of the representation $\sigma = \phi^*\tau$).
\end{proof}

\begin{proof}[Proof of the proposition]
We define $s:S_{\dot}(Y^X) \to S^{\infty}_{\dot}(Y^X)$ by $\sigma \mapsto H_{\sigma}(1,\,.\,)$. Condition (ii) of the lemma ensures that this is a morphism of simplicial sets, and by (iii) $s \circ i = \id_{S^{\infty}_{\dot}(Y^X)}$.

We construct a simplicial homotopy $i \circ s \sim \id_{S_{\dot}(Y^X)}$ using the $H_{\sigma}$: For $i = 0, \dots, p$ let $\alpha_i : \Delta^{p+1} \to I \times \Delta^p$ be the affine singular simplex sending $e_0 \mapsto (0, e_0), \dots e_i \mapsto (0, e_i), e_{i+1} \mapsto (1, e_i), \dots e_{p+1} \mapsto (1, e_p)$, where $e_0, \dots,e_{p+1}$ is the standard basis of $\R^{p+2}$ (this is just the standard decomposition of $I\times \Delta^p$ in $(p+1)$-simplices), and define $h_i : S_p(Y^X) \to S_{p+1}(Y^X)$ as $h_i(\sigma) = H_{\sigma} \circ \alpha_i$. Again it follows from condition (ii) of the lemma and the computations (16.10) -- (16.12) in \cite{Lee}, that the $h_i$ form a simplicial homotopy $i \circ s \sim \id_{S_{\dot}(Y^X)}$ in the sense of \cite[Definition 5.1]{May}.
\end{proof}

We apply this proposition in the case $Y = \GL_r(\C)$. Obviously, $S^{\infty}_{\dot}(\GL_r(\C)^X)$ and $S_{\dot}(\GL_r(\C)^X)$ are simplicial groups. Define $G_{\dot} = \dirlim_r S^{\infty}_{\dot}(\GL_r(\C)^X)$ and let $B_{\dot}G_{\dot}$ be its classifying space (Appendix \ref{app:SimplGrps}). We have the following chain of natural isomorphisms
\begin{multline*}
\pi_i(B_{\dot}G_{\dot}) \cong \pi_{i-1}(G_{\dot}) \cong \dirlim_r \pi_{i-1}(S_{\dot}^{\infty}(\GL_r(\C)^X)) \cong \\ 
\dirlim_r \pi_{i-1}(S_{\dot}(\GL_r(\C)^X))  \cong \dirlim_r\pi_{i-1}(\GL_r(\C)^X) \cong \dirlim_r\pi_{i-1}(U(r)^X) \cong \\
\dirlim_r \pi_i(BU(r)^X) = \pi_i(BU^X) = K^{-i}_{\top}(X),
\end{multline*}
where we used the fact, that $BU(r)^X$ is a classifying space for $U(r)^X$ (cf. the argument in the proof of the lemma in section 6.1 of Gillet's article in \cite{GilletTopK}), and lemma \ref{lem:HomotopyBU}, 
and $B_{\dot}G_{\dot}$ is our simplicial model for the topological $K$-theory of $X$.

\section{Relative $K$-theory} 
\label{sec:RelK}

Now let $X=\Spec(A)$ be a smooth affine scheme of finite type over $\C$. By abuse of notation, we denote the associated complex manifold by the same letter. Note, that $X$ has the structure of a finite dimensional CW complex, so our above description of the topological $K$-theory of $X$ applies.

\paragraph*{The map from algebraic to topological $K$-theory}

There are natural continuous homomorphisms $A \to \Sing^{\infty}(X) \to \Sing(X)$ from the ring of algebraic functions on $X$ to that of smooth resp. continuous complex valued functions on $X$, where $A$ is equipped with the discrete topology, $\Sing(X)$ with the compact-open topology and $\Sing^{\infty}(X)$ with the induced topology. These induce $\GL_r(A) \to \GL_r(\Sing^{\infty}(X)) \to \GL_r(\Sing(X)) = \GL_r(\C)^X$. Note, that the simplicial set of singular simplices of $\GL_r(A)$ is just the constant simplicial group $\GL_r(A)$. Thus we have natural morphisms of simplicial groups $\GL_r(A) \to S^{\infty}_{\dot}(\GL_r(\C)^X)$ and, taking the limit $r \to \infty$, $\GL(A) \to G_{\dot} = \dirlim_r S^{\infty}_{\dot}(\GL_r(\C)^X)$. Hence we get a map on the classifying simplicial sets $B_{\dot}\GL(A) \to B_{\dot}G_{\dot}$.

The geometric realization $|B_{\dot}\GL(A)|$ is a classifying space for the discrete group $\GL(A)$. Its fundamental group is $\pi_1(|B_{\dot}\GL(A)|) = \GL(A)$ with maximal perfect subgroup the commutator subgroup $\GL(A)' = [\GL(A), \GL(A)]$. The \emph{algebraic $K$-groups} of $X$ are by definition
\[
K_i(X) = \pi_i(|B_{\dot}\GL(A)|^+), \qquad i>0,
\]
where $|B_{\dot}\GL(A)|^+$ denotes Quillen's plus-construction with respect to $\GL(A)'$. Up to homotopy equivalence, $|B_{\dot}\GL(A)|^+$ is uniquely determined by the fact, that there is an acyclic cofibration $f: |B_{\dot}\GL(A)| \to |B_{\dot}\GL(A)|^+$ with $\ker(\pi_1(f)) = \GL(A)'$ \cite[Theorem (5.1)]{Berrick}. Now $\pi_1(|B_{\dot}G_{\dot}|) = K^{-1}_{\top}(X)$ is abelian, hence the image of $\GL(A)'$ under the map induced from $|B_{\dot}\GL(A)| \to |B_{\dot}G_{\dot}|$ on fundamental groups vanishes. By {\it loc. cit.} (5.2) $|B_{\dot}\GL(A)| \to |B_{\dot}G_{\dot}|$ factors up to homotopy uniquely through $|B_{\dot}\GL(A)|^+$. On homotopy groups this gives the desired map
\[
K_i(X) = \pi_i(|B_{\dot}\GL(A)|^+) \to \pi_i(|B_{\dot}G_{\dot}|) = K^{-i}_{\top}(X), \qquad i>0.
\]
\begin{rem}
It is easy to see, that this is the same map as defined e. g. in section 6 of Gillet's article in \cite{GilletTopK}.
\end{rem}

\paragraph*{Relative $K$-theory}

We define $F$ and $\widetilde F$ by the following pull-back diagrams:
\begin{equation}\label{diag1}\begin{split}
\xymatrix{
 F \ar@{}[dr]|{\lrcorner}\ar[d]\ar[r] &  \widetilde F \ar@{}[dr]|{\lrcorner}\ar[r]\ar[d] & |E_{\dot}G_{\dot}| \ar[d]^{|p|}\\
|B_{\dot}\GL(A)| \ar[r] & |B_{\dot}\GL(A)|^+ \ar[r] & |B_{\dot}G_{\dot}|
}
\end{split}
\end{equation}
Since $p: E_{\dot}G_{\dot} \to B_{\dot}G_{\dot}$ is a Kan fibration (Appendix \ref{app:SimplGrps}), the realization $|p|$ is a Serre fibration and so are the other two vertical arrows induced by $|p|$. Then, since $|B_{\dot}\GL(A)| \to |B_{\dot}\GL(A)|^+$ is acyclic, so is $F \to \widetilde F$ \cite[(4.1)]{Berrick}. Since $|E_{\dot}G_{\dot}|$ is contractible (lemma \ref{lem:EGcontractible}), $\widetilde F$ is homotopy equivalent to the homotopy fibre of the map $|B_{\dot}\GL(A)|^+ \to |B_{\dot}G_{\dot}|$ and we define the \emph{relative $K$-groups}
\[
 K_i^{\rel}(X) := \pi_i(\widetilde F), \qquad i>0.
\]
By construction we have a long exact sequence
\[
 \dots \to K^{-i-1}_{\top}(X) \to K_i^{\rel}(X) \to K_i(X) \to K^{-i}_{\top}(X) \to \dots.
\]

We also need the following simplicial description of the homology of $\widetilde F$. Define $\mathscr F$ by the following pullback diagram of simplicial sets:
\[\xymatrix{
\mathscr F \ar@{}[dr]|{\lrcorner}\ar[d]\ar[r] &  E_{\dot}G_{\dot} \ar[d]^p\\
B_{\dot}\GL(A) \ar[r]  & B_{\dot}G_{\dot}
}
\]
Since the realization functor $|\;.\;|$ commutes with finite limits \cite[Theorem in Ch. III.3]{GabrielZisman}, the natural map $|\mathscr F| \to F$ is a homeomorphism, and, since $F \to \widetilde F$ is acyclic, we have isomorphisms in homology
\[
H_*(\mathscr F, \Z) \cong H_*(|\mathscr F|, \Z) \xrightarrow{\cong} H_*(F, \Z) \xrightarrow{\cong} H_*(\widetilde F, \Z).
\]

\section{The relative Chern character}
\label{sec:RelChernChar}

Let $X = \Spec(A)$ be as before.
We define relative Chern character maps
\[
\Ch_{n,i}^{\rel}: K_i^{\rel}(X) \to H^{2n-i-1}(X, \C)/\Fil^nH^{2n-i-1}(X, \C) 
\]
as follows: By definition, $K_i^{\rel}(X) = \pi_i(\widetilde F)$, and we have the Hurewicz map
$K_i^{\rel}(X) \to H_i(\widetilde F, \Z) \cong H_i(\mathscr F, \Z)$. It is thus enough to construct a homomorphism $H_i(\mathscr F, \Z) \to H^{2n-i-1}_{\rel}(X, n) = H^{2n-i-1}(X, \C)/\Fil^nH^{2n-i-1}(X, \C)$. We will use the following
\begin{lemma}\label{lemma:cohom_X_tensor_S}
Let $S$ be a simplicial set and $X$ an algebraic variety. Form the simplicial variety $X_{\dot}:=X\otimes S$ as in Example \ref{ex:SimplicialBundles}. Then we have natural isomorphisms
\begin{align*}
&H^k_{\rel}(X_{\dot}, n) \cong \bigoplus_{p+q=k} \Hom(H_p(S, \Z), H^q(X,\C)/\Fil^nH^q(X, \C)),\\
&H^k_{\DB}(X_{\dot}, \Q(n)) \cong \bigoplus_{p+q=k} \Hom(H_p(S, \Z), H_{\DB}^q(X, \Q(n))).
\end{align*}
\end{lemma}
\begin{proof}
The proof is the same in both cases and we restrict to the first one. 
Choose a good compactification $X \hookrightarrow \ol X$. This induces a good compactification $X_{\dot} \hookrightarrow \ol X\otimes S =: \ol X_{\dot}$ and $H^*_{\rel}(X_{\dot}, n)$ is the cohomology of the (cosimplicial) complex $\mathscr G^*(X_{\dot}):= \Cone(\Fil^n\A^*(\ol X_{\dot}, \log D_{\dot}) \xrightarrow{\iota} \A^*(X_{\dot}))$. Let $\mathscr G^*(X)$ be the complex $\Cone(\Fil^n\A^*(\ol X, \log D) \to \A^*(X))$. 
Obviously, $\mathscr G^q(X_p) = \prod_{\sigma \in S_p} \mathscr G^q(X) = \Hom(\Z S_p, \mathscr G^q(X))$ where $\Z S_p$ is the free abelian group generated by $S_p$ and $\Hom$ is in the category of abelian groups. We form the chain complex $\Z S_*$ with the usual differential $\sum (-1)^i \del_i$, the $\del_i$'s denoting the face operators of $S$. Its homology is by definition $H_*(S, \Z)$.
Then the total complex of $\mathscr G^*(X_{\dot})$ is just the total Hom complex $\Hom_{\Z}(\Z S_*, \mathscr G^*(X))$ \cite[2.7.4]{Weibel} and there is a short exact sequence
\begin{multline*}
 0 \to \bigoplus_{p+q=k-1} \mathrm{Ext}^1_{\Z}(H_p(S, \Z), H^q(\mathscr G^*(X))) \to
H^k(\Hom_{\Z}(\Z S_*, \mathscr G^*(X))) \to\\ 
\bigoplus_{p+q=k} \Hom(H_p(S,\Z), H^q(\mathscr G^*(X))) \to 0
\end{multline*}
({\it loc. cit.} Exer. 3.6.1).
Since the $H^q(\mathscr G^*(X)) \cong H^q(X, \C)/\Fil^nH^q(X,\C)$ are $\Q$-vector spaces, the Ext term vanishes and the claim follows.
\end{proof}
\begin{rems}\label{rem:ExplicitDescriptionCocycle}
(i) Now it follows, that the relative cohomology $H^k_{\rel}(X_{\dot}, n)$ is identified with $H^k(X_{\dot}, \C)/\Fil^nH^k(X_{\dot}, \C)$ also in the case $X_{\dot}=X\otimes S$.

(ii) A similar statement also holds for the group $\Hyp^k(X_{\dot}, \Omega^{<n}_{X_{\dot}})$, which is computed by the complex $A^*(X_{\dot})/\Fil^nA^*(X_{\dot})$. We have a commutative diagram
\[\xymatrix{
H^k(X_{\dot}, \C)/\Fil^nH^k(X_{\dot},\C) \ar[d]\ar[r] & \Hyp^k(X_{\dot}, \Omega^{<n}_{X_{\dot}}) \ar[d]\\
\Hom(H_p(S, \Z), H^{k-p}(X, \C)/\Fil^n) \ar[r] & \Hom(H_p(S, \Z), \Hyp^{k-p}(X, \Omega^{<n}_X))
}
\]
and the right vertical arrow is given explicitely as follows: A class in $\Hyp^k(X_{\dot}, \Omega^{<n}_{X_{\dot}})$ may be represented by a form $\omega \in A^k(X_{\dot})$, closed modulo $\Fil^nA^{k+1}(X_{\dot})$. The simplicial form $\omega$ is given by a family of $k$-forms on $\Delta^q \times (X\otimes S)_q$, $q\geq 0$, and in particular we can consider the restriction $\sigma^*\omega$ of $\omega_p$ to the copy of $\Delta^p \times X$ corresponding to $\sigma \in S_p$. 
Integration along $\Delta^p$ gives the $(k-p)$-form $\int_{\sigma} \omega = \int_{\Delta^p} \sigma^*\omega \in \A^{k-p}(X)$. By linearity this extends to a map $\Z S_p \to \A^{k-p}(X)$, $\sigma \mapsto \int_{\Delta^p} \sigma^*\omega$, which induces a well defined homomorphism $H_p(S, \Z) \to H^{k-p}(\A^*(X)/\Fil^n\A^*(X)) = \Hyp^{k-p}(X, \Omega^{<n}_X)$.
\end{rems}

We return to our smooth affine $\C$-scheme of finite type $X = \Spec(A)$.
To construct the relative Chern character map on $K$-theory we thus have to construct classes in $H^{2n-1}(X\otimes \mathscr F, \C)/\Fil^nH^{2n-1}(X\otimes \mathscr F, \C)$. This is achieved as follows. 
First write $G_{r,\dot} := S^{\infty}_{\dot}(\GL_r(\C)^X)$, so that $G_{\dot} = \dirlim_r G_{r,\dot}$, and define $\mathscr F_r$  by the cartesian diagram of simplicial sets
\begin{equation}\label{diag2}\begin{split}
 \xymatrix{
\mathscr F_r \ar@{}[dr]|{\lrcorner}\ar[d]\ar[r] & E_{\dot}G_{r,\dot} \ar[d]^p\\
B_{\dot}\GL_r(A) \ar[r] &  B_{\dot}G_{r,\dot}.
}
\end{split}
\end{equation}

Then $\mathscr F = \dirlim_r \mathscr F_r$, $H_*(\mathscr F, \Z) = \dirlim_r H_*(\mathscr F_r, \Z)$ and by the lemma  
\[
H^*(X\otimes \mathscr F, \C)/\Fil^n = \projlim_r H^*(X\otimes \mathscr F_r, \C)/\Fil^n.
\]

By construction, a $p$-simplex in the simplicial group $G_{r,\dot}$ is a smooth map $\Delta^p \times X \to \GL_r(\C)$, and a $p$-simplex in $E_{\dot}G_{r,\dot}$ may be viewed as a smooth map $\Delta^p \times X \to E_p\GL_r(\C)$.
On the other hand, every $p$-simplex in $B_{\dot}\GL_r(A)$ may be seen as a morphism of varieties $X \to B_p\GL_r(\C)$.
As in example \ref{ex:Karoubi_top_bundles} the above diagram (\ref{diag2}) then gives rise to a commutative diagram
\[\xymatrix{
& E_{\dot}\GL_r(\C)\ar[d]^p\\
X\otimes \mathscr F_r \ar@{~>}[ur]^{\alpha_r}\ar[r]^{g_r} & B_{\dot}\GL_r(\C),
}
\]
where $g_r$ is a morphism of simplicial varieties.

Phrased differently, if we denote by $E_r$ the algebraic bundle classified by $g_r : X\otimes \mathscr F_r \to B_{\dot}\GL_r(\C)$ and by $T_r$ the trivial $\GL_r(\C)$-bundle, we have the trivialization $\alpha_r : T_r \to E_r$ of the underlying topological bundles and corresponding relative Chern character classes $\widetilde\Ch_n^{\rel}(T_r, E_r, \alpha_r) \in H^{2n-1}_{\rel}(X\otimes\mathscr F_r,n)= H^{2n-1}(X\otimes \mathscr F_r, \C)/\Fil^nH^{2n-1}(X\otimes \mathscr F_r, \C)$. We claim, that these classes are compatible for different $r$.
\begin{lemma} The class $\widetilde\Ch_n^{\rel}(T_{r+1}, E_{r+1}, \alpha_{r+1})$ maps to $\widetilde\Ch_n^{\rel}(T_r, E_r, \alpha_r)$
under the natural map $H^{2n-1}_{\rel}(X\otimes \mathscr F_{r+1}, n) \to H^{2n-1}_{\rel}(X\otimes \mathscr F_r,n)$ induced by the inclusion $j_r : \GL_r(\C) \hookrightarrow \GL_{r+1}(\C)$ in the upper left corner.
\end{lemma}
\begin{proof}
By abuse of notation, we write $j$ for all the morphisms induced by $j$. Then we have $\alpha_{r+1}\circ j = j\circ \alpha_r$ and hence we get a commutative diagram
\[
\xymatrix{
H^{E_{r+1}, 2n-1}_{\rel} (X\otimes \mathscr F_{r+1}, n) \ar[r]^-{j^*}\ar[d]_{\alpha_{r+1}^*} & H^{E_r,2n-1}_{\rel}(X\otimes \mathscr F_r, n) \ar[d]^{\alpha_r^*}\\
H^{2n-1}_{\rel}(X\otimes \mathscr F_{r+1},n) \ar[r]^{j^*} & H^{2n-1}_{\rel}(X\otimes \mathscr F_r, n).
}
\]
By construction it then suffices to show, that the refined class $\widetilde\Ch_n^{\rel}(E_{r+1}) \in H^{E_{r+1},2n-1}_{\rel}(X\otimes \mathscr F_{r+1}, n)$ is mapped to $\widetilde\Ch_n^{\rel}(E_{r})$ by $j^*$. By functoriality it is enough to show this for the universal bundles $E^{\univ}_{r+1}/B_{\dot}\GL_{r+1}(\C)$ and $E^{\univ}_r/B_{\dot}\GL_{r}(\C)$. But under the identification $H^{E^{\univ}_r, 2n-1}(B_{\dot}\GL_{r}(\C), n) \cong \Fil^nH^{2n}(B_{\dot}\GL_{r}(\C),\C)$ the $n$-th universal refined Chern character class ``is'' the $n$-th universal Chern character class $\Ch_n(E^{\univ}_r)$ and  $j^*\Ch_n(E^{\univ}_{r+1}) = \Ch_n(j^*E^{\univ}_{r+1}) = \Ch_n(E^{\univ}_r \oplus T_1) = \Ch_n(E^{\univ}_r)$, since $j: B_{\dot}\GL_r(\C) \hookrightarrow B_{\dot}\GL_{r+1}(\C)$ classifies the bundle $E_r^{\univ} \oplus T_1$ and the higher Chern classes of the trivial bundle $T_1$ vanish.
\end{proof}

\begin{dfn}
According to the preceding lemma, the family 
\[
(\widetilde\Ch_n^{\rel}(T_r, E_r, \alpha_r))_{r\geq 0}
\] 
defines a class in $H^{2n-1}(X\otimes \mathscr F, \C)/\Fil^nH^{2n-1}(X\otimes \mathscr F, \C)$. 
By lemma \ref{lemma:cohom_X_tensor_S} this class gives morphisms $H_i(\mathscr F, \Z) \to H^{2n-i-1}(X, \C)/\Fil^nH^{2n-i-1}(X, \C)$, $i = 0, \dots, 2n -1$. We define the relative Chern character $\Ch^{\rel}_{n,i}$ on $K_i^{\rel}(X)$ to be the composition
\begin{multline*}
\Ch_{n,i}^{\rel}: K_i^{\rel}(X) = \pi_i(\widetilde F) \xrightarrow{\mathrm{Hur.}} H_i(\widetilde F, \Z) \cong H_i(\mathscr F, \Z) \to \\
\to H^{2n-i-1}(X, \C)/\Fil^nH^{2n-i-1}(X, \C).
\end{multline*}
\end{dfn}

\begin{rems}
(i) For the construction of regulators, it would suffice to develop a theory of bundles, connections and characteristic classes only for simplicial varieties of the type $\Spec(A) \otimes S$ with a simplicial set $S$. In this case a $\GL_r$-bundle on $X\otimes S$ corresponds to a $\GL_r(A)$-fibre bundle on the simplicial set $S$. These bundles are the ones studied by Karoubi in \cite{Kar87}. To compare the relative Chern character with the Chern character in Deligne-Beilinson cohomology however, it is necessary to extend the theory to general simplicial varieties.

The idea to use relative Chern character classes (of bundles on simplicial sets) for the construction of a relative Chern character on $K$-theory is completely due to Karoubi.
\smallskip

(ii) We want to mention the relation to Karoubi's relative Chern character (\cite{Kar87}, \cite{CK}, see also example \ref{ex:Karoubi_top_bundles}). There the setup is a little bit different from ours. Let $A$ be a complex Fr{\'e}chet algebra and define the simplicial ring $A_{\dot}$ as $\Sing^{\infty}(\Delta^{\dot}) \widehat{\otimes}_{\pi} A$. Then $K^{-i}_{\top}(A)$ is by definition $\pi_i(B_{\dot}\GL(A_{\dot}))$ and $K_i^{\rel}(A)$ is by definition the $i$-th homotopy group of the homotopy fibre of $|B_{\dot}\GL(A)|^+ \to |B_{\dot}\GL(A_{\dot})|$. 

Let $\Omega_*(A)$ be the differential graded algebra of \emph{non-commutative differential forms} \cite[2.1]{CK}. The \emph{non-commutative de Rham homology} $\ol H_*(A)$ is the homology of the complex $\ol\Omega_*(A) := \Omega_*(A)/[\Omega_*(A),\Omega_*(A)]$, where we divide by the submodule generated by the graded commutators.\footnote{$\ol H_n(A) \cong \ker(\ol{HC}_n(A) \xrightarrow{B} H_{n+1}(A,A))$, where $\ol{HC}$ denotes reduced continuous cyclic homology, $H_*(A,A)$ is continuous Hochschild homology and $B$ is Connes' $B$-operator \cite[2.4]{CK}. Hence everything that follows, may also be formulated in terms of cyclic homology.}

Let $S$ be any simplicial set, and $E/S$ a $\GL_r(A_{\dot})$-fibre bundle on $S$ \cite[5.1]{Kar87}. 
Define $\Omega^*(S, A)$ to be the complex of de Rham--Sullivan forms on $S$ with coefficients in $\ol\Omega_*(A)$\footnote{If $A=\Sing^{\infty}(X)$ for a manifold $X$, this is a non-commutative analogue of Dupont's complex $A^*(X\otimes S)$.}. Thus an $n$-form in $\Omega^*(S,A)$ is a compatible family of $n$-forms $(\omega_{\sigma})_{\sigma\in S}$, where for each $p$-simplex $\sigma$ the form $\omega_{\sigma}$ lives in $\Omega^n(\sigma; A):= \bigoplus_{k+l=n} \A^k(\Delta^p) \widehat\otimes_{\pi} \ol\Omega_l(A)$.

Connections and curvature are defined as in our geometric situation using non-commutative differential forms. For example a connection is given by a family of matrices $\Gamma_i(\sigma) \in \Mat_r(\Omega^1(\sigma;A)), \sigma\in S_p, i=0, \dots, p$, satisfying similar relations as in definition \ref{def:Connection}. Then one constructs Chern character classes $\Ch_n(E) \in H^{2n}(\Omega^*(S, A)) \cong \bigoplus_{k+l=2n}\Hom(H_k(S), \ol H_l(A))$ in the same way as we did \cite[5.28]{Kar87}. 

Since each $\Omega^*(\sigma;A)$ is by definition the total complex associated with a double complex, the same is true for $\Omega^*(S,A)$. Hence we can filter $\Omega^*(S,A)$ with respect to the second index. 

If $E$ now is a $\GL_r(A)$-fibre bundle on $S$, it is easy to see, that it has well defined Chern character classes $\Ch_n(E) \in H^{2n}(\Fil^n\Omega^*(S,A)) \cong \bigoplus_{\substack{k+l=2n\\ k<l}}\Hom(H_k(S), \ol H_l(A)) \oplus \Hom(H_n(S), \ol Z_n(A))$, where $\ol Z_n(A)$ denotes the cycles of degree $n$ in $\ol\Omega_*(A)$.\footnote{Karoubi uses another subcomplex $\mathscr C^*(S,A)$ instead of $\Fil^n\Omega^*(S,A)$, which nevertheless has the same cohomology in degree $2n$.}

In the same way as we did in section \ref{sec:SecondaryClasses}, one can then construct secondary classes $\Ch_n^{\rel}(E,F,\alpha) \in H^{2n-1}(\Omega^*(S,A)/\Fil^n)$ for triples $(E,F,\alpha)$, where $E,F$ are $\GL_r(A)$-fibre bundles on $S$ and $\alpha$ is an isomorphism of the induced $\GL_r(A_{\dot})$-bundles.

In \cite{Kar87} Karoubi uses a geometric interpretation of $K_i(A)$ and $K_i^{\rel}(A)$ in terms of ``virtual'' $\GL(A)$-bundles on $i$-spheres to define Chern character maps $\Ch_{n,i}$ on $K_i(A)$ and relative Chern character maps $\Ch_{n,i}^{\rel} : K_i^{\rel}(A) \to \ol H_{2n-1-i}(A)$, if $i>n$, and $\Ch_{n,n}^{\rel}: K_n^{\rel}(A) \to \ol\Omega_{n-1}(A)/\ol B_{n-1}(A)$, $\ol B_{n-1}(A)$ denoting the boundaries in degree $n-1$ \cite[6.21, 6.22]{Kar87}. 
Note, that one can write this in the following uniform way: $\Ch_{n,i}^{\rel} : K_i^{\rel}(A) \to H^{2n-i-1}(\ol\Omega^{<n}(A))$, $i = 0, \dots, 2n-1$, where $\ol \Omega^{<n}(A)$ denotes as usual the truncated complex.
It is not hard to see (cf. \cite[5.17]{Kar87}, \cite[Th\'eor\`eme 3.4]{CK}), that this construction is ``the same'' as the one we used via the Hurewicz map.

Now assume, that $A = \mathscr C^{\infty}(X)$  is the ring of smooth functions on a manifold $X$. Since the algebra of smooth complex-valued differential forms on $X$, $\A^*(X)$, is a differential graded algebra with $\A^0(X) = A$, there is a unique morphisms of DGAs $\Omega_*(A) \to \A^*(X)$, which is the identity in degree $0$. Hence Karoubi's relative Chern character induces morphisms $K_i^{\rel}(\mathscr C^{\infty}(X)) \to H^{2n-i-1}(\A^{<n}(X))$. Insofar, our relative Chern character is analogous to Karoubi's one. 

If $X$ is a smooth separated scheme of finite type over $\C$, one can construct a natural map $K_i^{\rel}(X) \to K_i^{\rel}(\mathscr C^{\infty}(X))$. Moreover, the relative Chern character $\Ch_{n,i}^{\rel} : K_i^{\rel}(X) \to H^{2n-i-1}(X,\C)/\Fil^n$ may be composed with the natural maps $H^{2n-i-1}(X,\C)/\Fil^n \to \Hyp^{2n-i-1}(X, \Omega^{<n}_X) \to \Hyp^{2n-i-1}(X, \A^{<n}_X) = H^{2n-i-1}(\A^{<n}(X))$, and it is clear from the constructions, that the diagram
\[
\xymatrix{
K_i^{\rel}(X) \ar[r] \ar[d]_{\Ch_{n,i}^{\rel}} & K_i^{\rel}(\mathscr C^{\infty}(X)) \ar[d]^{\Ch_{n,i}^{\rel}}\\
H^{2n-i-1}(X,\C)/\Fil^n \ar[r] & H^{2n-i-1}(\A^{<n}(X))
}
\]
commutes.
\end{rems}

\section[Comparison with the Deligne-Beilinson Chern character]{Comparison with the Chern character in Deligne-Beilinson cohomology}
\label{sec:ComparisonChernCharacters}

The Chern character in Deligne-Beilinson cohomology is constructed in exactly the same way as the relative Chern character above:

Let $X= \Spec(A)$ be a smooth affine $\C$-scheme of finite type as in the previous section. We have again the natural  morphisms of simplicial varieties $X \otimes B_{\dot}\GL_r(A) \to B_{\dot}\GL_r(\C)$. Call the corresponding algebraic bundle $G_r$.
As in the relative case the  Chern character classes $\Ch_n^{\DB}(G_r) \in H^{2n}_{\DB}(X\otimes B_{\dot}\GL_r(A), \Q(n))$ are compatible with respect to the maps $H^{2n}_{\DB}(X\otimes B_{\dot}\GL_{r+1}(A),\Q(n)) \xrightarrow{j_r^*} H^{2n}_{\DB}(X\otimes B_{\dot}\GL_r(A),\Q(n))$ and thus yield a well defined class in
$H^{2n}_{\DB}(X\otimes B_{\dot}\GL(A),\Q(n))$. This class in turn yields maps
$H_i(B_{\dot}\GL(A), \Z) \to H^{2n-i}_{\DB}(X, \Q(n))$ and, for $i > 0$, we define the Chern character maps $\Ch_{n,i}^{\DB}$ on $K$-theory to be the composition
\begin{multline*}
\Ch_{n,i}^{\DB}: K_i(X) = \pi_i(|B_{\dot}\GL(A)|^+) \xrightarrow{\mathrm{Hur.}} H_i(|B_{\dot}\GL(A)|^+, \Z) \cong\\ 
\cong H_i(B_{\dot}\GL(A), \Z) \to H^{2n-i}_{\DB}(X,\Q(n)).
\end{multline*}
\begin{rem}
This is the construction used by Soul\'e \cite[2.3]{SouleReg}. It is just a down to earth version of the more general constructions of Chern characters in \cite[\S 4]{SchneiderBeil} or \cite{GilletRR}.
\end{rem}

\begin{thm}\label{thm:ComparisonChernCharacters}
The diagram
\[\xymatrix{
K_i^{\rel}(X) \ar[r]\ar[d]^{(-1)^{n-1}\Ch_{n,i}^{\rel}} & K_i(X) \ar[d]^{\Ch_{n,i}^{\DB}}\\
H^{2n-i-1}(X, \C)/\Fil^nH^{2n-i-1}(X, \C) \ar[r] & H^{2n-i}_{\DB}(X, \Q(n))
}
\]
commutes. 
\end{thm}
\begin{proof}
This is now an easy consequence of theorem \ref{thm:comparison} and the constructions.

We use the notations of the last two sections. Then $E_r/X\otimes \mathscr F_r$ is just the pullback of $G_r/X\otimes B_{\dot}\GL_r(A)$ by the morphism $X\otimes \mathscr F_r \to X\otimes B_{\dot}\GL_r(A)$. It follows from theorem \ref{thm:comparison} and functoriality, that $(-1)^{n-1}\widetilde\Ch_n^{\rel}(T_r, E_r, \alpha_r) \in H^{2n-1}(X\otimes \mathscr F_r, \C)/\Fil^n$ and $\Ch_n^{\DB}(G_r) \in H^{2n}_{\DB}(X\otimes B_{\dot}\GL(A), \Q(n))$ are mapped to the same class in $H^{2n}_{\DB}(X\otimes \mathscr F_r, \Q(n))$, namely to $\Ch_n^{\DB}(E_r)$. It follows, that we have commutative diagrams
\[\xymatrix@C+1.7cm{
H_i(\mathscr F_r, \Z) \ar[r]^-{(-1)^{n-1}\widetilde\Ch_n^{\rel}(T_r, E_r, \alpha_r)}\ar[d]\ar[dr]^-{\Ch_n^{\D}(E_r)} & H^{2n-i-1}(X, \C)/\Fil^nH^{2n-i-1}(X, \C) \ar[d]\\
H_i(B_{\dot}\GL_r(A)), \Z) \ar[r]_-{\Ch_n^{\D}(G_r)} & H^{2n-i}_{\DB}(X, \Q(n)),
}
\]
where the arrows are induced by the specified classes.
Going to the limit $r \to \infty$ and using the commutativity of diagram \eqref{diag1} the claim follows.
\end{proof}

\section{Non affine varieties}
\label{sec:NonAffineVar}

Using Jouanolou's trick, we may extend the definition of the relative Chern character to all smooth separated schemes of finite type (= varieties) over $\C$ (cf. section 6.2 of Gillet's article in \cite{GilletTopK}).

A \emph{Jouanolou torsor} over a scheme $X$ is an affine scheme $W$ together with an affine map $W \to X$ such, that, for some vector bundle $E$ over $X$, $W$ is a torsor for $E$. According to Jouanolou and Thomason, every smooth separated scheme of finite type over a field admits a Jouanolou torsor \cite[Proposition 4.4]{WeibelKH}. 

Let $X$ be a variety over $\Spec(\C)$ and fix a Jouanolou torsor $\pi: W \to X$. Since $X$ is smooth, so is $W$, and Quillen's algebraic $K$-theory of locally free $\O^{\alg}_X$-modules of finite rank $K_*(X)$ is isomorphic to the $K$-theory of coherent $\O^{\alg}_X$-modules $K'_*(X)$ \cite[\S 7.1]{Quillen}. By {\it loc. cit.} \S 7 Proposition 4.1 $\pi^* : K_*(X) = K'_*(X) \to K_*(W) =  K'_*(W)$ is an isomorphism.

On the other hand $\pi : W \to X$ is a homotopy equivalence\footnote{If $\mathscr E$ denotes the sheaf of $\mathscr C^{\infty}$-sections of the underlying vector bundle $E$ of $W$, then $H^1(X,\mathscr E) = 0$, since $\mathscr E$ is fine. Hence, topologically, $W$ is a trivial torsor, i.e. $W \cong E$ over $X$.}, hence it also induces an isomorphism in topological $K$-theory $\pi^*: K^*_{\top}(X) \xrightarrow{\cong} K^*_{\top}(W)$.

\begin{dfn} 
Let $X$ be a variety over $\Spec(\C)$. We define the relative $K$-theory $K_i^{\rel}(X)$ for $i \geq 1$ by
\[
K_i^{\rel}(X) = K_i^{\rel}(W),
\]
where $W$ is any Jouanolou torsor over $X$. The map $K_i^{\rel}(X) \to K_i(X)$ is given by the composition $K_i^{\rel}(W) \to K_i(W) \xrightarrow{(\pi^*)^{-1}} K_i(X)$ and the map $K_i(X) \to K^{-i}_{\top}(X)$ is given by the composition $K_i(X) \xrightarrow{\pi^*} K_i(W) \to K^{-i}_{\top}(W) \xrightarrow{(\pi^*)^{-1}} K^{-i}_{\top}(X)$.
\end{dfn}
Of course, this is only well defined up to isomorphism:
If $\pi' : W' \to X$ is a second Jouanolou torsor, the fibre product $W'' = W \times_X W'$ is again a Jouanolou torsor over $X$ and we have isomorphisms $K_*(W) \xleftarrow{\cong} K_*(W'') \xrightarrow{\cong} K_*(W')$ and similar for topological $K$-theory. By the five lemma, also $K_i^{\rel}(W) \xleftarrow{\cong} K_i^{\rel}(W'') \xrightarrow{\cong} K_i^{\rel}(W')$. Moreover, the map $K_i(X) \to K^{-i}_{\top}(X)$ is well defined.

Since $\pi : W \to X$ is a homotopy equivalence, it induces an isomorphism $H^*(X, \Z) \xrightarrow{\cong} H^*(W, \Z)$. This is a morphism of mixed Hodge structures, hence an isomorphism of mixed Hodge structures. In particular $\pi^*$ also induces isomorphisms $\Fil^nH^*(X, \C) \to \Fil^nH^*(W, \C)$, $H^*(X, \C)/\Fil^nH^*(X, \C) \xrightarrow{\cong} H^*(W, \C)/\Fil^nH^*(W, \C)$ and $H^*_{\DB}(X, \Q(n)) \xrightarrow{\cong} H^*_{\DB}(W, \Q(n))$, which we use to define the relative Chern character and the Chern character in Deligne-Beilinson cohomology (this is the method used by Schneider in \cite[\S 4]{SchneiderBeil}).
\begin{dfn}
The Chern character in Deligne-Beilinson cohomology 
$
\Ch_{n,i}^{\DB} : K_i(X) \to H^{2n-i}_{\DB}(X, \Q(n))
$
is given by the composition 
\[
K_i(X) \xrightarrow{\pi^*} K_i(W) \xrightarrow{\Ch_{n,i}^{\DB}} H^{2n-i}_{\DB}(W, \Q(n)) \xrightarrow{(\pi^*)^{-1}} H^{2n-i}_{\DB}(X, \Q(n)).
\]
Similarly, the relative Chern character 
$
\Ch_{n,i}^{\rel} : K_i^{\rel}(X) \to H^{2n-i-1}(X, \C)/\Fil^n
$
is given by the composition
\begin{multline*}
K_i^{\rel}(X) = K_i^{\rel}(W) \xrightarrow{\Ch_{n,i}^{\rel}} H^{2n-i-1}(W, \C)/\Fil^nH^{2n-i-1}(W, \C) \xrightarrow{(\pi^*)^{-1}} \\ H^{2n-i-1}(X, \C)/\Fil^nH^{2n-i-1}(X, \C).
\end{multline*}
\end{dfn}
From the constructions it is clear that theorem \ref{thm:ComparisonChernCharacters} remains valid in this situation:
\begin{thm}
The diagram
\[\xymatrix{
K_i^{\rel}(X) \ar[r]\ar[d]^{(-1)^{n-1}\Ch_{n,i}^{\rel}} & K_i(X) \ar[d]^{\Ch_{n,i}^{\DB}}\\
H^{2n-i-1}(X, \C)/\Fil^nH^{2n-i-1}(X, \C) \ar[r] & H^{2n-i}_{\DB}(X, \Q(n))
}
\]
commutes. 
\end{thm}
\begin{rem}
Note, that, if $X$ is smooth and projective, then for $i>0$ the map $K_i(X)\to K^{-i}_{\top}(X)$ has torsion image (cf. \cite[6.3]{GilletTopK}). Hence the map $K_i^{\rel}(X) \to K_i(X)$ is rationally surjective and thus the relative Chern character is in some sense the interesting part of the Deligne-Beilinson Chern character.
\end{rem}

\section[The regulators of Borel and Beilinson]{The case $X = \Spec(\C)$: The regulators of Borel and Beilinson}
\label{sec:Borel}

In this section we apply the above comparison result in the case $X=\Spec(\C)$ and obtain as a corollary a new proof of Burgos' theorem, that Borel's regulator is twice Beilinson's regulator. We first give a different description of the homotopy fibre of $B\GL_r(\C)^{\delta} \to B\GL_r(\C)$ and use it to give an explicit cocycle for the relative Chern character. This cocycle may then be compared with a representative of Borel's regulator in Lie algebra cohomology using the explicit description of the van Est isomorphism due to Dupont. At this point one sees how well-suited the Chern-Weil theoretic description of characteristic classes is for the computation of regulators.

Here and in the following we denote by $\GL_r(\C)^{\delta}$ the group of invertible complex $n\times n$-matrices equipped with the \emph{discrete} topology. 

\subsection{An explicit cocycle}
\label{sec:ExplicitCocycle}
In the notations of the previous sections we fix $A=\C$, $X = \Spec(\C)$. In particular, we have the simplicial groups $G_{r,\dot} = S_{\dot}^{\infty}(\GL_r(\C))$, whose realization is equivalent to $\GL_r(\C)$ with the usual topology, and the simplicial set $\mathscr F_r$, defined by diagram (\ref{diag2}) and homotopy equivalent to the homotopy fibre of $B_{\dot}\GL_r(\C) \to B_{\dot}G_{r,\dot}$. Recall that by construction the relative Chern character factors through the homology of the simplicial set $\mathscr F = \dirlim_r\mathscr F_r$. 

In the present situation there is another model for $\mathscr F_r$, that will be useful for us, see 
Appendix \ref{app:SimplGrps}\footnote{Actually, in the Appendix $\GL_r(\C)\backslash G_{r,\dot}$ is replaced by the isomorphic simplicial set $G_{r,\dot}/\GL_r(\C)$, the isomorphism being induced by $\sigma \mapsto \sigma^{-1}$. This is due to different conventions in the cited literature and I hope, it will not cause too much confusion.}:
We have a commutative diagram of simplicial sets
\[
\xymatrix{
& \mathscr F_r \ar@{}[dr]|{\lrcorner}\ar[d]\ar[r]_{\alpha_r} & E_{\dot}G_{r,\dot} \ar[d]^p\\
\GL_r(\C)\backslash G_{r,\dot} \ar@{.>}[ur]^{\eta_r}\ar@/^1.2cm/[urr]^{\beta_r}\ar[r]^-{\rho_r} & B_{\dot}\GL_r(\C)^{\delta} \ar[r] &  B_{\dot}G_{r,\dot},
}
\]
where $\beta_r$ is given in degree $p$ by $\beta_r(\sigma) = (\sigma(e_0)^{-1}\sigma, \dots, \sigma(e_p)^{-1}\sigma)$ and the map $\eta_r$, induced by $\beta_r$ and $\rho_r$, is a weak homotopy equivalence (lemma \ref{lem:HomotopyFibres}). 
Here $e_i$ denotes the $i$-th standard basis vector $(0, \dots, 1, \dots, 0)$ and $\sigma(e_i)$ is also the same as $\tau_i^*\sigma$ with $\tau:[0]\to [p], 0\mapsto i$.
This translates into a commutative diagram of topological morphisms of simplicial manifolds
\[\xymatrix{
&& E_{\dot}\GL_r(\C) \ar[d]^p\\
X\otimes \GL_r(\C)\backslash G_{r,\dot} \ar[r]^-{\eta_r} \ar@{~>}[urr]^-{\beta_r}& X\otimes\mathscr F_r \ar[r]^-{g_r} \ar@{~>}[ur]^{\alpha_r} & B_{\dot}\GL_r(\C).
}
\]

\begin{prop}\label{prop:ExplicitCocycle}
The composition 
\[
H_{2n-1}(\GL_r(\C)\backslash G_{r,\dot}, \Z) \xrightarrow{\cong} H_{2n-1}(\mathscr F_r, \Z) \xrightarrow{\widetilde\Ch_n^{\rel}(T_r, E_r, \alpha_r)} H^0(X, \C)/\Fil^n = \C
\]
is given by the cocycle
\[
\sigma \mapsto (-1)^n\frac{(n-1)!}{(2n-1)!} \Tr \int_{\Delta^{2n-1}}(\sigma^{-1}d\sigma)^{2n-1}.
\]
\end{prop}
\begin{rem}
Hamida obtained a similar result \cite{HamidaCR}.
\end{rem}

\begin{proof}
Since $r$ is fixed, we drop the subscript $r$ in the following. Since $X$ is proper, it makes no difference if we work with $\widetilde\Ch_n^{\rel}(T,E,\alpha)$ or with $\Ch_n^{\rel}(T,E,\alpha)$.
It is clear from the commutativity of the above diagram, that the composition in the statement of the proposition is induced by $\Ch_n^{\rel}(T, \eta^*E, \beta)$. This class can be computed explicitely: Since $X$ is a point, the standard connection on the bundle $\eta^*E$ is given by the zero matrix (cf. the formula in example \ref{ex:standard_connection}). Then the pullback to the trivial bundle via $\beta$ is given by $\beta_i^{-1}d\beta_i$ (see remark \ref{rem:FunctorialityConnection} (i)). 
$\beta_i$ is given on the $p$-simplex $\sigma \in \GL_r(\C)\backslash G_{r,p}$ by the matrix $\sigma(e_i)^{-1}\sigma \in G_{r,p}=\Sing^{\infty}(\Delta^p, \GL_r(\C))$, hence $\beta_i^{-1}d\beta_i = \sigma^{-1}d\sigma$ on the simplex $\sigma$.
We denote the corresponding simplicial form simply by $\sigma^{-1}d\sigma$.

By construction $\Ch_n^{\rel}(\Gamma^{T}, \Gamma^{\eta^*E}, \beta)$ is given by $\int_0^1 (i_{\del/\del t} \Ch_n(\Gamma))dt$, where $\Gamma$ is the connection given by $\Gamma_i = (1-t) \beta_i^{-1}d\beta_i = (1-t)\sigma^{-1}d\sigma$ on the trivial $\GL_r(\C)$-bundle on $(X\otimes (\GL_r(\C)\backslash G_{r,\dot})) \times \C$, $t$ denoting the coordinate on $\C$. 

The curvature of $\Gamma$ is given by
\begin{multline*}
R_i = d\Gamma_i + \Gamma_i^2 = -dt(\sigma^{-1}d\sigma) -(1-t)(\sigma^{-1}d\sigma)^2 + (1-t)^2(\sigma^{-1}d\sigma)^2\\
= -dt(\sigma^{-1}d\sigma) + (t^2-t)(\sigma^{-1}d\sigma)^2.
\end{multline*}
Hence $R_i^n = (t^2-t)^n(\sigma^{-1}d\sigma)^{2n} - ndt (t^2-t)^{n-1}(\sigma^{-1}d\sigma)^{2n-1}$ and 
\begin{eqnarray*}
\Ch_n^{\rel}(\Gamma^{T}, \Gamma^{\rho^*E}, \beta) &=& \frac{1}{n!}\int_0^1 i_{\del/\del t} \Tr(R_i^n) dt\\
&=& -\frac{n}{n!} \Tr\int_0^1 (t^2-t)^{n-1} (\sigma^{-1}d\sigma)^{2n-1}dt\\
&=& -\frac{1}{(n-1)!} (\int_0^1 (t^2-t)^{n-1} dt) \Tr\left((\sigma^{-1}d\sigma)^{2n-1}\right)\\
&=& (-1)^n\frac{(n-1)!}{(2n-1)!}\Tr\left((\sigma^{-1}d\sigma)^{2n-1}\right).
\end{eqnarray*}
Here we used that $\int_0^1 (t^2-t)^{n-1} dt = (-1)^{n-1}\int_0^1 t^{n-1}(1-t)^{n-1}dt = (-1)^{n-1} \mathrm B(n,n) = (-1)^{n-1}\frac{\Gamma(n)\cdot \Gamma(n)}{\Gamma(n+n)}=(-1)^{n-1}\frac{((n-1)!)^2}{(2n-1)!}$, where $\mathrm B$ is Euler's Beta function \cite[section 4.2]{Carlson}.
Now the claim follows from remark \ref{rem:ExplicitDescriptionCocycle}.
\end{proof}

\subsection{An explicit description of the van Est isomorphism}

Consider $\GL_r(\C)$ as a real Lie group with maximal compact subgroup $U(r)$. Denote the corresponding Lie algebras by $\gl_r$ resp. $\LieU_r$. If $V$ is a finite dimensional real vector space with a continuous $\GL_r(\C)$-action, the van Est isomorphism
\[
H^*(\gl_r, \LieU_r; V) \cong H^*_{\mathrm{cts}}(\GL_r(\C), V)
\]
relates relative Lie algebra cohomology with continuous group cohomology. 

Recall, that in general, if $G$ is a connected Lie group and $K \subset G$ a subgroup with Lie algebras $\mathfrak{g}$ and $\mathfrak{k}$ respectively, and $V$ a trivial $G$-module, the relative Lie algebra cohomology $H^*(\mathfrak g, \mathfrak k; V)$ is the cohomology of the complex $\A^*(G/K; V)^G$ of smooth $V$-valued differential forms on $G/K$, that are invariant under the left action of $G$ (see e. g. \cite[Example 5.39]{Burgos}).

To compare Borel's regulator with the relative Chern character, we need the following description of the composition of the van Est isomorphism with the natural map $H^*_{\mathrm{cts}}(\GL_r(\C), V) \to H^*_{\mathrm{grp}}(\GL_r(\C), V) = H^*(B_{\dot}\GL_r(\C)^{\delta}, V)$ from continuous to discrete group cohomology.
\begin{prop}\label{prop:ExplicitVanEst}
We have a commutative diagram
\[\xymatrix{
H^*_{\mathrm{cts}}(\GL_r(\C), V) \ar[r] & H^*(B_{\dot}\GL_r(\C)^{\delta}, V) \ar[r]^-{\rho_r^*} & H^*(\GL_r(\C)\backslash G_{r,\dot}, V),\\
H^*(\gl_r, \LieU_r; V) \ar[u]^{\mathrm{van~ Est}}_{\cong} \ar[r] &H^*(\gl_r; V) \ar[ur]_-{\phi}
}
\]
where $\phi$ is induced by the chain map $\phi$ sending a left invariant form $\omega$ to the simplicial cocycle
\begin{equation}\label{eq:cocycle}
\GL_r(\C) \backslash S^{\infty}_p(\GL_r(\C)) \ni \sigma \mapsto \int_{\Delta^p} \sigma^*\omega.
\end{equation}
\end{prop}
\begin{proof}
The proof is based on the explicit description of the van Est isomorphism by Dupont in \cite[Proposition 1.5]{Dup76} and \cite[(proof of) Proposition 9.10]{DupLNM}.

First of all \eqref{eq:cocycle} is well defined, since $\omega$ is left invariant, and $\phi$ is a chain map by Stoke's theorem.

In the following, we use the abbreviations $G:= \GL_r(\C)$, $K := U(r)$, $G_{\dot} :=  G_{r,\dot} = S^{\infty}_{\dot}(\GL_r(\C))$ and write $G^{\delta}$ when we consider $G = \GL_r(\C)$ as a group with the \emph{discrete} topology.

Consider the simplicial manifold $E_{\dot}G^{\delta} \underset{G^{\delta}}{\times} G/K = B_{\dot}(G^{\delta}; G/K)$ (the bundle with fibre $G/K$ associated with the principal bundle $E_{\dot}G^{\delta} \to B_{\dot}G^{\delta}$). It is given in degree $p$ by $B_p(G^{\delta};G/K) = B_pG^{\delta} \times G/K$ with face operators
\[
\del_i(g_1, \dots, g_p, gK) = 
\begin{cases}
(g_2, \dots, g_p, gK), &  i = 0,\\
(g_1, \dots, g_ig_{i+1}, \dots, g_p, gK), & 0<i<p,\\
(g_1, \dots, g_{p-1}, g_pgK), & i=p.
\end{cases}
\]
Denote by $\widetilde p$ the canonical projection $B_{\dot}(G^{\delta}; G/K) \to B_{\dot}G^{\delta}$. Since $G/K$ is contractible, $\widetilde p$ induces an isomorphism in de Rham cohomology (cf. \cite[proof of proposition 9.10]{DupLNM}).

On the other hand, we have a commutative diagram
\[\xymatrix{
& B_{\dot}(G^{\delta}; G/K) \ar[d]^{\widetilde p}\\
G\backslash G_{\dot} \ar@{~>}[ur]^{\gamma} \ar[r]^{\rho} & B_{\dot}G^{\delta},
}
\]
where the topological morphism $\gamma$ is given in degree $p$ by $\Delta^p \times G\backslash G_p \to B_p(G^{\delta}; G/K)$, $(t, \sigma) \mapsto (\sigma(e_0)^{-1}\sigma(e_1), \dots, \sigma(e_{p-1})^{-1}\sigma(e_p), \sigma(e_p)^{-1}\sigma(t)K)$, and $\rho$ was defined in the previous subsection.

If now $\omega$ is a left invariant $V$-valued differential form, its pullbacks 
along the projections $\Delta^p \times B_p(G^{\delta};G/K) \to G/K$ give a well defined simplicial $V$-valued form on $B_{\dot}(G^{\delta};G/K)$, where $V$-valued simplicial forms are defined similar as in section \ref{sec:SimplDeRham}. Thus we get a natural map of complexes $\mathrm{pr}_2^*:\A^*(G/K;V)^G \to A^*(B_{\dot}(G^{\delta}; G/K);V)$.

On cohomology we have the commutative diagram
\[\xymatrix@C-0.2cm{
H^*(\A^*(G/K;V)^G) \ar[r]^-{\mathrm{pr}_2^*} & H^*(A^*(B_{\dot}(G^{\delta};G/K);V)) \ar[dr]^{\gamma^*}\\
H^*(\gl_r, \LieU_r;V)\ar@{=}[u] & H^*(A^*(B_{\dot}G^{\delta};V)) \ar[u]_{\cong}^{\widetilde p^*} \ar[r]^{\rho^*} \ar[d]_I^{\cong} & H^*(A^*(G\backslash G_{\dot}; V)) \ar[d]_I^{\cong}\\
& H^*(B_{\dot}G^{\delta};V) \ar[r]^{\rho^*} & H^*(G\backslash G_{\dot}; V).
}
\]
Here $I$ is the isomorphism of theorem \ref{thm:Dupont} (with $V$-coefficients) in the special case of a simplicial set considered as a simplicial manifold and is given by integration over the standard simplex.

It follows from the explicit description of the van Est isomorphism in \cite[Proposition 1.5]{Dup76} and \cite[Proposition 9.10 and the remark following it]{DupLNM}, that the composition $I \circ (\widetilde p^*)^{-1} \circ \mathrm{pr}_2^*$ is the same as the composition of the van Est isomorphism with the natural map from continuous to discrete group cohomology. 

Hence the composition $H^*(\gl_r, \LieU_r; V) \to H^*(G\backslash G_{\dot}; V)$ we are looking for is given by $I \circ \gamma^* \circ \mathrm{pr}_2^*$. Since $\mathrm{pr}_2 \circ \gamma$ is given in degree $p$ by $\Delta^p \times (G\backslash G_p) \to G/K$, $(t, \sigma) \mapsto \sigma(e_p)^{-1}\sigma(t)K$, an invariant form $\omega$ is sent by $I \circ \gamma^* \circ \mathrm{pr}_2^*$ to the simplicial cocycle
\[
\sigma \mapsto \int_{\Delta^p} (L_{\sigma(e_p)^{-1}} \circ \sigma)^*\omega = \int_{\Delta^p} \sigma^*\omega,
\]
where -- by abuse of notation -- we still denote by $\sigma$ the composition of $\sigma$ with the natural projection $G \to G/K$ and $L_g$ denotes the left translation with $g$.

Now it is obvious, that $I \circ \gamma^* \circ \mathrm{pr}_2^*$ factors through the map $\phi$ as claimed.
\end{proof}

\subsection{Comparison of the regulators}
First we give the definitions of the regulators we use.

\begin{dfn}
The \emph{Beilinson regulator} is by definition the Chern character with values in \emph{real} Deligne-Beilinson cohomology:
\[
r_{\mathrm{Be}}: K_{2n-1}(\C) \xrightarrow{\Ch_{n,2n-1}^{\DB}} H^1_{\D}(\Spec(\C), \Q(n)) \to H^1_{\DB}(\Spec(\C), \R(n)).
\]
Here $H^1_{\DB}(\Spec(\C), \R(n))$ is the cohomology in degree $1$ of the complex $\R(n) \to \C$, hence canonically isomorphic to $\C/\R(n)$ which in turn is isomorphic to $\R(n-1)$ via the projection $\pi_{n-1}: \C \to \R(n-1)$, $z\mapsto \frac{1}{2}(z + (-1)^{n-1}\bar z)$, and we will view $r_{\mathrm{Be}}$ as a map with values in $\R(n-1)$.
\end{dfn}

Next we shortly recall the construction of Borel's regulator (see e. g. \cite[Ch. 9]{Burgos}).

Let $G$ be $\GL_r(\C)$ viewed as real Lie group with maximal compact subgroup $K = U(r)$ with Lie algebras $\mathfrak g = \gl_r$ and $\mathfrak k = \LieU_r$ respectively. Let $\mathfrak g = \mathfrak k \oplus \mathfrak p$ be the corresponding Cartan decomposition. If $V$ is a finite dimensional real vector space with trivial $G$-action, there are canonical isomorphisms 
\[
H^*_{\mathrm{cts}}(G, V) \underset{\text{van Est}}{\cong} H^*(\mathfrak g, \mathfrak k; V) = H^*(\mathfrak g, \mathfrak k) \otimes V \cong (\bigwedge\nolimits^*\mathfrak p^{\vee})^K \otimes V.
\]
The right hand side is the complex of $K$-invariant alternating forms on $\mathfrak p$ with values in $V$.
It is computed as follows.

The complexification of $G=\GL_r(\C)$ is $G_{\C} = \GL_r(\C) \times \GL_r(\C)$ and the compact real form $U$ of $G_{\C}$ is $U(r) \times U(r)$ with Lie algebra $\LieU = \mathfrak k \oplus i\mathfrak p \subset \mathfrak g_{\C}$.

Now we have a chain of isomorphisms
\begin{align*}
H^{2n-1}(G, \R(n)) &\cong H^{2n-1}(K, \R(n)) &&\text{since $G/K$ is contractible}\\
&\cong H^{2n-1}(U/K, \R(n)) &&\text{since $K\cong U/K$}\\
&\cong H^{2n-1}(\LieU, \mathfrak k; \R(n)) &&\text{since $U$ is compact}\\
&\cong \Hom_{\mathfrak k}(\bigwedge\nolimits^{2n-1}(i\mathfrak p), \R(n)) \\
&\cong \Hom_{\mathfrak k}(\bigwedge\nolimits^{2n-1}\mathfrak p, \R(n-1)) &&\text{multiplication with $i^{2n-1}$}\\
&\cong H^{2n-1}(\mathfrak g, \mathfrak k; \R(n-1)) \\
&\cong H^{2n-1}_{\mathrm{cts}}(G, \R(n-1)) &&\text{van Est}
\end{align*}
and therefore natural maps
\begin{multline*}
H^{2n}(B_{\dot}\GL_r(\C), \R(n)) \xrightarrow{\text{suspension}} H^{2n-1}(\GL_r(\C), \R(n)) \\
\cong H^{2n-1}_{\mathrm{cts}}(\GL_r(\C), \R(n-1)) \to H^{2n-1}_{\mathrm{grp}}(\GL_r(\C), \R(n-1)).
\end{multline*}

The image $\mathrm{Bo}_n$ of the $n$-th universal Chern character class $\Ch_n^{\top}(E^{\univ}) \in H^{2n}(B_{\dot}\GL_r(\C), \R(n))$ in the group cohomology $H^{2n-1}_{\mathrm{grp}}(\GL_r(\C), \R(n-1))$ then induces (for $r$ large enough) the Borel regulator $r_{\mathrm{Bo}}$ via
\[
K_{2n-1}(\C) \xrightarrow{\text{Hur.}} H_{2n-1}(B_{\dot}\GL(\C)^{\delta}, \Z) \cong H_{2n-1}(B_{\dot}\GL_r(\C)^{\delta}, \Z) \xrightarrow{\mathrm{Bo}_n} \R(n-1).
\]

We also denote by $\mathrm{Bo}_n$ the image of $\Ch_n^{\top}(E^{\univ})$ in the relative Lie algebra $H^{2n-1}(\gl_r, \LieU_r; \R(n-1))$. We need Burgos' description of its image in absolute Lie algebra cohomology:
\begin{lemma}\label{lemma:ExplicitLieCocycle}
The image of $\mathrm{Bo}_n$ 
in $H^{2n-1}(\gl_r, \R(n-1))$ is represented by the left invariant differential form 
\[
-2\frac{(n-1)!}{(2n-1)!}\pi_{n-1}\circ \Tr((g^{-1}dg)^{2n-1}),
\]
$g^{-1}dg$ denoting the Maurer-Cartan form on $\GL_r(\C)$ and $\pi_{n-1}$ the projection $\C\to \R(n-1)$.
\end{lemma}
\begin{proof}
Obviously, the above form is left invariant. At the unit element the Maurer-Cartan form is just the identity $\gl_r \to \gl_r$. Hence the above form corresponds to the alternating form on $\gl_r$, that is given by
\[
x_1\wedge\dots\wedge x_{2n-1} \mapsto -2\frac{(n-1)!}{(2n-1)!}\pi_{n-1}\left(  
  \sum_{\tau\in\Sy_{2n-1}} \sgn(\tau)\Tr(x_{\tau(1)}\cdots x_{\tau(2n-1)}) \right),
\]
where $\Sy_{2n-1}$ denotes the symmetric group on $2n-1$ elements.

It follows from \cite[Proposition 9.26]{Burgos}, that this represents the image of $\mathrm{Bo}_n$ in $H^{2n-1}(\gl_r, \R(n-1))$. Remark that Burgos' cocycle differs from ours by the factor $(-1)^n$. This is explained by the fact, that Burgos uses another normalization of the Chern classes. His ``twisted Chern character class'' $\mathrm{ch}_n$ is $(-1)^n\Ch_n^{\top}$, cf. remark \ref{rem:NormalizationChernClasses}.
\end{proof}

\begin{thm}[Burgos \cite{Burgos}]
\[
r_{\mathrm{Bo}} =  2 r_{\mathrm{Be}}.
\]
\end{thm}
\begin{rem}
Beilinson \cite{Beilinson} proved, that both regulators coincide up to a non zero rational factor. Many details of Beilinson's proof were provided by Rapoport \cite{Rapoport}. Dupont, Hain, and Zucker \cite{DHZ} conjectured that the factor should be $2$. This was proven by Burgos using Beilinson's original argument and making all the normalizations and identifications precise.
\end{rem}

\begin{proof}
Since the odd topological $K$-theory of $\Spec(\C)$ vanishes, the map $K_{2n-1}^{\rel}(\C) \to K_{2n-1}(\C)$ is surjective. 
By construction of the regulators resp. the relative Chern character and the comparison result of theorem \ref{thm:ComparisonChernCharacters} it then suffices to show, that the diagram
\[\xymatrix@R-0.2cm{
H_{2n-1}(\GL_r(\C)\backslash G_{r,\dot}, \Z) \ar[r]^{\rho_r^*} \ar[d]_{(-1)^{n-1}\Ch^{\rel}_{n,2n-1}} & H_{2n-1}(B_{\dot}\GL_r(\C)^{\delta}, \Z) \ar@/^/[ddr]^{\frac{1}{2}\mathrm{Bo}_n}\\
H^0(\Spec(\C), \C)/\Fil^n \ar@{=}[d] \ar[r] & H^1_{\DB}(\Spec(\C), \R(n)) \ar@{=}[d] \\ 
\C \ar[r] & \C/\R(n) \ar[r]^{\pi_{n-1}}_{\cong}& \R(n-1)
}
\]
commutes. Note that by our constructions, the map $\C \to \C/\R(n)$ is really the projection.
\smallskip

According to proposition \ref{prop:ExplicitCocycle}, $\pi_{n-1} \circ ((-1)^{n-1}\Ch^{\rel}_{n,2n-1})$ is induced by the cocycle $\sigma \mapsto  -\pi_{n-1}\left(\frac{(n-1)!}{(2n-1)!} \Tr\int_{\Delta^{2n-1}} (\sigma^{-1}d\sigma)^{2n-1} \right)$.
\smallskip

On the other hand, by lemma \ref{lemma:ExplicitLieCocycle}, the image of $\mathrm{Bo}_n$ in $H^{2n-1}(\gl_r; \R(n-1))$ is given by the invariant differential
\[
-2\frac{(n-1)!}{(2n-1)!}\pi_{n-1}\circ \Tr((g^{-1}dg)^{2n-1}).
\]

Hence, by proposition \ref{prop:ExplicitVanEst}, the composition $\frac{1}{2}\mathrm{Bo}_n \circ \rho_r^*$ is induced by the cocycle
\[
\GL_r(\C)\backslash G_{r,\dot} \ni \sigma \mapsto -\frac{(n-1)!}{(2n-1)!}\pi_{n-1} \Tr \int_{\Delta^{2n-1}} (\sigma^{-1}d\sigma)^{2n-1},
\]
thus proving the theorem.
\end{proof}

\part{The $p$-adic theory}

\chapter*{Introduction}

As mentioned in the main introduction, the goal of this second part is to construct a relative Chern character for smooth affine $R$-schemes of finite type, where $R$ is a complete discrete valuation ring 
with field of fractions $K$ of characteristic $0$ and residue field $k$ of characteristic $p >0$, and compare it with the $p$-adic Borel regulator in the case of the ring of integers in a finite extension of $\Q_p$. Thus the structure of part II is parallel to that of part I. 

Let us only mention the following points: Whereas in the complex situation we had nice functorial complexes computing de Rham cohomology, namely the complex of $\mathscr C^{\infty}$-differential forms, this is not the case for dagger spaces. The de Rham cohomology of a dagger space $X$, $H^*_{\dR}(X/K) = \Hyp^*(X, \Omega^*_X)$, equals the cohomology of the complex $\Omega^*(X)$ in general only, if $X$ is acyclic for the cohomology of coherent sheaves. Thus one can compute the de Rham cohomology of a simplicial dagger space $X_{\dot}$ by simplicial differential forms in the style of Dupont only, if each $X_p$ is acyclic for the cohomology of coherent sheaves. For instance, this is the case, if each $X_p$ is affinoid or the dagger space associated with an affine $K$-scheme, e.g. the classifying space $B_{\dot}\GL_{r,K}^{\dag}$ or the universal principal bundle $E_{\dot}\GL_{r,K}^{\dag}$. This does not cause any problems for the construction of Chern-Weil theoretic classes, but, for example, a pullback map on de Rham cohomology for a topological morphism $Y_{\dot} \rightsquigarrow X_{\dot}$ is a priori only well defined, if each $X_p$ is acyclic for the cohomology of coherent sheaves. Thus everything works fine, if one restricts to the affine case, which is enough for the construction of regulators on $K$-theory, but things become more complicated if one wants a nice general theory.

Next a few comments on the relation with the syntomic Chern character.
For simplicity let us assume, that $R$ is the ring of integers in a finite \emph{unramified} extension $K$ of $\Q_p$. Let $X$ be a smooth affine $R$-scheme of finite type. There is a natural map from the algebraic de Rham cohomology of $X_K$ to the rigid cohomology of $X_k$ and on the latter, there is a natural Frobenius $\phi$. The rigid syntomic cohomology $H^*_{\syn}(X,n)$ of $X$ as developed systematically by Besser \cite{Besser}, is the cohomology of the complex
\[
\Cone\left(\Fil^n\Rhyp\Gamma_{\dR}(X_K/K) \xrightarrow{1-\frac{\phi}{p^n}} \Rhyp\Gamma_{\rig}(X_k/K)\right)[-1].
\]
On the other hand, the generic fibre $\hat X_K$ of the weak completion of $X$ (section \ref{sec:PrelimSpaces}) is a dagger space, whose de Rham cohomology is naturally isomorphic to the rigid cohomology of $X_k$ \cite[Kap. 8]{GK}. If we define relative cohomology groups $H^*_{\rel}(X,n) = H^*(\Cone(\Fil^n\Rhyp\Gamma_{\dR}(X_K/K) \to \Rhyp\Gamma_{\dR}(\hat X_K/K)))$, this enables us to construct a natural map $H^{*-1}_{\rel}(X,n) \to H^*_{\syn}(X,n)$ (induced by the natural map $H^*_{\dR}(\hat X_K/K) \to H^*_{\rig}(X_k/K)$ and $1-\frac{\phi}{p^n}:H^*_{\rig}(X_k/K) \to H^*_{\rig}(X_k/K)$) and conjecturally (cf. Besser's talk \cite{BesserTalk}), the diagram
\[
\xymatrix{
K_i^{\rel}(X) \ar[d]_{\Ch_{n,i}^{\rel}} \ar[r] & K_i(X) \ar[d]^{\Ch_{n,i}^{\syn}}\\
H^{2n-i-1}_{\rel}(X, n) \ar[r] & H^{2n-i}_{\syn}(X,n)
}
\]
commutes (up to a sign). 

One can try to prove this as in part I. What one has to prove is, that for an algebraic $\GL_r$-bundle $E$ on a smooth simplicial (affine) $R$-scheme $X_{\dot}$, whose induced bundle on the simplicial dagger space $(\widehat X_{\dot})_K$ is topologically trivialized by $\alpha$, the class $\widetilde\Ch_n^{\rel}(T,E,\alpha) \in H^{2n-1}_{\rel}(X_{\dot}, n)$ is mapped to $\Ch_n^{\syn}(E)\in H^{2n}_{\syn}(X_{\dot}, n)$. As in chapter \ref{ch:CharClassesAlg} (see chapter \ref{ch:RefinedSecClassesAlgBundles} for precise definitions) one constructs refined classes $\widetilde \Ch_n^{\rel}(E) \in H^{E,2n-1}_{\rel}(X_{\dot}, n)$ and a pullback map $\alpha^*\colon H^{E,2n-1}_{\rel}(X_{\dot}, n) \to H^{2n-1}_{\rel}(X_{\dot}, n)$, such that $\alpha^*\widetilde\Ch_n^{\rel}(E) = -\widetilde\Ch_n^{\rel}(T,E,\alpha)$. One may also define the groups $H^{E,*}_{\syn}(X_{\dot}, n)$ (cf. the proof of theorem \ref{thm:comparison}) and show, that in the diagram
\[
\xymatrix{
H^{E,2n-1}_{\rel}(X_{\dot}, n) \ar[r] \ar@/^1pc/[d]^{\alpha^*} & H^{E, 2n}_{\syn}(X_{\dot}, n) \ar@/^1pc/@{.>}[d]^{\exists ?}\\
H^{2n-1}_{\rel}(X_{\dot}, n) \ar[u]^{p^*} \ar[r] & H^{2n}_{\syn}(X_{\dot}, n) \ar[u]^{p^*}
}
\]
the refined class $\widetilde\Ch_n^{\rel}(E)$ is mapped to $p^*(\Ch_n^{\syn}(E))$ by the upper horizontal map. Thus it would suffice to construct the dotted arrow, which has to be a left inverse of $p^*$ and compatible with $\alpha^*$ on the left hand side. It is not hard to see, that one can construct a left inverse of $p^*$ on the right hand side, which is induced by $\alpha$, but I was not able to show the compatibility with $\alpha^*$ on the left hand side. The fundamental problem, which occurs, is that it is not clear, that the Frobenius map on the rigid cohomology of the special fibre of $X_{\dot}$ is compatible with the pullback by a topological morphism on the de Rham cohomology of the generic fibre of the weak completion under the identification mentioned above.

In the special case, where $X = \Spec(R)$, one can in fact use the above methods to compare the relative Chern character with the syntomic Chern character, but there the result also follows from the comparison of the relative Chern character with the $p$-adic Borel regulator, which is achieved in section \ref{sec:pAdicBorel}, and the comparison of the $p$-adic Borel regulator with the syntomic regulator in \cite{HK}.

\chapter{Preliminaries}

This chapter recalls the relevant definitions and facts about dagger spaces, that will be used in the subsequent chapters.

\section{Affinoid algebras}
\label{sec:PrelimAlgebras}

Let $R$ be a complete discrete valutation ring with maximal ideal $(\pi)$, perfect residue field $R/(\pi) = k$ of characteristic $p > 0$ and field of fractions $K$ of characteristic $0$. We denote by $|\;.\;|$ an absolute value on $K$.

We define the \emph{$K$-Tate algebra in $n$ variables} to be the algebra of power series converging on the unit disc
\[
T_n := K\<x_1,\dots, x_n\> = \{ \sum a_{\nu}x^{\nu} \in K[[x_1, \dots, x_n]]\;|\; |a_{\nu}| \xrightarrow{|\nu|\to\infty} 0\}.
\]
Here $\nu=(\nu_1, \dots, \nu_n)$ runs over the multiindices $\N_0^n$ and $x^{\nu}$ is by definition $x_1^{\nu_1}\cdots x_n^{\nu_n}$.
This is a $K$-Banach algebra with respect to the \emph{Gau{\ss} norm} $|\sum a_{\nu}x^{\nu}| = \max_{\nu} |a_{\nu}|$. The \emph{$R$-Tate algebra} is the subalgebra 
\[
R\<x_1, \dots,x_n\> = \{f\in K\<x_1, \dots, x_n\>\;|\; |f|\leq 1\},
\]
i.e. the algebra consisting of convergent power series with coefficients in $R$.

The \emph{$K$-Washnitzer algebra} $K\<x_1, \dots,x_n\>^{\dag}$ is the subalgebra of the $K$-Tate algebra consisting of \emph{overconvergent power series}, i.e.
\[
W_n := K\<x_1,\dots, x_n\>^{\dag} = \{ \sum a_{\nu} x^{\nu} \in K[[x_1, \dots, x_m]] \;|\; \exists \rho > 1 : |a_{\nu}|\rho^{|\nu|} \to 0\}.
\]
Finally the \emph{$R$-Washnitzer algebra} is the algebra of overconvergent power series with coefficients in $R$:
\[
R\<x_1, \dots,x_n\>^{\dag} = \{ f\in K\<x_1,\dots, x_n\>^{\dag}\;|\; |f| \leq 1\}.
\]

We will sometimes also use the algebras of power series converging on a disc of radius $\rho >0$
\[
T_n(\rho) := \{ \sum a_{\nu}x^{\nu} \in K[[x_1, \dots, x_n]]\;|\; |a_{\nu}|\rho^{|\nu|} \xrightarrow{|\nu|\to\infty} 0\}
\]
with norm $|\sum a_{\nu}x^{\nu}|_{\rho} := \max_{\nu} |a_{\nu}|\rho^{|\nu|}$. 
When we want to specify the names of the variables, we will sometimes denote this algebra by $K\<\rho^{-1}x_1,\dots, \rho^{-1}x_n\>$.
These are Banach algebras as well and the Washnitzer algebra (as an abstract algebra) may be written as the direct limit $W_n = \dirlim_{\rho\searrow 1} T_n(\rho) = \bigcup_{\rho > 1} T_n(\rho)$.

\medskip

A $K$- resp. \emph{$R$-affinoid algebra} is a homomorphic image of a $K$- resp. $R$-Tate algebra, a $K$- resp. \emph{$R$-dagger algebra} is a homomorphic image of a $K$- resp. $R$-Washnitzer algebra.
$R$-dagger algebras are also called \emph{weakly complete, weakly finitely generated $R$-algebras} (\emph{wcfg-algebras} for short).
All these algebras are Noetherian (\cite[5.2.6. Theorem 1]{BGR}, \cite[Korollar 1.3]{GK}, \cite[Theorem 2.1]{MW}).

\medskip

We also write $\ul x$ for the set of variables $x_1, \dots, x_n$. According to \cite[5.2.7/8]{BGR}  resp. \cite[Proposition 1.5]{GK}  all ideals in $K\<\ul x\>$ resp. $K\<\ul x\>^{\dag}$ are closed. The same is true for $T_n(\rho)$ \cite[section 2.1]{Berkovich}.
If $A$ is $K$-affinoid, hence of the form $K\<\ul x\>/I$ with an ideal $I$, then $A$ may be equipped with the residue norm of the Gau{\ss} norm. In fact, all the norms arising in this way are equivalent \cite[6.1]{BGR}. The corresponding statement for $K$-dagger algebras also holds \cite[Satz 1.9]{GK}. If $A$ is an $R$-affinoid or an $R$-dagger algebra, then $A_K := A \otimes_R K$ is a $K$-affinoid resp. $K$-dagger algebra.

\medskip

Every ideal $I$ in $R\<\ul x\>$ is finitely generated, hence $\pi$-adically separated and complete. In particular, $I$ is closed for the $\pi$-adic topology on $R\<\ul x\>$. Since the $\pi$-adic topology on $R\<\ul x\>$ coincides with the topology induced by the Gau{\ss} norm, we see that every $R$-affinoid algebra $A = R\<\ul x\>/I$ equipped with the residue norm is still an ultrametric Banach ring.

\medskip

By \cite[Proposition 1.11]{GK} $K$-dagger algebras are \emph{weakly complete} in the sense that, if $A$ is a $K$-dagger algebra and $a_1, \dots, a_n$ are power bounded elements of $A$, then the natural homomorphism $K[x_1,\dots, x_n] \to A, x_i\mapsto a_i$, admits a continuous extension to $K\<x_1,\dots, x_n\>^{\dag} \to A$.

Let $A$ be a $K$-dagger algebra and choose a representation $A = K\<\ul x\>^{\dag}/I$. then the completion $\hat A$ of $A$ is the $K$-affinoid algebra $K\<\ul x\>/IK\<\ul x\>$ \cite[Proposition 1.7]{GK}. Similarly, if $A = R\<\ul x\>^{\dag}/I$ is an $R$-dagger algebra, its $\pi$-adic completion is $\hat A = R\<\ul x\>/IR\<\ul x\>$.

If $A$ is an $R$-algebra, its $\pi$-adic completion is $\hat A = \projlim_n A/\pi^nA$.
If $A$ is of finite type, $\hat A$ is an $R$-affinoid algebra.
The \emph{weak completion} $A^{\dag}$ of $A$ is by definition the subalgebra of $\hat A$ consisting of all elements $z \in \hat A$ having a representation $z = \sum_{j=0}^{\infty} p_j(y_1, \dots, y_n)$, where $y_1, \dots, y_n \in A$, $p_j \in \pi^jR[x_1, \dots, x_n]$, and there exists a constant $c$ such that $\deg p_j \leq c\cdot(j+1)$ for all $j$ \cite[Definition 1.1]{MW}.
$A$ is called \emph{weakly complete} if $A \to A^{\dag}$ is bijective. The weak completion $A^{\dag}$ is always weakly complete \cite[Theorem1.2]{MW}.
Explicitely, if $A = R[\ul x]/I$, then $A^{\dag} = R\<\ul x\>^{\dag}/IR\<\ul x\>^{\dag} \subset \hat A = R\<\ul x\>/IR\<\ul x\>$.

\medskip

Morphisms of $R$- resp. $K$-dagger and affinoid algebras are morphisms of algebras. They are automatically continuous (clear for the $R$-case, \cite[6.1.3. Theorem 1]{BGR} resp. \cite[Proposition 1.8]{GK} in the $K$-case). All the four corresponding categories admit coproducts (cf. \cite[6.1.1. Proposition 11]{BGR}, \cite[Satz 1.19]{GK}). E.g. if $A= R\<\ul x\>^{\dag}/I$ and $B = R\<\ul Y\>^{\dag}/J$ are $R$-dagger algebras, their coproduct is given by $A \otimes_R^{\dag} B := R\<\ul x, \ul y\>^{\dag}/(I+J)$ (using \cite[Theorem 1.5]{MW} it is easy to check the universal property directly).

\section{Dagger spaces, weak formal schemes}
\label{sec:PrelimSpaces}

The general reference for dagger spaces is \cite{GK}. 
We do not recall all the details of the definition \cite[Kapitel 2]{GK}, which is parallel to the case of rigid spaces. If $A$ is a $K$-dagger algebra, we denote by $\Sp(A)$ the set of maximal ideals in $A$. This set is then endowed with a Grothendieck topology and a structure sheaf $\O_{\Sp(A)}$, so that $(\Sp(A), \O_{\Sp(A)})$ is a locally G-ringed space \cite[Proposition 2.9]{GK}, an \emph{affinoid $K$-dagger space}.  A general $K$-dagger space is then a locally G-ringed space $(X,\O_X)$, whose underlying Grothendieck topology is saturated (i.e. satisfying $(G_0), (G_1), (G_2)$ of \cite[9.1.2.]{BGR}), and which is locally isomorphic to an affinoid dagger space.

\medskip

There exists a \emph{``dagger analytification functor''} $(\,.\,)^{\dag}$ from the category of $K$-schemes of finite type to the category of $K$-dagger spaces \cite[Korollar 2.18]{GK}, more precisely: For any $K$-scheme $X$ there is an associated dagger space $X^{\dag}$ together with a morphism of locally G-ringed spaces $X^{\dag} \to X$, which is final in the category of all morphisms from a $K$-dagger space to $X$. It follows from this universal property, that $(\,.\,)^{\dag}$ commutes with products.

If $X$ is affine, $X^{\dag}$ may be described explicitely as follows (cf. \cite[1.13]{Bosch} for the rigid analogue and details of the proof): Choose a representation $A = K[\ul x]/I$ and $c\in K, |c|>1$. Define $K\<c^{-n}\ul x\>^{\dag}$ to be the algebra of power series, which are overconvergent on the disc of radius $|c|^n$, i.e.  series $\sum_{\nu} a_{\nu} \ul x^{\nu}$ satisfying $|a_{\nu}|\rho^{|\nu|} \xrightarrow{|\nu|\to \infty} 0$ for some $\rho > |c|^n$. Each $K\<c^{-n}\ul x\>^{\dag}$ may be identified with $K\<\ul x\>^{\dag}$ via $\ul x \mapsto c^n\cdot\ul x$, in particular is a $K$-dagger algebra. We have natural inclusions $K[\ul x] \subset K\<c^{-(n+1)}\ul x\>^{\dag} \subset K\<c^{-n}\ul x\>^{\dag}$ and hence
\[
K[\ul x]/I \to \dots \to K\<c^{-(n+1)}\ul x\>^{\dag}/(I) \to K\<c^{-n}\ul x\>^{\dag}/(I) \to \dots \to K\<\ul x\>^{\dag}/(I)
\]
inducing a sequence of inclusions as affinoid subdomains
\[
\Sp(K\<\ul x\>^{\dag}/(I)) \to \Sp(K\<c^{-1}\ul x\>^{\dag}/(I)) \to \Sp(K\<c^{-2}\ul x\>^{\dag}/(I)) \to \dots,
\]
whose union is $X^{\dag}$.

\bigskip

Next we recall the definition and basic properties of weakly formal ($R$-)schemes (\cite[Kapitel 3]{GK} and originally \cite{Mer}). Let $A$ be an $R$-dagger algebra and $\ol A = A/(\pi)$. Then $D(\bar f) \mapsto (A_f)^{\dag}$, where $f \in A$ is a preimage of $\bar f$ and $(A_f)^{\dag}$ denotes the weak completion of the localization $A_f$, is a sheaf of local rings on the topological space underlying $\Spec(\ol A)$. The corresponding locally ringed space is the \emph{affine weak formal $R$-scheme} $\Spwf(A)$. A general weak formal ($R$-)scheme is a locally ringed space, which locally is isomorphic to an affine weak formal $R$-scheme. 

\medskip

The weak completion of $R$-algebras of finite type induces a functor $\widehat{(\,.\,)}$ from the category of $R$-schemes of finite type to that of weak formal schemes. 

\medskip

On the other hand, if $A$ is an $R$-dagger algebra, then $A_K =A \otimes_R K$ is a $K$-dagger algebra, and we get the \emph{generic fibre functor} $(\,.\,)_K$ from weak formal schemes to $K$-dagger spaces sending $X = \Spwf(A)$ to $X_K=\Sp(A_K)$. Moreover, for any weak formal $R$-scheme $X$ there exists a natural morphism of ringed sites
\[
\spec: X_K \to X, \quad \text{such that } \spec_*\O_{X_K} \cong \O_X \otimes_R K,
\]
called the \emph{specialization map}.

\medskip

If $X$ is an $R$-scheme of finite type, there exists a natural morphism of dagger spaces $(\hat X)_K \to (X_K)^{\dag}$, which is an open immersion if $X$ is separated and an isomorphism if $X$ is proper over $R$ (cf. \cite[Proposition 0.3.5]{BerCohomRig}).

\chapter{Chern-Weil theory for simplicial dagger spaces}

In this chapter we formulate and prove the analogue of Dupont's theorem \ref{thm:Dupont} for simplicial dagger spaces and use it to develop Chern-Weil theory in this setting. 
Fix $R$ and $K$ as in the previous chapter.

\section{De Rham cohomology}
Let $X$ be a smooth $K$-dagger space (cf. \cite[p. 40]{GK} for the definition of smoothness) and $\Omega^1_X = \Omega^1_{X/\Sp(K)}$ the locally free sheaf of $1$-forms on $X$ \cite[Lemma 5.3]{GK}. We denote its global sections simply by $\Omega^1(X)$. If $U=\Sp(A) \subset X$ is affinoid, then $\Omega^1_X(U) = \Omega^1_U(U)=\Omega^1(A)$ is the \emph{universally finite} differential module, i.e. $d : A \to \Omega^1(A)$ is universal for $K$-derivations from $A$ in finite $A$-modules \cite[Lemma 5.1]{GK}.

We define the sheaf of $n$-forms as $\bigwedge^n_{\O_X} \Omega^1_X$ and get as usual the complex of sheaves $\Omega^*_X$. The \emph{de Rham cohomology of $X$} is by definition 
\[
H^*_{\dR}(X/K) := \Hyp^*(X, \Omega^*_X).
\]
As usual, if $X_{\dot}$ is a smooth simplicial dagger space, the sheaves $\Omega^n_{X_p}$ on $X_p$, $p\geq 0$, together with the pullback maps form a sheaf on the  simplicial dagger space $X_{\dot}$ and the de Rham cohomology of $X_{\dot}$ is by definition $H^*_{\dR}(X_{\dot}/K) := \Hyp^*(X_{\dot}, \Omega^*_{X_{\dot}})$.

\medskip

We need the analogue of Dupont's theorem in the dagger context.
The analogues of the standard simplices are the affinoid dagger spaces 
\[
\Delta^p := \Sp\bigl(K\<x_0, \dots, x_p\>^{\dag}\big/(\sum_i x_i -1)\bigr), \quad p \geq 0.
\]
Then $\Omega^1(\Delta^p) = \bigoplus_{i=0}^p \frac{K\<x_0, \dots, x_p\>^{\dag}}{\sum_i x_i -1}dx_i \big/(\sum_i dx_i)$.
In fact, it is easy to see, that $d: K\<x_0, \dots, x_p\>^{\dag}/(\sum_i x_i -1) \to \Omega^1(\Delta^p)$, $f\mapsto \sum_i \frac{\del f}{\del x_i}dx_i$, is universal for $K$-derivations of $K\<x_0, \dots, x_p\>^{\dag}/(\sum_i x_i -1)$ in finite modules (cf. \cite[2.2.5]{Berger}).

For any increasing map $\phi: [p] \to [q]$, we define $\phi_{\Delta}: \Delta^p \to \Delta^q$ by $K\<x_0, \dots, x_q\>^{\dag}\big/(\sum_i x_i -1) \ni x_i \mapsto \sum_{j:\phi(j)=i} x_j \in K\<x_0, \dots, x_p\>^{\dag}\big/(\sum_i x_i -1)$. This map is well defined, since the elements $\sum_{j:\phi(j)=i} x_j$ have norm $\leq 1$, hence are power bounded (cf. section \ref{sec:PrelimAlgebras}). In particular, $[p]\mapsto \Delta^p$ defines a cosimplicial dagger space. 

\begin{dfn}
A simplicial $n$-form on the simplicial dagger space $X_{\dot}$ is a family of $n$-forms $(\omega_p)_{p\geq 0}$, where $\omega_p \in \Omega^n(\Delta^p \times X_p)$ and for all $p \geq 0, i = 0, \dots p$
\[
(\delta^i \times 1)^*\omega_p = (1\times \del_i)^*\omega_{p-1} \text{ in } \Omega^n(\Delta^{p-1}\times X_p).
\]
The space of simplicial $n$-forms is denoted by $D^n(X_{\dot})$.
We get a commutative differential graded $K$-algebra $D^*(X_{\dot})$ by applying the wedge product and the exterior differential component-wise.
\end{dfn}

\begin{rems}\label{rem:DiffFormsDoubleComplex}
(i) Let $X$ and $Y$ be two dagger spaces and consider their product $X \times Y$ with projections $p_1:X\times Y\to X, p_2: X\times Y\to Y$. Then there is a natural isomorphism 
\[
\Omega^1_{X\times Y} = p_1^*\Omega^1_X \oplus p_2^*\Omega^1_Y.
\]
In fact, the question being local, it suffices to consider the case, where $X$ and $Y$ are affinoid, and there the result follows as in \cite[2.2.2.a)]{Berger}. Hence we get a decomposition
\[
\Omega^n_{X\times Y} = \bigoplus_{k+l=n} \Omega^{k,l}_{X\times Y}, \quad \text{where} \quad \Omega^{k,l}_{X\times Y} := p_1^*\Omega^k_X \otimes_{\O_{X\times Y}} p_2^*\Omega^l_Y.
\]
Obviously, the differential $\Omega^n_{X\times Y} \xrightarrow{d} \Omega^{n+1}_{X\times Y}$ sends $\Omega^{k,l}_{X\times Y}$ to $\Omega^{k+1,l}_{X\times Y} \oplus \Omega^{k,l+1}_{X\times Y}$ and we denote the two components of $d$ by $d_X$ and $d_Y$ respectively. Since $dd=0$, it follows, that $d_Xd_X = 0, d_Yd_Y=0, d_Xd_Y=-d_Yd_X$. In other words, $(\Omega^*_{X\times Y},d)$ is the total complex associated with the double complex $(\Omega^{*,*}_{X\times Y}, d_X,d_Y)$. This double complex is functorial in $X$ and $Y$.

\smallskip

(ii) If $X$ is a dagger space, the complex of global sections $\Omega^*(\Delta^p\times X)$ is the total complex associated with the double complex $(\Omega^{*,*}(\Delta^p\times X), d_{\Delta}, d_X)$. It follows, that if $X_{\dot}$ is a (strict) simplicial dagger space, then $D^*(X_{\dot})$ is the total complex associated with the double complex $(D^{*,*}(X_{\dot}), d_{\Delta}, d_X)$, where $D^{k,l}(X_{\dot})$ consists of those forms $\omega=(\omega_p)_{p\geq 0}$, such that $\omega_p \in \Omega^{k,l}(\Delta^p\times X_p)$ for all $p\geq 0$.

We denote by $\Fil^*D^*(X_{\dot})$ the filtration of $D^*(X_{\dot})$ with respect to the second index:
\[
\Fil^nD^*(X_{\dot}) = \bigoplus_{k+l=*, l\geq n} D^{k,l}(X_{\dot}).
\]
\end{rems}

Our goal is to construct a filtered homotopy equivalence $D^*(X_{\dot}) \to \Omega^*(X_{\dot})$ given by integration along the standard simplices, similar to the classical case. Here on the right hand side $\Omega^*(X_{\dot})$ denotes the total complex of the cosimplicial complex $[p]\mapsto \Omega^*(X_p) = \Gamma(X_p, \Omega^*_{X_p})$.

First we have to introduce some more notation: Let $I := \Sp(K\<t\>^{\dag})$. Then $\Omega^1(I) = K\<t\>^{\dag}dt$, $\Omega^n(I) = 0$, if $n > 1$. 
If $X = \Sp(A)$ is affinoid, then $I\times X= \Sp(A\<t\>^{\dag})$, where $A\<t\>^{\dag} := A \otimes^{\dag}_K K\<t\>^{\dag}$. Explicitely, if $A= K\<\ul x\>^{\dag}/I$, then $A\<t\>^{\dag} = K\<\ul x, t\>^{\dag}/(I)$.
\begin{lemma}\label{lem:FormalIntegration}
There exists a unique $A$-linear map $\int_0^1 \;.\;dt: A\<t\>^{\dag} \to A$, that sends $t^k$ to $\frac{1}{k+1}$.
If $f \in A\<t\>^{\dag}$, its formal derivative with respect to $t$, $\frac{\del f}{\del t} \in A\<t\>^{\dag}$, is well-defined and $\int_0^1 \frac{\del f}{\del t}dt = f(1) - f(0)$.
\end{lemma}
This is the crucial point, where overconvergence and hence dagger spaces come into play.
\begin{proof}
We first consider the case $A = W_n = K\<x_1, \dots, x_n\>^{\dag}$. Then $A\<t\>^{\dag} = W_{n+1} = K\<x_1, \dots, x_n, t\>^{\dag}$. If $f \in W_{n+1}$, there exists $\rho > 1$, such that $f \in T_{n+1}(\rho)$ (see section \ref{sec:PrelimAlgebras}). Write $f = \sum_{k=0}^{\infty} g_k t^k$, $g_k = \sum_{\nu\in \N_0^n} a_{k,\nu} \ul x^{\nu}$.
Then clearly $g_k \in T_n(\rho)$. Moreover
\[
|g_k|_{\rho} = \max_{\nu\in\N_0^n} |a_{k,\nu}|\rho^{|\nu|} = \rho^{-k}\max_{\nu\in\N_0^n} |a_{k,\nu}|\rho^{|\nu|+k}
\leq \rho^{-k}|f|_{\rho}.
\]
Hence $|g_k\frac{1}{k+1}|_{\rho} \leq |\frac{1}{k+1}|\cdot \rho^{-k}\cdot |f|_{\rho}$. But $|\frac{1}{k+1}|\cdot \rho^{-k}$ tends to $0$ as $k$ tends to infinity,\footnote{The valuation on $\Q$ induced by the valuation $|\;.\;|$ on $K$ is equivalent to the $p$-adic valuation $|\;.\;|_p$; hence there exists a constant $c>0$ such that $|\frac{1}{k+1}| = |\frac{1}{k+1}|_p^c\leq (k+1)^c$.}\label{footnote:Const} hence $\sum_{k=0}^{\infty} \frac{1}{k+1}g_k$ converges in $T_n(\rho) \subset W_n$. We define $\int_0^1 fdt := \sum_{k=0}^{\infty}\frac{1}{k+1}g_k$. 

Clearly $\frac{\del f}{\del t} = \sum_{k=0}^{\infty} (k+1)g_{k+1}t^k \in W_{n+1}$ is well-defined and $\frac{\del}{\del t} : W_{n+1} \to W_{n+1}$ is $W_n$-linear. The last formula of the assertion follows directly from the constructions.

In general, $A$ may be written as a quotient $A = W_n/I$. Then $A\<t\>^{\dag} = W_{n+1}/I\cdot W_{n+1}$ and by linearity we have $\int_0^1 (I\cdot W_{n+1})dt \subset I$. Hence, $\int_0^1\;.\;dt: W_{n+1} \to W_n$ induces the desired map $A\<t\>^{\dag} \to A$. Similarly $\frac{\del}{\del t}$ induces the morphism $\frac{\del}{\del t} : A\<t\>^{\dag} \to A\<t\>^{\dag}$ and the last formula of the assertion follows from the case of the Washnitzer algebra treated before.
\end{proof}
\begin{rem}\label{rem:ContIntegration}
For later reference we observe the following: Let $\rho > 1$ and $f= \sum_{k=0}^{\infty}g_kt^k\in T_{n+1}(\rho)$ be as in the above proof. There exists a constant $C > 0$, such that $|\frac{1}{k+1}|\rho^{-k} \leq C$ for all $k\in \N$. Hence $|\int_0^1 fdt|_{\rho} \leq C\cdot |f|_{\rho}$ and $\int_0^1\,.\,dt: T_{n+1}(\rho) \to T_{n}(\rho)$ is continuous.

Iterated application of the integration operator constructed in the lemma gives a $K$-linear morphism $K\<x_1, \dots, x_n\>^{\dag} \xrightarrow{\int_0^1\,.\,dx_n} K\<x_1,\dots, x_{n-1}\>^{\dag} \to \dots \to K$ and it follows from the above, that the induced map $T_n(\rho) \to K, f \mapsto \int_0^1\dots\int_0^1 fdx_1\dots dx_n$ is continuous.
\end{rem}

If $X$ is a dagger space, let $p: I \times X \to X$ denote the projection and $i_j : X \hookrightarrow I \times X$ the inclusion induced by $K\<t\>^{\dag} \to K$, $t \mapsto j$, $j=0,1$. We have the pullback map $i_j^* : \Omega^*_{I\times X} \to (i_j)_*\Omega^*_X$ and, applying $p_*$, $i_j^*: p_*\Omega^*_{I\times X} \to p_*(i_j)_*\Omega^*_X = \Omega^*_X$.
\begin{lemma}\label{lem:HomotopyOperator}
There is a natural $\O_X$-linear morphism 
\[
K: p_*\Omega^n_{I\times X} \to \Omega^{n-1}_X
\]
satisfying $dK+Kd= i_1^* -i_0^*$. 
\end{lemma}
\begin{proof}
Let $U = \Sp(A) \subset X$ be open affinoid. Then $p^{-1}(U) = I \times U = \Sp(A\<t\>^{\dag})$. We have
\begin{align*}
p_*\Omega^n_{I\times X}(U) & = \Omega^n(I\times U) = \Omega^{0,n}(I\times U) \oplus \Omega^{1,n-1}(I\times U)\\
& = A\<t\>^{\dag} \otimes_A \Omega^n(A) \oplus A\<t\>^{\dag}dt \otimes_A \Omega^{n-1}(A).
\end{align*}
We define $K(U) : p_*\Omega^n_{I\times X}(U) \to \Omega^{n-1}(U)$ to be equal to the zero map on the first summand and $fdt\otimes \omega \mapsto (\int_0^1fdt)\cdot\omega$ on the second summand.

If $V = \Sp(B) \subset U$ is an admissible open affinoid, given by a morphism of dagger algebras $A \to B$, the maps $K(U)$ and $K(V)$ are clearly compatible with respect to the restriction map. Hence we get a well defined morphism of $\O_X$-modules $p_*\Omega^n_{I\times X} \to \Omega^{n-1}_X$. This map is clearly natural in $X$.

Finally, we show that 
\begin{equation}\label{eq:DiffAndInt}
dK+Kd_X = 0, \qquad Kd_I = i_1^* -i_0^*.
\end{equation}
This in particular implies the last formula of the claim.  Since \eqref{eq:DiffAndInt} is local on $X$, we may assume that $X = \Sp(A)$ is affinoid. Choose a presentation $A = K\<x_1,\dots, x_r\>^{\dag}/I$, i.e. a closed immersion $X \hookrightarrow \Sp(K\<x_1,\dots, x_r\>^{\dag}) =: \mathbf B^r$. By the naturality of $K$ we have a commutative diagram
\[
\xymatrix{
\Omega^n(I\times \mathbf B^r) \ar[d]_K \ar@{->>}[r] & \Omega^n(I \times X) \ar[d]^K\\
\Omega^{n-1}(\mathbf B^r) \ar@{->>}[r] & \Omega^{n-1}(X),
}
\]
where the horizontal maps are surjections. Hence it suffices to prove the claim for $X=\mathbf B^r$, where it can be checked by direct computation: An $n$-form on $I \times \mathbf B^r$ is a sum of forms of the types $g(\ul x, t)dtdx_{i_1}\dots dx_{i_{n-1}}$ and $f(\ul x,t) dx_{j_1} \dots dx_{j_n}$ with $g(\ul x,t), f(\ul x, t) \in K\<x_1,\dots, x_r, t\>^{\dag}$. Let us check the first formula for $\omega= g(\ul x, t) dtdx_{i_1}\dots dx_{i_{n-1}}$ as an example. Write $g(\ul x,t) = \sum_{k=0}^{\infty} g_k(\ul x)t^k$. Then $d_X\omega = -\sum_{j=1}^r\sum_{k=0}^{\infty} \frac{\del g_k(\ul x)}{\del x_j}t^k dtdx_jdx_{i_1}\dots dx_{i_{n-1}}$ and hence
\begin{align*}
K(d_X\omega) & = -\sum_{j=1}^r \sum_{k=0}^{\infty} \frac{1}{k+1}\frac{\del g_k(\ul x)}{\del x_j} dx_jdx_{i_1}\dots dx_{i_{n-1}}\\
 & = - \sum_{j=1}^r \frac{\del}{\del x_j} (\int_0^1 g(\ul x,t)dt)dx_jdx_{i_1}\dots dx_{i_{n-1}}\\
& = -dK(\omega).
\end{align*}
The remaining identities are shown similarly.
\end{proof}

Let $X_{\dot}$ be a simplicial dagger space. For each $l\geq 0$ we can consider 
the  cosimplicial group $[p] \mapsto \Omega^l(X_p)$. The associated complex is denoted by  $(\Omega^{*,l}(X_{\dot}), \delta)$.
\begin{thm}\label{thm:p-adicDupont}
Let $X_{\dot}$ be a simplicial dagger space. For each $l$ the two chain complexes $(D^{*, l}(X_{\dot}), d_{\Delta})$ and $(\Omega^{*,l}(X_{\dot}), \delta)$ are naturally chain homotopy equivalent.

In fact, there are natural maps $I : D^{k,l}(X_{\dot}) \rightleftarrows \Omega^{k,l}(X_{\dot}) : E$ and chain homotopies $s: D^{k,l}(X_{\dot}) \to D^{k-1,l}(X_{\dot})$ such that
\begin{align}
\label{eq:11} I \circ d_{\Delta} & = \delta \circ I,&  I \circ d_X & = d_X \circ I,\\
\label{eq:22} d_{\Delta} \circ E & = E \circ \delta,&  E\circ d_X & = d_X \circ E,\\
\label{eq:33} I\circ E & = \id,\\
\label{eq:44} E\circ I - id &= s\circ d_{\Delta} + d_{\Delta} \circ s,&  s\circ d_X & = d_X \circ s.
\end{align}
In particular we get isomorphisms $H^*(\Fil^nD^*(X_{\dot})) \cong H^*(\Omega^{\geq n}(X_{\dot}))$ and $H^*(D^*(X_{\dot})/\Fil^nD^*(X_{\dot})) \cong H^*(\Omega^{<n}(X_{\dot}))$ for any $n \geq 0$.
\end{thm}
\begin{rem}\label{rem:AffinoidCovering}
In general, the natural map $H^*(\Omega^*(X_{\dot})) \to \Hyp^*(X_{\dot}, \Omega^*_{X_{\dot}}) = H^*_{\dR}(X_{\dot}/K)$ is \emph{not} an isomorphism. However, this will be the case as soon as each $X_p$ is acyclic for coherent sheaves \cite[(5.2.3)]{HodgeIII}, e.g. affinoid or a Stein space, e.g. the dagger space associated with an affine $K$-scheme of finite type (cf. \cite{GK} Lemma 4.3 and p. 25 for the definition of a Stein space).
\end{rem}
\begin{proof}[Proof of the Theorem]
We adapt Dupont's proof of Theorem \ref{thm:Dupont} \cite[Theorem 2.3]{Dup76}.

For any $j = 0, \dots, p$ consider the morphism $g_j : I \times \Delta^p \to \Delta^p$ given on dagger algebras by $K\<x_0,\dots, x_p\>^{\dag}/(\sum_i x_i-1) \to K\<x_0, \dots, x_p, t\>^{\dag}/(\sum_i x_i -1)$, $x_i \mapsto \delta_{ij} \cdot t + (1-t)\cdot x_i$, where $\delta_{ij}$ is the Kronecker delta. This is well-defined since the target elements are power bounded and $\sum_i (\delta_{ij}t + (1-t)x_i) = 1$ in $K\<x_0,\dots, x_p, t\>^{\dag}/(\sum_i x_i -1)$.
Thus $g_j$ is a homotopy between $\id_{\Delta^p}$ and the constant map $e_j:\Delta^p \to \Delta^p$ given by $x_i \mapsto \delta_{ij}$.

For any dagger space $Y$ we can now define the homotopy operator $h_{(j)}$ to be the composition
\[
h_{(j)}:\Omega^n(\Delta^p\times Y) \xrightarrow{(g_j\times\id_Y)^*} \Omega^n(I\times \Delta^p \times Y) \xrightarrow{K} \Omega^{n-1}(\Delta^p\times Y).
\]
We have the analogue of \cite[Lemma 2.9]{Dup76}:
\begin{lemma} The operators $h_{(j)}$, $j = 0, \dots, p$, satisfy
\begin{gather*}
h_{(j)}\circ d_{\Delta} + d_{\Delta} \circ h_{(j)} = (e_j\times \id_Y)^* - \id,\\
h_{(j)} \circ d_Y + d_Y \circ h_{(j)} = 0
\end{gather*}
and for $i = 0, \dots, p$
\begin{gather*}
(\delta^i \times \id_Y)^*\circ h_{(j)} = h_{(j)} \circ (\delta^i \times \id_Y)^*, \quad i > j,\\
(\delta^i \times \id_Y)^*\circ h_{(j)} = h_{(j-1)} \circ (\delta^i \times \id_Y)^*, \quad i < j.
\end{gather*}
\end{lemma}
\begin{proof}
Since everything follows by formal computation, we only check the first statement.
Thus take $\omega \in \Omega^n(\Delta^p\times Y)$. Then
\begin{align*}
&\qquad h_{(j)} \circ d_{\Delta}(\omega) + d_{\Delta}\circ h_{(j)}(\omega) =\\
&= K((g_j\times \id_Y)^*d_{\Delta}\omega) + d_{\Delta}K((g_j\times\id_Y)^*\omega) \\
&= K(d_{I\times\Delta}(g_j\times\id_Y)^*\omega) - K(d_{\Delta}(g_j\times\id_Y)^*\omega) && \text{cf. \eqref{eq:DiffAndInt}}\\
&= K(d_I(g_j\times \id_Y)^*\omega)\\
&= (i_1^* - i_0^*)(g_j\times\id_Y)^*\omega &&\text{by \eqref{eq:DiffAndInt} again}\\
&= (e_j\times \id_Y)^*\omega - \omega.
\end{align*}
Here we used the naturality of the double complex of remark \ref{rem:DiffFormsDoubleComplex} (i) and applied the first formula of \eqref{eq:DiffAndInt} with $X= \Delta^p\times Y$ only for the $\Delta$-component $d_{\Delta}$ of the differential $d_{\Delta\times Y}$.
\end{proof}
We define the integration map $I : D^{k,l}(X_{\dot}) \to \Omega^l(X_k)$ as in the classical case:
\begin{equation}\label{eq:I}
I(\omega) = (-1)^k(e_k\times\id_{X_k})^*(h_{(k-1)}\circ\dots\circ h_{(0)})(\omega_k).
\end{equation}
Using the lemma and the compatibility condition of simplicial differential forms one checks \eqref{eq:11}.

Similarly, $E$ is defined by the same formula as in the classical case: If $\omega\in\Omega^l(X_k)$, the simplicial form $E(\omega) \in D^{k,l}(X_{\dot})$ is given on $\Delta^p \times X_p$ by $0$, if $p < k$, and else by
\begin{multline*}
E(\omega)_p =\\
=k! \sum_{\phi:[k]\hookrightarrow[p]}\left(\sum_{j=0}^k (-1)^j x_{\phi(j)}dx_{\phi(0)} \smash \dots \smash
\widehat{(dx_{\phi(j)})}\smash \dots \smash dx_{\phi(k)} \right) \smash \phi_X^*\omega.
\end{multline*}
Note, that this really defines a $(k+l)$-form on the dagger space $\Delta^p\times X_p$. It is easy to see that $E(\omega)$ defines a simplicial form on $X_{\dot}$ and that $E$ satisfies \eqref{eq:22} and \eqref{eq:33}.

Also the homotopy operator $s : D^{k,l}(X_{\dot}) \to D^{k-1,l}(X_{\dot})$ is defined by the same formula as in the complex situation, which again gives a well-defined differential form also in the dagger context. That $s(\omega)$ really defines a \emph{simplicial} differential form and that $s$ satisfies \eqref{eq:44} follows again from the above lemma.
(Most of the computations are also carried out in \cite[proof of Theorem 2.16]{DupLNM}.)
\end{proof}

\begin{rem}\label{rem:IntegrationStandardSimplex} 
If $Y$ is any dagger space, \eqref{eq:I} defines an operator $I\colon \Omega^n(\Delta^k \times Y) \to \Omega^{n-k}(Y)$, which we denote by $\int_{\Delta^k}$. It may also be described as follows:
Define a morphism $\psi:I^k = \Sp(K\<t_1, \dots, t_k\>^{\dag}) \to \Delta^k = \Sp(K\<x_0, \dots, x_k\>^{\dag}/(\sum_i x_i -1))$ by $x_i \mapsto t_1\cdots t_i(1-t_{i+1})$, $i=0, \dots, k$, where we let $t_{n+1}=0$.\footnote{This is the analogue of the diffeomorphism $[0,1]^k \to \Delta^k_{\R}$ given by $(s_1, \dots, s_k) \mapsto (1-s_1, s_1(1-s_2), s_1s_2(1-s_3), \dots, s_1\cdots s_{k-1}(1-s_k), s_1\cdots s_k)$.}
It follows directly from the definitions, that $\int_{\Delta^k}$ is simply the composition
\[
\Omega^n(\Delta^k\times Y) \xrightarrow{(\psi\times1)^*} \Omega^n(I^k \times Y) \xrightarrow{K^k} \Omega^{n-k}(Y).
\]

In particular we have the integration map $\int_{\Delta^n} : \Omega^n(\Delta^n) \to K$. For later use, we record the following continuity property of $\int_{\Delta^n}$.

Fix $\rho > 1$ and $1< \eta<\rho^{\frac{1}{n}}$. Since $|t_1\cdots t_i(1-t_{i+1})|_{\eta} \leq \eta^n < \rho = |x_i|_{\rho}$ the morphism $\psi$ above restricts to a continuous morphism of Banach algebras $K\<\rho^{-1}x_0, \dots, \rho^{-1}x_n\>/(\sum x_i -1) \to K\<\eta^{-1}t_1, \dots, \eta^{-1}t_n\>$ (see section \ref{sec:PrelimAlgebras} for the notations).

We have a natural map $K\<\rho^{-1}\ul x\>/(\sum_i x_i-1) \otimes_K \bigwedge^n_K\frac{\bigoplus_{i=0}^nKdx_i}{\sum_i dx_i} \to \Omega^n(\Delta^n)$ and by the above, the composition 
\begin{equation}\label{eq:ContIntegration}
K\<\rho^{-1}\ul x\>/(\sum_i x_i-1) \otimes_K \bigwedge^n_K\frac{\bigoplus_{i=0}^nKdx_i}{\sum_i dx_i}  \ni \omega \mapsto \int_{\Delta^n}\omega  \in K
\end{equation}
is equal to the composition
\begin{multline*}
K\<\rho^{-1}\ul x\>/(\sum_i x_i-1) \otimes_K \bigwedge^n_K\frac{\bigoplus_{i=0}^nKdx_i}{\sum_i dx_i} \xrightarrow{\psi^*} \\
K\<\eta^{-1}\ul t\>\otimes_K\bigwedge^n_K\left(\bigoplus_{i=1}^n Kdt_i\right) \to   \Omega^n(I^n) \xrightarrow{K^n} K.
\end{multline*}
Using the continuity of $\psi$  and remark \ref{rem:ContIntegration}, it follows, that \eqref{eq:ContIntegration} is continuous as well.
\end{rem}

\section{Simplicial bundles and connections}

Let $\GL_{r,K}^{\dag}$ be the dagger space associated with the affine $K$-scheme $\GL_{r,K}$. 
The following lemma is certainly well-known, but I could not find a reference.
\begin{lemma}
If $X$ is any dagger space, the morphisms of dagger spaces $X \to \GL_{r,K}^{\dag}$ are in one to one correspondence with the group $\GL_r(\O_X(X))$.
\end{lemma}
\begin{proof}
By the sheaf property of $\GL_r(\O_X)$ and of the morphisms $U \to \GL_{r,K}^{\dag}$, $U \subset X$ admissible open, it suffices to treat the case, where $X = \Sp(A)$ is affinoid. Let $C = K[x_{ij},y]/(\det(x_{ij})\cdot y -1)$ such that $\Spec(C) = \GL_{r,K}$. 
Fix $c \in K$ with $|c|>1$ and write $C_n = K\<c^{-n}x_{ij}, c^{-n}y\>^{\dag}/(\det(x_{ij})\cdot y-1)$. Then $\GL_{r,K}^{\dag} = \bigcup_{n\geq 0} \Sp(C_n)$ (cf. section \ref{sec:PrelimSpaces}).

The set $\GL_r(A)$ corresponds bijectively to the set of $K$-algebra homomorphisms $C \to A$.  Given such a morphism $\sigma: C \to A$, choose $n$ large enough, such that $|\sigma(x_{ij})|, |\sigma(y)| \leq |c|^n$. Then the elements $c^{-n}\sigma(x_{ij}), c^{-n}\sigma(y) \in A$ are power bounded and by the weak completeness of $A$ and the fact, that $C \subset C_n$ is dense, there exists a unique extension of $\sigma$ to a morphism $C_n \to A$ (cf. section \ref{sec:PrelimAlgebras}). This in turn gives a well-defined morphism of dagger spaces $\Sp(A) \to \Sp(C_n) \subset \GL_{r,K}^{\dag}$.

On the other hand, any morphism of dagger spaces $\Sp(A) \to \GL_{r,K}^{\dag}$ gives, composed with the morphism of locally G-ringed spaces $\GL_{r,K}^{\dag} \to \GL_{r,K}$, a morphism of locally G-ringed spaces $\Sp(A) \to \GL_{r,K} = \Spec(C)$ and on global sections a morphism of rings $C \to A$, i.e. an element of $\GL_r(A)$.

Using the uniqueness of the extension of $\sigma$ above it is now easy to see, that both constructions are inverse to each other.
\end{proof}

Now the formalism of sections \ref{sec:Bundles} and \ref{sec:connections} carries over to the setting of simplicial dagger spaces:

Let $X_{\dot}$ and $Y_{\dot}$ be simplicial dagger spaces. A \emph{topological morphism} $f: X_{\dot} \rightsquigarrow Y_{\dot}$ is a family of morphisms of dagger spaces $\Delta^p \times X_p \to Y_p$ satisfying a compatibility condition for every increasing map $\phi: [p] \to [q]$ as in definition \ref{def:TopMorph}. A \emph{topological $\GL_r$-bundle} on $X_{\dot}$ is a topological morphism $g : X_{\dot} \rightsquigarrow B_{\dot}\GL_{r,K}^{\dag}$. A \emph{morphism} $\alpha: g \to h$ of topological bundles on $X_{\dot}$ is a topological morphism $\alpha: X_{\dot} \rightsquigarrow E_{\dot}\GL_{r,K}^{\dag}$ satisfying $\alpha \cdot g = h$.
An \emph{analytic $\GL_r$-bundle} is a morphism of simplicial dagger spaces $X_{\dot} \to B_{\dot}\GL_{r,K}^{\dag}$. 
\smallskip

A \emph{connection} in a topological $\GL_r$-bundle $g: X_{\dot} \rightsquigarrow B_{\dot}\GL_{r,K}^{\dag}$ is given by the following data:
For any $p \geq 0$ and any $i \in [p] = \{0, \dots, p\}$ a matrix valued $1$-form $\Gamma_i = \Gamma_i^{(p)} \in \Mat_r(\Omega^1(\Delta^p \times X_p))$ subject to the  conditions
\begin{enumerate}
\item $(\phi_{\Delta}\times \id)^*\Gamma_{\phi(i)}^{(q)} = (\id \times \phi_X)^*\Gamma_i^{(p)}$ for any increasing map $\phi: [p] \to [q]$ and
\item $\Gamma_i = g_{ji}^{-1} dg_{ji} + g_{ji}^{-1}\Gamma_j g_{ji}$.
\end{enumerate}
The notations are the same as in section \ref{sec:connections}. By the previous lemma we view the morphism $g_{ji} : \Delta^p\times X_p \to \GL_{r,K}^{\dag}$ as an element of $\GL_r(\O_{\Delta^p\times X_p}(\Delta^p \times X_p))$, hence $dg_{ji} \in \Mat_r(\Omega^1(\Delta^p\times X_p))$.

A connection $\Gamma = \{\Gamma_i\}$ on an analytic bundle is called \emph{analytic} if $\Gamma_i \in \Omega^{0,1}(\Delta^p \times X_p)$ for all $p \geq 0, i = 0, \dots, p$.
For example, the standard connection (example \ref{ex:standard_connection}) on any analytic bundle is analytic.

\smallskip

The \emph{curvature} of the connection $\{\Gamma_i\}$ is defined as the family of matrix valued 2-forms
\[
 R_i := R_i^{(p)} := d\Gamma_i^{(p)} + \left(\Gamma_i^{(p)}\right)^2 \in \Mat_r(\Omega^2(\Delta^p \times X_p)),  
\]
$p \geq 0, i = 0, \dots p$.

\smallskip

We define the \emph{$n$-th Chern character form} $\Ch_n(\Gamma)$ of the connection $\Gamma =\{\Gamma_i\}$ to be the family of forms $\frac{1}{n!} \Tr\left(\left(R_i^{(p)}\right)^n\right)$ on $\Delta^p\times X_p$, $p \geq 0$. According to lemma \ref{lem:CurvatureCompatible}, this form does not depend on $i$. We have the analogue of proposition \ref{prop:CharClasses}:
\begin{prop}\label{prop:DaggerCharClasses}
Let $g: X_{\dot} \rightsquigarrow B_{\dot}\GL_{r,K}^{\dag}$ be  a topological bundle and $\Gamma$ a connection on $g$.
\begin{enumerate}
\item $\Ch_n(\Gamma)$ is a closed $2n$-form on $X_{\dot}$, i.e. belongs to $D^{2n}(X_{\dot})$ and $d\Ch_n(\Gamma) = 0$.
\item The cohomology class of $\Ch_n(\Gamma)$ does not depend on the connection chosen.
\item If the bundle $g$ and the connection $\Gamma$ are analytic, $\Ch_n(\Gamma) \in \Fil^nD^{2n}(X_{\dot})$. Moreover, the class of $\Ch_n(\Gamma)$ in $H^{2n}(\Fil^nD^*(X_{\dot}))$ does not depend on the analytic connection chosen.
\item If $h : X_{\dot} \rightsquigarrow B_{\dot}\GL_{r,K}^{\dag}$ is a second bundle, and $\alpha : h \to g$ is a morphism, then $\Ch_n(\alpha^*\Gamma) = \Ch_n(\Gamma)$.
\item If $f: Y_{\dot} \rightsquigarrow X_{\dot}$ is a topological morphism, $\Ch_n(f^*\Gamma) = f^*\Ch_n(\Gamma)$.
\end{enumerate}
\end{prop}
\begin{dfn}
If $E/X_{\dot}$ is a topological bundle and $\Gamma$ is any connection on $E$, we write $\Ch_n(E)$ for the image of the class of $\Ch_n(\Gamma)$ in $H^{2n}(D^*(X_{\dot}))$ under the natural map $H^{2n}(D^*(X_{\dot})) \to H^{2n}_{\dR}(X_{\dot}/K)$.
If $E$ and $\Gamma$ are analytic, we still write $\Ch_n(E)$ for the image of the class of $\Ch_n(\Gamma)$
in $H^{2n}(\Fil^nD^*(X_{\dot}))$ under the natural map $H^{2n}(\Fil^nD^*(X_{\dot})) \to \Hyp^{2n}(X_{\dot}, \Omega^{\geq n}_{X_{\dot}})$.
\end{dfn}

Also, we can construct \emph{Chern character classes of vector bundles} as in the complex case. Here we freely use some results of section \ref{sec:Bundles}, which were stated there for complex manifolds, but which obviously carry over to the setting of dagger spaces with the appropriate modifications (e.g. the coverings considered have all to be admissible):

Let $\E_{\dot}$ be a vector bundle of rank $r$ on the  simplicial dagger space $X_{\dot}$ (definition \ref{def:VectorBundle}). As in lemma \ref{lemma:open_covering} there exists a morphism of simplicial dagger spaces $U_{\dot} \to X_{\dot}$, such that each $U_p$ is a disjoint union $\coprod_{\alpha\in A} U_{p,\alpha}$, where $\{U_{p,\alpha}\}_{\alpha\in A}$ is an \emph{admissible} open covering of $X_p$,
and $\E_{\dot}|_{U_{\dot}}$ is degree-wise trivial. Define $X_{\dot}'$ to be the diagonal of the associated \v{C}ech nerve $N_{X_{\dot}}(U_{\dot})$. Then $\E_{\dot}' := \E_{\dot}|_{X_{\dot}'}$ is degree-wise trivial, too, hence corresponds to an analytic $\GL_r$-bundle $E'/X'_{\dot}$. Moreover $X_{\dot}' \to X_{\dot}$ induces an isomorphism in cohomology.
We define $\Ch_n(\E_{\dot})$ to be the inverse image of $\Ch_n(E')$ under the isomorphism $\Hyp^{2n}(X'_{\dot}, \Omega^{\geq n}_{X'_{\dot}}) \xrightarrow{\cong} \Hyp^{2n}(X_{\dot}, \Omega^{\geq n}_{X_{\dot}})$. 
As in the complex case one shows, that this class is well defined and that the Whitney sum formula holds.

\section{Secondary classes}
\label{sec:DaggerSecondaryClasses}
\label{SEC:DAGGERSECONDARYCLASSES}

Let again $X_{\dot}$ be a simplicial dagger space and  let $E, F$ be two topological bundles on $X_{\dot}$ with connection $\Gamma^E$ and $\Gamma^F$ respectively and $\alpha: E\to F$ a morphism. Recall that $I = \Sp(K\<t\>^{\dag})$ and let $\pi\colon X_{\dot}\times I \to X_{\dot}$ be the projection. We equip the bundle $\pi^*E$ with the connection $\Gamma = t\pi^*\Gamma^E + (1-t)\pi^*\alpha^*\Gamma^F$ as in \eqref{eq:HomotopyConnection}, which is obviously well-defined also in the present context. We define the secondary form
\[
\Ch_n^{\rel}(\Gamma^E, \Gamma^F, \alpha) = K(\Ch_n(\Gamma)) \in D^{2n-1}(X_{\dot})
\]
using the homotopy operator $K$ of lemma \ref{lem:HomotopyOperator} componentwise. It has the same formal properties as its complex counterpart (section \ref{sec:SecondaryClasses}).
\smallskip

Similar arguments as in the complex situation then show, that, if $E$ and $F$ are \emph{analytic} bundles equipped with \emph{analytic} connections, then $\Ch_n^{\rel}(E,F,\alpha)$ gives a well-defined cohomology class in $H^{2n-1}(D^*(X_{\dot})/\Fil^nD^*(X_{\dot})) = H^{2n-1}(\Omega^{<n}(X_{\dot}))$, independent of the chosen analytic connections. Its image in $\Hyp^{2n-1}(X_{\dot}, \Omega^{<n}_{X_{\dot}})$ is denoted by
\[
\Ch_n^{\rel}(E,F,\alpha).
\]

\section{Chern character classes for algebraic bundles}

We recall the construction of Chern character classes in algebraic de Rham cohomology and compare them with the classes constructed via Chern-Weil theory above. The construction is the same as in the complex case, only that holomorphic differential forms are replaced with algebraic differential forms.

Let $X_{\dot}$ be a separated smooth simplicial $K$-scheme of finite type. 
Choose a good compactification $j : X_{\dot} \hookrightarrow \ol X_{\dot}$, i.e. an open immersion of smooth strict simplicial schemes of finite type over $K$, such that $\ol X_{\dot}$ is proper over $K$ and each $D_p = \ol X_p - X_p$ is a divisor with normal crossings. 
We have the logarithmic de Rham complex $\Omega^*_{\ol X_{\dot}}(\log D_{\dot}) \subset j_*\Omega^*_{X_{\dot}}$ and $\Hyp^*(\ol X_{\dot}, \Omega^*_{\ol X_{\dot}}(\log D_{\dot})) \cong \Hyp^*(X_{\dot}, \Omega^*_{X_{\dot}}) = H^*_{\dR}(X_{\dot}/K)$ (cf. \cite[Lemma 3.4]{Jannsen}).
By definition, the Hodge filtration on $H^*_{\dR}(X_{\dot}/K)$ is given by
\[
\Fil^nH^*_{\dR}(X_{\dot}/K) = \Im(\Hyp^*(\ol X_{\dot}, \Omega^{\geq n}_{\ol X_{\dot}}(\log D_{\dot}))\to \Hyp^*(X_{\dot}, \Omega^*_{X_{\dot}})).
\]
This is independent of the chosen good compactification, and it follows by the Lefschetz principle and GAGA from the corresponding fact over $\C$, that the map $\Hyp^*(\ol X_{\dot}, \Omega^{\geq n}_{\ol X_{\dot}}(\log D_{\dot}))\to \Hyp^*(X_{\dot}, \Omega^*_{X_{\dot}})$ is injective (cf. \cite[(8.7.2)]{Katz}).

\medskip

The first Chern class of line bundles $c_1 : H^1(X_{\dot}, \O_{X_{\dot}}^*) \to \Hyp^2(X_{\dot}, \Omega^{\geq 1}_{X_{\dot}})$ is again induced from the morphism of complexes $d\log: \O_{X_{\dot}}^*[-1] \to \Omega^{\geq 1}_{X_{\dot}}$, and one checks as in the complex case, that the image in fact lies in $\Fil^1H^2_{\dR}(X_{\dot}/K) \subset \Hyp^2(X_{\dot}, \Omega^{\geq 1}_{X_{\dot}})$.

Also, higher Chern classes and Chern character classes 
\[
\widetilde\Ch_n(\E_{\dot}) \in \Fil^nH^{2n}_{\dR}(X_{\dot}/K)
\]
for algebraic vector bundles $\E_{\dot}$ on $X_{\dot}$ are constructed as in the complex case using the splitting principle. 
\medskip

Denote by $X^{\dag}_{\dot}$ the  simplicial dagger space associated with $X_{\dot}$.
There is a natural morphism of simplicial locally G-ringed spaces $\iota: X_{\dot}^{\dag} \to X_{\dot}$.
We have the following chain of morphisms in the derived category $D^+(\ol X_{\dot})$ of bounded below complexes of abelian sheaves on $\ol X_{\dot}$:
\[
\Omega^{\geq n}_{\ol X_{\dot}}(\log D_{\dot}) \to \Rhyp j_*\Omega^{\geq n}_{X_{\dot}} \to \Rhyp j_*\Rhyp\iota_* \Omega^{\geq n}_{X_{\dot}^{\dag}}
\]
and hence natural maps $\Fil^nH^*_{\dR}(X_{\dot}/K) \to \Hyp^*(X_{\dot}^{\dag}, \Omega^{\geq n}_{X_{\dot}^{\dag}})$.

Let $\E_{\dot}/X_{\dot}$ be an algebraic vector bundle of rank $r$ and denote the induced vector bundle on $X_{\dot}^{\dag}$ by $\E_{\dot}^{\dag}$.
\begin{prop}\label{prop:AlgDaggerChernClasses}
$\widetilde\Ch_n(\E_{\dot})$ is mapped to $(-1)^n\Ch_n(\E_{\dot}^{\dag})$ under the natural morphism $\Fil^nH^{2n}_{\dR}(X_{\dot}/K) \to \Hyp^{2n}(X_{\dot}^{\dag}, \Omega^{\geq n}_{X_{\dot}^{\dag}})$.
\end{prop}
\begin{proof}[Sketch of proof]
The proof is the same as in the complex case: First one checks the case of the first Chern character class of a line bundle as in lemma \ref{lemma:ChLineBundle}. As in proposition \ref{prop:AlgChernCharClasses} the general case is then reduced to the case already treated using the splitting principle:
\begin{lemma}
Let $X_{\dot}$ be a smooth simplicial $K$-scheme of finite type, $\E_{\dot}$ an algebraic vector bundle of rank $r$ on $X_{\dot}$ and $\P(\E_{\dot}) \xrightarrow{\pi} X_{\dot}$ the associated projective bundle. Write $\xi = c_1(\O(1)^{\dag}) \in \Hyp^2(\P(\E_{\dot})^{\dag}, \Omega^{\geq 1}_{\P(\E_{\dot})^{\dag}})$. Then 
\[
\sum_{i=0}^{r-1} \pi^*(\,.\,)\cup \xi^i:\bigoplus_{i=0}^{r-1} \Hyp^{m-2i}(X_{\dot}^{\dag}, \Omega_{X_{\dot}^{\dag}}^{\geq n-i}) \to \Hyp^m(\P(\E_{\dot})^{\dag}, \Omega_{\P(\E_{\dot})^{\dag}}^{\geq n})
\]
is an isomorphism.
\end{lemma}
\begin{proof}[Sketch of proof]
By the same spectral sequence arguments as in the complex situation (lemmata \ref{lem:CohomProjBundle1} and \ref{lem:CohomProjBundle}) one is reduced to show, that, if $X$ is an ordinary smooth $K$-scheme of finite type and $\E$ an algebraic vector bundle of rank $r$ on $X$, then
\[
\bigoplus_{i=0}^{r-1} \Omega^{p-i}_{X^{\dag}}[-i] \xrightarrow{\bigoplus_{i=0}^{r-1} \pi^*(\underline{\;\;})\cup \xi^i} \Rhyp\pi_*\Omega^p_{\P(\E)^{\dag}}
\]
is an isomorphism in $D^+(X^{\dag})$. This in turn can be shown exactly as in the complex analytic setting \cite[Th\'eor\`eme 2]{Verdier}. One only has to use the fact, that the GAGA-principle holds for the dagger analytification of proper $K$-schemes \cite[Korollar 4.5]{GK}. 
\end{proof}
\end{proof}

\begin{rem}\label{rem:KiehlThm}
According to \cite[Theorem 2.4]{Kiehl} and \cite[Korollar 4.6 together with Beispiel (iv) on p. 25]{GK}, the natural map $H^*_{\dR}(X_{\dot}/K) \to H^*_{\dR}(X_{\dot}^{\dag}/K)$ is an isomorphism. Hence it follows, that the morphism $\Fil^nH^*_{\dR}(X_{\dot}/K) \to \Hyp^*(X_{\dot}^{\dag}, \Omega^{\geq n}_{X_{\dot}^{\dag}})$ is injective.
In particular the Chern character classes $\Ch_n(E)$ of \emph{algebraic} $\GL_r$-bundles indeed lie in $\Fil^nH^{2n}_{\dR}(X_{\dot}/K) \subset \Hyp^{2n}(X_{\dot}^{\dag}, \Omega^{\geq n}_{X_{\dot}^{\dag}})$.
\end{rem}

\chapter{Refined and secondary classes for algebraic bundles}
\label{ch:RefinedSecClassesAlgBundles}

This chapter constructs refined and secondary classes for algebraic bundles in analogy to the constructions in section \ref{sec:RelativeChernCharClasses}. Due to the problems mentioned in the introduction to part II, this is more complicated than in the complex case and we restrict the construction of secondary classes to affine simplicial schemes. This will be enough for the construction of Chern character maps on $K$-theory.

There are several possible variants. The direct analogue of the secondary classes of section \ref{sec:RelativeChernCharClasses} are secondary classes for algebraic bundles on a simplicial $K$-scheme $X_{\dot}$ together with a topological trivialization of the induced bundle on the simplicial dagger space $X_{\dot}^{\dag}$. These classes are constructed in the first section and then compared with the Chern-Weil theoretic secondary classes in the second section. Since we will define topological and relative $K$-theory in chapter \ref{ch:UltrametricRelKAndRegulators} for $R$-schemes, the secondary classes needed for the construction of the relative Chern character on $K$-theory are classes for algebraic bundles on a simplicial $R$-scheme together with a topological trivialization of the induced bundle on the generic fibre $(\hat X_{\dot})_K$ of the weak completion of $X_{\dot}$. These are constructed in section \ref{sec:RelClassesRSchemes}.

\section{Construction}

Let $X_{\dot}$ be a smooth separated simplicial $K$-scheme of finite type. We denote the associated simplicial dagger space by $X_{\dot}^{\dag}$ and by $\iota$ the canonical morphism $\iota: X_{\dot}^{\dag} \to X_{\dot}$.

Let $E/X_{\dot}$ be an algebraic $\GL_r$-bundle and $E^{\dag}/X_{\dot}^{\dag}$ the associated analytic bundle, classified by $g^{\dag} : X_{\dot}^{\dag} \to B_{\dot}\GL_{r,K}^{\dag}$. Define the associated principal bundle $E_{\dot}^{\dag} \xrightarrow{p} X_{\dot}^{\dag}$ to be the pullback of the universal bundle $E_{\dot}\GL_{r,K}^{\dag} \xrightarrow{p} B_{\dot}\GL_{r,K}^{\dag}$ along $g^{\dag}$:
\[
\xymatrix@C+0.3cm{
E_{\dot}^{\dag} \ar[r]\ar[d]_p\ar@{}[dr]|{\lrcorner} & E_{\dot}\GL_{r,K}^{\dag} \ar[d]^p\\
X_{\dot}^{\dag} \ar[r]^-{g^{\dag}} & B_{\dot}\GL_{r,K}^{\dag}.
}
\]
\begin{rem}\label{rem:ExplicitStructurePrincipleBundles}
Since $E_p\GL_{r,K}^{\dag} \to B_p\GL_{r,K}^{\dag} \times \GL_{r,K}^{\dag}$, $(g_0, \dots, g_p) \mapsto (g_0g_1^{-1}, \dots, g_{p-1}g_p^{-1}, g_p)$ is an isomorphism over $B_p\GL_{r,K}^{\dag}$, we have an isomorphism 
\[
E_p^{\dag} \cong X_p^{\dag} \times \GL_{r,K}^{\dag}.
\]
\end{rem}

Choose a good compactification $j: X_{\dot} \hookrightarrow \ol X_{\dot}$ and write $D_p = \ol X_p - X_p$. We have natural morphisms
\[
\xymatrix{
\Omega^{\geq n}_{\ol X_{\dot}}(\log D_{\dot}) \ar[r] & j_*\Omega^*_{X_{\dot}} \ar[r] & j_*\iota_*\Omega^*_{X_{\dot}^{\dag}}\ar[r]^{p^*}\ar[d] & j_*\iota_*p_*\Omega^*_{E_{\dot}^{\dag}}\ar[d]\\
&& \Rhyp(j_*\iota_*)\Omega^*_{X_{\dot}^{\dag}} \ar[r] & \Rhyp(j_*\iota_*p_*)\Omega^*_{E_{\dot}^{\dag}}.
}
\]
\begin{dfn}
Define the \emph{relative cohomology groups}
\begin{align*}
&H^{E,*}_{\rel}(X_{\dot}, n) := \Hyp^*(\ol X_{\dot}, \Cone(\Omega^{\geq n}_{\ol X_{\dot}}(\log D_{\dot}) \to \Rhyp(j_*\iota_*p_*)\Omega^*_{E_{\dot}^{\dag}})) \text{ and}\\
&H^{*}_{\rel}(X_{\dot}, n) := \Hyp^*(\ol X_{\dot}, \Cone(\Omega^{\geq n}_{\ol X_{\dot}}(\log D_{\dot}) \to \Rhyp(j_*\iota_*)\Omega^*_{X_{\dot}^{\dag}})).
\end{align*}
\end{dfn}
\begin{rems}\label{rem:DefOfRelCohom}
(i) Here we represent $\Rhyp(j_*\iota_*)\Omega^*_{X_{\dot}^{\dag}}$ by $j_*\iota_*I^*_{X_{\dot}^{\dag}}$, where 
$\Omega^*_{X_{\dot}^{\dag}} \trivcof I^*_{X_{\dot}^{\dag}}$ is an injective quasiisomorphism of complexes of abelian sheaves on $X_{\dot}^{\dag}$ and each $I^k_{X_{\dot}^{\dag}}$ is injective (cf. Appendix \ref{app:StrictSimplicial}), and similarly for $\Rhyp(j_*\iota_*p_*)\Omega^*_{E_{\dot}^{\dag}}$.
Since $p_*$, the functor $p^{-1}$ being exact, maps injective sheaves to injectives, there exists a morphism $I^*_{X_{\dot}^{\dag}} \to p_*I^*_{E_{\dot}^{\dag}}$ making the diagram 
\[
\xymatrix{
\Omega^*_{X_{\dot}^{\dag}} \ar[r]^{p^*}\ar@{^{(}->}[d]^{\sim} & p_*\Omega^*_{E_{\dot}^{\dag}} \ar[d]\\
I^*_{X_{\dot}^{\dag}} \ar@{.>}[r] & p_*I^*_{E_{\dot}^{\dag}}
}
\]
commute, and this morphism is unique up to homotopy \emph{under} $\Omega^*_{X_{\dot}^{\dag}}$ (cf. lemma \ref{lem:HomotopyEquivCones}). Hence we get a map $\Rhyp(j_*\iota_*)\Omega^*_{X_{\dot}^{\dag}} = j_*\iota_*I^*_{X_{\dot}^{\dag}} \to j_*\iota_*p_*I^*_{E_{\dot}^{\dag}} = \Rhyp(j_*\iota_*p_*)\Omega^*_{E_{\dot}^{\dag}}$ and hence a map of the cones
$\Cone(\Omega^{\geq n}_{\ol X_{\dot}}(\log D_{\dot}) \to \Rhyp(j_*\iota_*)\Omega^*_{X_{\dot}^{\dag}}) \to  \Cone(\Omega^{\geq n}_{\ol X_{\dot}}(\log D_{\dot}) \to \Rhyp(j_*\iota_*p_*)\Omega^*_{E_{\dot}^{\dag}})$, which is well defined up to homotopy (cf. lemma \ref{lem:HomotopyOnCones}). We thus have a \emph{canonical} morphism 
\[
p^*: H^*_{\rel}(X_{\dot}, n) \to H^{E,*}_{\rel}(X_{\dot}, n).
\]
This morphism fits in a long exact sequence
\begin{equation}\label{seq:RelCohomKSchemes}
\begin{small}
\xymatrix@C-0.45cm{
\dots \ar[r] & H^{E,i-1}_{\rel}(X_{\dot}, n) \ar[r] & \Fil^nH^i_{\dR}(X_{\dot}/K) \ar[r] & H^i_{\dR}(E_{\dot}^{\dag}/K)  \ar[r] & H^{E,i}_{\rel}(X_{\dot}, n)\ar[r] & \dots\\
\dots \ar[r] & H^{i-1}_{\rel}(X_{\dot}, n) \ar[r]\ar[u]^{p^*} & \Fil^nH^i_{\dR}(X_{\dot}/K) \ar[r]\ar@{=}[u] & H^i_{\dR}(X_{\dot}^{\dag}/K) \ar[u]^{p^*} \ar[r] & H^i_{\rel}(X_{\dot}, n) \ar[u]^{p^*} \ar[r] & \dots.
}
\end{small}
\end{equation}
\smallskip

(ii) As for the Hodge filtration of the de Rham cohomology one shows, that the definition of the relative cohomology groups does up to isomorphism not depend on the particular choice of the compactification $\ol X_{\dot}$. Since the family of all good compactifications is directed, one could take a colimit over all good compactifications to get a definition independent of choices.
\smallskip

(iii) By remark \ref{rem:KiehlThm} $H^i_{\rel}(X_{\dot},n) \cong H^i_{\dR}(X_{\dot}/K)/\Fil^nH^i_{\dR}(X_{\dot}/K)$ similar to the complex case.
\end{rems}

If $f : Y_{\dot} \to X_{\dot}$ is a morphism of smooth simplicial $K$-schemes of finite type and $E/X_{\dot}$ as before, we can consider the pullback $f^*E$ and the associated principal bundle $f^*E_{\dot}^{\dag}$. Whereas in the complex case we had functorial complexes defining the cone for the relative cohomologies, we have to be a little bit careful with functoriality here.
\begin{lemma}
There are well defined pullback maps $f^*: H^{E,*}_{\rel}(X_{\dot}, n) \to H^{f^*E, *}_{\rel}(Y_{\dot}, n)$ and $f^*: H^*_{\rel}(X_{\dot}, n) \to H^*_{\rel}(Y_{\dot}, n)$
\end{lemma}
\begin{proof}
The proof is the same for both maps, and we restrict to the second one.
Given good compactifications $Y_{\dot} \hookrightarrow \ol Y_{\dot}$ and $X_{\dot} \hookrightarrow \ol X_{\dot}$, whose complement will as usual be denoted simply by $D_{\dot}$, one can construct a good compactification $Y_{\dot}\hookrightarrow \ol Y_{\dot}'$ together with maps $\ol f: \ol Y_{\dot}' \to \ol X_{\dot}$ and $\ol Y_{\dot}' \to \ol Y_{\dot}$ fitting in a commutative diagram
\[
\xymatrix{
Y_{\dot} \ar[rr]^f \ar@{^{(}->}[d] \ar@{^{(}->}[dr] && X_{\dot} \ar@{^{(}->}[d]\\
\ol Y_{\dot} & \ol Y_{\dot}' \ar[l] \ar[r]^{\ol f} & \ol X_{\dot}.
}
\]
Hence we may assume without loss of generality, that the given morphism $f$ extends to a morphism $\ol f: \ol Y_{\dot} \to \ol X_{\dot}$ of the compactifications. Thus we have a commutative diagram
\[
\xymatrix{
Y_{\dot}^{\dag} \ar[r]^{\iota_Y} \ar[d]^{f^{\dag}} & Y_{\dot} \ar[d]^f \ar@{^{(}->}[r]^{j_Y} & \ol Y_{\dot}\ar[d]^{\ol f}\\
X_{\dot}^{\dag} \ar[r]^{\iota_X} & X_{\dot}  \ar@{^{(}->}[r]^{j_X} & \ol X_{\dot}.
}
\]
Choose injective resolutions $\Omega^*_{X_{\dot}^{\dag}} \trivcof I^*_{X_{\dot}^{\dag}}$ and similarly for $Y_{\dot}^{\dag}$. 
Since $f^{\dag}_*$ maps injective sheaves to injectives, the dotted arrow in the diagram
\[
\xymatrix{
\Omega^*_{X_{\dot}^{\dag}} \ar[d]\ar@{^{(}->}[r]^{\sim} & I^*_{X_{\dot}^{\dag}} \ar@{.>}[d]\\
f^{\dag}_*\Omega^*_{Y_{\dot}^{\dag}} \ar[r] & f^{\dag}_*I^*_{Y_{\dot}^{\dag}}
}
\]
exists and is unique up to homotopy under $\Omega^*_{X_{\dot}^{\dag}}$.
Applying $(j_X)_*(\iota_X)_*$ and composing with the natural maps $\Omega^{\geq n}_{\ol X_{\dot}}(\log D_{\dot}) \to (j_X)_*(\iota_X)_*\Omega^*_{X_{\dot}^{\dag}}$ resp. $\ol f_*\Omega^{\geq n}_{\ol Y_{\dot}}(\log D_{\dot}) \to (j_X)_*(\iota_X)_*f^{\dag}_*\Omega^*_{Y_{\dot}^{\dag}}$ we get a morphism
\begin{multline}\label{eq:Morph1}
\Cone\left(\Omega^{\geq n}_{\ol X_{\dot}}(\log D_{\dot}) \to (j_X)_*(\iota_X)_*I^*_{X_{\dot}^{\dag}}\right) \to\\
\Cone\left(\ol f_*\Omega^{\geq n}_{\ol Y_{\dot}}(\log D_{\dot}) \to (j_X)_*(\iota_X)_*f^{\dag}_*I^*_{Y_{\dot}^{\dag}}\right),
\end{multline}
which is well defined up to homotopy (lemma \ref{lem:HomotopyOnCones} again).

Choose an injective resolution $\Omega^{\geq n}_{\ol Y_{\dot}}(\log D_{\dot}) \trivcof J^*$. As before, the dotted arrow in the diagram
\[
\xymatrix{
\Omega^{\geq n}_{\ol Y_{\dot}}(\log D_{\dot}) \ar[d] \ar@{^{(}->}[r]^-{\sim} & J^* \ar@{.>}[d] \\
(j_Y)_*(\iota_Y)_*\Omega^*_{Y_{\dot}^{\dag}} \ar@{^{(}->}[r] & (j_Y)_*(\iota_Y)_*I^*_{Y_{\dot}^{\dag}}
}
\]
exists and is unique up to homotopy under $\Omega^{\geq n}_{\ol Y_{\dot}}(\log D_{\dot})$.
This induces a quasiisomorphism
\begin{multline*}
\Cone\left(\Omega^{\geq n}_{\ol Y_{\dot}}(\log D_{\dot}) \to (j_Y)_*(\iota_Y)_*I^*_{Y_{\dot}^{\dag}}\right) \trivcof
\Cone\left( J^* \to (j_Y)_*(\iota_Y)_*I^*_{Y_{\dot}^{\dag}}\right),
\end{multline*}
where the complex on the right hand side is well defined up to homotopy equivalence (cf. lemma \ref{lem:HomotopyEquivCones}).
Applying $\ol f_*$, we get the natural map
\begin{multline*}
\Cone\left(\ol f_*\Omega^{\geq n}_{\ol Y_{\dot}}(\log D_{\dot}) \to \ol f_*(j_Y)_*(\iota_Y)_*I^*_{Y_{\dot}^{\dag}}\right) \to\\
\ol f_*\Cone\left( J^* \to (j_Y)_*(\iota_Y)_*I^*_{Y_{\dot}^{\dag}}\right).
\end{multline*}
Composing this with \eqref{eq:Morph1} and noting that the last complex represents $\Rhyp\ol f_*\Cone\left(\Omega^{\geq n}_{\ol Y_{\dot}}(\log D_{\dot}) \to \Rhyp((j_Y)_*(\iota_Y)_*)\Omega^*_{Y_{\dot}^{\dag}}\right)$, we get the desired map $f^*$ on relative cohomology groups. Similar one shows, that the map on cohomology does not depend on the choices of $I^*_{X_{\dot}^{\dag}}, I^*_{Y_{\dot}^{\dag}}$ or $J$. 
\end{proof}
\begin{rem}
That there is a unique way of defining $f^*$ on $H^*_{\rel}(X_{\dot}, n)$, is clear from remark \ref{rem:DefOfRelCohom}(iii). Later on we have to use this lemma also in a slightly modified situation, where the conclusion of remark \ref{rem:DefOfRelCohom}(iii) no longer holds.
\end{rem}

Since lemma \ref{lem:CohomEG} applies equally in the dagger context, we have the analogue of proposition \ref{prop:CRefinedClasses}:
\begin{prop}\label{prop:RefinedClasses}
There exists a class $\widetilde\Ch_n^{\rel}(E) \in H^{2n-1,E}_{\rel}(X_{\dot},n)$, which is mapped to the $n$-th Chern character class $\Ch_n(E)$ in $\Fil^nH^{2n}_{\dR}(X_{\dot}/K)$, and which is functorial in $X$. Moreover, the assignment $E \mapsto \widetilde\Ch_n^{\rel}(E)$ is uniquely determined by these two properties.
\end{prop}

\begin{dfn}
If $X_{\dot}$ is a smooth separated simplicial $K$-scheme and $E/X_{\dot}$ an algebraic $\GL_r$-bundle, the class $\widetilde\Ch_n^{\rel}(E) \in H^{2n-1,E}_{\rel}(X_{\dot},n)$ is called the \emph{$n$-th refined Chern character class of $E$}.
\end{dfn}

Assume, that the bundle $E^{\dag}$ induced by $E$ on $X_{\dot}^{\dag}$ admits a topological trivialization $\alpha$, i.e. there exists a topological morphism $\alpha\colon X_{\dot}^{\dag} \rightsquigarrow E_{\dot}\GL_{r,K}^{\dag}$, such that $p\circ \alpha = g^{\dag}$, the classifying map of $E^{\dag}$. Then $\alpha$ induces a topological morphism $\alpha\colon X_{\dot}^{\dag} \rightsquigarrow E_{\dot}^{\dag}$ right inverse to $p: E_{\dot}^{\dag} \to X_{\dot}^{\dag}$. To define secondary classes as in section \ref{sec:RelativeChernCharClasses}, we have to define a pullback $\alpha^*\colon H^{E,*}_{\rel}(X_{\dot}, n) \to H^*_{\rel}(X_{\dot}, n)$. Again, this takes a little bit more work than in the complex analogue. 
For simplicity we restrict to the affine case (but see the remark below). This is enough for the construction of regulators.
\begin{lemma}\label{lemma:TopPullback}
In the above situation assume in addition that $X_{\dot}$ is affine. Then $\alpha$ induces compatible left inverses $\alpha^*$ of $p^*: H^*_{\dR}(X_{\dot}^{\dag}/K) \to H^*_{\dR}(E_{\dot}^{\dag}/K)$ and of $p^*: H^*_{\rel}(X_{\dot}, n) \to H^{E,*}_{\rel}(X_{\dot}, n)$.
\end{lemma}
\begin{proof}
Using Theorem \ref{thm:p-adicDupont} $\alpha$ induces a section of $p^*:\Omega^*(X_{\dot}^{\dag}) \to \Omega^*(E_{\dot}^{\dag})$, defined in the notation of the Theorem as the composition $I \circ\alpha^*\circ E$, which we also denote by $\alpha^*$.
Since $E^{\dag}_p = (X_p \times \GL_{r,K})^{\dag}$ (cf. remark \ref{rem:ExplicitStructurePrincipleBundles}) is the dagger space associated with an affine $K$-scheme, it is a Stein space and hence acyclic for the cohomology of coherent sheaves. Hence $\Omega^*(E_{\dot}^{\dag}) \to \Rhyp\Gamma(E_{\dot}^{\dag}, \Omega^*_{E_{\dot}^{\dag}})$ is a quasiisomorphism and on de Rham cohomology $\alpha^*$ is induced by the maps $\Rhyp\Gamma(E_{\dot}^{\dag}, \Omega^*_{E_{\dot}^{\dag}}) \xleftarrow{\sim} \Omega^*(E_{\dot}^{\dag}) \xrightarrow{\alpha^*} \Omega^*(X_{\dot}^{\dag}) \xrightarrow{\sim} \Rhyp\Gamma(X_{\dot}^{\dag}, \Omega^*_{X_{\dot}^{\dag}})$.

On  relative cohomology groups $\alpha^*$ is constructed as follows: 
Choose injective resolutions $\Omega^*_{E_{\dot}^{\dag}} \trivcof I^*_{E_{\dot}^{\dag}}$, $\Omega^*_{X_{\dot}^{\dag}} \trivcof I^*_{X_{\dot}^{\dag}}$ and $\Omega^{\geq n}_{\ol X_{\dot}}(\log D_{\dot}) \trivcof J^*$.
Then we get a commutative diagram
\[
\xymatrix{
\Omega^{\geq n}_{\ol X_{\dot}}(\log D_{\dot}) \ar@{^{(}->}[d]^{\sim} \ar[r] & j_*\iota_*\Omega^*_{X_{\dot}^{\dag}} \ar[d]\ar[r] & j_*\iota_* p_*\Omega^*_{E_{\dot}^{\dag}} \ar[d]\\
J^* \ar@{.>}[r] & j_*\iota_*I^*_{X_{\dot}^{\dag}} \ar@{.>}[r] & j_*\iota_*p_*I^*_{E_{\dot}^{\dag}}.
}
\]
Taking global sections we get the diagram
\[
\xymatrix{
A^* \ar[d]^{\sim} \ar[r] & \Gamma^*(X_{\dot}^{\dag}, \Omega^*_{X_{\dot}^{\dag}}) \ar[d]^{\sim} \ar[r]^{p^*} & \Gamma^*(E_{\dot}^{\dag}, \Omega^*_{E_{\dot}^{\dag}}) \ar[d]^{\sim}\\
\Gamma^*(\ol X_{\dot}, J^*) \ar[r] & \Gamma^*(X_{\dot}^{\dag}, I^*_{X_{\dot}^{\dag}}) \ar[r]^{p^*} & \Gamma^*(E_{\dot}^{\dag}, I^*_{E_{\dot}^{\dag}}),
}
\]
where $\Gamma^*$ denotes the total complex associated with the obvious (strict) cosimplicial complex, and $A^*$ is defined by requiring that the left hand square is a quasi-pullback. In particular the left hand square commutes up to \emph{canonical} homotopy, whereas the right hand square strictly commutes.
Hence we get quasiisomorphisms
\begin{align*}
&\Cone\left(A^* \to \Gamma^*(E_{\dot}^{\dag}, \Omega^*_{E_{\dot}^{\dag}})\right) \xrightarrow{\sim}  \Cone\left(\Gamma^*(\ol X_{\dot}, J^*) \to \Gamma^*(E_{\dot}^{\dag}, I^*_{E_{\dot}^{\dag}})\right),\\
&\Cone\left(A^* \to \Gamma^*(X_{\dot}^{\dag}, \Omega^*_{X_{\dot}^{\dag}})\right) \xrightarrow{\sim}  \Cone\left(\Gamma^*(\ol X_{\dot}, J^*) \to \Gamma^*(X_{\dot}^{\dag}, I^*_{X_{\dot}^{\dag}})\right).
\end{align*}
Note, that the upper complex on the right hand side, hence also the complex on the left hand side, represents $\Rhyp\Gamma(\ol X_{\dot}, \Cone(\Omega^{\geq n}_{\ol X_{\dot}}(\log D_{\dot}) \to \Rhyp(j_*\iota_*p_*)\Omega^*_{E_{\dot}^{\dag}}))$ and similar for $X_{\dot}^{\dag}$.
Clearly $\alpha^*$ induces a section of $p^*: \Cone(A^*\to \Gamma^*(X_{\dot}^{\dag}, \Omega^*_{X_{\dot}^{\dag}})) \to \Cone(A^* \to \Gamma^*(E_{\dot}^{\dag}, \Omega^*_{E_{\dot}^{\dag}})$, which gives the desired pullback on relative cohomology groups. This map is obviously compatible with the morphism $\alpha^*$ on de Rham cohomology constructed above.
\end{proof}
\begin{rem}
One can extend this to the case of separated smooth simplicial $K$-schemes of finite type as follows: Given such $X_{\dot}$ and an algebraic $\GL_r$-bundle $E/X_{\dot}$, topologically trivialized by $\alpha: X_{\dot}^{\dag} \rightsquigarrow E_{\dot}\GL_{r,K}^{\dag}$ as before, there exists a \emph{strict} simplicial scheme $U_{\dot} \to X_{\dot}$ such that each $U_p$ is a disjoint union of open affine subschemes of $X_p$, which cover $X_p$. Define $X'_{\dot,\dot}$ to be the \v{C}ech nerve of $U_{\dot}\to X_{\dot}$. Since $X_{\dot}$ is separated, $X'_{\dot,\dot}$ is affine, too. Moreover, the natural augmentation $X'_{\dot,\dot} \to X_{\dot}$ induces an isomorphism in cohomology.\footnote{We can not replace $X'_{\dot,\dot}$ by the diagonal simplicial scheme, since the Theorem of Eilenberg and Zilber fails for strict bisimplicial objects.} Now the pullback ${E'}^{\dag}_{\dot,\dot}$ of $E_{\dot}^{\dag}$ to $X'_{\dot,\dot}$ is a bisimplicial Stein space and ${E'}^{\dag}_{\dot,\dot}\to E_{\dot}^{\dag}$ induces an isomorphism in cohomology, too. By base change, $\alpha$ induces a topological morphism\footnote{defined similarly as in the simplicial case} $X'_{\dot,\dot} \to {E'}^{\dag}_{\dot,\dot}$, which, using the extension of Dupont's theorem \ref{thm:p-adicDupont} to the strict bisimplicial case, allows one to define the desired map $\alpha^*$ on the de Rham and relative cohomology groups as in the lemma.
\end{rem}

\begin{dfn}
Let $X_{\dot}$ be a smooth affine simplicial $K$-scheme of finite type, $E/X_{\dot}$ an algebraic $\GL_r$-bundle and $\alpha: X_{\dot}^{\dag} \rightsquigarrow E_{\dot}\GL_{r,K}^{\dag}$ a topological trivialization of the induced bundle $E^{\dag}/X_{\dot}^{\dag}$. Then we define
\[
\widetilde\Ch_n^{\rel}(T,E,\alpha) := -\alpha^*\widetilde\Ch_n^{\rel}(E) \in H^{2n-1}_{\rel}(X_{\dot}, n).
\]
\end{dfn}

\section{Comparison with the secondary classes of section \ref{sec:DaggerSecondaryClasses}}

We have to compare these classes with those already constructed in section \ref{sec:DaggerSecondaryClasses}.
To do this, we give an alternative construction of the latter along the above lines: If $E^{\dag}$ is an analytic $\GL_r$-bundle on the simplicial dagger space $X_{\dot}^{\dag}$ with associated principal bundle $E_{\dot}^{\dag} \xrightarrow{p} X_{\dot}^{\dag}$ as above, we define  
\begin{align*}
& H^{E^{\dag},*}_{\rel}(X_{\dot}^{\dag}, n) := \Hyp^*(X_{\dot}^{\dag}, \Cone(\Omega^{\geq n}_{X_{\dot}^{\dag}} \to \Rhyp p_*\Omega^*_{E_{\dot}^{\dag}})) \text{ and}\\
& H^*_{\rel}(X_{\dot}^{\dag}, n) := \Hyp^*(X_{\dot}^{\dag}, \Cone(\Omega^{\geq n}_{X_{\dot}^{\dag}} \to \Omega^*_{X_{\dot}^{\dag}})) \cong \Hyp^*(X_{\dot}^{\dag}, \Omega^{<n}_{X_{\dot}^{\dag}}).
\end{align*}

As in proposition \ref{prop:RefinedClasses} there exists a unique way to assign to every analytic $\GL_r$-bundle $E^{\dag}/X_{\dot}^{\dag}$ a \emph{refined} class 
\[
\Ch_n^{\rel}(E^{\dag}) \in H^{E^{\dag},2n-1}_{\rel}(X_{\dot}^{\dag},n),
\]
which is functorial with respect to $X_{\dot}^{\dag}$ and which by the natural morphism $H^{E^{\dag},2n-1}_{\rel}(X_{\dot}^{\dag},n) \to \Hyp^{2n}(X_{\dot}^{\dag}, \Omega^{\geq n}_{X_{\dot}^{\dag}})$ is mapped to the class $\Ch_n(E^{\dag})$ constructed in proposition \ref{prop:DaggerCharClasses}.

\medskip

Now assume in addition, that $X_{\dot}^{\dag}$ is Stein and that the bundle $E^{\dag}$ has a topological trivialization $\alpha$. As in lemma \ref{lemma:TopPullback} we have a map $\alpha^*: H^{E^{\dag},*}_{\rel}(X_{\dot}^{\dag}, n) \to H^*_{\rel}(X_{\dot}^{\dag},n)$ and we claim:
\begin{lemma}
$\Ch_n^{\rel}(T,E^{\dag},\alpha) = -\alpha^*\Ch_n^{\rel}(E^{\dag})$ in $\Hyp^{2n-1}(X_{\dot}^{\dag}, \Omega^{<n}_{X_{\dot}^{\dag}}) \cong H^{2n-1}_{\rel}(X_{\dot}^{\dag},n)$
\end{lemma}
\begin{proof}
Since all the simplicial dagger spaces $X_{\dot}^{\dag}, E_{\dot}^{\dag}, B_{\dot}\GL_{r,K}^{\dag}, E_{\dot}\GL_{r,K}^{\dag}$ are Stein, we can work with the functorial complexes $D^*(\,.\,)$ resp. $\Fil^nD^*(\,.\,)$. Then the claim follows by an analogue (even easier) computation as in the proof of proposition \ref{prop:ComparisonOfRelClasses}.
\end{proof}

\medskip

Now let $X_{\dot}$ be a smooth affine simplicial $K$-scheme with good compactification $j: X_{\dot} \hookrightarrow \ol X_{\dot}$, whose complement we denote as usual by $D_{\dot}$, and let $X_{\dot}^{\dag} \xrightarrow{\iota} X_{\dot}$ be the dagger analytification morphism. 
Since $\Omega^{\geq n}_{\ol X_{\dot}}(\log D_{\dot}) \to \Rhyp j_*\Rhyp\iota_*\Omega^*_{X_{\dot}^{\dag}}$ factors through $\Rhyp j_*\Rhyp\iota_*\Omega^{\geq n}_{X_{\dot}^{\dag}}$, we get a natural map
\[
H^*_{\rel}(X_{\dot}, n) \to H^*_{\rel}(X_{\dot}^{\dag}, n).
\]

If $E$ is an algebraic $\GL_r$-bundle on $X_{\dot}$ and $E^{\dag}$ the associated bundle on $X_{\dot}^{\dag}$, there is also a natural map
\[
H^{E,*}_{\rel}(X_{\dot}, n) \to H^{E^{\dag},*}_{\rel}(X_{\dot}^{\dag}, n).
\]
\begin{lemma}
The refined class $\widetilde\Ch_n^{\rel}(E) \in H^{E,2n-1}_{\rel}(X_{\dot},n)$ is mapped to $\Ch_n^{\rel}(E^{\dag}) \in H^{E^{\dag},2n-1}_{\rel}(X_{\dot}^{\dag}, n)$ by the above morphism. 
\end{lemma}
\begin{proof}
This follows from the unicity and the defining property of the refined classes (cf. proposition \ref{prop:RefinedClasses}).
\end{proof}

Putting everything together we have now achieved the proof of
\begin{prop}\label{prop:ComparisonDaggerRelClasses}
Let $E$ be an algebraic $\GL_r$-bundle on the smooth affine simplicial $K$-scheme $X_{\dot}$ and $\alpha$ a topological trivialization of the associated analytic bundle $E^{\dag}$ on the simplicial dagger space $X_{\dot}^{\dag}$. Then
$\widetilde\Ch_n^{\rel}(T,E,\alpha)$ is mapped to $\Ch_n^{\rel}(T,E^{\dag},\alpha)$ by the natural map $H^{2n-1}_{\rel}(X_{\dot}, n) \to H^{2n-1}_{\rel}(X_{\dot}^{\dag}, n) = \Hyp^{2n-1}(X_{\dot}^{\dag}, \Omega^{<n}_{X_{\dot}^{\dag}})$.
\end{prop}

\section{Variant for $R$-schemes}
\label{sec:RelClassesRSchemes}

For the construction of the relative Chern  character on the relative $K$-theory of a smooth affine $R$-scheme, we need the following variant of the classes constructed above.

Let $X_{\dot}$ be a smooth simplicial $R$-scheme of finite type, $X_{K,\dot}$ its generic fibre, $X_{K,\dot}^{\dag}$ its associated dagger space, $\widehat X_{\dot}$ the weak completion of $X_{\dot}$ and $(\widehat X_{\dot})_K$ its generic fibre. Choose a good compactification $j : X_{K,\dot} \hookrightarrow \ol X_{K,\dot}$. Then we have the following picture:
\[
(\widehat X_{\dot})_K \underset{\substack{\text{admissible}\\\text{open}}}{\subset} X_{K,\dot}^{\dag} \xrightarrow{\iota} X_{K,\dot} \overset{j}{\hookrightarrow} \ol X_{K,\dot}.
\]

Let $E/X_{\dot}$ be an algebraic $\GL_r$-bundle classified by a morphism $g : X_{\dot} \to B_{\dot}\GL_{r,R}$. Let $E_{\dot} \xrightarrow{p} X_{\dot}$ be the associated principal bundle, i.e. the pullback of $E_{\dot}\GL_{r,R} \to B_{\dot}\GL_{r,R}$ along $g$, and denote by $(\widehat E_{\dot})_K$ the generic fibre of the weak completion of $E_{\dot}$. 
Since weak completion and generic fibre commute with base change, this is also the pullback of $E_{\dot}(\widehat \GL_{r,R})_K \to B_{\dot}(\widehat\GL_{r,R})_K$ along $\widehat g_K$.

On the other hand, $E$ induces analytic bundles $\hat E_K$ on $(\widehat X_{\dot})_K$ and $E_K^{\dag}$ on $X_{K,\dot}^{\dag}$, which are classified by $\widehat g_K : (\widehat X_{\dot})_K \to B_{\dot}(\widehat\GL_{r,R})_K \subset B_{\dot}\GL_{r,K}^{\dag}$ and $g_K^{\dag} : X_{K,\dot}^{\dag} \to B_{\dot}\GL_{r,K}^{\dag}$ respectively. If $(\hat E_K)_{\dot}$ (resp. $(E_K^{\dag})_{\dot}$) denotes the principal bundle associated with $\hat E_K$ (resp. $E_K^{\dag}$), then $(\widehat E_{\dot})_K \subsetneqq (\hat E_K)_{\dot} \subset (E_K^{\dag})_{\dot}$ are admissible open. The following picture, where we indicated the fibres of the bundles, might help to clarify the situation:
\[
\xymatrix@C+1cm{
(\widehat E_{\dot})_K \ar[r]^{\subset} \ar[dr]_{(\widehat\GL_{r,R})_K} & (\hat E_K)_{\dot} \ar[d]^{\GL_{r,K}^{\dag}} \ar[r]^{\subset}\ar@{}[dr]|{\lrcorner} & (E_K^{\dag})_{\dot}\ar[d]^{\GL_{r,K}^{\dag}}\\
& (\widehat X_{\dot})_K \ar[r]_{\subset} & X_{K,\dot}^{\dag}
}
\]

\medskip

We have  natural morphisms of complexes of sheaves
\[
\Omega^{\geq n}_{\ol X_{K,\dot}}(\log D_{\dot}) \to \Rhyp j_*\Rhyp\iota_*\Omega^*_{(\widehat X_{\dot})_K} \xrightarrow{p^*} \Rhyp j_*\Rhyp\iota_*\Rhyp p_* \Omega^*_{(\widehat E_{\dot})_K},
\]
where, by abuse of notation, the composition $(\widehat X_{\dot})_K \subset X_{K,\dot}^{\dag} \xrightarrow{\iota} X_{K,\dot}$ is still denoted by $\iota$,
and define 
\begin{align*}
H^*_{\rel}(X_{\dot}/R, n) & := \Hyp^*\left(\ol X_{K,\dot},\Cone\left(\Omega^{\geq n}_{\ol X_{K,\dot}}(\log D_{\dot}) \to \Rhyp j_*\Rhyp\iota_*\Omega^*_{(\widehat X_{\dot})_K}\right)\right),\\
H^{E,*}_{\rel}(X_{\dot}/R,n) & := \Hyp^*\left(\ol X_{K,\dot},\Cone\left(\Omega^{\geq n}_{\ol X_{K,\dot}}(\log D_{\dot}) \to \Rhyp j_*\Rhyp\iota_*\Omega^*_{(\widehat E_{\dot})_K}\right)\right).
\end{align*}

\medskip

The factorisations

\begin{minipage}{6cm}
\[
\xymatrix@C-1.7cm@R-0.5cm{
\Omega^{\geq n}_{\ol X_{K,\dot}}(\log D_{\dot}) \ar[rr] \ar[dr] && \Rhyp j_*\Rhyp\iota_* \Omega^*_{(\widehat X_{\dot})_K}\\
& \Rhyp j_*\Rhyp\iota_*\Omega^{\geq n}_{(\widehat X_{\dot})_K} \ar@{.>}[ur]
}
\]
\end{minipage} and
\begin{minipage}{6cm}
\[
\xymatrix@C-1.7cm@R-0.5cm{
\Omega^{\geq n}_{\ol X_{K,\dot}}(\log D_{\dot}) \ar[rr] \ar[dr] && \Rhyp j_*\Rhyp\iota_* \Omega^*_{(\widehat X_{\dot})_K}\\
& \Rhyp j_*\Rhyp\iota_*\Omega^*_{X_{K,\dot}^{\dag}} \ar@{.>}[ur]
}
\]
\end{minipage}
induce natural maps
\[
H^*_{\rel}(X_{\dot}/R, n) \to H^*_{\rel}((\widehat X_{\dot})_K, n)\quad \text{ and }\quad H^*_{\rel}(X_{K,\dot}, n) \to H^*_{\rel}(X_{\dot}/R, n)
\]
respectively. Similarly, we have a natural map 
\[
H^{E_K,*}_{\rel}(X_{K,\dot}, n) \to H^{E,*}_{\rel}(X_{\dot}/R, n)
\]
and in particular we can consider the image of the refined class $\widetilde\Ch_n^{\rel}(E_K)$ in $H^{E,2n-1}_{\rel}(X_{\dot}/R, n)$. It will be denoted by $\widetilde\Ch_n^{\rel}(E/R)$.

\begin{rem}\label{rem:RelCohomProper}
Assume, that $X_{\dot}$ is a smooth \emph{proper} simplicial $R$-scheme. In this case all the different relative cohomology groups coincide: The natural map $(\widehat X_{\dot})_K \to X_{K,\dot}^{\dag}$ is an isomorphism (see section \ref{sec:PrelimSpaces}) and a good compactification $j$ of $X_{K,\dot}$ is given by the identity. Hence
\[
H^*_{\rel}(X_{K,\dot}, n) = \Hyp^*\left(X_{K,\dot}, \Cone(\Omega^{\geq n}_{X_{K,\dot}} \to \Rhyp\iota_* \Omega^*_{X_{K,\dot}^{\dag}})\right)
\]
and 
\[
H^*_{\rel}(X_{\dot}/R, n) = \Hyp^*\left(X_{K,\dot}, \Cone(\Omega^{\geq n}_{X_{K,\dot}} \to \Rhyp\iota_* \Omega^*_{(\widehat X_{\dot})_K})\right)
\]
are isomorphic. Moreover, the last group is isomorphic to
\[
H^*_{\rel}((\widehat X_{\dot})_K, n)= H^*_{\rel}(X_{K,\dot}^{\dag}, n) = \Hyp^*\left(X_{K,\dot}^{\dag}, \Cone(\Omega^{\geq n}_{X_{K,\dot}^{\dag}} \to \Omega^*_{X_{K, \dot}^{\dag}})\right)
\]
by the GAGA-principle \cite[Korollar 4.5]{GK}.
\end{rem}

Now assume, that $X_{\dot}$ is \emph{affine} and that the bundle $\hat E_K$ admits a topological trivialization $\alpha$ such that $\alpha: (\hat X_{\dot})_K \rightsquigarrow E_{\dot}\GL_{r,K}^{\dag}$ factors through the admissible open subspace $E_{\dot}(\widehat\GL_{r,R})_K \subset E_{\dot}\GL_{r,K}^{\dag}$.
By base change, $\alpha$ induces a topological morphism $\alpha : (\hat X_{\dot})_K \rightsquigarrow (\hat E_{\dot})_K$. 
Similar as in lemma \ref{lemma:TopPullback} (use the fact, that, if $U$ is an affine $R$-scheme, then $\hat U_K$ is an affinoid dagger space, hence acyclic for coherent sheaves) we may then construct $\alpha^*: H^{E,*}_{\rel}(X_{\dot}/R, n) \to H^*_{\rel}(X_{\dot}/R, n)$ left inverse to $p^*$.
Hence we can define
\[
\widetilde\Ch_n^{\rel}(T,E,\alpha/R) := -\alpha^*\widetilde\Ch_n^{\rel}(E/R) \in H^{2n-1}_{\rel}(X_{\dot}/R, n).
\]

We record the following properties:
\begin{prop}\label{prop:ComparisonRelClassesRSchemes}
Let $E$ be an algebraic $\GL_r$-bundle on the smooth affine simplicial $R$-scheme $X_{\dot}$. We have induced bundles $E_K$ on $X_{K,\dot}$, $E_K^{\dag}$ on $X_{K,\dot}^{\dag}$ and $\hat E_K$ on $(\hat X_{\dot})_K$. Assume, that $\alpha: (\hat X_{\dot})_K \rightsquigarrow E_{\dot}(\widehat\GL_{r,R})_K$ is a topological trivialization of $\hat E_K$ as above.
\begin{enumerate}
\item  $\widetilde\Ch_n^{\rel}(T,E,\alpha/R)$ is mapped to $\Ch_n^{\rel}(T, \hat E_K, \alpha)$ by the natural map $H^{2n-1}_{\rel}(X_{\dot}/R,n) \to H^{2n-1}_{\rel}((\hat X_{\dot})_K,n)$.
\item If $\alpha$ extends to a topological trivialization $\alpha_K: X_K^{\dag} \rightsquigarrow E_{\dot}\GL_{r,K}^{\dag}$ of $E_K^{\dag}$,  $\widetilde\Ch_n^{\rel}(T,E_K,\alpha_K)$ is mapped to $\widetilde\Ch_n^{\rel}(T,E,\alpha/R)$
by the natural map $H^{2n-1}_{\rel}(X_{K,\dot}, n) \to H^{2n-1}_{\rel}(X_{\dot}/R, n)$.
\end{enumerate}
\end{prop}
\begin{proof}
(i) may be shown as the analogue proposition \ref{prop:ComparisonDaggerRelClasses}.

(ii) follows from the commutativity of the diagram
\[
\xymatrix{
H^{E_K,*}_{\rel}(X_{K,\dot}, n)\ar[r]\ar[d]_{\alpha_K^*} & H^{E,*}_{\rel}(X_{\dot}/R, n)\ar[d]^{\alpha*}\\
H^*_{\rel}(X_{K,\dot}, n)\ar[r] & H^*_{\rel}(X_{\dot}/R, n).
}
\]

\end{proof}

\chapter{Relative $K$-theory and regulators}
\label{ch:UltrametricRelKAndRegulators}

In this chapter we finally construct the relative Chern character for a smooth affine $R$-scheme\footnote{Since the dagger space associated with an affine $K$-scheme is not affinoid, I am not quite sure what the ``right'' definition of topological $K$-theory of a $K$-scheme is.}. First of all, we recall the definition of the topological $K$-groups of ultrametric Banach rings due to Karoubi and Villamayor and show, that one can similarly define topological $K$-groups for dagger algebras fitting in the context of the previous chapters. Having done this, we can define relative $K$-theory and the relative Chern character (sections 7.1 and 7.2) exactly as in the complex case. The comparison with the $p$-adic Borel regulator is done in section 7.4. 

\section{Topological $K$-theory of affinoid and dagger algebras}
\label{sec:UltrametricK}

Let $(A, |\;.\;|)$ be an ultrametric Banach ring, i.e. a ring $A$ together with a map $|\;.\;|:A \to \R_{\geq 0}$ such that $|x|=0$ iff $x=0$, $|x|=|-x|$, $|xy|\leq |x||y|$ and $|x+y| \leq \max\{|x|,|y|\}$, and such that $A$ is complete for the metric $(x,y) \mapsto |y-x|$. For example, any $K$- or $R$-affinoid algebra with a chosen norm is an ultrametric Banach ring. In \cite{KV} Karoubi and Villamayor define topological $K$-groups $K^{-*}_{\top}(A)$ for arbitrary Banach rings and sketch a particular approach for ultrametric Banach rings (using convergent power series instead of absolutely converging power series, see below), studied further by Adina Calvo \cite{Calvo}. For unitary Banach rings it may be formulated as follows: Define 
\begin{equation}\label{eq:DefA_dot}
A_p := A\<x_0, \dots, x_p\>/(\sum_i x_i -1),
\end{equation}
where
\[
A\<x_0,\dots, x_p\> = \{ \sum a_{\nu}x^{\nu} \in A[[x_0, \dots, x_p]]\;|\; |a_{\nu}| \xrightarrow{|\nu|\to\infty} 0\}.
\]
If $\phi : [p] \to [q]$ is an increasing map, we define $\phi^* : A_q \to A_p$ by $x_i \mapsto \sum_{j\in[p]:\phi(j)=i} x_j$. This is well defined, since $A\<x_0,\dots, x_p\>$ is complete (w.r.t. the Gau{\ss} norm) and the target elements are power bounded \cite[Proposition 1.4.3/1]{BGR}, and gives a simplicial ring $A_{\dot}$.
\begin{dfn}
The \emph{topological $K$-groups of $A$} are given by
\[
K^{-i}_{\top}(A) := \pi_i(B_{\dot}\GL(A_{\dot})) = \pi_{i-1}\GL(A_{\dot}), \quad i\geq 1.
\]
\end{dfn}
\begin{rems}\label{rem:top_K}
(i) It is clear from the definition, that the topological $K$-groups of an ultrametric Banach ring do not depend on the particular norm chosen. In particular, the topological $K$-groups of an $R$- or $K$-affinoid algebra are well defined.
\smallskip

(ii) For any simplicial group $G_{\dot}$, let $\bar G_p := \bigcap_{i=1}^p\ker(\del_i) \subset G_p$. Then $(\bar G_p, \del_0)_{p\geq 0}$ is a chain complex of (non abelian) groups, whose homology groups are the homotopy groups of $G_{\dot}$. Symmetrically, one can also use the chain complex $(\bigcap_{i=0}^{p-1} \ker(\del_i), \del_p)$ (see e.g. \cite[Proposition 17.4]{May}). 
\smallskip

(iii) The ring $A_{\dot}$ is additive contractible, i.e. the identity on $A_{\dot}$ is homotopic to the zero map by a homotopy which is compatible with the abelian group structure of $A_{\dot}$: The $1$-simplex $x_0 \in A_1$ corresponds to a map $f_{x_0}:\Delta[1] \to A_{\dot}$ such that $f_{x_0}(0) = \del_1(x_0) = 1, f_{x_0}(1) = \del_0(x_0) = 0$, where $0:= \delta^1$ and $1 := \delta^0 \in \Delta[1]_0= \Hom_{\Delta}([0],[1])$ are the two vertices of $\Delta[1]$\footnote{For $n\geq 0$ $\Delta[n]$ denotes the simplicial set $\Hom_{\Delta}(\,.\, ,[n]):\Delta^{\mathrm{op}} \to Sets$. Its geometric realisation is the standard simplex $\Delta^n\subset \R^{n+1}$. For any simplicial set $X_{\dot}$ there is a natural isomorphism $\Hom(\Delta[n],X_{\dot})= X_n$ (Yoneda lemma).}. 
The desired homotopy is then given by $A_{\dot}\times \Delta[1] \to A_{\dot}, (a, t) \mapsto f_{x_0}(t)\cdot a$.
In particular, the homotopy groups $\pi_*(A_{\dot})$ vanish. More generally the simplicial ring of $r\times r$ matrices $\Mat_r(A_{\dot})$ is additive contractible.
\smallskip

(iv) We also need the following refinement of the last remark. Equip $A\<x_0, \dots, x_p\>$ with the Gau{\ss} norm and $A_p$ with the residue semi-norm and denote it by $\|\,.\,\|$. It is easy to see, that for any $\phi : [p] \to [q]$, the induced homomorphism $\phi^* : A_q \to A_p$ is contractive, i.e. $\|\phi^*(f)\| \leq \|f\|$. Write $A_p^0 := \{f\in A_p\,|\, \|f\|\leq 1\}$.  It follows, that $A_{\dot}^0$ is a simplicial subring of $A_{\dot}$. 

The semi-norm on $\Mat_r(A_p)$ is defined to be the maximum of the semi-norms of the entries. Write $\Mat_r(A_p)^{00} := \{g \in \Mat_r(A_p)\,|\, \|g\| < 1\}$. Then $\Mat_r(A_{\dot})^{00}$ is a simplicial subgroup of $\Mat_r(A_{\dot})$, which is moreover an $A_{\dot}^0$-module.

Since $x_0 \in A_1^0$, the argument of the last remark shows, that $\Mat_r(A_{\dot})^{00}$ is additive contractible, too. 

\end{rems}

The definition given above is \emph{not} the one given by Karoubi--Villamayor and Calvo. Since the equivalence of both definitions is proved in the literature only in the case of \emph{discrete} Banach rings, we give a proof here.
\begin{prop}
The topological $K$-groups defined above coincide with those defined by Karoubi-Villamayor and Calvo for $i\geq 1$.
\end{prop}
\begin{proof}
The agrument in the discrete case is due to Anderson \cite[Theorem 1.6]{Anderson}. First we recall Calvo's definition.  It is best, to work in the category of ultrametric Banach rings without unit. Let $A$ be such a ring. Then $\GL_r(A)$ is by definition the group of $r\times r$-matrices invertible w.r.t. the formal group law $(M,N)\mapsto M \odot N := MN+M+N$. If $A$ has a unit, this is clearly equivalent to the usual definition via $M \mapsto M+1$. Denote by $\GL'_r(A)$ the subgroup of $\GL_r(A)$ generated by the topologically nilpotent matrices. Define $\GL(A)$ and $\GL'(A)$ as the usual colimits. Then $\ol K^{-1}(A) := \GL(A)/\GL'(A)$. 

For any $A$ as above, the path ring $EA$ is the kernel of $p_0:A\<t\> \to A, t\mapsto 0$, i.e. $EA = tA\<t\>$, the loop ring $\Omega A$ is the kernel of $p_1:EA \to A,$ ``$t\mapsto 1$'', i.e. $\sum_i a_it^i\mapsto \sum_i a_i$. These are again ultrametric Banach rings and the higher Karoubi--Villamayor $K$-groups are defined by 
\[
\ol K^{-i}(A) := \ol K^{-1}(\Omega^{i-1}A).
\]

A matrix $M \in \GL_r(A)$ is called \emph{null-homotopic}, if there exists $\widetilde M \in \GL_r(EA)$, such that $p_1(\widetilde M) = M$. Two matrices $M,N$ are called \emph{homotopic}, if $M\odot N^{-1}$ is null-homotopic. Denote by $\GL_r^0(A)=\im\left(p_1: \GL_r(EA) \to \GL_r(A)\right)$ the group of null-homotopic matrices and let $\GL^0(A) = \colim_r \GL^0_r(A)$. It follows as in \cite[Appendice 3]{KV} that $\GL'(A) = \GL^0(A)$ (the arguments given there and in the cited references for unitary rings also work in the non-unitary context). In other words, $p_1:\GL(EA) \to \GL(A)$ induces an isomorphism
$\GL(EA)/\GL(\Omega A) \cong \GL'(A)$. 

Now let $A$ be an unitary Banach ring and form the simplicial Banach ring $A_{\dot}$ as in \eqref{eq:DefA_dot}. Applying the above isomorphism to the simplicial Banach ring $\Omega^iA_{\dot}$, we get an isomorphism of simplicial groups
\begin{equation}\label{eq:IsomGL'}
\GL(E\Omega^iA_{\dot})/\GL(\Omega^{i+1}A_{\dot}) \cong \GL'(\Omega^iA_{\dot}), \quad i\geq 0.
\end{equation}
We claim, that $\GL(E\Omega^iA_{\dot})$ is contractible. In fact, the ring morphisms $h_j : E\Omega^iA_n = t(\Omega^iA)\<t\> \to E\Omega^iA_{n+1}$, $j = 0, \dots, n$, given by the degeneracy $s_j$ on $\Omega^iA_n$ and by $h_j(t) = t(x_0+\dots +x_j)$ define a simplicial homotopy between $0$ and $\id_{E\Omega^iA_{\dot}}$ in the sense of \cite[\S 5]{May}. Hence they induce a contracting homotopy of $\GL(E\Omega^iA_{\dot})$. By \eqref{eq:IsomGL'} we get isomorphisms
\begin{equation}\label{eq:IsomOmega}
\pi_n(\GL'(\Omega^iA_{\dot})) \xrightarrow{\cong} \pi_{n-1}(\GL(\Omega^{i+1}A_{\dot})), \quad i\geq 0, n \geq 1.
\end{equation}
Next, by definition we have an isomorphism of simplicial groups 
\begin{equation}\label{eq:IsomK1}
\GL(\Omega^iA_{\dot})/\GL'(\Omega^iA_{\dot}) = \ol K^{-1}(\Omega^i A_{\dot})
\end{equation}
and we claim, that this last group is a \emph{constant} simplicial group. Since $A_n \cong A\<x_0,\dots, x_{n-1}\>$ it suffices to show, that for any Banach ring $A$ the inclusion $A \hookrightarrow A\<x\>$ induces an isomorphism on $\ol K^{-1}$. Since this inclusion is split by $x\mapsto 0$, the induced map on $\ol K^{-1}$ is an injection. Now consider $h: A\<x\> \to A\<x\>\<t\>=A\<x,t\>, x\mapsto tx$. Then $p_0\circ h(x) = 0, p_1\circ h(x) = x$. Hence, any $M\in \GL(A\<x\>)$ is homotopic to $M(0) \in \GL(A)\subset \GL(A\<x\>)$, the null-homotopy for $M\odot M(0)^{-1}$ being given by $h(M)\odot M(0)^{-1} \in \GL(EA\<x\>)$. It follows, that $\ol K^{-1}(A) \to \ol K^{-1}(A\<x\>)$ is also surjective.

Hence we get from \eqref{eq:IsomK1}, that 
\[
\pi_n(\GL'(\Omega^iA_{\dot})) \cong \pi_n(\GL(\Omega^iA_{\dot})), \quad i\geq 0, n\geq 1. 
\]
Combining this with \eqref{eq:IsomOmega}, we get 
\[
\pi_n(\GL(A_{\dot})) \cong \pi_{n-1}(\GL(\Omega A_{\dot})) \cong \dots \cong \pi_0(\GL(\Omega^nA_{\dot})), n\geq 0.
\]
Since the left hand side is $K^{-n-1}_{\top}(A)$ by definition, it suffices to show, that $\pi_0(\GL(\Omega^nA_{\dot})) = \ol K^{-n-1}(A) = \ol K^{-1}(\Omega^nA)$. We have isomorphisms $A_1 \cong A\<x\>,x_0\mapsto x, x_1\mapsto 1-x,$ and $A_0 \cong A, x_0 \mapsto 1$, under which $\del_0, \del_1$ are given by $p_0\colon x\mapsto 0, p_1\colon x\mapsto 1$ respectively. Hence by remark \ref{rem:top_K}(ii) 
\begin{align*}
\pi_0(\GL(\Omega^nA_{\dot})) &= \coker\left(p_1\colon \ker\bigl(p_0\colon \GL(\Omega^nA\<x\>)\to \GL(\Omega^nA)\bigr) \to \GL(\Omega^nA)\right)\\
& = \coker(p_1\colon \GL(E\Omega^nA) \to \GL(\Omega^nA))\\
& = \GL(\Omega^nA)/\GL^0(\Omega^nA) = \ol K^{-1}(\Omega^nA).\qedhere
\end{align*}
\end{proof}

Now let $R$ be a complete discrete valuation ring with maximal ideal $(\pi)$, perfect residue field $R/(\pi) =k$ of characteristic $p > 0$ and field of fractions $K$ of characteristic $0$.
We want to show, that the topological $K$-groups of affinoid algebras may also be computed using \emph{overconvergent} power series. More precisely: Define the $R$-dagger algebra $R_n^{\dag} := R\<x_0, \dots, x_n\>^{\dag}/(\sum_i x_i -1)$ and similarly the $K$-dagger algebra $K_n^{\dag}$. As above, we get simplicial $R$- resp. $K$-dagger algebras $R_{\dot}^{\dag}$ and $K_{\dot}^{\dag}$.
\begin{dfn}
Let $A$ be an $R$-dagger algebra. We define the topological $K$-groups
\[
K^{-i}_{\top}(A) := \pi_i(B_{\dot}\GL(A \otimes_R^{\dag} R_{\dot}^{\dag})) = \pi_{i-1}(\GL(A \otimes_R^{\dag} R_{\dot}^{\dag})), \quad i\geq 1.
\]
If $A$ is a $K$-dagger algebra, the definition is the same with $R$ replaced by $K$.
\end{dfn}

\begin{prop}
Let $A$ be an $R$- or $K$-dagger algebra, and $\hat A$ its completion, an $R$- resp. $K$-affinoid algebra. Then
\[
K^{-n}_{\top}(A) \cong K^{-n}_{\top}(\hat A),\quad  n\geq 1.
\]
\end{prop}
\begin{rem}
In \cite{Kar97} Karoubi states, that one can use ``indefinitely integrable power series'' to define the topological $K$-theory of ultrametric Banach algebras and uses these for the construction of the relative Chern character. The difference here is, that we do not use the full Banach algebra $\hat A$, but only the overconvergent part $A$ of it. 
\end{rem}

\begin{proof}[Proof of the proposition]
The proof is the same for $R$- and $K$-dagger algebras and we restrict to the case of $R$-dagger algebras.
Choose a representation $A = R\<\ul y\>^{\dag}/I$ and write $A_n := A \otimes_R^{\dag} R_n^{\dag} = R\<\ul y, x_0, \dots, x_n\>^{\dag}/(I, \sum_i x_i - 1)$. 

The completion of $A$ is given by $\hat A = R\<\ul y\>/(I)$ and the ring $(\hat A)_n$ appearing in the definition of the topological $K$-theory of $A$ is given by $(\hat A)_n = (R\<\ul y\>/(I))\<x_0, \dots, x_n\>/(\sum_i x_i -1) = R\<\ul y, x_0, \dots, x_n\>/(I, \sum_i x_i -1) = \widehat{(A_n)}$.

Since $\pi_n(\GL(A_{\dot})) = \dirlim_r \pi_n(\GL_r(A_{\dot}))$, it suffices to show, that for any $r\geq 1$ the natural map $\pi_n(\GL_r(A_{\dot})) \to \pi_n(\GL_r(\hat A_{\dot}))$ is an isomorphism.
\smallskip

We begin with the surjectivity. A class in $\pi_n(\GL_r(\hat A_{\dot}))$ is represented by a $g \in \GL_r(\hat A_n)$ such that $\del_i g=1$, $i = 0, \dots, n$. By lemma \ref{lem:ApproxPoly} below, there is a sequence of matrices $g_N \in \Mat_r(A_n)$ converging to $g$, where each $g_N$ satisfies $\del_i g_N = 1$, $i=0, \dots, 1$. Since $\GL_r(\hat A_n) \subset \Mat_r(\hat A_n)$ is open, $g_N \in \GL_r(\hat A_n)$ for $N$ large enough, and by lemma \ref{lem:InvertiblePoly} $g_N \in \GL_r(A_n)$. We claim, that for $N$ large enough, $[g_N] = [g]$ in $\pi_n(\GL_r(\hat A_{\dot}))$, thus showing surjectivity.

We use remark \ref{rem:top_K}(iv). Choose $N$ large enough, so that $g_Ng^{-1} -1 \in \Mat_r(\hat A_n)^{00}$. Since $\Mat_r(\hat A_{\dot})^{00}$ is contractible, there exists $h \in \Mat_r(\hat A_{n+1})^{00}$, such that $\del_0(h)=g_Ng^{-1}-1$, $\del_i(h) = 0$, $i>0$. Since $\| h\| < 1$, $1+h \in \GL_r(\hat A_{n+1})$.
Moreover, $\del_0(1+h) = g_Ng^{-1}$ and $\del_i(1+h) = 1$ for $i>0$, hence $[g_Ng^{-1}]=[1]$ and $[g_N] = [g]$ as claimed.
\smallskip

Next we prove the injectivity. Thus let $g \in \GL_r(A_n)$ with $\del_i(g) = 1$, $i=0, \dots, n$, and assume that there exists $h \in \GL_r(\hat A_{n+1})$, such that $\del_0(h) = g$, $\del_i(h) = 1$ if $i>0$. As in remark \ref{rem:top_K}(iii) $\pi_*(\Mat_r(A_{\dot})) =0$. Hence there exists a matrix $\widetilde h \in \Mat_r(A_{n+1})$ such that $\del_0(\widetilde h) = g$, $\del_i(\widetilde h) = 1$, $i=1, \dots, n+1$. Now we can apply lemma \ref{lem:ApproxPoly} to $h - \widetilde h$ to obtain a sequence of matrices $h_N \in \Mat_r(A_{n+1})$ converging to $h - \widetilde h$ and satisfying $\del_i(h_N) = 0$ for $i=0, \dots, n+1$ and all $N$. Then $h_N + \widetilde h \in \Mat_r(A_{n+1})$ converges to $h \in \GL_r(\hat A_{n+1})$, hence $h_N + \widetilde h \in \GL_r(A_{n+1})$ for $N$ large enough, again by the openness of $\GL_R(\hat A_{n+1})$ and lemma \ref{lem:InvertiblePoly}. Moreover $\del_0(h_N+\widetilde h) = g$, $\del_i(h_N+\widetilde h) = 1$, $i=1, \dots, n+1$, hence $[g]=[1]$ in $\pi_n(\GL_r(A_{\dot}))$.
\end{proof}

\begin{lemma}\label{lem:ApproxPoly}
We use the notations of the above proof.
Let $g \in \Mat_r(\hat A_n)$ be such that $\del_i g = 0$, $i = 0, \dots, n$. There exists a sequence of matrices $g_N \in \Mat_r(A_n)$, $N\geq 0$, which converges to $g$ in $\Mat_r(\hat A_n)$ and satisfies $\del_ig_N=0$, for $i=0, \dots, n$ and all $N$.
\end{lemma}
\begin{proof}
As in remark \ref{rem:top_K}(iii) $\pi_*(\Mat_r(\hat A_{\dot})) = 0$. Hence there exists $h \in \Mat_r(\hat A_{n+1})$, such that $\del_0(h) = g$, $\del_i(h) = 0$, $i > 0$.

We have an isomorphism $\hat A_n = R\<\ul y, x_1, \dots, x_n\>/(I)$ given by $x_0 \mapsto 1 - \sum_{i\geq 1} x_i, x_i \mapsto x_i, i>0$, and similar $\hat A_{n+1} = R\<\ul y, x_1, \dots, x_{n+1}\>/(I)$ . In terms of these isomorphisms $\del_i$, $i > 0$, is given by $x_j \mapsto x_j$, if $j < i$, $x_i \mapsto 0$ and $x_j \mapsto x_{j-1}$, if $j>i$. 

Represent $h$  by a power series 
\[
\tilde h = \sum_{\nu\in \N_0^{n+1}} a_{\nu}(\ul y)x_1^{\nu_1}\cdots x_{n+1}^{\nu_{n+1}} \in \Mat_r(R\<\ul y,x_1, \dots, x_{n+1}\>),
\]
where $a_{\nu}(\ul y) \in \Mat_r(R\<\ul y\>)$.
Let $i>0$. Since 
\[
\del_i\tilde h = \sum_{\nu\in\N_0^{n+1}\colon \nu_i=0} a_{\nu}(\ul y)x_1^{\nu_1}\cdots x_{i-1}^{\nu_{i-1}} x_i^{\nu_{i+1}} \cdots x_n^{\nu_{n+1}}
\]
represents $\del_ih=0$, $a_{\nu}(\ul y)$ has entries in $(I) \subset R\<\ul y\>$ as soon as  $\nu_i=0$ for one $i$, and 
we may assume without loss of generality, that $a_{\nu}(\ul y) = 0$ for those $\nu$.
Now let $\tilde h_N \in \Mat_r(R[\ul y,x_1, \dots, x_{n+1}])$ be the polynomial arising from $\tilde h$ by deleting all terms of total degree greater than $N$ and denote by $h_N$ its image in $\Mat_r(A_{n+1})$. By construction, we  have $\del_i(\tilde h_N) = 0$, if $i>0$, and for $N \to \infty$ the $\tilde h_N$ converge to $\tilde h$. Hence also $\del_i(h_N)=0$ for $i>0$ and the $h_N$ converge to $h$ in $\Mat_r(\hat A_{n+1})$. Now define $g_N := \del_0(h_N) \in \Mat_r(A_n)$. Then the $g_N$ converge to $\del_0(h) =g$, and $\del_i(g_N) = \del_i\del_0(h_N) = \del_0 \del_{i+1}(h_N) = 0$, $i=0, \dots, n$.
\end{proof}

\begin{lemma}\label{lem:InvertiblePoly}
Let $g \in \Mat_r(A_n)$ be a matrix, whose image in $\Mat_r(\hat A_n)$ is invertible. Then $g$ itself is invertible.
\end{lemma}
\begin{proof}
Equip $R\<\ul y,\ul x\>$ with the Gau{\ss} norm and $A_n, \hat A_n, \Mat_r(A_n)$, etc., with the induced norms.

Let $h \in \GL_r(\hat A_n)$ be the inverse of $g$. Since $A_n$ is dense in $\hat A_n$, we may approximate $h$ by matrices $h_N \in \Mat_r(A_n)$. Then $h_N\cdot g \xrightarrow{N\to \infty} 1$ in $\Mat_r(A_n)$, and, for $N$ large enough, $\|1- h_Ng\| \leq |\pi|$ (recall that $\pi$ is a uniformizer for $R$). Then we can represent $1-h_Ng$ by a matrix of power series $f_N$ in $\Mat_r(R\<\ul y, \ul x\>^{\dag})$, with $\|f_N\| \leq |\pi|$. Then $f_N = \pi f_N'$, where $f_N'$ is a matrix of power series with $\|f_N'\| \leq 1$. Since $R\<\ul y, \ul x\>^{\dag}$ is weakly complete (cf. section \ref{sec:PrelimAlgebras}), the series $\sum_{k=0}^{\infty} \pi^k(f_N')^k = \sum_{k=0}^{\infty} f_N^k$ converges in $\Mat_r(R\<\ul y,\ul x\>^{\dag})$ and defines an inverse of $1 - f_N$. Hence its image $e_N$ in $\Mat_r(A_n)$ is an inverse of $1- (1 - h_Ng) = h_Ng$. Then $e_N\cdot h_N$ is a left inverse of $g$ in $\Mat_r(A_n)$. By the same argument applied to $1 - gh_N$, $g$ also possesses a right inverse, hence is invertible in $\Mat_r(A_n)$.
\end{proof}

\begin{prop}\label{prop:TopAlgK}
Let $A$ be an $R$-affinoid or $R$-dagger algebra and assume, that $A/\pi A$ is regular. Then
\[
K^{-n}_{\top}(A) \cong K_n(A/\pi A), \quad n\geq 1.
\]
\end{prop}
\begin{proof}
This follows from Calvo's Proposition 2.1 \cite{Calvo} and Gersten's result, that Karoubi--Villamayor theory for discrete noetherian regular rings coincides with  Quillen's $K$-theory \cite[Proposition 3.14]{Gersten}. In fact, similar methods as above show that $\pi_{n-1}(\GL(A_{\dot})) = \pi_{n-1}(\GL((A/\pi A)_{\dot}))$, where $(A/\pi A)_n = (A/\pi A)[x_0, \dots, x_n]/(\sum x_i -1)$, and the right hand side is the Karoubi--Villamayor $K$-group $K^{-n}(A/\pi A)$.
\end{proof}

\section{Relative $K$-theory}

Let $R$ be as before. Let $X = \Spec(A)$ be an affine $R$-scheme of finite type. Let $\hat A$ denote the $\pi$-adic completion of $A$, an $R$-affinoid algebra, and $A^{\dag} \subset \hat A$ the weak completion of $A$, an $R$-dagger algebra. We define the topological $K$-groups of $X$ to be
\[
K^{-i}_{\top}(X) := K^{-i}_{\top}(\hat A) = K^{-i}_{\top}(A^{\dag}) = \pi_i(B_{\dot}\GL(A^{\dag}\otimes_R^{\dag} R_{\dot}^{\dag})), \quad i\geq 1.
\]
Recall that $K_i(X) = \pi_i(|B_{\dot}\GL(A)|^+)$. Since $\pi_1(B_{\dot}\GL(A^{\dag}\otimes_R^{\dag} R_{\dot}^{\dag})) = K^{-1}_{\top}(\hat A)$ is abelian, the natural morphism $|B_{\dot}\GL(A)| \to |B_{\dot}\GL(A^{\dag}\otimes_R^{\dag} R_{\dot}^{\dag})|$ factors up to homotopy uniquely through $|B_{\dot}\GL(A)| \to |B_{\dot}\GL(A)|^+$. We abbreviate the simplicial group $\GL(A^{\dag}\otimes_R^{\dag} R_{\dot}^{\dag})$ by $G_{\dot}$. As in the complex case (cf. section \ref{sec:RelK}), we define the space $\widetilde F$ and the simplicial set $\mathscr F$ by the pullback diagrams
\[\xymatrix{
\widetilde F \ar@{}[dr]|{\lrcorner}\ar[r]\ar[d] & |E_{\dot}G_{\dot}| \ar[d]^{|p|}  & \mathscr F \ar@{}[dr]|{\lrcorner}\ar[r]\ar[d] & E_{\dot}G_{\dot}\ar[d]^p\\
|B_{\dot}\GL(A)|^+ \ar[r] & |B_{\dot}G_{\dot}|, & B_{\dot}\GL(A) \ar[r] & B_{\dot}G_{\dot}.
}
\]
Again there is an acyclic map $|\mathscr F| \to \widetilde F$. We define the relative $K$-groups
\[
K_i^{\rel}(X) := \pi_i(\widetilde F), \quad i\geq 1.
\]

\section{The relative Chern character}

Let $X=\Spec(A)$ be a smooth affine $R$-scheme of finite type. 
Recall, that $\hat X_K= \Sp(A^{\dag}\otimes_R K)$ denotes the generic fibre of the weak completion of $X$. 
The relative cohomology groups $H^*_{\rel}(X/R, n)$ are defined as in the simplicial case (section \ref{sec:RelClassesRSchemes}).

We want to construct relative Chern character maps
\[
\Ch_{n,i}^{\rel} : K_i^{\rel}(X) \to H^{2n-i-1}_{\rel}(X/R, n). 
\]

This is done as in the complex case: First of all, define the simplicial set $\mathscr F_r$ by the pullback diagram
\begin{equation}\label{diag:UltraRelK}\begin{split}
\xymatrix{
\mathscr F_r \ar@{}[dr]|{\lrcorner}\ar[r]\ar[d] & E_{\dot}\GL_r(A^{\dag} \otimes_R^{\dag} R_{\dot}^{\dag}) \ar[d]^p\\
B_{\dot}\GL_r(A) \ar[r] & B_{\dot}\GL_r(A^{\dag} \otimes_R^{\dag} R_{\dot}^{\dag}),
}
\end{split}
\end{equation}
so that $\mathscr F = \dirlim_r \mathscr F_r$.
\begin{lemma}\label{lem:MatrixMorphism}
Any matrix $g\in \GL_r(A^{\dag}\otimes_R^{\dag} R_p^{\dag})$ induces a morphism of dagger spaces
\[
\Delta^p \times \hat X_K \to (\widehat\GL_{r,R})_K.
\]
\end{lemma}
\begin{proof}
First of all $(\widehat\GL_{r,R})_K = \Sp(K\<x_{ij}, y\>^{\dag}/(\det(x_{ij})y-1))$ and $\hat X_K = \Sp(A^{\dag} \otimes_R K)$. The matrix $g$ is determined by a morphism of $R$-algebras
\[
R[x_{ij},y]/(\det(x_{ij})y-1) \to A^{\dag} \otimes_R^{\dag} R_p^{\dag}.
\]
Since $A^{\dag} \otimes_R^{\dag} R_p^{\dag}$ is an $R$-dagger algebra as well, this morphism extends uniquely to a morphism
\[
R\<x_{ij},y\>^{\dag}/(\det(x_{ij})y-1) \to A^{\dag} \otimes_R^{\dag} R_p^{\dag}
\]
\cite[Theorem 1.5]{MW}, which in turn induces a morphism of dagger spaces
\[
\Sp((A^{\dag} \otimes_R^{\dag} R_p^{\dag})\otimes_R K) \to \Sp(K\<x_{ij},y\>^{\dag}/(\det(x_{ij})y-1)),
\]
that is,
\[
\hat X_K \times \Delta^p \to (\widehat\GL_{r,R})_K. \qedhere
\]
\end{proof}

Hence the above diagram \eqref{diag:UltraRelK} gives rise to a morphism of simplicial $R$-schemes 
\[
g_r : X \otimes \mathscr F_r \to B_{\dot}\GL_{r,R},
\]
together with a commutative diagram
\[
\xymatrix@C+1cm{
& E_{\dot}(\widehat\GL_{r,R})_K \ar[d]^p\\
\hat X_K \otimes \mathscr F_r \ar[r]^{(\hat g_r)_K} \ar@{~>}[ur]^{\alpha_r} & B_{\dot}(\widehat\GL_{r,R})_K
}
\]
of topological morphisms of dagger spaces.
Thus we are exactly in the situation of section \ref{sec:RelClassesRSchemes} and have relative Chern character classes
\[
\widetilde\Ch_n^{\rel}(T_r,E_r,\alpha_r/R) \in H^{2n-1}_{\rel}(X\otimes \mathscr F_r/R, n),\footnote{Here we tacitly extended the definition of relative cohomology to the case of simplicial schemes of the form $X\otimes S$,  which obviously does make sense.}
\]
where we denote by $T_r$ the trivial $\GL_r$-bundle and by $E_r$ the algebraic $\GL_r$-bundle classified by $g_r$.
Similar as for the complex analogue one shows, that these classes are compatible for different $r$:
\begin{lemma}
The class $\widetilde\Ch_n^{\rel}(T_{r+1},E_{r+1},\alpha_{r+1}/R)$ is mapped to the class $\widetilde\Ch_n^{\rel}(T_r,E_r,\alpha_r/R)$ by the natural map $H^{2n-1}_{\rel}(X\otimes \mathscr F_{r+1}/R, n) \to H^{2n-1}_{\rel}(X\otimes \mathscr F_r/R, n)$ induced by the inclusion $j : \GL_r \hookrightarrow \GL_{r+1}$ in the upper left corner.
\end{lemma}

The other input we need for the definition of Chern character maps is 
\begin{lemma}
Let X be a smooth separated $R$-scheme of finite type and $S$ a simplicial set. Then we have natural isomorphisms
\[
H^k_{\rel}(X\otimes S/R, n) \cong \bigoplus_{p+q=k} \Hom(H_p(S), H^q_{\rel}(X/R, n)).
\]
\end{lemma}

\begin{proof}
Choose a good compactification $j : X_K \hookrightarrow \ol X_K$ with complement $D$. This induces a good compactification $j: X_K\otimes S \hookrightarrow \ol X_K \otimes S$. Denote the natural morphisms $\hat X_K \to X_K$ and $\hat X_K \otimes S \to X_K \otimes S$ by $\iota$. Choose a complex $I^*$ of injective sheaves on $\ol X_K$ representing $\Cone(\Omega^{\geq n}_{\ol X_K}(\log D) \to \Rhyp j_*\Rhyp\iota_*\Omega^*_{\hat X_K})$. 
Thus $\Gamma(\ol X_K, I^*)$ is a complex computing $H^*_{\rel}(X/R,n)$.

In an obvious way, $I^*$ induces a complex of sheaves $I^*\otimes S$ on the simplicial scheme $\ol X_K \otimes S$, which represents $\Cone(\Omega^{\geq n}_{\ol X_K\otimes S}(\log D\otimes S) \to \Rhyp j_*\Rhyp\iota_*\Omega^*_{\hat X_K \otimes S})$, since $\Rhyp j_*\Rhyp\iota_*\Omega^*_{\hat X_K\otimes S}$ can be computed ``degree-wise'' \cite[(5.2.5)]{HodgeIII}.
Moreover, each $I^q\otimes S_p$ is an injective sheaf on $\ol X_K \otimes S_p$, and hence the relative cohomology $H^*_{\rel}(X\otimes S/R,n)$ is just the cohomology of the total complex associated with the cosimplicial complex $[p]\mapsto \Gamma(\ol X_K \otimes S_p, I^*\otimes S_p) = \prod_{\sigma\in S_p} \Gamma(\ol X_K, I^*)$. Now the claim follows as in lemma \ref{lemma:cohom_X_tensor_S}.
\end{proof}

Putting everything together, we can now define:
\begin{dfn*}
Let $X=\Spec(A)$ be a smooth affine $R$-scheme of finite type. The \emph{relative Chern character} 
\[
\Ch_{n,i}^{\rel}: K_i^{\rel}(X) \to H^{2n-i-1}_{\rel}(X/R, n)
\]
is given by the composition
\begin{multline*}
K_i^{\rel}(X) = \pi_i(\widetilde F) \xrightarrow{\text{Hur.}} H_i(\widetilde F, \Z) \cong H_i(\mathscr F,\Z) =\\
\dirlim_r H_i(\mathscr F_r, \Z) \xrightarrow{\dirlim_r\widetilde\Ch_n^{\rel}(T_r,E_r,\alpha_r/R)} H^{2n-i-1}_{\rel}(X/R, n).
\end{multline*}
\end{dfn*}

\begin{rem}
The construction of the relative Chern character for ultrametric Banach algebras is also due to Karoubi \cite{KarCR, Kar97}. Instead of overconvergent power series he uses indefinitely integrable power series and in contrast to our construction his relative Chern character takes values in the cohomology of the truncated de Rham complex of the \emph{rigid} space $\Sp(\hat A_K)$. That is, he does not take logarithmic singularities or overconvergence into account.
\end{rem}

\section[Comparison with the $p$-adic Borel regulator]{The case $X = \Spec(R)$: Comparison with the $p$-adic Borel regulator}
\label{sec:pAdicBorel}

In this section we study in more detail the situation $X = \Spec(R)$, where $R$ is the ring of integers in a finite extension of $\Q_p$. This will be used to compare the relative Chern character with the $p$-adic Borel regulator.
Thus, throughout this section we fix a finite extension $K$ of $\Q_p$ with ring of integers $R \subset K$, uniformizer $\pi\in R$ and residue field $k$.

Similar as for the relative Chern character for $\Spec(\C)$, the Chern-Weil theoretic description of secondary classes yields an explicit cocycle defining the relative Chern character on the simplicial set $\GL(R_{\dot}^{\dag})/\GL(R)$ (the homotopy fibre of $B_{\dot}\GL(R) \to B_{\dot}\GL(R_{\dot}^{\dag})$). Since the explicit description of the Lazard isomorphism due to Huber and Kings describes the map from locally analytic group cohomology to Lie algebra cohomology (in contrast to Dupont's description of the van Est isomorphism) and since the $p$-adic Borel regulator is defined by a Lie algebra cocycle, we could take a locally analytic cocycle on $B_{\dot}\GL(R)$, which induces the relative Chern character on $\GL(R_{\dot}^{\dag})/\GL(R)$, and then check, that it is mapped to the Lie algebra cocycle defining the $p$-adic Borel regulator by the explicit Lazard map. Since it is not so easy, to find such a locally analytic cocycle, we use the following approach. The Lazard isomorphism factors through the locally analytic group cohomology of $U(R) := \ker(\GL(R) \to \GL(k))$. We construct a section $\nu$ of the map $\GL(R_{\dot}^{\dag})/\GL(R) \to B_{\dot}\GL(R)$, which is only defined on $B_{\dot}U(R)\subset B_{\dot}\GL(R)$, and show, that it induces a surjection $\nu_*: H_*(B_{\dot}U(R), \Q) \to H_*(\GL(R_{\dot}^{\dag})/\GL(R), \Q)$. This is done by showing, that the map $H_*(\GL(R_{\dot}^{\dag})/\GL(R), \Q) \to H_*(B_{\dot}\GL(R), \Q)$ is in fact an isomorphism (section 7.4.1). Via $\nu$ our explicit cocycle for the relative Chern character gives a group cocycle on $U(R)$ and we show in section 7.4.3, that it is in fact locally analytic. Hence we can apply the Lazard map and show that this cocycle is (up to a constant) mapped to the Lie algebra cocycle defining the $p$-adic Borel regulator (section 7.4.4). By the surjectivity of $\nu_*$ this implies the desired comparison.

\subsection{The homology of the fibre of $B_{\dot}\GL(R) \to B_{\dot}\GL(R_{\dot}^{\dag})$.}

We proceed as in the complex case. We abbreviate the simplicial group $\GL_r(R_{\dot}^{\dag}) =: G_{r,\dot}$, $\GL(R_{\dot}^{\dag}) = G_{\dot} = \dirlim_r G_{r,\dot}$. 
Again we have a homotopy equivalence $\eta_r: G_{r,\dot}/\GL_r(R) \to \mathscr F_r = B_{\dot}\GL_r(R) \times_{B_{\dot}G_{r,\dot}} E_{\dot}G_{r,\dot}$ of two models for the homotopy fibre of the map $B_{\dot}\GL_r(R) \to B_{\dot}G_{r,\dot}$ (lemma \ref{lem:HomotopyFibres}). 

\begin{thm}\label{thm:RationalHomolIsom}
The natural map 
\[
\rho_*: H_*(G_{\dot}/\GL(R), \Q) \to H_*(B_{\dot}GL(R), \Q)
\]
is an isomorphism.
\end{thm}
The proof uses the Serre spectral sequence for the homotopy fibration $G_{\dot}/\GL(R) \to B_{\dot}\GL(R) \to B_{\dot}\GL(R_{\dot}^{\dag})$.
We have to study this in more detail. Write $G := G_0 = \GL(R)$. 
By lemma \ref{lem:HomotopyFibres} we have a diagram
\begin{equation}\label{diag:hofib}\begin{split}
\xymatrix@C+1cm{
G_{\dot}/G \ar@{=}[d] \ar[r]^{\sigma\mapsto (\sigma, \dots, \sigma)} & E_{\dot}G_{\dot}/G \ar@{->>}[r]^p & B_{\dot}G_{\dot},\\
G_{\dot}/G \ar[r]^-{\rho} & E_{\dot}G/G \cong B_{\dot}G \ar[u]^{\mathrm{incl.}}_{\sim} \ar[ur]
}
\end{split}
\end{equation}
where the inclusion $E_{\dot}G/G \to E_{\dot}G_{\dot}/G$ is a homotopy equivalence and the left square commutes up to homotopy. 
Since $E_{\dot}G_{\dot}/G \overset{p}{\twoheadrightarrow} B_{\dot}G_{\dot}$ is a Kan fibration with fibre $G_{\dot}/G$, we have the associated Serre spectral sequence \cite[Kap. VI, \S 6]{Lamotke}
\begin{equation}\label{eq:SerreSS}
E^2_{p,q} = H_p(B_{\dot}G_{\dot}, \mathscr{H}_q(p, \Q)) \Longrightarrow H_{p+q}(E_{\dot}G_{\dot}/G, \Q).
\end{equation}

Here $\mathscr{H}_q(p, \Q)$ denotes the $q$-th homology local system of the fibration $p$ with rational coefficients.

\begin{lemma} 
The action of $\pi_1(B_{\dot}G_{\dot})$ on $H_q(G_{\dot}/G, \Q)$ is trivial for every $q \geq 0$.
\end{lemma}
All the simplicial sets occuring have natural base points represented by $1 \in G_0=G$ or the single element in $B_0G_0$. They will all be denoted by the same symbol $1$ and all constructions depending on base points (like fibres) are made with respect to these without further reference.
\begin{proof}
Recall the operation $H_q(G_{\dot}/G, \Q) \times \pi_1(B_{\dot}G_{\dot}) \to H_q(G_{\dot}/G, \Q)$: Denote by $\Delta[1]$ the simplicial set $\Hom_{\Delta}(\,.\,, [1])$ with vertices $0 := \delta^1$ and $1 := \delta^0 \in \Delta[1]_0$. Any class $[g]$ in $\pi_1(B_{\dot}G_{\dot})$ is represented by a 1-simplex $g \in B_1G_{\dot}$, which corresponds to the unique morphism $g: \Delta[1] \to B_{\dot}G_{\dot}$ sending $\id_{[1]} \in \Delta[1]_1$ to $g$. Consider the diagram
\begin{equation}
\begin{split}
\xymatrix@R+0.5cm{
G_{\dot}/G \ar[r]^{\phi} \ar@{^{(}->}[d]_{(\id, 0)} & E_{\dot}G_{\dot}/G \ar@{->>}[d]^p\\
G_{\dot}/G \times \Delta[1] \ar@{.>}[ur]^h \ar[r]^-{g\circ\mathrm{pr}_2} & B_{\dot}G_{\dot},
}
\end{split}
\end{equation}
where $\phi$ is the inclusion of the fibre of $p$ induced by $\sigma \mapsto (\sigma, \dots, \sigma)$.
The dotted arrow exists by the homotopy lifting property for Kan fibrations \cite[Kap. I, Satz 6.5]{Lamotke}. The restriction of $h$ to $G_{\dot}/G \times 1$ factors through the fibre $G_{\dot}/G$, and hence induces a map $\hat g : G_{\dot}/G \to G_{\dot}/G$. Now the action of $[g] \in \pi_1(B_{\dot}G_{\dot})$ on $H_q(G_{\dot}/G, \Q)$ is given by the homomorphism $H_q(\hat g)$ \cite[Kap. VI, 5.3]{Lamotke}.

We want to make this explicit. Since $G_{\dot}/G$ is obviously connected, the natural map $\pi_1(B_{\dot}G) = G \to \pi_1(B_{\dot}G_{\dot})$ is surjective. Thus we may choose the representative $g$ in $G \subset G_1$. Consider $(1,g^{-1}) \in E_1G_{\dot} = G_1 \times G_1$. Then $p(1,g^{-1}) = g$, $\del_0(1,g^{-1}) = g^{-1}$ and $\del_1(1,g^{-1}) = 1$. Hence $(1,g^{-1})$ corresponds to a morphism $\tilde g: \Delta[1] \to E_{\dot}G_{\dot}$ sending $0$ to $1$ and $1$ to $g^{-1}$ and such that $p \circ \tilde g = g$.

Recall that $E_{\dot}G_{\dot}$ is a simplicial group which operates from the left on $E_{\dot}G_{\dot}/G$ and the projection $p$ is equivariant for this action. Then it makes sense to consider the map $h : G_{\dot}/G \times \Delta[1] \to E_{\dot}G_{\dot}/G$ defined as $h = (\tilde g \circ \mathrm{pr}_2)\cdot (\phi\circ\mathrm{pr}_1)$. We claim that $h$ makes the above diagram commutative. First $p \circ h = (\tilde g \circ \mathrm{pr}_2)\cdot (p\circ\phi\circ\mathrm{pr}_1) = (\tilde g\circ\mathrm{pr}_2)\cdot 1 = p\circ\tilde g \circ\mathrm{pr}_2 = g\circ\mathrm{pr}_2$, i.e. the lower triangle commutes. Next, for $\sigma \in G_p/G$ we have $h(\sigma, 0) = \tilde g(0) \cdot \phi(\sigma) = 1\cdot \phi(\sigma) = \phi(\sigma)$, i.e. the upper triangle commutes. 

On the other hand, $h(\sigma, 1) = \tilde g(1)\cdot \phi(\sigma) = g^{-1}\cdot\phi(\sigma) = (g^{-1}\sigma, \dots, g^{-1}\sigma) = \phi(g^{-1}\sigma)$ and hence the action of the class $[g]$ on $H_q(G_{\dot}/G, \Q)$ is induced by the map $\hat g : G_{\dot}/G \to G_{\dot}/G, \sigma \mapsto g^{-1}\sigma$. We want to show that this map induces the identity map on homology.

Recall that $G_{r,\dot} = \GL_r(R_{\dot}^{\dag})$ and $\dirlim_r G_{r,\dot} = G_{\dot}$ and hence $H_*(G_{\dot}/G, \Q) = \dirlim_r H_*(G_{r,\dot}/\GL_r(R), \Q)$. For $r$ big enough we have $g \in \GL_r(R)$ and it clearly suffices to show that $\hat g: G_{r,\dot}/\GL_r(R) \to G_{2r,\dot}/\GL_{2r}(R)$, $\sigma \mapsto \left(\begin{smallmatrix} g^{-1}\sigma & \\ &1 \end{smallmatrix}\right)$, is homotopic to the identity and hence induces the identity on homology. But since $\left(\begin{smallmatrix} 1&\\&g\end{smallmatrix}\right)$ is an element of $\GL_{2r}(R)$, this last map is the same as the left multiplication by $\left(\begin{smallmatrix} g^{-1} &\\ & g\end{smallmatrix}\right)$. By the Whitehead lemma (see e.g. \cite[(1.9)]{Berrick}) this matrix is a product of elementary matrices. An elementary matrix is a matrix of the form $e_{ij}(a)$, $i \not= j$, with $1$'s on the diagonal and $a\in R$ in the $(i,j)$-slot. Clearly, every elementary matrix is homotopic to the identity matrix, more precisely $e_{ij}(ax_1) \in \GL_{2r}(R_1^{\dag}) = G_{2r,1}$ satisfies $\del_0e_{ij}(ax_1) = e_{ij}(a), \del_1e_{ij}(ax_1)=1$. It follows, that there exists a matrix $H \in G_{2r,1}$ such that $\del_0H = \left(\begin{smallmatrix} g^{-1} & \\&g\end{smallmatrix}\right)$, $\del_1H = 1$. Again, $H$ corresponds to a morphism $H : \Delta[1] \to G_{2r,\dot}$ such that $H(0) = 1, H(1) = \left(\begin{smallmatrix} g^{-1}&\\&g\end{smallmatrix}\right)$. The required homotopy $G_{r,\dot}/\GL_r(R) \times \Delta[1] \to G_{2r,\dot}/\GL_{2r}(R)$ is now given by $(\sigma, \tau) \mapsto H(\tau)\cdot \sigma$.
\end{proof}

\begin{cor}
In the spectral sequence \eqref{eq:SerreSS}
\[
E^2_{p,q} = \begin{cases} H_q(G_{\dot}/G, \Q), & \text{ if } p=0,\\ 0 & \text{ else.} \end{cases}
\]
\end{cor}
\begin{proof}
Since $\pi_1(B_{\dot}G_{\dot})$ acts trivially on $H_q(G_{\dot}/G, \Q)$ and we are working with rational coefficients, we have an isomorphism 
\[
E^2_{p,q}=H_p(B_{\dot}G_{\dot}, \mathscr{H}_q(p,\Q)) \cong H_p(B_{\dot}G_{\dot},\Q) \otimes_{\Q} H_q(G_{\dot}/G,\Q)
\]
\cite[Kap. VI, 8.1]{Lamotke}. 

Recall that $k$ denotes the residue field of $R$, and define the simplicial ring $k_{\dot}$ by $k_p = k[x_0, \dots, x_p]/(\sum_i x_i -1)$ with the usual structure maps. Then $\pi_*(B_{\dot}\GL(k_{\dot}))$ is the Karoubi-Villamayor $K$-theory of $k$. The geometric realization of the natural map $B_{\dot}\GL(k) \to B_{\dot}\GL(k_{\dot})$ factors through $|B_{\dot}\GL(k)|^+$, which gives the isomorphism between Quillen's $K_*(k)$ and the Karoubi-Villamayor $K$-groups of $k$.

With proposition \ref{prop:TopAlgK} it follows, that we have weak equivalences $|B_{\dot}G_{\dot}| = |B_{\dot}\GL(R_{\dot}^{\dag})| \xrightarrow{\sim} |B_{\dot}\GL(k_{\dot})| \xleftarrow{\sim} |B_{\dot}\GL(k)|^+$. Hence we get isomorphisms $H_*(B_{\dot}G_{\dot}, \Q) \cong H_*(B_{\dot}\GL(k), \Q)= \dirlim_r H_*(B_{\dot}\GL_r(k), \Q)$. But since $H_*(B_{\dot}\GL_r(k),\Q)$ is just group homology of the finite group $\GL_r(k)$ with rational coefficients, it vanishes in positive degrees and equals $\Q$ in degree $0$. Now the claim follows.
\end{proof}

\begin{proof}[Proof of the theorem]
It follows from the corollary, that the edge morphism $E^2_{0,q} = H_q(G_{\dot}/G, \Q) \to H_q(E_{\dot}G_{\dot}/G, \Q)$ is an isomorphism. By \cite[Kap. VI, 6.7 b)]{Lamotke}, this is just the homomorphism induced by the inclusion $G_{\dot}/G \to E_{\dot}G_{\dot}/G$. Since the inclusion $B_{\dot}G \cong E_{\dot}G/G \hookrightarrow E_{\dot}G_{\dot}/G$ is a homotopy equivalence and diagram \eqref{diag:hofib} is homotopy commutative, it follows, that $\rho : G_{\dot}/G \to B_{\dot}G$ also induces an isomorphism in rational homology.
\end{proof}

Define $U_r(R) := \ker(\GL_r(R) \to \GL_r(k)) = 1 + \pi\Mat_r(R)$ and $U(R) = \dirlim_r U_r(R) = \ker(\GL(R) \to \GL(k))$.
\begin{lemma}\label{lem:ConstructionNu}
There is a natural map of simplicial sets $\nu: B_{\dot}U_r(R) \to G_{r,\dot}/\GL_r(R)$, fitting in a commutative diagram 
\[\xymatrix{
& B_{\dot}U_r(R) \ar[dl]_{\nu} \ar@{^{(}->}[d]^{\mathrm{incl.}} \\
G_{r,\dot}/\GL_r(R) \ar[r]^{\rho} & B_{\dot}\GL_r(R).
}
\]
Explicitely $\nu$ is given in degree $p$ by
\[
\ul g = (g_1, \dots, g_p) \mapsto \nu(\ul g) = \sum_{i=0}^p x_ig_{i+1}\cdots g_p.
\]
Going to the limit $r \to \infty$ we get a map $B_{\dot}U(R) \to G_{\dot}/\GL(R)$, that induces a surjection 
\[
H_*(B_{\dot}U(R), \Q) \twoheadrightarrow H_*(G_{\dot}/\GL(R), \Q).
\]
\end{lemma}
\begin{proof}
First of all we have to show, that the above formula for $\nu(\ul g)$ really defines an element in $\GL_r(R^{\dag}_p)$.
Thus, take $\ul g = (g_1, \dots, g_p) \in B_pU_r(R)$. 
Write $h_i := 1 - g_{i+1}\cdots g_p$. Since $g_i \in U_r(R)$ for all $i$, we have $h_i \in \pi\Mat_r(R)$. 
Define $h:=\sum_{i=0}^p x_ih_i \in \Mat_r(R[x_0, \dots, x_p])$ and denote its image in $\Mat_r(R^{\dag}_p)$ by the same letter. Then $\nu(\ul g) = \sum_{i=0}^p x_i(1- h_i) = 1 - h$ in $\Mat_r(R_p^{\dag})$.

Choose $1 < \rho < |\pi|^{-1}$ and consider the Banach algebra $T_{p+1}(\rho)$ in the variables $x_0, \dots, x_p$ with the $\rho$-norm $|\,.\,|_{\rho}$ (see section \ref{sec:PrelimAlgebras}). Define the norm $|\,.\,|_{\rho}$ on $\Mat_r(T_{p+1}(\rho))$ to be the maximum of the $\rho$-norms of the entries. Then $\Mat_r(T_{p+1}(\rho))$ obviously becomes a Banach algebra as well and $|h|_{\rho} \leq \max_i |x_ih_i|_{\rho} \leq \rho\cdot |\pi| < 1$ by definition of the $\rho$-norm and since $h_i\in \pi\Mat_r(R)$.
Hence $\sum_{k=0}^{\infty} h^k$ converges in $\Mat_r(T_{p+1}(\rho))\subset \Mat_r(K\<x_0, \dots, x_p\>^{\dag})$. Obviously, all the coefficients lie in $R$, hence $\sum_{k=0}^{\infty} h^k$ defines in fact an element in $\Mat_r(R\<x_0, \dots, x_p\>^{\dag})$. Its image in $\Mat_r(R_p^{\dag})$ clearly gives an inverse of $\nu(\ul g) = 1 - h$.

It is easy to check, that $\nu$ is a morphism of simplicial sets. For example, $\nu(\del_p(g_1, \dots, g_p)) = \sum_{i=0}^{p-1} x_ig_{i+1}\cdots g_{p-1} = \sum_{i=0}^{p-1} x_ig_{i+1}\cdots g_p = \del_p(\nu(g_1, \dots, g_p))$ in $\GL_r(R^{\dag}_{p-1})/\GL_r(R)$.

Recall, that  $\rho$ is given by $\sigma \mapsto (\sigma(e_0)\sigma(e_1)^{-1}, \dots, \sigma(e_{p-1})\sigma(e_p)^{-1})$. Clearly $\nu(\ul g)(e_{i-1})\nu(\ul g)(e_i)^{-1} = g_i$ and hence $\rho \circ \nu : B_{\dot}U_r(R) \to B_{\dot}\GL_r(R)$ is just the inclusion.

Since $k$ is finite, $U_r(R)$ has finite index in $\GL_r(R)$. Since $H_*(B_{\dot}U_r(R), \Q)$ is just group homology with rational coefficients, $H_*(B_{\dot}U_r(R), \Q) \to H_*(B_{\dot}\GL_r(R), \Q)$ is surjective by the usual restriction-corestriction argument. Going to the limit $r \to \infty$, $H_*(B_{\dot}U(R), \Q) \to H_*(B_{\dot}\GL(R), \Q)$ is surjective. Since $H_*(G_{\dot}/\GL(R), \Q) \to H_*(B_{\dot}\GL(R), \Q)$ is an isomorphism by theorem \ref{thm:RationalHomolIsom}, the claim follows.
\end{proof}

\subsection{The $p$-adic Borel regulator}

Here we recall the construction of the $p$-adic Borel regulator and the explicit description of the Lazard isomorphism. 

As before, $K$ denotes a finite extension of $\Q_p$ with ring of integers $R$ and uniformizer $\pi$. Recall, that $U_r(R) = 1 + \pi \Mat_r(R) \subset \GL_r(R)$. Denote by $\gl_r$ the $K$-Lie algebra of $\GL_r(R)$, viewed as a locally $K$-analytic Lie group, and by $\O^{\mathrm{la}}(X)$ the ring of locally analytic functions on a locally $K$-analytic manifold $X$. 
We denote by $H^*_{\mathrm{la}}(\GL_r(R), K)$ the \emph{locally analytic group cohomology} defined as the cohomology of the complex associated with the cosimplicial $K$-vector space $[p]\mapsto \O^{\mathrm{la}}(B_p\GL_r(R)) = \O^{\mathrm{la}}(\GL_r(R)^{\times p})$. Recall, that the Lie algebra cohomology $H^*(\gl_r, K)$ is the cohomology of the complex $\bigwedge^*\gl_r^{\vee}$ with differential induced by the Lie bracket (see e.g. \cite[Corollary 7.7.3]{Weibel}), where $\gl_r^{\vee}$ denotes the $K$-dual of $\gl_r$.

Huber and Kings prove the following version of Lazard's theorem:
\begin{thm}[Lazard, Huber--Kings]
There are isomorphisms 
\[
H^k_{\mathrm{la}}(\GL_r(R), K) \xrightarrow{\cong} H^k_{\mathrm{la}}(U_r(R), K) \xrightarrow{\cong} H^k(\gl_r, K).
\]
On the level of cochains the map to Lie algebra cohomology is induced by the map 
\[
\Phi: \O^{\mathrm{la}}(\GL_r(R)^{\times k}) \to \bigwedge^k \gl_r^{\vee},
\]
which is given on topological generators by $f_1\otimes \dots \otimes f_k \mapsto df_1(1) \wedge \dots \wedge df_k(1)$,
where $df(1)$ is the differential of $f$ at the unit element $1\in \GL_r(R)$.
\end{thm}
\begin{proof}
This is proven in \cite[Theorems 1.2.1. and 4.7.1]{HK}. See also \cite[Theorem 4.3.1]{HKN}.
\end{proof}

\begin{dfn}[\cite{HK} Definitions 0.4.5 and 1.2.3]\label{def:PrimitiveElem}
For $n \leq r$ the (primitive) element $p_n = p_{n,r} \in H^{2n-1}(\gl_r, K)$ is the class represented by the cocycle
\[
X_1 \wedge \dots\wedge X_{2n-1} \mapsto \frac{((n-1)!)^2}{(2n-1)!} \sum_{\sigma\in\Sy_{2n-1}} \sgn(\sigma)\Tr(X_{\sigma(1)} \cdots X_{\sigma(2n-1)}).
\]
Here $\Sy_{2n-1}$ denotes the symmetric group on $2n-1$ elements.
Define $b_{n,r} \in H^{2n-1}(\GL_r(R), K)$ to be the image of $p_{n,r}$ under the composition
\[
H^{2n-1}(\gl_r, K) \xleftarrow{\cong} H^{2n-1}_{\mathrm{la}}(\GL_r(R), K) \to H^{2n-1}(B_{\dot}\GL_r(R), K),
\]
where the right hand map is the canonical map from locally analytic to discrete group cohomology. Obviously, the $b_{n,r}$ are compatible for different $r$.

The \emph{$p$-adic Borel regulator} is the composition
\[
r_p: K_{2n-1}(R) \xrightarrow{\mathrm{Hur.}} H_{2n-1}(\GL(R), \Q) = \dirlim_r H_{2n-1}(\GL_r(R), \Q) \xrightarrow{\dirlim_r b_{n,r}} K.
\]
\end{dfn}
\smallskip

For later use, we record the following alternative description of the map $\Phi$ in the theorem above.
Consider $\GL_r(R)$ as a $K$-Lie group and let $\exp$ be the exponential map of $\GL_r(R)$ defined on a neighbourhood of zero in $\gl_r$. For a
locally analytic function $f \in \mathcal{O}^{\mathrm{la}}(\GL_r(R)^{\times k})$ we define $\Delta f \in \bigwedge^k \gl_r^{\vee}$ by 
\begin{multline*}
\Delta f (X_1, \dots, X_k) := \\
\sum_{\sigma\in\Sy_k} \sgn(\sigma) \frac{d^k}{dt_1 \dots dt_k} f(\exp(t_1X_{\sigma 1}), \dots, \exp(t_k X_{\sigma k}))\Big|_{t_1 = \dots = t_n = 0}.
\end{multline*}
If $f$ is of the special form $f = f_1 \otimes \dots \otimes f_k$, one has 
\begin{multline*}
\frac{d}{dt_i}f(\exp(t_1X_{\sigma 1}), \dots, \exp(t_k X_{\sigma k}))\Big|_{t_i=0} =\\
=f_1(\exp(t_1X_{\sigma 1}))\cdots df_i(1)(X_{\sigma i})  \cdots f_k(\exp(t_kX_{\sigma k}))
\end{multline*}
and therefore
\begin{eqnarray*}
\Delta f (X_1, \dots, X_k) &=& \sum_{\sigma\in\Sy_k} \sgn(\sigma) df_1(1)(X_{\sigma 1})\cdots df_k(1)(X_{\sigma k})\\
&=& df_1(1)\wedge \dots \wedge df_k(1) (X_1, \dots, X_k) = \Phi(f) (X_1, \dots, X_k).
\end{eqnarray*}

The vector space $\O^{\mathrm{la}}(\GL_r(R)^k)$ carries a natural locally convex topology \cite[\S 12]{Schneider}. Using proposition 12.4 of {\it loc. cit.}, it is easy to see, that both, $\Phi$ and $\Delta$, are continuous for this topology.
Since moreover the functions of the form $f_1 \otimes \cdots \otimes f_k$ are topological generators of 
$\mathcal{O}^{\mathrm{la}}(\GL_r(R)^{\times k})$, we get:
\begin{cor}\label{cor:Lazard}
The Lazard isomorphism 
$H^k_{\mathrm{la}}(\GL_r(R), K) \xrightarrow{\isom} H^k(\gl_r, K)$
is induced by $\Delta: \mathcal{O}^{\mathrm{la}}(\GL_r(R)^{\times k}) \to \bigwedge^k \gl_r^{\vee}$.
\end{cor}

The same description applies for $U_r(R)$ instead of $\GL_r(R)$.

\subsection{Local analyticity of the relative Chern character}

Recall that the relative Chern character $\Ch^{\rel}_{n,2n-1} : K^{\rel}_{2n-1}(\Spec(R)) \to H^0_{\rel}(\Spec(R)/R, n) = K$ is  determined by a compatible family of homomorphisms $H_{2n-1}(\mathscr F_r, \Z) \to K$ and that we have a natural homotopy equivalence $\eta_r: G_{r,\dot}/\GL_r(R) \to \mathscr F_r$. Similar as in the complex situation we have:
\begin{prop}\label{prop:ExplicitCocyclePAdicRelReg}
The composition
\[
H_{2n-1}(G_{r,\dot}/\GL_r(R), \Z) \xrightarrow{\cong} H_{2n-1}(\mathscr F_r, \Z) \xrightarrow{\widetilde\Ch_n^{\rel}(T_r, E_r, \alpha_r)} K
\]
is given by the cocycle 
\[
\sigma \mapsto (-1)^{n-1}\frac{(n-1)!}{(2n-1)!} \Tr \int_{\Delta^{2n-1}} (d\sigma\cdot\sigma^{-1})^{2n-1}.
\]
\end{prop}
\begin{proof}
Write $X = \Spec(R)$.
For the proof note the following: The morphism $H_{2n-1}(\mathscr F_r, \Z) \to H^0_{\rel}(X/R, n)$ is induced by the class $\widetilde\Ch_n^{\rel}(T_r,E_r,\alpha_r/R) \in H^{2n-1}_{\rel}(X\otimes\mathscr F_r/R,n)$. By remark \ref{rem:RelCohomProper} this group is isomorphic to $H^{2n-1}_{\rel}((\widehat X)_K\otimes \mathscr F_r, n)$ and by proposition \ref{prop:ComparisonRelClassesRSchemes} the class $\widetilde\Ch_n^{\rel}(T_r,E_r,\alpha_r/R)$ corresponds under this isomorphism to the class $\Ch_n^{\rel}(T_r, (\widehat E_r)_K, \alpha_r)$, where $(\widehat E_r)_K$ is the bundle induced by $E_r$ on $(\widehat X)_K\otimes \mathscr F_r$. Hence we may work with this class constructed via Chern-Weil theory and there the same (up to a sign) computation as in the complex case (proposition \ref{prop:ExplicitCocycle}) applies.
\end{proof}

Next recall, that we constructed maps of simplicial sets $\nu : B_{\dot}U_r(R) \to G_{r,\dot}/\GL_r(R)$, which induce a surjection $H_*(B_{\dot}U(R), \Q) \to H_*(G_{\dot}/\GL(R), \Q)$. The composition
\[
H_{2n-1}(B_{\dot}U_r(R), \Q) \xrightarrow{\nu} H_*(G_{r,\dot}/\GL_r(R), \Q) \xrightarrow{\Ch_{n,2n-1}^{\rel}} K
\]
is then given by the cocycle
\begin{align}
& U_r(R)^{\times (2n-1)} \to K,\notag \\
& \ul g = (g_1, \dots, g_{2n-1}) \mapsto (-1)^{n-1}\frac{(n-1)!}{(2n-1)!}\Tr\int_{\Delta^{2n-1}} (d\nu(\ul g)\cdot\nu(\ul g)^{-1})^{2n-1} \label{eq:cocyclePAd}
\end{align}
where $\nu(\ul g) = \sum_{i=0}^{2n-1}x_ig_{i+1}\cdots g_{2n-1}$.

We want to show, that this cocycle is locally analytic, hence may be compared with the Lie algebra cocycle defining the $p$-adic Borel regulator using the Lazard isomorphism.

Let $\rho >1$ and consider the Banach algebra $T_{n+1}(\rho) = K\<\rho^{-1}x_0, \dots, \rho^{-1}x_n\>$. 
Write $\Omega^n(\Delta^n)_{\rho} := K\<\rho^{-1}\ul x\>/(\sum_i x_i-1) \otimes_K \bigwedge^n_K\frac{\bigoplus_{i=0}^nKdx_i}{\sum_i dx_i}$ (cf. remark \ref{rem:IntegrationStandardSimplex}). This is a $K$-Banach space and we have  a natural map $\Omega^n(\Delta^n)_{\rho} \to \Omega^n(\Delta^n)$, which may be composed with the integration map $\int_{\Delta^n}\colon \Omega^n(\Delta^n) \to K$.

\begin{lemma}\label{lem:IntegrationLocAn}
Let $M$ be a locally $K$-analytic manifold and $F : M \to \Omega^n(\Delta^n)_{\rho}$  a locally analytic function. Then 
\[
M \ni u \mapsto \int_{\Delta^n} F(u) \in K
\]
is also locally analytic and will be denoted by $\int_{\Delta^n} F$.

If $dF(u): T_uM\to \Omega^n(\Delta^n)_{\rho}$ denotes the differential of $F$ at $u \in M$ and $v$ is a tangent vector to $M$ at $u$, we have
\[
d\left(\int_{\Delta^n} F\right)(u)(v) = \int_{\Delta^n}\left(dF(u)(v)\right).
\]
\end{lemma}
Note, that the ``$d$'' is the differential on $M$ and has nothing to do with the differential on $\Omega^*(\Delta^n)$.
\begin{proof}
As noted in remark \ref{rem:IntegrationStandardSimplex} the composition $\Omega^n(\Delta^n)_{\rho} \to K, \omega \mapsto \int_{\Delta^n} \omega$ is continuous. Hence $u \mapsto \int_{\Delta^n} F(u)$ being the composition of a bounded linear map with a locally analytic map is locally analytic as well. The second assertion is simply the chain rule. 
\end{proof}

If $F : M \to \Mat_r(\Omega^n(\Delta^n)_{\rho})$ is a locally analytic function with values in the Banach space of $r\times r$-matrices with coefficients in $\Omega^n(\Delta^n)_{\rho}$, we get the locally analytic function $\int_{\Delta^n} F$ with values in $\Mat_r(K)$ applying the integral component-wise.

\begin{lemma}\label{lem:CocycleLocAn}
The cocycle \eqref{eq:cocyclePAd} is locally analytic.
\end{lemma}
\begin{proof}
We introduce some more notation. For any $K$-Banach space $(V, \|\,.\,\|)$ and $\epsilon > 0$ we denote by $\mathcal F_{\epsilon}(K^m, V)$ the $K$-Banach space of $\epsilon$-convergent power series in $m$ variables with coefficients in $V$, i.e. formal power series $\sum_{\nu} v_{\nu}\ul x^{\nu}$, such that $\|v_{\nu}\|\epsilon^{|\nu|} \xrightarrow{|\nu|\to\infty} 0$, equipped with the norm $\|\sum_{\nu} v_{\nu}\ul x^{\nu}\|_{\epsilon} = \max_{\nu} \|v\|\cdot \epsilon^{|\nu|}$. If $A$ is a $K$-Banach algebra, then $\mathcal F_{\epsilon}(K^m, A)$, equipped with the usual multiplication of power series, becomes a $K$-Banach algebra as well. 

We show, that $\nu^{-1}: U_r(R)^{\times (2n-1)} \to \GL_r(R^{\dag}_{2n-1}) \subset \Mat_r(K^{\dag}_{2n-1})$, $\ul g \mapsto (\sum_i x_ig_{i+1}\cdots g_{2n-1})^{-1}$ factors through a locally analytic map $U_r(R)^{\times (2n-1)} \to \Mat_r(T_{2n}(\rho))$ for any $1<\rho<|\pi|^{-1}$. 

Thus fix $1 < \rho < |\pi|^{-1}$. Consider the locally analytic function 
\begin{align*}
h_i\colon&  U_r(R)^{\times(2n-1)} \to \pi\Mat_r(R) \subset \Mat_r(K),\\
&(g_1, \dots, g_{2n-1}) \mapsto 1 - g_{i+1}\cdots g_{2n-1}.
\end{align*}
Then $h := \sum_{i=0}^{2n-1} x_ih_i\colon U_r(R)^{\times(2n-1)}\to \Mat_r(T_{2n}(\rho))$ is also locally analytic and $\nu(\ul g)^{-1}$ is the image of $\sum_{k=0}^{\infty} h(\ul g)^k \in \Mat_r(T_{2n}(\rho))$ in $\Mat_r(K^{\dag}_{2n-1})$ (cf. the proof of lemma \ref{lem:ConstructionNu}). Hence we have to show, that $\sum_{k=0}^{\infty} h^k: U_r(R)^{\times (2n-1)} \to \Mat_r(T_{2n}(\rho))$ is locally analytic.
We have the chart $U_r(R)^{\times (2n-1)} \xrightarrow{\psi} \pi\Mat_r(R)^{\times (2n-1)} \subset \Mat_r(K)^{\times (2n-1)} \cong K^{r^2(2n-1)}$, whose inverse is given by $(M_1, \dots, M_{2n-1}) \mapsto (1 + M_1, \dots, 1+M_{2n-1})$. Then $h\circ \psi^{-1}$ is given by 
\[
(M_1, \dots, M_{2n-1})\mapsto \sum_i x_i\left(1-(1+M_{i+1})\cdots(1+M_{2n-1})\right).
\]
This map is clearly given by a power series $F$ (in fact a polynomial) in $\mathcal F_{|\pi|}\left(K^{r^2(2n-1)}, T_{2n}(\rho)\right)$ with $\|F\|_{|\pi|} \leq \rho\cdot |\pi|<1$ [Note that here the $x_i$'s are the coefficients, and the $M_i$'s are the variables. Since $1 - (1+M_{i+1})\cdots(1+M_{2n-1}))$ has no constant term and only integral coefficients, we have $\| 1 - (1+M_{i+1})\cdots(1+M_{2n-1}))\|_{|\pi|} \leq |\pi|$. On the other hand $|x_i|_{\rho} = \rho$.].
Consequently $\sum_{k=0}^{\infty} F^k$ converges in $\mathcal F_{|\pi|}\left(K^{r^2(2n-1)}, T_{2n}(\rho)\right)$ to a power series representing $(\sum_{k=0}^{\infty} h^k)\circ \psi^{-1}$, i.e. $\sum_{k=0}^{\infty} h^k$ is locally analytic.

Since sums and products of locally analytic functions with values in $T_{2n}(\rho)$ are again locally analytic, it follows, that 
$(d\nu \cdot \nu^{-1})^{2n-1} : U_r(R)^{\times(2n-1)} \to \Mat_r(\Omega^{2n-1}(\Delta^{2n-1})_{\rho})$ is locally analytic hence the cocycle \eqref{eq:cocyclePAd} is locally analytic by lemma \ref{lem:IntegrationLocAn}.
\end{proof}

\subsection{Comparison of the $p$-adic Borel regulator and the relative Chern character}

According to lemma \ref{lem:CocycleLocAn} the cocycle \eqref{eq:cocyclePAd} defines a class in $H^{2n-1}_{\mathrm{la}}(U_r(R), K)$ and we have:
\begin{thm}
The class of the cocycle \eqref{eq:cocyclePAd} is mapped to $\frac{(-1)^n}{(n-1)!}p_n$ by the Lazard isomorphism $H^{2n-1}_{\mathrm{la}}(U_r(R), K) \xrightarrow[\cong]{\Phi=\Delta} H^{2n-1}(\gl_r, K)$.
\end{thm}
Here $p_n$ denotes the primitive element of definition \ref{def:PrimitiveElem}.
\begin{proof}
Denote the cocycle \eqref{eq:cocyclePAd} by $f$. 
We show that $\Delta(f) = (-1)^n\frac{1}{(n-1)!} p_n$. 
Write $\del_i$ instead of $\frac{d}{dt_i}$. We have
\begin{multline*}
\Delta(f)(X_1, \dots, X_{2n-1})=
(-1)^{n-1}\frac{(n-1)!}{(2n-1)!}\sum_{\sigma\in\Sy_{2n-1}} \sgn(\sigma) \del_1\dots\del_{2n-1}\Big|_{t_1 = \dots = 0}\\
 \Tr \int_{\Delta^{2n-1}} (d\nu \cdot \nu^{-1})^{2n-1} (\exp(t_1X_{\sigma 1}), \dots, \exp(t_{2n-1} X_{\sigma(2n-1)})).
\end{multline*}
By lemma \ref{lem:IntegrationLocAn} we may interchange differentiation and integration.  Let us first consider the $\sigma = 1$ summand. 
Write
\begin{align*}
\omega& := \sum_{i=0}^{2n-1}dx_i \exp(t_{i+1}X_{i+1})\cdots\exp(t_{2n-1}X_{2n-1}),\\
\omega'& := \sum_{i=0}^{2n-1}x_i \exp(t_{i+1}X_{i+1})\cdots\exp(t_{2n-1}X_{2n-1}).
\end{align*}
Then 
\[
(d\nu \cdot \nu^{-1})^{2n-1} (\exp(t_1X_1), \dots, \exp(t_{2n-1} X_{2n-1})) = (\omega\cdot\omega'^{-1})^{2n-1}.
\]
Note, that $\omega|_{t_1=\dots=t_{2n-1}=0} = \sum_{i=0}^{2n-1} dx_i = 0$. It follows, that when we calculate $\del_1\dots\del_{2n-1} (\omega\omega'^{-1})^{2n-1}$ using the Leibniz rule repeatedly and then set all the $t_i$ equal to zero, we get
\begin{multline*}
\del_1\dots\del_{2n-1} (\omega\omega'^{-1})^{2n-1}\Big|_{t_1=\dots=t_{2n-1}=0}=\\
\sum_{\tau\in\Sy_{2n-1}} \del_{\tau(1)}(\omega\omega'^{-1}) \cdots \del_{\tau(2n-1)}(\omega\omega'^{-1})\Big|_{t_1=\dots=t_{2n-1}=0}.
\end{multline*}
On the other hand, using $\omega'\big|_{t_1=\dots=0}= \sum_{i=0}^{2n-1} x_i = 1$ we get
\begin{align*}
\del_j \left(\omega\omega'^{-1}\right)\big|_{t_1=\dots=0} &= (\del_j\omega)\omega'^{-1}\big|_{t_1=\dots=0} + \omega(\del_j\omega'^{-1})\big|_{t_1=\dots=0}= \\
&=(\del_j\omega)\big|_{t_1=\dots=0}= \sum_{i=0}^{j-1} dx_i \cdot X_j.
\end{align*}
Alltogether we obtain
\begin{multline*}
\del_1\dots\del_{2n-1} (\omega\omega'^{-1})^{2n-1}\Big|_{t_1=\dots=t_{2n-1}=0}=\\
= \sum_{\tau \in\Sy_{2n-1}} \left(\sum_{i=0}^{\tau(1)-1} dx_i\cdot X_{\tau(1)}\right) \cdots \left(\sum_{i=0}^{\tau(2n-1)-1} dx_i\cdot X_{\tau(2n-1)}\right)\\
= \sum_{\tau \in\Sy_{2n-1}} X_{\tau(1)}\cdots X_{\tau(2n-1)} \left(\sum_{i=0}^{\tau(1)-1} dx_i\right) \cdots \left(\sum_{i=0}^{\tau(2n-1)-1} dx_i\right) \\
= \sum_{\tau \in\Sy_{2n-1}} \sgn(\tau) X_{\tau(1)}\cdots X_{\tau(2n-1)} dx_0 dx_1\dots dx_{2n-2}. 
\end{multline*}
It follows that 
\begin{multline*}
\sum_{\sigma\in\Sy_{2n-1}} \sgn(\sigma)\del_1\dots\del_{2n-1} \\
(d\nu \cdot \nu^{-1})^{2n-1} (\exp(t_1X_{\sigma 1}), \dots, \exp(t_{2n-1} X_{\sigma(2n-1)}))\Big|_{t_1=\dots=t_{2n-1}=0} \\ 
=\sum_{\sigma\in\Sy_{2n-1}}\sgn(\sigma)\sum_{\tau \in\Sy_{2n-1}} \sgn(\tau) X_{\sigma\tau(1)}\cdots X_{\sigma\tau(2n-1)} dx_0 dx_1\dots dx_{2n-2}\\
= (2n-1)! \sum_{\sigma\in\Sy_{2n-1}}\sgn(\sigma) X_{\sigma(1)}\cdots X_{\sigma(2n-1)} dx_0 dx_1\dots dx_{2n-2}.
\end{multline*}
Because 
\begin{align*}
\int_{\Delta^{2n-1}}dx_0\dots dx_{2n-2} &= -\int_{\Delta^{2n-1}} dx_{2n-1}dx_1\dots dx_{2n-2} \\
&= -\int_{\Delta^{2n-1}} dx_1\dots dx_{2n-1} = -\frac{1}{(2n-1)!}
\end{align*}
by a direct computation,
we finally obtain
\[
\Delta(f)(X_1, \dots, X_{2n-1}) = (-1)^n\frac{(n-1)!}{(2n-1)!} \sum_{\sigma\in\Sy_{2n-1}} \sgn(\sigma) \Tr(X_{\sigma 1}\cdots X_{\sigma(2n-1)}),
\]
that is $\Delta(f) = \frac{(-1)^n}{(n-1)!}p_n$.
\end{proof}

\begin{cor}\label{cor:ComparisonPAdBorel}
The diagram
\[\xymatrix{
K_{2n-1}^{\rel}(\Spec(R)) \ar[rr]\ar[dr]_{\Ch^{\rel}_{n,2n-1}} && K_{2n-1}(\Spec(R)) \ar[ld]^{\frac{(-1)^n}{(n-1)!}r_p}\\
& K
}
\]
commutes.
\end{cor}
\begin{rems}
i) That there should be a direct connection between the $p$-adic Borel regulator and the relative Chern character was mentioned by Karoubi.

ii) A formula for the relative Chern character similar to \eqref{eq:cocyclePAd} has also been obtained by Hamida \cite{HamidaCR}.

iii) In connection with the work of Huber and Kings \cite{HK} this result implies the comparison of the syntomic Chern character with the relative Chern character in the case $X= \Spec(R)$. Another strategy for the proof of a general comparison result was given by Besser in his talk \cite{BesserTalk}.

iv) Note, that the horizontal map in the above diagram has finite kernel and cokernel. This follows from the long exact sequence connecting relative, algebraic and topological $K$-theory together with the fact, that $K_{\top}^{-(2n-1)}(R) \cong K_{2n-1}(k)$ for $n\geq 1$, hence is finite by Quillen's computation of the algebraic $K$-theory of finite fields.

v) The factor $\frac{(-1)^n}{(n-1)!}$ shows up, since the construction of the relative Chern character uses Chern character classes, whereas  the construction of the $p$-adic Borel regulator in \cite{HK} uses Chern classes. One could define a renormalized version of the $p$-adic Borel regulator using Chern character classes instead and then the factor would disappear (cf. the different normalizations of the primitive element $p_n$ in \cite[Definition 0.4.5]{HK} and \cite[Example 5.37]{Burgos}).

\end{rems}

\begin{proof}[Proof of the corollary]
By construction of the two regulators it suffices to show, that the diagram
\[\xymatrix{
H_{2n-1}(G_{\dot}/\GL(R), \Q) \ar[dr]_{\Ch_{n,2n-1}^{\rel}} \ar[rr]^{\cong} && H_{2n-1}(B_{\dot}\GL(R), \Q) \ar[dl]^{\frac{(-1)^n}{(n-1)!}r_p}\\
&K
}
\]
commutes. By lemma \ref{lem:ConstructionNu} we have a surjection $H_{2n-1}(B_{\dot}U(R), \Q) \twoheadrightarrow H_{2n-1}(G_{\dot}/\GL(R), \Q)$ and it follows from the last theorem and the definition of the $p$-adic Borel regulator, that the two possible compositions agree on $H_{2n-1}(B_{\dot}U_r(R), \Q)$ for any $r$, hence they agree on $H_{2n-1}(B_{\dot}U(R), \Q)$ and the claim follows.
\end{proof}

\appendix

\chapter{}
\section{Some homological algebra}
\label{app:HomolAlg}
Let $A, B, C$ be three (cohomological) complexes in an abelian category. Given morphisms $f: A \to C$ and $g : B \to C$, we define the \emph{quasi-pullback} $A \widetilde\times_{C} B$ to be the complex
\[
\Cone(A \oplus B \xrightarrow{f-g} C)[-1].
\]
We have the short exact sequence
\[
0 \to C[-1] \to A \widetilde\times_{C} B \xrightarrow{p_A\oplus p_B} A \oplus B \to 0.
\]
\begin{lemma}
The diagram 
\[\xymatrix{
A \widetilde\times_{C} B \ar[d]^{p_B}\ar[r]^-{p_A} & A \ar[d]^f\\
B \ar[r]^g & C
}
\]
commutes up to canonical homotopy. If $f$ is a quasiisomorphism, so is $p_B$.
\end{lemma}
\begin{proof}
The homotopy $h: (A \widetilde\times_{C} B)^n \to C^{n-1}$ is given explicitely by $(a, b, c) \mapsto c$.
The short exact sequence above yields the following exact sequence of cohomology groups
\begin{multline*}
H^{i-1}(A) \oplus H^{i-1}(B) \xrightarrow{f-g} H^{i-1}(C) \to H^i(A \widetilde\times_{C} B) \xrightarrow{p_A \oplus p_B}\\ 
\to H^i(A) \oplus H^i(B) \xrightarrow{f-g} H^i(C)
\end{multline*}
and by a little diagram chase it follows, that if $f$ is an isomorphism on cohomology groups, so is $p_B$.
\end{proof}

\begin{lemma}\label{lem:HomotopyOnCones}
Suppose given a  commutative diagram of complexes
\[
\xymatrix{
A \ar[r]^{\alpha}\ar[d]^f & B \ar@<-0.7ex>[d]_{g_0} \ar@<0.7ex>[d]^{g_1}\\
A' \ar[r]^{\alpha'} & B'
}
\]
and a homotopy $h$ between $g_0$ and $g_1$ \emph{under} $A$\footnote{i.e. $h(\alpha(a)) =0 \forall a\in A$}, the two maps $\Cone(A \to B) \rightrightarrows \Cone(A'\to B')$ induced by $g_0$ and $g_1$ respectively are homotopic.
\end{lemma}
\begin{proof}
The induced maps on the cones are given by $(a, b) \mapsto (f(a), g_i(b))$, $i = 0,1$, and a homotopy between them is given by $(a,b) \mapsto (0, h(b))$. In fact, 
$
(dh+hd)(a,b) = d(0, h(b)) + h(-da, db-\alpha(a)) = (0, dh(b)) + (0, h(db) - h(\alpha(a))) = (0, dh(b) + h(db)) = (0, g_0(b) - g_1(b)) = (f(a), g_0(b)) - (f(a), g_1(b)).
$
\end{proof}

\begin{lemma}\label{lem:HomotopyEquivCones}
Let $A, I, J$ be non-negative complexes in a Grothendieck abelian category\footnote{a cocomplete abelian category satisfying AB5) and admitting a generator} and assume that $I$ and $J$ consist of injective objects. Let $A \trivcof I$ be an injective quasiisomorphism and $f: A \to J$ any morphism. Then the dotted arrow in the diagram
\[
\xymatrix{
A \ar@{^{(}->}[d]_{\sim} \ar[r]^f & J\\
I \ar@{.>}[ur]
}
\]
exists and is unique up to homotopy under $A$. 

Assume, that $g_0$ and $g_1$ are two morphisms $I \to J$ making the above diagram commute. The choice of a homotopy between $g_0$ and $g_1$ under $A$ determines a quasiisomorphism $\Cone(g_0) \xrightarrow{\sim} \Cone(g_1)$. This quasiisomorphism does up to homotopy not depend on the chosen homotopy between $g_0$ and $g_1$.
\end{lemma}
\begin{proof}
The first part simply follows from the fact, that the non-negative cochain complexes in a Grothendieck abelian category form a model category, where the weak equivalences are the quasiisomorphisms, the cofibrations are the injections and the fibrations are the morphisms, which are surjective in positive degree, and have degree-wise an injective kernel (cf. \cite[Theorem 2.3.13]{Hovey} for an explicit description in the case of $R$-modules and \cite[Proposition 3.13]{Beke} in the general case), and general facts on model categories.

For the second part, note that with $I$ and $J$ also $\Cone(g_i)$, $i=0,1$, is fibrant (w.r.t. to the aforementioned model structure) and that we have a commutative diagram
\[
\xymatrix{
\Cone(f) \ar@{^{(}->}[d]_{\sim} \ar@{^{(}->}[r]^{\sim} & \Cone(g_1)\\
\Cone(g_0). \ar[ur]
}
\]
Again by general facts on model categories, the diagonal arrow is a homotopy equivalence, which is well defined up to homotopy (under $\Cone(f)$).
\end{proof}

\section{Cohomology on strict simplicial (dagger) spaces}
\label{app:StrictSimplicial}

Recall, that $\Delta^{\mathrm{str}}$ denotes the subcategory of $\Delta$ with the same objects, but with morphisms the strictly increasing maps $[p]\hookrightarrow [q]$, and that a strict (co)simplicial object in a category $\mathscr C$ is a co- resp. contravariant functor $\Delta^{\mathrm{str}} \to \mathscr C$. 

The sheaf theory on  strict simplicial spaces is essentially the same as in the simplicial case (cf. \cite[\S\S 1 and 2]{Fri}, \cite[\S 5]{HodgeIII}).

Let $X_{\dot}$ be a strict simplicial (dagger) space. Let $T(X_{\dot})$ denote the category, whose objects are (admissible) opens $U \subset X_p$ for some $p\geq 0$, and whose morphisms are pairs $(U \to V, \phi)$, where $U \subset X_p, V\subset X_q, \phi:[q]\hookrightarrow [p]$ and $U \to V\subset X_q$ is the restriction of $\phi_X: X_p \to X_q$ to $U$. An \emph{(admissible) covering} of $U\subset X_p$ is a usual (admissible) covering $\{U_i\}_{i\in I}$ of $U$. This defines a Grothendieck topology on $T(X_{\dot})$.

By abuse of language we say ``(abelian) sheaf on $X_{\dot}$'' instead of ``(abelian) sheaf on $T(X_{\dot})$''. Explicitely, a sheaf $\mathscr F$ on $X_{\dot}$ is given by a family $\{\mathscr F_p\}_{p\geq 0}$ of sheaves $\mathscr F_p$ on $X_p$ together with morphisms $\phi_X^*: \phi_X^{-1}\mathscr F_q \to \mathscr F_p$ for any $\phi: [q]\hookrightarrow [p]$, which are compatible in an obvious sense \cite[(5.1.6)]{HodgeIII}.

As for any Grothendieck site, the category of abelian sheaves on $X_{\dot}$ is a Grothendieck abelian category \cite[Theorem 2.1.4]{Artin} and in particular has enough injectives. A sequence $\mathscr F' \to \mathscr F \to \mathscr F''$ of sheaves on $X_{\dot}$ is exact, if and only if each sequence $\mathscr F'_p \to \mathscr F_p \to \mathscr F''_p$ on $X_p$ is exact \cite[proof of proposition 2.2]{Fri}.

Let $\mathscr F$ be a sheaf on $X_{\dot}$. Its global sections are by definition 
\[
\Gamma(X_{\dot}, \mathscr F) := \ker(\Gamma(X_0, \mathscr F_0) \underset{\del_1^*}{\overset{\del_0^*}{\rightrightarrows}} \Gamma(X_1, \mathscr F_1)).
\]
If $\mathscr F$ is an \emph{abelian} sheaf, then $\Gamma(X_{\dot}, \mathscr F) = \Hom_{\mathrm{AbSh}(X_{\dot})}(\Z, \mathscr F)$, where $\Z$ is the constant abelian sheaf $\Z$, as one easily checks, and
\[
H^i(X_{\dot}, \mathscr F) := R^i\Gamma(X_{\dot}, \mathscr F) = \mathrm{Ext}^i_{\mathrm{AbSh}(X_{\dot})}(\Z, \mathscr F).
\]
We have the usual spectral sequence (the proof of \cite[Proposition 2.4]{Fri} also works in the strict case)
\[
E_1^{p,q} = H^q(X_p, \mathscr F_p) \Longrightarrow H^{p+q}(X_{\dot}, \mathscr F).
\]
The recipe \cite[(5.2.7)]{HodgeIII} for the computation of the hypercohomology of a complex of abelian sheaves on $X_{\dot}$ carries over to the strict case. Note, that the Godement resolution Deligne uses, which is not available for dagger spaces, is not needed, since we can clearly take injective resolutions instead. 

If $I$ is an injective sheaf on $X_{\dot}$, each $I_p$ on $X_p$ is also injective \cite[proof of proposition 2.4]{Fri}.

Now let $X_{\dot}$ be a \emph{simplicial} space and $X_{\dot}^{\mathrm{str}}$ the associated strict simplicial space. Obviously, the natural functor $U: \mathrm{AbSh}(X_{\dot}) \to \mathrm{AbSh}(X_{\dot}^{\mathrm{str}})$ is exact. If $\mathscr I$ is an injective abelian sheaf on $X_{\dot}$, each $I_p$ on $X_p$ is injective and the cochain complex associated with $[p] \mapsto \mathscr I(X_p)$ is acyclic (cf. the remark after proposition 2.4 in \cite{Fri}). By the  arguments of that remark, the same is true for $U(\mathscr I)$ and hence $U(\mathscr I)$ is an acyclic sheaf on $X_{\dot}^{\mathrm{str}}$. Since $\Gamma(X_{\dot}, \,.\,) = \Gamma(X_{\dot}^{\mathrm{str}}, U(\,.\,))$, we get, that the natural map $H^*(X_{\dot}, \mathscr F) \to H^*(X_{\dot}^{\mathrm{str}}, U(\mathscr F))$ is an isomorphism for all abelian sheaves $\mathscr F$ on $X_{\dot}$.

\section{Simplicial groups}
\label{app:SimplGrps}

For any group $G$, we define the simplicial sets $E_{\dot}G$ and $B_{\dot}G$ as in definition \ref{def:ClassSpace}.
Now let $G_{\dot}$ be a simplicial group. We define $E_{\dot}G_{\dot}$ to be the diagonal of the bisimplicial set $([p],[q])\mapsto E_pG_q$. Note, that $E_{\dot}G_{\dot}$ is itself a simplicial group, the multiplication being defined component-wise, and in particular is a Kan set \cite[Kap. I, Folgerung 9.6]{Lamotke}.

\begin{lemma}\label{lem:EGcontractible}
$E_{\dot}G_{\dot}$ is contractible.
\end{lemma}
\begin{proof}
For $i=0, \dots, p$ define $h_i: E_pG_p \to E_{p+1}G_{p+1}$ by $(g_0 \dots, g_p) \mapsto (s_i(g_0), \dots, s_i(g_i), 1, \dots, 1)$. It is easy to see, that these define a simplicial homotopy between the constant map $1$ and the identity in the sense of \cite[\S 5]{May}.
\end{proof}

We define $B_{\dot}G_{\dot}$ to be the diagonal of the bisimplicial set $([p],[q])\mapsto B_pG_q$. Recall, that $G_{\dot}$ acts from the right on $E_{\dot}G_{\dot}$ and the map 
\[
E_{\dot}G_{\dot} \to B_{\dot}G_{\dot}, (g_0, \dots, g_p) \mapsto (g_0g_1^{-1}, \dots, g_{p-1}g_p^{-1})
\]
induces an isomorphism $E_{\dot}G_{\dot}/G_{\dot} \xrightarrow{\cong} B_{\dot}G_{\dot}$. By \cite[Kap. I, Satz 9.5]{Lamotke} $E_{\dot}G_{\dot}\to E_{\dot}G_{\dot}/G_{\dot} \cong B_{\dot}G_{\dot}$ is a Kan fibration, hence $B_{\dot}G_{\dot}$ is a Kan set by {\it loc. cit.} Kap. I, Folgerung 6.3.
From the contractibility of $E_{\dot}G_{\dot}$ and the long exact sequence of the fibration $E_{\dot}G_{\dot}\to B_{\dot}G_{\dot}$ we get isomorphisms 
\[
\pi_i(B_{\dot}G_{\dot}) \xrightarrow{\cong} \pi_{i-1}(G_{\dot}), \quad i \geq 1,
\]
where all simplicial sets occuring are equipped with the natural base point $1$.
The same conclusion holds for $B_{\dot}G_{\dot}$ replaced by the quotient of any contractible Kan set by a free $G_{\dot}$-operation, and we will call any such simplicial set a ``classifying space'' for $G_{\dot}$.

Next let $G=G_0$ considered as a constant simplicial subgroup of $G_{\dot}$. 
We want to study the homotopy fibre of the natural map $B_{\dot}G \to B_{\dot}G_{\dot}$.
With $G_{\dot}$ also $G$ operates freely from the right on $E_{\dot}G_{\dot}$, hence $E_{\dot}G_{\dot} \to E_{\dot}G_{\dot}/G$ is a Kan fibration and $E_{\dot}G_{\dot}/G$ is a model for the classifying space of $G$.
\begin{lemma}
The projection $E_{\dot}G_{\dot}/G \to E_{\dot}G_{\dot}/G_{\dot}$ is a Kan fibration.
\end{lemma}
\begin{proof}
Here we use the fact, that simplicial sets form a model category with fibrations the Kan fibrations, cofibrations the monomorphisms and weak equivalences the maps, which become weak equivalences after geometric realization \cite[Theorem I.11.3]{GJ}.
A Kan fibration is by definition a morphism satisfying the right lifting property with respect to the (trivial) cofibrations $\Lambda^n_k \tilde\hookrightarrow \Delta^n$, where $\Lambda^n_k$ is the $k$-th $n$-horn ({\it loc. cit.} p. 6), and (as in any model category) has the right lifting property with respect to all trivial cofibrations. Choose any vertex $* \in \Lambda^n_k$. Then $* \hookrightarrow \Lambda^n_k$ is a trivial cofibration. 
Now the claim follows as pictured in the following diagram:
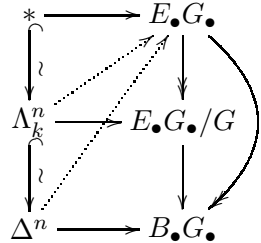
\[\xymatrix{
\ast \ar[r] \ar@{^{(}->}[d]^{\wr} & E_{\dot}G_{\dot}\ar@{->>}[d] \ar@{->>}@/^1cm/[dd]\\
\Lambda^n_k \ar[r]\ar@{^{(}->}[d]^{\wr} \ar@{.>}[ur]& E_{\dot}G_{\dot}/G\ar[d]\\
\Delta^n \ar[r] \ar@{.>}[uur]& B_{\dot}G_{\dot}
}
\]
The upper dotted arrow exists, since $E_{\dot}G_{\dot}\to E_{\dot}G_{\dot}/G$ is a fibration, and then the lower dotted arrow exists, since $E_{\dot}G_{\dot}\to E_{\dot}G_{\dot}/G_{\dot}$ is a fibration.
\end{proof}
In the following we use the standard notations to denote fibrations ($\twoheadrightarrow$), cofibrations ($\hookrightarrow$) and weak equivalences ($\xrightarrow{\sim}$).
There is a natural inclusion $B_{\dot}G \cong E_{\dot}G/G \hookrightarrow E_{\dot}G_{\dot}/G$, which is a weak equivalence as follows from the following morphism of fibre sequences
\[
\xymatrix{
G \ar@{^{(}->}[r] & E_{\dot}G_{\dot} \ar@{->>}[r] & E_{\dot}G_{\dot}/G\\
G \ar@{^{(}->}[r] \ar@{=}[u] & E_{\dot}G \ar@{^{(}->}[u]^{\sim} \ar@{->>}[r] & E_{\dot}G/G. \ar@{^{(}->}[u]
}
\]
Next we have a commutative diagram
\[
\xymatrix{
E_{\dot}G_{\dot}/G \ar@{->>}[dr]\\
B_{\dot}G \cong E_{\dot}G/G \ar@{^{(}->}[u]^{\sim} \ar[r] & E_{\dot}G_{\dot}/G_{\dot} \cong B_{\dot}G_{\dot},
}
\]
and hence the homotopy fibre of $B_{\dot}G \to B_{\dot}G_{\dot}$ can be taken to be the fibre of the map $E_{\dot}G_{\dot}/G \twoheadrightarrow E_{\dot}G_{\dot}/G_{\dot}$, that is $G_{\dot}/G$. Here $G_{\dot}/G$ is embedded in $E_{\dot}G_{\dot}/G$ diagonally ($[\sigma] \mapsto [(\sigma, \dots, \sigma)]$). Pulling back the fibration $E_{\dot}G_{\dot}/G \twoheadrightarrow E_{\dot}G_{\dot}/G_{\dot} \cong B_{\dot}G_{\dot}$ along $E_{\dot}G_{\dot} \twoheadrightarrow B_{\dot}G_{\dot}$, we get the fibration $E_{\dot}G_{\dot}/G \times_{B_{\dot}G_{\dot}} E_{\dot}G_{\dot} \twoheadrightarrow E_{\dot}G_{\dot}$. Since the base of this fibration is contractible, the inclusion of the fibre $G_{\dot}/G \hookrightarrow E_{\dot}G_{\dot}/G \times_{B_{\dot}G_{\dot}} E_{\dot}G_{\dot}$, which is given by $[\sigma] \mapsto ([(\sigma, \dots, \sigma)],(1, \dots, 1))$, is a weak equivalence.

Next, the inclusion $E_{\dot}G \hookrightarrow E_{\dot}G_{\dot}$ induces a weak equivalence 
$B_{\dot}G \times_{B_{\dot}G_{\dot}} E_{\dot}G_{\dot} \xrightarrow{\sim} E_{\dot}G_{\dot}/G \times_{B_{\dot}G_{\dot}} E_{\dot}G_{\dot}$
as follows from the morphism of fibre sequences
\[
\xymatrix{
G_{\dot} \ar@{^{(}->}[r] & E_{\dot}G_{\dot}/G \times_{B_{\dot}G_{\dot}} E_{\dot}G_{\dot} \ar@{->>}[r] & E_{\dot}G_{\dot}/G\\
G_{\dot} \ar@{^{(}->}[r] \ar@{=}[u] &  B_{\dot}G \times_{B_{\dot}G_{\dot}} E_{\dot}G_{\dot} \ar[u] \ar@{->>}[r] & B_{\dot}G. \ar@{^{(}->}[u]^{\sim}
}
\]
Hence we have weak equivalences 
\[
G_{\dot}/G \trivcof E_{\dot}G_{\dot}/G \times_{B_{\dot}G_{\dot}} E_{\dot}G_{\dot} \xleftarrow{\sim} B_{\dot}G \times_{B_{\dot}G_{\dot}} E_{\dot}G_{\dot}.
\]
For $\sigma \in G_p$ write $\sigma(e_i) := \tau_i^*\sigma \in G_0 = G$, with $\tau_i : [0]\to [p], 0\mapsto i$, $i = 0, \dots, p$. Then we define a map $\eta: G_{\dot}/G \to B_{\dot}G \times_{B_{\dot}G_{\dot}} E_{\dot}G_{\dot}$ by $\sigma \mapsto ((\sigma(e_0)\sigma(e_1)^{-1}, \dots, \sigma(e_{p-1})\sigma(e_p)^{-1}), (\sigma(e_0)\sigma^{-1}, \dots, \sigma(e_p)\sigma^{-1}))$.

\begin{lemma}\label{lem:HomotopyFibres}
The diagram
\[
\xymatrix{
G_{\dot}/G \ar@{^{(}->}[r]^-{\sim} & E_{\dot}G_{\dot}/G \times_{B_{\dot}G_{\dot}} E_{\dot}G_{\dot}\\
G_{\dot}/G \ar@{=}[u] \ar[r]^-{\eta} & B_{\dot}G \times_{B_{\dot}G_{\dot}} E_{\dot}G_{\dot} \ar@{^{(}->}[u]^{\sim}
}
\]
is homotopy commutative. In particular, $\eta$ is a weak equivalence, too.
\end{lemma}
\begin{proof} (Cf. \cite[proof of proposition 6.16]{Kar87}) Define $\chi : G_{\dot}/G \to E_{\dot}G_{\dot}$ by $\sigma \mapsto (\sigma\sigma(e_0)^{-1}, \dots, \sigma\sigma(e_p)^{-1})$. Since $E_{\dot}G_{\dot}$ is contractible, $\chi$ is homotopic to the constant map $1$. Recall, that $E_{\dot}G_{\dot}$ is a simplicial group, which acts from the left on $E_{\dot}G_{\dot}/G$ and on $E_{\dot}G_{\dot}/G_{\dot}$ and, since all projections are equivariant for this action, also on $E_{\dot}G_{\dot}/G \times_{B_{\dot}G_{\dot}} E_{\dot}G_{\dot}$. 

The composition $G_{\dot}/G \xrightarrow{\eta} B_{\dot}G \times_{B_{\dot}G_{\dot}} E_{\dot}G_{\dot} \trivcof E_{\dot}G_{\dot}/G \times_{B_{\dot}G_{\dot}} E_{\dot}G_{\dot}$ is given by $\sigma\mapsto ([(\sigma(e_0), \dots, \sigma(e_p))], \sigma(e_0)\sigma^{-1}, \dots, \sigma(e_p)\sigma^{-1})$. Multiplying this composition from the left with $\chi$, we get the homotopic map $\sigma \mapsto ([(\sigma, \dots, \sigma)], 1, \dots, 1)$, which is precisely the upper horizontal map in the diagram.
\end{proof}

\bibliography{bib}
\bibliographystyle{smfalpha}

\end{document}